\documentclass[11pt, sort&compress]{elsarticle} % Do USE A4  option; 10-point type preferred

\usepackage{tikz}
\usetikzlibrary{arrows}
\usetikzlibrary{shapes}
\usetikzlibrary{positioning}

\usepackage{fullpage}
\usepackage{amsfonts} 
\usepackage{graphics}

\usepackage{mathtools} % pour symbole \bigtimes

\allowdisplaybreaks[4] 

\usepackage{refcheck} %pour infos sur references/citation utilisees ou pas
\norefnames	%pas de labels en marge
\nocitenames 	%pas de ref biblio en marge
\ignoreunlbld %pas de test label equations
\setoffmsgs %pas de msg du package en log

\usepackage{arydshln} % dotted vertical line in tabular

\usepackage{url}      % ecriture url dans refs

\usepackage{graphicx}
\usepackage{amsmath, amsthm, amssymb}
\usepackage{amssymb,amsmath,graphics,color}

\usepackage{epsfig}
\usepackage{epstopdf}

\usepackage{multirow}

\usepackage{tocloft}
\renewcommand{\cftsecfont}{}
\usepackage{xr} % pouvoir referencer a autre document

\usepackage{xcolor}

\def\red #1{\textcolor{red}{#1}}

\definecolor{darkbrown}{rgb}{.5,.1,.1} 
\usepackage{here} %pour forcer latex a placer image ou on veut avec option [H]

\usepackage{relsize} %augmetner/reduire taille en mode math

\renewcommand{\today}
{%
\number \day 
\ifnum  \day =1 st
\else
	\ifnum  \day =2 nd
	\else
		\ifnum  \day =3 rd
		\else
			\ifnum  \day =21 st
			\else
				\ifnum  \day =22 nd
				\else
					\ifnum  \day =23 rd
					\else
						\ifnum  \day =31 st
						\else
							th
						\fi
					\fi
				\fi
			\fi
		\fi
	\fi
\fi
\ifcase \month \or January\or February\or March\or April\or May \or June\or July\or August\or September\or October\or November\or December\fi 
{ \number \year}%
} 

\def\boxit [#1]#2{\vbox{\hrule\hbox{\vrule
     \vbox spread #1{\vss\hbox spread#1{\hss #2\hss}\vss}%
        \vrule}\hrule}}

\newbox\algo

\font\algofont=cmtt10 scaled 1100
\font\smallalgofont=cmtt10 scaled 950

\newenvironment {boxedalgorithme}{\smallskip \smallskip
\setbox\algo=\vbox\bgroup\algofont\parindent= 0mm}{\egroup
\boxit[5pt]{\box\algo} \smallskip\smallskip}

\newenvironment {algorithme}{\smallskip \smallskip
\bgroup\algofont\parindent= 0mm}{\egroup
\smallskip\smallskip}

\newtheorem{thm}{Theorem}[section]
\newtheorem{cor}[thm]{Corollary} % share numbers with theorems
\newtheorem{prop}[thm]{Proposition} % share numbers with theorems
\newtheorem{lemma}[thm]{Lemma}
\newtheorem{remark}[thm]{Remark}
\newtheorem{observation}[thm]{Observation}
\newtheorem{definition}[thm]{Definition}
\newtheorem{property}[thm]{Property}

\newtheorem{hors-property}[subsubsection]{Property}
\newtheorem{hors-lemma}[subsubsection]{Lemma}
\newtheorem{hors-remark}[subsubsection]{Remark}
\newtheorem{hors-definition}[subsubsection]{Definition}
\def\newline{\hfill\break}

\def\newpage{\vfill\break}

\def\Int{\mathrm{Int}}

\def\Ext{\mathrm{Ext}}

\def\min{\mathrm{min}}
\def\Min{\mathrm{min}}

\def\max{\mathrm{max}}

\def\ass{\mathrm{Part}}

\def\ep{\varepsilon}
\def\io{\iota}

\def\n{\omega}

\def\ov{\overrightarrow}
\def\s{\setminus}
\def\bk{\backslash}
\def\ovl{\overline}

\def\tiret{--- }
\def\bs{\bigskip}
\def\ms{\medskip}
\def\ss{\smallskip}
\def\n{\omega}

\def\choice{{\algofont CHOICE }}

\def\finsi{}

\def\plus{\uplus} % union disjointe

\font\ptirm=cmr10 scaled 900

\def\bul{$\bullet$ }

\def\G{{\ov G}}

\def\AA{\mathrm{acl}}

\def\bysame{\rule{1cm}{0.1mm}}

\newcommand{\overbar}[1]{\mkern 1.5mu\overline{\mkern-1.5mu#1\mkern-1.5mu}\mkern 1.5mu}
\def\bar{\overbar}    %redefinition

\def\eme#1{}
\def\emeref#1{}
\def\emevder#1{}

\title{The active bijection for graphs\tnoteref{t1}}
\tnotetext[t1]{Preprint (submitted). This version: {\today}.}
\author[lirmm]{Emeric Gioan\corref{cor1}}
\ead{emeric.gioan@lirmm.fr}

\author[paris]{Michel Las Vergnas\corref{cor2}}

\cortext[cor1]{Email: emeric.gioan@lirmm.fr}
\cortext[cor2]{R.I.P.}
\address[lirmm]{CNRS, LIRMM, Universit\'e de Montpellier,  France\emevder{respecte-je la consigne ? c'est : LIRMM, Universit\'e de Montpellier, CNRS, Montpellier, France}}
\address[paris]{CNRS, Paris, France}

\begin{document}

\eme{dessous en commentaire abstract plus long}%

\begin{abstract} % Short abstract
\noindent
\small
%\red{fob section indepnendant... could have been other paper}
The active bijection %in oriented matroids %on a linearly ordered ground set 
%the short name for 
%a set of results
forms a package of results
studied
%introduced and investigated  
by the authors in a series of papers in %hyperplane arrangements and 
oriented matroids. 
%
%All these results generalize 
%to (oriented) matroids but 
The present paper is intended to state the main results 
in the particular case, and more widespread language, of graphs.
%%
%In contrast, the present paper %, in contrast with other papers of the series, 
%uses the graph language and 
%is intended to give the main results for a reader primarily interested in graph~theory.
%%
%%The particular case of graphs 
%%%on a linearly ordered set of edges 
%%is interesting on its own as it 
%The particular case of graphs
%provides simplified constructions, or specific properties, or simply a more widespread framework. 
%%The present paper provides the full construction from orientations to spanning trees.
%
%The central achievement is to 
We associate 
any directed graph, defined on a linearly ordered set of edges, with one particular of its spanning trees, which we call its \emph{active spanning tree}. 
For any graph on a linearly ordered set of edges, this yields a surjective mapping from orientations onto spanning trees,
which preserves activities (for orientations in the sense of Las Vergnas, for spanning trees in the sense of Tutte), as well as some 
%active 
partitions (or filtrations) of the edge set associated with orientations and spanning trees.  It yields a canonical bijection between classes of orientations and spanning trees, as well as a refined bijection between all orientations and edge subsets, containing various noticeable bijections, for instance: between orientations in which smallest edges of cycles and cocycles have a fixed orientation and spanning trees; or between acyclic orientations and no-broken-circuit subsets. 
%A first paper on this case has been published [\emph{Activity preserving bijections between spanning trees and
%  orientations in graphs}, Disc. Math. \textbf{298} (2005), 169--188].
% The present paper %updates and 
% completes this paper,  and provides %notably 
% the full construction from orientations to spanning trees.
% , whereas the previous paper relied on the inverse construction. 
% It %also
%gives either new results or results that were underlying the previous paper but not explicitly given.
%\red{canonical ?}
Several constructions of independent interest are involved.
The basic case concerns bipolar orientations, which are in bijection with their fully optimal spanning trees, 
%as proved in a previous paper. 
as proved in a previous paper,
and as computed in a companion paper.
%and as addressed from the computational viewpoint in a companion paper.
%The problem of computing this bijection is addressed in a companion paper.
We give a canonical decomposition of  a directed graph on a linearly ordered set of edges into acyclic/cyclic bipolar directed graphs. 
 Considering all orientations of a graph, we obtain
% a general decomposition theorem in terms of particular sequences of 2-connected minors  and, numerically, 
 an expression of the Tutte polynomial 
 %using only 
 in terms of
 products of beta invariants of minors, a remarkable partition of the set of orientations into activity classes, and a simple expression of the Tutte polynomial using four orientation activity parameters.
We derive a similar decomposition theorem for spanning trees. 
%We combine the previous constructions with a bijection in the bipolar case 
%%from a previous paper in order 
%%to obtain  the canonical and refined active bijections.
%to obtain  the above bijections.
%%,and we mention a simple linear inverse construction from spanning trees to orientations.
%%We~provide deletion/contraction constructions of these %(canonical) bijections, and  a general deletion/contraction framework for various activity preserving correspondences.
We also provide a general deletion/contraction  framework for these bijections and relatives.

\end{abstract}

\maketitle   % This line must be here, following the abstract
% Note that the Plain TeX commands \over and \atop do not work in
% AMS-LaTeX; fractions must be specified as \frac{1}{2} and 
% binomial coefficients as \binom{n}{k}

%\bigskip
%\noindent{\it \textbf{Nota Bene.} This preprint will soon be put on arXiv and submitted, up to a few possible corrections. Please prefer the arXiv version when it is available.}

%%%%%%%%%%%%%%%%%%%%%%%%%%%%%%%%%%%%%%%%%%%%%%%%%%%%%%%%%%%%%%%%%%
% TABLE OF CONTENTS

%\newpage

%\hrule %\ss
\vspace{-7mm}

\setcounter{tocdepth}{2}  %default
      \renewcommand*\contentsname{\small Summary\vspace{-3mm}}
  \renewcommand\cftsecfont{\small\rm}
    \renewcommand\cftsubsecfont{\small\rm}
 \renewcommand{\cftsecpagefont}{\small\normalfont}
 \pagenumbering{arabic}
    \setlength{\cftbeforesecskip}{0cm}
    \setlength{\cftparskip}{0cm}

%\newpage
\tableofcontents

\bs
\hrule

\setcounter{page}{2}

%\ms
%\hrule
%%%%%%%%%%%%%%%%%%%%%%%%%%%%%%%%%%%%%%%%%%%%%%%%%%%%%%%%%%%%%%%

%\normalsize
%%%%%%%%%%%%%%%%%%%%%%%%%%%%%%%%%%%%%%%%%%%%%%%%%%%%%%%%%%%%%%%%%%
%\vspace{-2mm}

\emevder{faire corrections d'anglais, comme dans chapter, voir fichier anglais.tex  passage "SUR CHAPTER"}%

\emevder{inserer preuve de covolution formula par orientations?}

\section{Introduction}
\label{sec:intro}

\emevder{*** HYPER IMPORTANT  *** remplacer 'greater than' par 'greater than or equal to' quand necessaire *****}

\emevder{penser a rejeter un oeil aux vieilles macros eme a la fin}

\emevder{refs a chapters!}%

\emevder{comparer intros de ABG2 et AB2-b}%

\emevder{??? REMPLACER LE PLUS POSSIBLE on a lienarly ordered edge set PAR ordered ? MEME DANS INTRO ?}%
%\section{Overview on the active bijection in graphs}
%
%\subsection{Introduction}
%\label{sub:intro}

%\red{ajotuer ref a LV ancienne correspondance pour decomposition des activites annoncee}

%EUROCOMB 09 :

%The active bijection in graphs, hyperplane arrangements and oriented matroids on a linearly ordered %ground 
%set is the subject of several papers by the present authors%
%, e.g. \cite{GiLV05}...\cite{AB3}.
%%%%% \cite{GiLV05, GiLV06, GiLV07, AB1, AB3}.
%It defines an activity preserving correspondence between orientations and spanning trees of a graph resp. reorientations and bases of an oriented matroid,  with various related bijections and characterizations. 
%This paper studies general directed graphs on a linearly ordered set of edges, in terms of structural properties, constructions,  enumerative properties, invariants and bijections.

%\vspace{-0.5mm}

The general setting of this paper is to relate orientations and spanning trees of graphs, and to study graphs on a linearly ordered set of edges,
% (which we may call ordered graphs for short), 
in terms of structural properties, fundamental constructions, decompositions,  enumerative properties, %invariants 
and bijections.
The original motivation was to provide a bijective interpretation
and a structural understanding
of the equality of two classical expressions of the Tutte polynomial, one in terms of spanning tree activities by Tutte~\cite{Tu54}:

\centerline{$\displaystyle t(G;x,y)=\sum_{\io,\ep}t_{\io,\ep}x^\io y^\ep$}
\noindent where $t_{\io,\ep}$ is the number of spanning trees of the graph $G$
with internal activity $\io$ and external activity $\ep$,
 the other in terms of orientation activities by Las Vergnas \cite{LV84a}:
 
 \centerline{$\displaystyle t(G;x,y)=\sum_{\io,\ep}o_{\io,\ep}\ \Bigl({x\over 2}\Bigr)^\io \ \Bigl({y\over 2}\Bigr)^\ep$}
\noindent where $o_{\io,\ep}$ is the number of orientations of $G$ with  dual-activity $\io$ and activity $\ep$,
which contains various famous enumerative results from the literature, such as counting  acyclic reorientations (and more generally regions in hyperplane arrangements and oriented matroids), e.g.
%\cite{Wi66}\cite{St73}\cite{Za75}\cite{GrZa83}.
\cite{Wi66, St73, Za75, GrZa83}.

\ss

Our solution is made of several constructions and results of independent interest, whose central achievement 
%is the definition of a canonical mapping 
%of a general  kind, denoted $\alpha$, which associates any given directed graph on a linearly ordered set of edges  with one particular of its  spanning trees.
is to associate, in a canonical way,
any directed graph $\G$ defined on a linearly ordered set of edges with one particular of its spanning trees, denoted $\alpha(\G)$, which we call \emph{the active spanning tree of $\G$}. 
For any graph on a linearly ordered set of edges, this yields a surjective mapping from orientations onto spanning trees,
%which preserves activities (for orientations in the sense of Las Vergnas, for spanning trees in the sense of Tutte),
which preserves the above activities, and
%for a given (undirected) graph, it yields a surjection from orientations to spanning trees of this graph, 
such that exactly $2^{\io+\ep}$ orientations with orientation activities $(\io,\ep)$  are associated with the same spanning tree with spanning tree  activities $(\io,\ep)$. This yields a canonical bijection between remarkable orientation classes and spanning trees (depending only on the ordered undirected graph), along with 
%various related results such as %Tutte polynomial expressions, and 
a naturally refined bijection between orientations and edge-subsets (depending in addition on any fixed reference orientation). % obtained by Crapo's interval decomposition.
%\red{$\io+\ep=1$}
%Overall, this forms a three-level construction, involving at the same time various results of independent interest.

\ss

%\red{\it From oriented matroids to graphs.}
Before turning into the graph language, 
%\red{abreger} for the sake of briefly presenting the general setting of the active bijection and the specificities of the graph language,
 let us give to the reader one of the shortest possible complete definition of the active basis in the general setting of oriented matroids.
 % of the general canonical active mapping $\alpha$ of ordered oriented matroids. 
%
For any %every
oriented matroid $M$ on a linearly ordered set $E$, 
%the image $\alpha(M)$ (which turns out to be a basis of $M$) satisfies the three following properties:
\emph{the active basis $\alpha(M)$ of $M$} is  determined~by:%
\vspace{-2mm}
\begin{itemize}% [(i) $\bullet$]
\itemsep=-0.5mm

%\itemlabel{it1-def-alpha}{(F.o.b.)}
\item %[--] %[\rm \itemref{it1-def-alpha}]
%\label{it1-def-alpha}
%\itemlabel{it1-def-alpha}{F.o.b.}
\emph{Fully optimal basis of a bounded region.} If $M$ is 
%bounded with respect to $p=\min(E)$ 
acyclic and every positive cocircuit of $M$ contains $\min(E)$, then $\alpha(M)$ is the unique %(so-called fully optimal) 
basis $B$ of $M$ such~that:

%\smallskip
\begin{itemize}
\vspace{-2mm}
\item %[--] 
 for all $b\in B\setminus p$, the signs of $b$ and $\min(C^*(B;b))$ are opposite in $C^*(B;b)$;
\vspace{-1mm}

\item %[--]   
for all $e\in E\s B$, the signs of $e$ and $\min(C(B;e))$ are opposite in $C(B;e)$.
\end{itemize}
\vspace{-2mm}

\item %[--] %[\rm \itemref{it2-def-alpha}]
\emph{Duality.} 
$\alpha(M)=E\s \alpha(M^*).$

\item %[--] 
\emph{Decomposition.} 
$\alpha(M)=\alpha(M/ F)\ \uplus\ \alpha(M(F))$
where $F$ is the union of all positive circuits of $M$ whose smallest element is the greatest possible smallest element of a positive circuit of~$M$.%
\end{itemize}
%\ss

\vspace{-1mm}

This definition, developed in \cite{Gi02, AB1,AB2-a,AB2-b}, applies to directed graphs: for a directed graph $\G$ on a linearly ordered set of edges $E$, we have $\alpha(\G)=\alpha(M(\G))$ where $M(\G)$ is the usual oriented matroid on $E$ associated with $\G$.
In Section \ref{sec:edges} of the present paper, we directly define $\alpha(\G)$ in terms of graphs.
However, a specificity of graphs is their lack of duality, which implies that the definitions have to be adapted.
Throughout the paper, in comparison with \cite{AB1,AB2-a,AB2-b}, the fact that a graph does not have a dual graph forces us to repeat some definitions and some proofs, first from the primal viewpoint and second from the dual viewpoint, which is usually a simple translation using cycle/cocycle duality, whereas, in (oriented) matroids, definitions and proofs can be shortened by applying them directly  to the dual. 
Other specificities of the graph case will be mentioned later.
\ms

What we call \emph{the active bijection} is actually a three-level  construction, summarized in the diagram of Figure \ref{fig:diagram}.
It is based on several results of independent interest, including various Tutte polynomial expressions as shown in this diagram,  yielding various bijections listed in Table \ref{table:intro}, and forming a consistent whole.
Let us describe all this along with the organization of the paper.
In what follows, $G$ is a graph on a linearly ordered set of edges, also called ordered graph for short.
\ms

%\hrule

\def\Gref{{\G_{\hbox{\small ref}}}}
\def\Gref{{\G}}

\def\hdistance{10cm}
\def\vdistance{6.5cm}

\def\hsdistance{3.5cm}
\def\hsldistance{4.2cm}
\def\hsmdistance{4cm}
\def\vsdistance{1.5cm}

\def\diagrdistance{2.6cm}
\def\diagldistance{2.5cm}

%\begin{figure}[H]
\begin{figure}
\centering

%\hspace{-1.5cm}
\scalebox{0.75}{
\begin{tikzpicture}%
%[->,>=triangle 90, thick,shorten >=1pt,auto,node distance=\hdistance,   thick,main node/.style={rectangle,fill=blue!10,draw,font=\sffamily\large}]
[->,>=triangle 90, thick,shorten >=1pt,auto, node distance=\hdistance,  thick,  
main node/.style={rectangle,fill=blue!10,draw,font=\sffamily\large},
   TP node/.style={rectangle,dotted,draw,font=\sffamily\it\tiny}]
   
  \node[main node] (1a) {orientation $-_A\Gref$};
  \node[main node] (1b) [right of=1a] {subset $\alpha_\Gref(A)$};
%  \node[main node] (2a) [below of=1a, node distance=\vdistance] {\begin{tabular}{c}(activity class of) \\ orientation $\G$\\6 (up to activity class)\end{tabular}};
   \node[main node] (2a) [below of=1a, node distance=\vdistance] {\begin{tabular}{c}orientation $\G$\\ (activity class)\end{tabular}};
  \node[main node] (2b) [right of=2a] {spanning tree $\alpha(\G)$};
  \node[main node] (3a) [below of=2a, node distance=\vdistance] {\begin{tabular}{c} acyclic/cyclic bipolar \\ orientation $\G$ \\ {\small(up to opposite)}\end{tabular}};
  \node[main node] (3b) [right of=3a] {\begin{tabular}{c}internal/external uniactive \\ spanning tree $\alpha(\G)$\end{tabular}};
  
    \node[TP node] (TP1a) [left of=1a, node distance=\hsldistance] {\begin{tabular}{c} Tutte polynomial\\ in terms of \\ \scriptsize refined orientation activities\\ \tiny (Theorem \ref{EG:th:expansion-orientations})\end{tabular}};
  \node[TP node] (TP1b) [right of=1b, node distance=\hsdistance] {\begin{tabular}{c} Tutte polynomial\\ in terms of \\ \scriptsize subset activities \\ \tiny (Theorem \ref{EG:th:Tutte-4-variables})\end{tabular}};
  \node[TP node] (TP2a) [left of=2a, node distance=\hsldistance] {\begin{tabular}{c} Tutte polynomial\\ in terms of \\ \scriptsize orientation activities\\ \tiny (Section \ref{subsec:orientation-activity})\end{tabular}};
  \node[TP node] (TP2b) [right of=2b, node distance=\hsdistance] {\begin{tabular}{c} Tutte polynomial\\ in terms of \\ \scriptsize spanning tree \\ \scriptsize activities\\ \tiny (Section \ref{subsec:prelim-sp-trees})\end{tabular}};
  \node[TP node] (TP3a) [left of=3a, node distance=\hsldistance] {\begin{tabular}{c} \scriptsize $\beta$ invariant\\ \tiny (Section \ref{subsec:prelim-beta})\end{tabular}};
  \node[TP node] (TP3b) [right of=3b, node distance=\hsmdistance] {\begin{tabular}{c} \scriptsize $\beta$ invariant\\ \tiny (Section \ref{subsec:prelim-beta})\end{tabular}};
  \node[TP node] (TP4a) [above right of=TP3a, node distance=\diagrdistance]  {\begin{tabular}{c} Tutte polynomial in terms of \\ \scriptsize  $\beta$ invariants of minors \\
(Theorem \ref{th:tutte})\end{tabular}};
  \node[TP node] (TP4b) [above left of=TP3b, node distance=\diagldistance] {\begin{tabular}{c} Tutte polynomial in terms of \\ \scriptsize  $\beta$ invariants of minors \\(Theorem \ref{th:tutte}) \end{tabular}};
  
%\draw[<->,>=latex] (1a) --  (1b);
\path[every node/.style={font=\sffamily}]
	
	(1a) edge [->, >=latex, line width=1.2mm] node [bend right] 
		{\begin{tabular}{c}
			\bf refined active bijection $\alpha_\Gref$
		\end{tabular}} 
	(1b)
	
	(1b) edge [-, >=latex] node [bend right] 
		{\begin{tabular}{c}
			\small (Definition \ref{EG:def:act-bij-ext}, Theorem \ref{EG:th:ext-act-bij})
		\end{tabular}} 
	(1a)
		
	(2a) edge [->, >=latex, line width=1.2mm] node [bend right] 
		{\begin{tabular}{c}
			\bf canonical active bijection $\alpha$
		\end{tabular}}  
	(2b)
	
	(2b) edge [-, >=latex] node [bend right] 
		{\begin{tabular}{c}
			\small (Definitions \ref{def:alpha-ind-decomp} \& \ref{def:alpha-seq-decomp}, Theorem \ref{EG:th:alpha})
		\end{tabular}}  
	(2a)
		
	(3a) edge [->, >=latex, line width=1.2mm] node [bend right] 
%		{\begin{tabular}{c}
%		{\bf	active bijection $\alpha$} \cite{GiLV05}\cite{AB1}
%		\end{tabular}}   
		{\begin{tabular}{c}
%		{\bf \small canonical uniactive	 }\\
%		{\bf active bijection $\alpha$} \cite{GiLV05} \cite{AB1}
		{\bf uniactive	 }\\
		{\bf bijection $\alpha$} \cite{GiLV05, AB1}
		\end{tabular}} 
	(3b)

	(3b) edge [-, >=latex] node [bend left] 
		{\begin{tabular}{c}
			\small (Def. \ref{def:acyc-alpha} \& \ref{def:cyc-alpha1} \& \ref{def:cyc-alpha2}, Th. \ref{thm:bij-10})
		\end{tabular}}   
	(3a)

	(1a) edge [<-, bend left=25] node [bend left] 
		{\begin{tabular}{c}
%			linear single pass algorithm \cite{AB2-b}\\
			single pass algorithm \cite{AB2-b}\\
%			\small (Theorem \ref{th:basori}, proved in \cite{AB2-b})
			\small (Theorem \ref{th:basori})			
		\end{tabular}}    
	(1b)
	
	(2a) edge [<-, bend left=25] node [bend left] 
		{\begin{tabular}{c}
%			linear single pass algorithm \cite{AB2-b}\\
			single pass algorithm \cite{AB2-b}\\
%			\small (Theorem \ref{th:basori}, proved in \cite{AB2-b})
			\small (Theorem \ref{th:basori})
		\end{tabular}} 
	(2b)	
	
	(3a) edge [<-, bend left=25] node [bend left] 
		{\begin{tabular}{c}
%			linear single pass algorithm \cite{GiLV05, AB1}\\
			single pass algorithm \cite{GiLV05, AB1}\\
%			\small (Proposition \ref{prop:alpha-10-inverse}, and \cite{GiLV05}\cite{AB1}) 
			\small (Proposition \ref{prop:alpha-10-inverse}) 			
		\end{tabular}}  
	(3b)

	(1a)	 edge [<->,bend right=40]  node [bend right] 
		{\begin{tabular}{c} 
			deletion/contraction\\ 
			\small (Theorem \ref{thm:ind-gene-refined})
		\end{tabular} } 
	(1b)
	
	(2a)	 edge [<->,bend right=40]  node [bend right] 
		{\begin{tabular}{c} 
			deletion/contraction\\ 
			\small (Theorem \ref{thm:ind-gene})
		\end{tabular} } 
	(2b)
	
	(3a)	 edge [<->,bend right=40]  node [bend right] 
		{\begin{tabular}{c} 
			deletion/contraction\\ 
			\small (Theorem \ref{thm:ind-10})
		\end{tabular} } 
	(3b)

%	(3a)	 edge [->,bend right=90]  node [bend right] 
%		{\begin{tabular}{c} 
%			full optimality algorithm\\ 
%			\small (Theorem \ref{th:fob})
%		\end{tabular} }
%	 (3b)
	
	(3b)	 edge [<-,bend left=75]  node [bend right] 
		{\begin{tabular}{c} 
%			full optimality algorithm \cite[Theorem 4.2]{ABG2LP} %\\ 
			full optimality algorithm \cite[companion paper]{ABG2LP} %\\ 
%			\small (Theorem \ref{th:fob})
		\end{tabular} }
	 (3a)

	(1a) edge [<->] node [left] 
		{\begin{tabular}{c} 
			partition of $2^E$ \\ 
			into activity classes\\
			of orientations\\
			\small (Prop. \ref{prop:act-classes}, Def. \ref{def:act-class})
		\end{tabular}} 
	(2a)
	
	(1b) edge [<->] node [right] 
		{\begin{tabular}{c} 
			partition of $2^E$ \\ 
			into  intervals\\
	%		w.r.t. spanning tree activities\\
			associated with spanning trees\\
			\small (\cite{Cr69, Da81}, see Section \ref{subsec:refined})
		\end{tabular}} 
	(2b)

	(2a) edge [<->] node [left] 
		{\begin{tabular}{c} 
			decomposition  \\ 
			of orientation activities\\
			\small (Def. \ref{def:act-seq-dec}, Theorem \ref{th:dec-ori})
		\end{tabular}} 
	(3a)
	
	(2b) edge [<->] node [right] 
		{\begin{tabular}{c} 
			decomposition \\ 
			of spanning tree activities\\
			\small (Def. \ref{def:sp-tree-dec-seq}, Theorem \ref{EG:th:dec_base})
		\end{tabular}} 
	(3b)
	
	% aretes avec pointilles tutte poly
	
	(1a) edge [-, dotted]  
	(TP1a)
	
	(1b) edge [-, dotted]  
	(TP1b)
	
	(2a) edge [-, dotted]  
	(TP2a)

	(2b) edge [-, dotted]  
	(TP2b)		
	
	(3a) edge [-, dotted]  
	(TP3a)

	(3b) edge [-, dotted]  
	(TP3b)
	
	;		
\end{tikzpicture}
}
\label{fig:diagram}
\caption{Diagram of results and constructions for the active bijection of an ordered graph $G$.
Horizontal arrows indicate in which ways the constructions or definitions apply
(the deletion/contraction constructions can be used to 
build the whole bijections as matchings rather than mappings). Vertical arrows indicate how objects are related. 
Dotted rectangles indicate how the Tutte polynomial is involved or transforms through the constructions. All results quoted in the diagram are proved in terms of graphs in the paper (or the companion paper \cite{ABG2LP}), except Theorem \ref{th:basori} \cite{AB2-b} and Theorem \ref{thm:bij-10} \cite{GiLV05, AB1} (alternative or more general proofs of all results can be found in \cite{AB1, AB2-a, AB2-b, AB3, AB4}). 
%See the caption of Table~\ref{table:intro} for some definitions.
}
\end{figure}
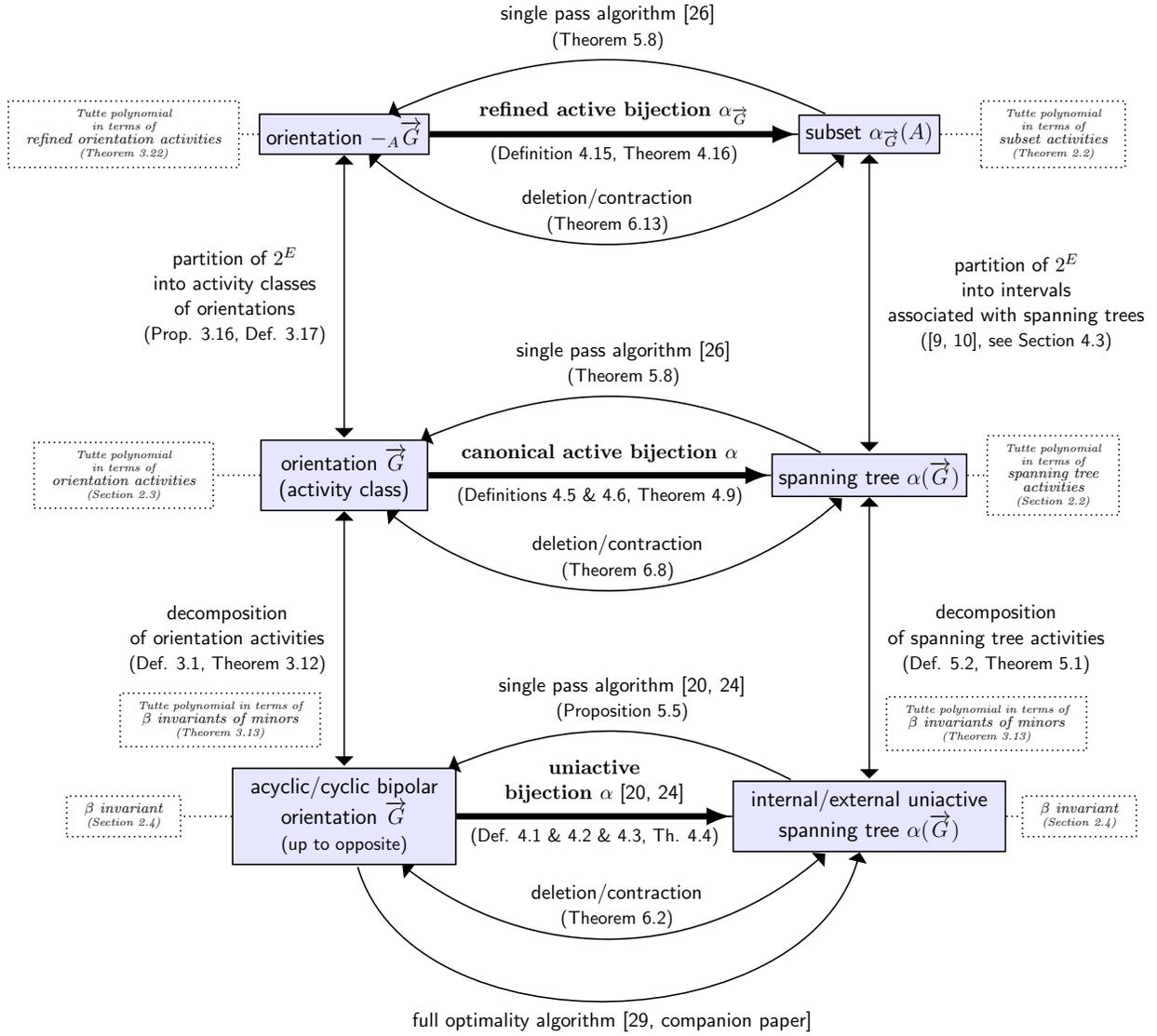

\begin{table}[h]
	\centering
	{\ptirm

\renewcommand{\arraystretch}{1}

\noindent
%\center
%\caption{The active bijection in graphs}
\begin{tabular}{|l|l|c |}

\hline
\multicolumn{1}{|c|}{\bf orientations} & \multicolumn{1}{|c|}{\bf spanning trees/subsets} & \multicolumn{1}{|c|}{} \\
\hline
\multicolumn{3}{|c|}{{\it \small canonical active bijection of an ordered undirected graph}}\\
\hline
 activity classes of orientations  \ 		&  spanning trees 			&  $t(G;1,1)$ \\ 
 activity classes of acyclic orientations	& internal spanning trees			&  $t(G;1,0)$ \\ 
% & & \\
 activity classes of strongly connected orientations   & external  spanning trees			& $t(G;0,1)$ \\ %\bysame  \\ 
 	 			 bipolar orientations {\scriptsize (up to opposite)}	&   uniactive internal spanning trees	  		&  
 	 			{$t_{1,0}$} \\ % \bysame  \\ 
% 	 			{\scriptsize $\beta(G)=t_{1,0}$} \\ % \bysame  \\ 
 	 			 cyclic-bipolar orientations {\scriptsize (up to opposite)}	&   uniactive external spanning trees	  		& 
 	 			 {$t_{0,1}$} \\ 
%  	 			 {\scriptsize $\beta^*(G)=t_{0,1}$  } \\ 

 \hline
 \multicolumn{3}{|c|}{{\it \small refined active bijection w.r.t. a given reference orientation}}\\
 \hline
  				 orientations	& subsets of the edge set							& $t(G;2,2)$ \\ 

orientations with fixed orientation for active edges  & 	forests	&  $t(G;2,1)$ \\ %\bysame  \\ 
orientations with fixed orientation for dual-active edges	& connected spanning subgraphs			&  $t(G;1,2)$ \\ 

%  & & \\
 				 acyclic orientations	& no-broken-circuit subsets			& $t(G;2,0)$ \\ 
% & & \\
 				 strongly connected orientations & {\scriptsize supersets of external spanning trees} & $t(G;0,2)$\\

  \begin{tabular}{@{}l@{}}
orientations with fixed orientation\\
\hskip 1cm  for active edges and dual-active edges 
\end{tabular} 	
		&  spanning trees 			&   $t(G;1,1)$\\ 
 \begin{tabular}{@{}l@{}}
acyclic orientations  with\\
\hskip 1cm  fixed orientation for dual-active edges 
\end{tabular} 	
& internal spanning trees			&  $t(G;1,0)$ \\ 
 \begin{tabular}{@{}l@{}}
  strongly connected orientations   with \\
\hskip 1cm  fixed orientation for active edges
\end{tabular}  
& external  spanning trees			& $t(G;0,1)$ \\ 
%  				{\scriptsize strongly connected orientations with fixed orientation for active edges } & external  spanning trees			& $t(G;0,1)$ \\ 

\hline
 \multicolumn{3}{|c|}{{\it \small particular cases}}\\
 \hline
\emph{\small (for suitable orderings)}
 unique sink acyclic orientations 	& internal  spanning trees			&  \cite[Section 6]{GiLV05}  \\ 

% & & \\
		
%permutations & increasing trees							& \cite{GiLV06} for details, or consider the graph $K_n$ \\ 
\emph{\small (complete graph seen as a chordal graph)}
permutations & increasing trees							& \cite[Section 5]{GiLV06}\\ 
\hline
\end{tabular}
}
%\end{table}
%
\caption{Active bijections and noticeable restrictions. 
%Several activity preserving (and active partition preserving) bijections derived from the canonical active bijection are listed in the following~table (Theorems \ref{EG:th:alpha} and \ref{EG:th:ext-act-bij}).
The first blocks of lines come from Theorems \ref{EG:th:alpha} and \ref{EG:th:ext-act-bij}.  The third column indicates the evaluation (or the coefficient) of the Tutte polynomial that counts the involved objects.
%All these bijections generalize in hyperplane arrangements and oriented matroids, except the one before last, which requires the vertex notion.
%See \cite{GiLV07}\cite{AB2} for more general lists.\red{o.m. a mettre ?}
Activity classes of orientations are obtained by arbitrarily reorienting parts of the active partition/filtration (see Section \ref{sec:tutte}), that is by arbitrarily reorienting unions of all directed cycles or cocycles whose smallest edges are greater than a given edge.
Internal, resp. external, spanning trees are those  whose external activity, internal activity equals 0, and they are uniactive when the other activity equals 1 (see Section \ref{subsec:prelim-sp-trees}).
%The parameter $\beta^*$ equals $\beta$ except for an isthmus or a loop (see Section \ref{subsec:prelim-beta}).
Active/dual-active edges of an ordered directed graph are smallest edges of directed cycles/cocycles (see Section \ref{subsec:orientation-activity}).
An orientation is said to have fixed orientation for some edge it this edge has the same direction as in a given reference orientation of the graph (see Section \ref{subsec:act-classes}).
The two last lines recall particular cases addressed in other papers.
%\red{IMPORTANT POUR DEUX REVIEWERS: defs + beta* (ou alors apres terminology section)}
}
\label{table:intro}
\vspace{-3mm}
\end{table}

At the first level, the \emph{uniactive bijection} of $G$
concerns the case where $\io=1$ and $\ep=0$ in the above setting,
hence the term \emph{uniactive}, which includes also the case where $\io=0$ and $\ep=1$ by some dual construction.
This case is addressed for graphs in \cite{GiLV05, ABG2LP} (see \cite{AB1,AB3} in oriented matroids, or \cite[Section 5]{AB2-b} for a summary\emevder{\cite[Section 5]{AB2-b} verifier section AB2-b}).  Let us sum it up below.
Details are recalled in~Section~\ref{subsec:fob}.
%
%is %obtained by
% the mapping $\G\mapsto \alpha(\G)$ from  bipolar  orientations of $G$, which have a unique fixed source and a unique fixed sink connected by the smallest edge of $E$, to spanning trees with internal activity $1$ and external activity $0$.
%This is the case where $\io=1$ and $\ep=0$ in the above setting, hence the name ` uniactive', which includes also the case where $\io=0$ and $\ep=1$ by some dual construction.

%Briefly,
In \cite{GiLV05} (see also \cite{AB1}), we showed that a bipolar directed graph $\G$ on a linearly ordered set of edges,  with adjacent unique source and sink connected by the smallest edge, has a unique so-called fully optimal spanning tree $\alpha(\G)$ that satisfies a simple criterion on fundamental cycle/cocycle directions
(let us point out that this is a tricky theoretical result, with various interpretations, see the summary \cite[Section 5]{AB2-b}\emevder{\cite[Section 5]{AB2-b} verifier section AB2-b} for details and \cite{ABG2LP} for complexity issues).
%uniactive internal spanning trees, % (fully optimal bases), 
%which is a deep and difficult combinatorial result
 Associating bipolar orientations of $G$ (with fixed orientation for the smallest edge) with their fully optimal spanning trees provides a canonical bijection with spanning trees with internal activity 1 and external activity 0 (called uniactive internal). % in the sense of the Tutte polynomial \cite{Tu54}. 
 It is a classical result from \cite{Za75}, also implied by \cite{LV84a}, that those two sets have the same size,  also known as the $\beta$ invariant of the graph \cite{Cr67}, that is $\beta(G)=t_{1,0}=(1/2).o_{1,0}$.  %

In the companion paper \cite{ABG2LP}, we address the problem of computing the fully optimal spanning tree. 
%Its existence and uniqueness is a tricky combinatorial theorem, 
%As recalled in Section \ref{subsec:fob}, 
The inverse mapping, producing a bipolar orientation for which a given spanning tree is fully optimal, is very easy to compute by a single pass over the ordered set of edges. % (Proposition \ref{prop:alpha-10-inverse}). 
But the direct computation is complicated and it had not been addressed in previous papers. When generalized to real hyperplane arrangements, the problem contains and strengthens the real linear programming problem (as shown in~\cite{AB1}, hence the name \emph{fully optimal}). 
%Hence the mapping can be seen (at least in the general context of oriented matroids) as a sort one way-function.
%
This ``one way function'' feature is a noteworthy aspect of the active bijection.
%Here, we give two constructions to compute the mapping, that is to compute the fully optimal spanning tree of an ordered bipolar digraph. 
In general, we give a direct construction by means of elaborations on linear programming \cite{GiLV09, AB3}, allowing for a polynomial time computation.\emevder{ (in the real case).}
This construction is translated and adapted in the graph case in the companion paper~\cite{ABG2LP}.
\emevder{, implicitly using  the fact that graphs are binary matroids.}
% (i.e. the uniform matroid $U_{2,4}$ is an excluded minor for graphical matroids).
% allows us to simplify the general linear programming type algorithm developed in \cite{GiLV09, AB3} that builds the fully optimal basis of a bounded region. Indeed, the multiobjective programming part comes down to a linear ordering of cocycles induced by a simple weight function. See Section~\ref{sec:fob}.

Finally, from \cite[Section 4]{GiLV05}  (see also \cite[Section 5]{AB1}  or \cite[Section 5]{AB2-b}\emevder{ for a complete overview }\emevder{\cite[Section 5]{AB2-b} verifier section AB2-b}\emevder{--- ATTENTION peut-eter trop de repetition de citation \cite[Section 5]{AB2-b} peut-etre les regrouper ?---}), 
the bijection between bipolar orientations and their fully optimal spanning trees  directly yields a  bijection between orientations obtained from bipolar orientations by reversing %the direction of 
the source-sink edge, namely cyclic-bipolar orientations, and spanning trees with internal activity $0$ and external activity~$1$.
%
%Let us mention that %beyond the scope of the present paper, 
% this construction 
This framework
involves a remarkable duality property, 
the so-called active duality, \emevder{???donner ref result?? donner diagram?}
essentially meaning that those two bijections are related to each other consistently with cycle/cocycle duality (that is oriented matroid duality, which extends planar graph duality, see Section \ref{subsec:fob}). 
Let us mention that this duality property can be also seen as a strengthening of linear programming duality (see \cite[Section 5]{AB1}), and that it is also related to the equivalence of  two dual  formulations in the deletion/contraction construction of the uniactive bijection (see Section \ref{subsec:alpha-10-ind} or the companion paper \cite{ABG2LP}).

%\hrule

\bs

%Let us now get more precisely into the substance and into the organization of the paper and of the constructions. 
%
%%

%\new{The specificities of the graph case are of different types.
%
%In Section...}

%\item It can be built or characterized by several equivalent manners, see Figure \ref{fig:diagram} for an overview.

%\red{\it A three-level construction.}

%\red{second level =mian}

%\red{uniactive ? refs ? $\io+\ep=1$ ?}
At the second level, which is the central achievement of the whole construction and of this paper, we define  the active spanning tree $\alpha(\G)$ 
%of a directed graph $\G$ on a linearly ordered set of edges, 
of an ordered directed graph $\G$, 
from the previous bipolar and cyclic-bipolar cases, by means of some decompositions of orientations and spanning trees.
Then, the \emph{canonical active bijection} is
the bijection between preimages and images of the surjective mapping $\G\mapsto \alpha(\G)$, from orientations of $G$ to spanning trees of $G$, where preimages are characterized as natural equivalence classes in terms of the above decomposition of orientations, called activity classes.
In other words, it is the combination of the uniactive bijection and those two  decompositions for orientations and spanning trees.
It is called canonical because it is built from those three independent canonical constructions, and because it is an intrinsic attribute of the undirected ordered graph $G$ (depending on the ordering but not depending on any orientation of $G$).
These constructions have been shortly defined without proofs in \cite{GiLV05}.
%
%This second level is the main one in this paper .
%
The definition and properties of the canonical active bijection are addressed in Section \ref{subsec:alpha-def-decomp}. The decomposition of orientations and its various implications is addressed in Section \ref{sec:tutte} (those results are generalized to oriented matroids in \cite{AB2-b}).
The decomposition of spanning trees is briefly addressed in Section \ref{sec:spanning-trees}, obtained as corollaries of the previous results (see below). Now, let us precise the section contents.%
%

%In Section \ref{sec:tutte}, we address the decompostion of orientations. It is based on 

%\ref{sec:tutte}, 
In Section \ref{subsec:act-part-decomp}, we define the active partition/filtration of the set of edges of an ordered directed graph, a notion already introduced in \cite{GiLV05} (see \cite{AB2-b} for a geometrical interpretation in oriented matroids, and \cite{Gi18} for a generalization to oriented matroid perspectives and hence, in a sense, to directed graph homomorphisms\emevder{verifier}). We show how to decompose a directed graph on a linearly ordered set of edges into a sequence of minors that are either bipolar or cyclic-bipolar.
This construction refines the usual partition of the edge set into the union of directed cycles (yielding a strongly connected minor) and the union of directed cocycles (yielding an acyclic minor).
We mention that the notion of active partition turns out to generalize a notion of vertex partition which is relevant in \cite{CaFo69,Vi86, Ge01, La01}, 
see Remark \ref{rk:active-partitions-vs-components} 
%in Section \ref{subsec:act-part-decomp}
%see the introduction of Section \ref{sec:tutte}
\emevder{---verifier selon remaniements--- ou faire remarque avec ces refs---ou mettre ce paragraphe ici?---}%
 for more information 
and \cite[Section 7]{GiLV05} for details. 
%
%

% (that is bipolar up to the reversion of the smallest edge, a duality property in fact \cite{GiLV05}).
%Such a decomposition technique was announced in \cite{LV84b} and described in \cite{GiLV05} without proofs.
%Here, we give this construction in full details. 
In Section \ref{subsec:act-part-tutte}, considering all orientations of a graph, and building on a uniqueness property in the previous decomposition, we derive
 a general decomposition theorem for the set of all orientations, in terms of particular sequences of 2-connected minors (Theorem \ref{th:dec-ori}).
The involved sequences of subsets provide a remarkable notion of filtrations for ordered graphs.
%This achieves the detailed construction of the active bijection from orientations to spanning trees:
%one can compute this decomposition and glue together the fully optimal spanning trees associated with those minors in order to get the spanning tree associated with the initial orientation.
%
Enumeratively, this decomposition can be seen as an expression of the Tutte polynomial in terms of products of beta invariants of minors (Theorem \ref{th:tutte}).
This formula refines at the same time the formulas in terms of spanning tree activities \cite{Tu54}, of orientation activities \cite{LV84a}, and the convolution formula \cite{EtLV98, KoReSt99}.
Actually, it can be also seen as the enumerative interpretation of a spanning tree decomposition, see below (and in this context, it is generalized to matroids in \cite{AB2-a}).
%

%This result can be quite directly generalized to oriented matroids.
%
In Section \ref{subsec:act-classes},
we define activity classes of orientations,
obtained by reversing independently all parts in the active partition/filtration.
Activity classes are isomorphic to boolean lattices and form a remarkable partition of the set of orientations.
We show how this directly yields a simple expression of the Tutte polynomial using four orientation activity parameters (Theorem \ref{EG:th:expansion-orientations}), as announced in \cite{LV12}.
This expression is the counterpart for orientations of a similar four parameter formula for subsets/supersets of spanning trees \cite{GoTr90, LV13} (Theorem \ref{EG:th:Tutte-4-variables}).
Furthermore, in each activity class, there is one and only one 
representative orientation with fixed direction for smallest edges of directed cycles or cocycles.
%
%Let us mention in particular a construction specific to graphs: as it involves vertices, 
In particular, as shown in \cite[Section 6]{GiLV05}, given a vertex and a suitable ordering of the edge set 
(when all branches of the smallest spanning tree are increasing from the vertex), 
there is one and only one acyclic 
orientation with this vertex as a unique sink in each activity class of acyclic orientations.
%(for suitable orderings of the edge set), 
%(when all branches of the smallest spanning tree are increasing), .
%
This discussion is continued in Section \ref{subsec:refined} about 
%We will see later that 
the refined active bijection, which %notably 
relates the two above four parameter Tutte polynomial expressions.

%
%Let us mention also that, geometrically, active partitions and activity classes of orientations describe the position of a region of a (real central) hyperplane arrangement with respect to the minimal flag.
%

In Section \ref{subsec:alpha-def-decomp}, 
we define the active spanning tree as explained above, by gluing together the images, by the uniactive bijection of Section \ref{subsec:fob}, of the bipolar and cyclic-bipolar minors of the decomposition of Section \ref{sec:tutte}.
Equivalently, this definition can be formulated in a recursive way, as in the beginning of this introduction.
This yields a canonical bijection between activity classes of orientations and spanning trees   (Theorem \ref{EG:th:alpha}), as shown in Table \ref{table:intro}. Furthermore, this bijection  not only preserves activities and active edges, but also active partitions that one can also %independently 
define for spanning trees, as explained below.
\bs

Section \ref{sec:spanning-trees} has a special status in the paper, as it addresses the constructions from the spanning tree viewpoint, whereas the rest of the paper is focused on the orientation viewpoint.
First, in Section \ref{subsec:dec-bases}, we state counterparts in terms of spanning trees of the aforementioned decomposition of orientations.
The main result is a decomposition theorem for spanning trees of an ordered graph in terms of the same filtrations, or the same particular sequences of minors, as above, into spanning trees with internal/external activities equal to $1/0$ or $0/1$ (Theorem \ref{EG:th:dec_base}). It refines the decomposition into two internal/external parts from \cite{EtLV98}.
As far as proofs are concerned, in this paper, we essentially prove this spanning tree decomposition in Section \ref{subsec:alpha-def-decomp}, at the same time as the canonical active bijection properties, building on the decomposition of orientations.
It could also be defined and proved independently of the rest of constructions,  directly in terms of spanning trees (which is the approach used in \cite{AB2-a} to define these decompositions in matroids).
Here we take advantage of the fact that graphs are orientable (in contrast with matroids: such proofs using orientations are not possible in non-orientable matroids).
% allows us to focus on orientations. From this, first, we can derive the Tutte polynomial formula in terms of beta invaiants of minors (see Section \ref{subsec:act-part-tutte}).
%Moreover, second, with little work on spanning trees, we can prove that the bijection in the bipolar case from \cite{GiLV05, AB1} extends to a bijection for all orientations, and then derive the a decomposition result for spanning trees (see Section \ref{subsec:dec-bases}).
%Those two results generalize to matroids in \cite{AB2-a}, but such proofs using orientations are not possible in non-orientable matroids. 

Second, in Section \ref{subsec:basori}, we give reformulations of the definitions of the active bijection starting from spanning trees, and we give a simple construction
%by some propagation from the smallest to the greatest element,
building, for a given spanning tree, at the same time the active partition of this spanning tree and its preimage under the canonical active bijection.
It  consists in a single pass over the set of edges
and uses only fundamental cycles and cocycles.
This section is given for completeness of the paper, but it is proved in \cite{AB2-a, AB2-b} (in contrast with the rest of the paper which is self-contained).
Actually, it is the combination of a single pass construction of the active partition of (the fundamental graph of) a matroid basis \cite{AB2-a}, and the single pass inverse construction for the uniactive bijection alluded to above and  recalled in Section~\ref{subsec:fob}.
This construction also readily applies to the refined active bijection (the third level of the active bijection addressed below).
The simplicity of the construction from spanning trees to orientations is again a noteworthy aspect of the active bijection.
\bs

At the third level of the active bijection, in Section \ref{subsec:refined}, we choose a reference orientation $\G$ of $G$, and we define the \emph{refined active bijection of $G$ w.r.t. $\G$}, denoted $\alpha_\G$, which is a mapping from $2^E$ to $2^E$. Precisely, it applies to $A\subseteq E$, by:
% and an edge subset $A$, we have:%
%\alpha_\G(A)=\alpha(-_A\G)\ \setminus\ \Bigl(A\cap O^*(-_A\G)\Bigr)\ \cup\ \Bigl(A\cap O(-_A\G)\Bigr),$

\centerline{$\alpha_\G(A)=\alpha(-_A\G)\ \setminus\ \bigl(A\cap O^*(-_A\G)\bigr)\ \cup\ \bigl(A\cap O(-_A\G)\bigr),$}

\noindent where $O(-_A\G)$, resp. $O^*(-_A\G)$, denotes the set of smallest edges of a directed cycle, resp. cocycle, of $-_A\G$. \emevder{donner cette def technique ???}%
This mapping provides a bijection between all subsets of  edges  $A\subseteq E$, thought of as orientations $-_A\G$, and all subsets of edges, thought of as subsets/supersets of spanning trees (Theorem \ref{EG:def:act-bij-ext}), along with various interesting restrictions as shown in Table \ref{table:intro}.
In particular,  $\alpha_\G(A)$ equals the spanning tree $\alpha(-_A\G)$ when $A$ does not meet $O^*(-_A\G)$ nor $O(-_A\G)$, that is when the directions of smallest edges of directed cycles and cocycles agree with their directions in the reference orientation~$\G$, \emevder{donner cette precision technique ??? en tout cas bien de parler de choice of representative of activity classes !}%
which yields a bijection between spanning trees and these representatives of activity classes.
%
%The point of this natural refinement of the canonical active bijection  is the following (it has been briefly introduced in \cite{Gi02,GiLV07} and we develop it into the details).
This natural refinement of the canonical active bijection has been briefly introduced in \cite{Gi02,GiLV07} and we develop it into the details.
The construction is the following.
The canonical active bijection maps an activity class of orientations onto a spanning tree. On one hand, the activity class is isomorphic to a boolean lattice, and activity classes partition the set of orientations. On the other hand, each spanning tree $T$ is associated with a classical subset interval $[T\setminus \Int(T), T\cup \Ext(T)]$, where $Int(T)$, resp. $Ext(T)$, denotes the set of internally, resp. externally, active edges of $T$ \cite{Cr69} (see also \cite{Da81, GoMM97, LV13} for generalizations). 
These intervals are also isomorphic to boolean lattices, and partition the power set of $E$.
The canonical active bijection can be seen as associating each activity class to an isomorphic spanning tree interval. Then the choice of a reference orientations $\G$ allows for breaking the symmetry in the two boolean lattices and specifying a boolean lattice isomorphism for each such couple.
By this way, this refined active bijection preserves the four refined activity parameters alluded to above for orientations and for subsets about Section \ref{subsec:act-classes}.
% parameter activity considered in \cite{GoTr90, LV13}.
%

\emevder{A ETE MIS AVANT FINALEMENT : Let us finally mention that the inverse construction from spanning trees in Section \ref{subsec:basori} readily applies to the refined active bijection.}%

\bs

%To complete the presentation of the three levels of the active bijection, 
Let us point out that the constructions used at the three levels 
of the active bijection
are fundamentally independent of each other. As explained in Section \ref{subsec:act-map-class-decomp}, one can get a whole large class of activity preserving bijections following the same decomposition framework: start at the first level with any arbitrary bijection between bipolar orientations and uniactive spanning trees, extend it at the second level using the same recursive definition, and set arbitrary boolean lattice isomorphisms at the third level.
The active bijection is obtained by a canonical %constructive 
choice at each level.%\emevder{raccourcir???}

%\blue{The refined active bijection, between orientations and subsets,  is obtained with respect to any given reference orientation, and refines the canonical one with boolean lattice isomorphisms between subset intervals and orientations classes that partition the power set of the edge set.}

%In contrast to the canonical active bijection, which depends only on the unoriented underlying structure, the refined one depends on the choice of a reference orientation. 

In Section \ref{sec:induction},  
we complete the paper by providing deletion/contraction constructions of the above active bijections: the uniactive one (Theorem \ref{thm:ind-10}, extract from the companion paper \cite{ABG2LP} 
which addresses the problem of computing the fully optimal spanning tree of an ordered bipolar digraph), 
the canonical one (Theorem \ref{thm:ind-gene}), and the refined one (Theorem \ref{thm:ind-gene-refined}). 
We point out that those deletion/contraction constructions provide 
a global approach: they can be used to build the whole bijections at once, as a matching between orientations and spanning trees, rather than as a mapping
(see Remark \ref{rk:ind-gene}, see also \cite[Remark \ref{ABG2LP-rk:ind-10}]{ABG2LP} in terms of complexity).\emevder{---verifier refs---}
We also present a general deletion/contraction framework for building  correspondences/bijections between orientations and spanning trees/edge subsets involving more or less constraining activity preservation properties.
Here again, the active bijection is determined by canonical choices.
%We shall compare it with other known results from the litterature.
%
%\bs

At the end, in Section \ref{sec:example}, we completely analyze the example of $K_3$ and $K_4$ (much more illustrations and details on the same example can be found in \cite{AB2-a, AB2-b}).
%, and ignored in recent references from the litterature providing similar bijections but not activity preserving and specific to graphs.
%\ss
%\bs

\paragraph{Further notes on the scope of this paper}
%Let us end with a few more remarks.
This paper is intended %to give the main results 
for a reader primarily interested in graph~theory. It is essentially self-contained and written in the graph language.
Meanwhile, it is inspired from oriented matroid theory, meaning for example that the technique and constructions %usually 
do not use the vertices of the graph at all, and often manipulates or highlights minors, combinations of cycles/cocycles, as well as
%, more importantly, 
cycle/cocycle duality.
\emevder{The authors imagine that an appeal of this paper could be to help popularizing 
%\red{??? familiarize ???} 
%this \red{useful but} not very widespread approach of graphs \red{in the graph community}.
%this useful approach of graphs, not very widespread in the graph community.
this useful but  not very widespread approach of~graphs.
\red{???BOF ???}
}%

Beyond graphs,
%{\it Summary of the authors papers on the subject.}
this work is the subject of several papers by the present authors
% in various structures, mainly in oriented matroids
%\red{notably} 
%\red{\cite{GiLV03}???}
%\cite{GiLV04, GiLV05, GiLV06, GiLV07, GiLV09, AB1, AB2, AB3, AB4}.
\cite{GiLV04, GiLV05, GiLV06, GiLV07, GiLV09, AB1, AB2-a, AB2-b, AB3, AB4}.
%, whose original and most general context is that of oriented matroids.
%
In a much more general context, the active bijection has a geometrical flavour, in real hyperplane arrangements or pseudosphere arrangements.
The main papers, which provide the whole construction in oriented matroids, 
%in this general context, 
are \cite{AB1, AB2-a, AB2-b, AB3, AB4}, and the reader can see the introduction of \cite{AB2-b} for a more general and detailed overview.
%
%In contrast, the present paper %, in contrast with other papers of the series, 
%uses the graph language and 
%is intended to give the main results for a reader primarily interested in graph~theory.
%%
%%The particular case of graphs 
%%%on a linearly ordered set of edges 
%%is interesting on its own as it 
%The particular case of graphs is interesting on its own as it
%provides  simplified constructions, or specific properties, or simply a more widespread framework. 
%%
%A first paper on the particular case of graphs has been published 
%\cite{GiLV05}.
%Its main key theorem, recalled in Section \ref{subsec:fob} and  generalized to oriented matroids in \cite{AB1},
%allows us to define the active bijection restricted to bipolar orientations, that is when $\io+\ep=1$ in the above setting.
%%ensures that the active mapping is well-defined and yields a bijection when restricted to bipolar orientations (recalled here as Theorem \ref{thm:bij-10}).
%%
The previous paper on graphs \cite{GiLV05} was a graphical version of \cite{AB1}.
Now, roughly, as mentioned in the above introduction, the present paper condenses the papers  \cite{AB2-a, AB2-b,  AB4} 
and adapts them in terms of graphs
(the main results of \cite{AB2-a}, available in matroids, are
%almost skipped here \red{???} and 
 derived here from graph orientability),
and the companion paper \cite{ABG2LP} condenses and adapts \cite{AB3}.
More examples, figures, results and details, which apply in particular to graphs, can  be found in these papers \cite{AB1, AB2-a, AB2-b, AB3, AB4}.
%Let us also mention two short notes \cite{GiLV07, GiLV09} presenting the whole construction.
Summaries %of the whole construction 
can be found in \cite{GiLV07, GiChapterOriented}
(and a survey had been given in \cite{GiLV03}, partial translation of  \cite{Gi02} in english and obsolete as for today).
%See also \cite{GiLV07, GiChapterOriented} for summaries.
\emevder{citer Chapter en tant que resume plus global du contexte ?}

Let us also highlight \cite{GiLV06} which addresses the case of chordal graphs, also called triangulated graphs, in the more general context of supersolvable hyperplane arrangement (see \cite[Example 3.2]{GiLV06}). In particular, for acyclic orientations of the complete graph 
%(equivalent to regions of the braid arrangement, or to permutations) 
with a suitable edge set ordering, 
the active bijection coincides with a well-known bijection between permutations and increasing trees
%% \emevder{raccourcir ?}%
%%In the complete graph case, with a suitable edge set ordering, 
%%as detailed in \cite[Section 5]{GiLV06}, 
%%our bijection turns out to be a well-known bijection between permutations and increasing trees.
% as detailed in \cite[Section 5]{GiLV06}, 
%See \cite{GiLV06} for references on this classical bijection.
% as detailed in , 
(see \cite[Section~5]{GiLV06} for details and references).%
%\eme{Let us mention also that, for a suitable consistent combinatorial map ordering, the bijection from \cite{Be08} turns out to be also the same bijection [Bernardi, Gioan, Las Vergnas. Informal discussion.]}%

%\begin{color}{purple}
%Let us mention that 
Originally, the question of relating spanning tree and orientation activities came from a paper by the second author
%the second author %Las Vergnas  
%in 
\cite{LV84a}, following on from which, in \cite{LV84b}, a definition for a correspondence between spanning trees and orientations of graphs was proposed. %, with no proof.
%(not proved).
 It was based on an algorithm, given without a proof%
    \footnote{
 Besides the fact that no proof 
 exist, 
 %had  been given, 
% has been supplied, %given,
% Furthermore, 
 the %present 
 authors %strongly 
 suspect that, anyway,
  this algorithm would not yield a proper correspondence %anyway
   if its formulation was extended beyond regular matroids. 
   %Notably, 
   Its technicalities and its non-natural behaviour with respect to duality, in contrast with the active bijection, made %them 
  the authors 
  abandon this algorithm.}, 
% \footnote{
% Besides the fact that no proof had  been supplied, %given,
%% Furthermore, 
% the %present 
% authors %strongly 
% suspect that %, anyway,
%  this algorithm would not yield a proper correspondence anyway if its formulation was extended beyond regular matroids. Notably, its technicalities and its non-natural behaviour with respect to duality, in contrast with the active bijection, made the authors 
%  abandon this algorithm.}, 
 which %directly 
 inspired the 
 %idea of decomposing 
 decomposition of 
 activities developed for the active bijection, but which does not yield the correspondence given by the active bijection (not for general activities, nor for the restriction to $1/0$ activities, and nor with respect to duality).
 %the same correspondence as the active bijection.
%\EMEvquatorze{, and may most likely not be generalized beyond regular matroids.}%
%
%  defined by an algorithm. However, this correspondence is different from the active bijection and may probably not be generalized beyond regular matroids. Still, the decomposition of activities developed in what follows is directly inspired by this paper.
%
%
Also, let us %also 
mention that a different notion of activities for graph orientations had been introduced even earlier %by  Berman 
in \cite{Be77}, 
%with a claimed correspondence between orientations and spanning trees, 
%yielding incorrect results according to \cite{LV84a}\footnote{
along with incorrect constructions according to \cite{LV84a}\footnote{
The construction in \cite{Be77} consisted in defining some active directed cycles/cocycles in a complex way, instead of active edges, and in enumerating those cycles/cocycles. It claimed to yield a Tutte polynomial formula which was formally similar to that of Las Vergnas \cite{LV84a} using those different activities,
and  a correspondence between orientations and spanning trees.
% (different from the active bijection).
According to \cite[footnote page 370]{LV84a}, those constructions were not correct.
%See the footnote on page 370 of \cite{LV84a} for details.
}%
.
%, and different from the active bijection.
%but whose validity 
%%has not been checked by the present authors.
%seem obscure to the present authors.
Finally, the active bijection %, %for oriented matroids, 
%in the form 
%as 
%addressed in the present paper 
has been introduced %and developed 
in 
%Gioan's 
the
Ph.D. thesis of the first author 
\cite{Gi02}, where most of the results from 
\cite{GiLV04, GiLV05, GiLV06, GiLV07, GiLV09, AB1, AB2-a, AB2-b, AB3, AB4}
%\cite{GiLV04}--\cite{ABG2} 
%were either given or announced.\red{OK????}
were given, at least in a preliminary form.

\paragraph{Further literature notes}
%
%\new{donner les traldi, les viennot ici ou pas ? Further literature notes are given in section introductions.}
Information on literature %closely 
related to 
%concerning 
specific results of the paper is given throughout the paper.
To end this introduction, let us give further references on results involving orientations and spanning trees in graphs, distinct from the active bijection.

The equality between the number of unique sink acyclic orientations and internal spanning trees comes from \cite{Za75}.
A bijection between these objects appeared in \cite{GeSa00}, 
%It is not activity preserving, on the contrary with ours in \cite[Section 6]{GiLV05}, which answers a question in \cite[(a) p. 145]{GeSa00}.
and our %particular 
more involved bijection \cite[Section 6]{GiLV05} (see also Theorem \ref{EG:def:act-bij-ext}) answers a question in this paper \cite[(a) p. 145]{GeSa00}.

%The notion of active partition for a directed graph on a linearly ordered set of edges generalizes the notion of components of acyclic orientations with a unique sink, as studied in \cite{La01} in relation with the chromatic polynomial, in \cite{Vi86} in terms of heaps of pieces, and in \cite{CaFo69} in terms of non-commutative monoids.
%\red{DEJA DANS INTRO.... intro AUSSI ! intro seulement ???}
%Our generalization allows us to consider any linear ordering of the edge set, rather than particuar ordering of the vertex set, see \cite[Section 7]{GiLV05}.
%An enumeration of acyclic orientations with a unique sink
%in a graph, using the coefficients of the chromatic polynomial, 
%has been described in \cite{La01},
%in relation with results by P. Cartier, D. Foata, I. Gessel and X. Viennot [18].
%We prove in this section that the decomposition of an acyclic orientation with a 
%unique sink 
%into {\it $V$-components}, constructed in [13] by means of a linear ordering of 
%the vertices, is a particular case of the active partition of the present paper,
%for some suitably defined linear ordering of the edges.

According to our knowledge, the first bijection between acyclic orientations and no-broken-circuit subsets in graph appeared in \cite{BlSa86}.
Another bijection between orientations and no-broken-circuit subsets appeared  in \cite{BeChTe10} in the context of parking functions.

Other bijections between acyclic orientations, resp. strongly connected orientations, resp. general orientations, and internal-type, resp. external-type spanning trees, resp. edge subsets % (equivalent to subgraphs), 
appeared in \cite{Be08}. They rely on a different notion of activities, for spanning trees only, and depending on rotation schemes of combinatorial maps instead of linear orderings of the edge set. 
\eme{The extension to orientations and edge subsets from the acyclic and strongly connected cases uses a similar construction than the one for our bijection detailed in Section \red{???},
as already given in
% which was almost direct from \cite{GiLV05} and had appeared in 
%\cite{Gi02, GiLV03, GiLV07} 
\cite{Gi02, GiLV07} 
in more general structures. ****PAS UTILE***}%

%\red{XXX repet XXX} As mentioned above, In the complete graph case, with a suitable edge set ordering, as detailed in \cite[Section 5]{GiLV06}, 
%our bijection turns out to be a well-known bijection between permutations and increasing trees.
%See \cite{GiLV06} for references on this classical bijection.
%\eme{Let us mention also that, for a suitable consistent combinatorial map ordering, the bijection from \cite{Be08} turns out to be also the same bijection [Bernardi, Gioan, Las Vergnas. Informal discussion.]}%
%%\red{ou bien en mettre au moins une ici ? ou toutes ?}
%%This paper \cite{Be08} did not quote this similarity and %anyway 
%%did not quote our other similar bijections in full generality.
%%

%, which does not quote ours from 
%\cite{Gi02, GiLV03, GiLV05, GiLV06, GiLV07} neither. 
%\cite{Gi02, GiLV05, GiLV06, GiLV07} neither. 
%
However, none of the above bijections \cite{GeSa00, BlSa86, Be08, BeChTe10} 
%preserve activities (this can be directly checked on the examples given in these papers),
is intended to preserve activities,
 and none of them seems to generalize to hyperplane arrangements nor to oriented matroids. %\com{*autre-refs?*}

Lastly, let us mention the recent work \cite{BaHoTr18} which gathers both subset activity parameters (addressed in Section \ref{prelim:subset-activities}) and orientation activity parameters (addressed in Section \ref{subsec:act-classes}) in a large Tutte polynomial expansion formula %involving twelve parameters 
in the context of graph fourientations. 
%Furthermore, this work 
%Let us point out that 
%
This work also extends to graph fourientations 
a deletion/contraction property addressed in Section \ref{subsec:ind-framework}, see Remark~\ref{rk:fourientations}. 
\section{Preliminaries}
\label{sec:prelim}

\emevder{attention il fallait enlever trucs intuiles, notamment de section LP, normalement ca a ete fait, OK !}%

\emevder{remplacer le plus possible "on a linearly ordered set of edges" par "ordered graph" ?}

\subsection{Generalities}
\label{subsec:prelim-gene}

For the sake of simple exposition, graphs in this paper are usually assumed 
to be connected, %considered as being connected, 
but the results apply to non-connected graphs as well, up to direct adaptations such as replacing spanning trees with spanning forests.
Graphs can have loops and multiple edges.
The 2-connectivity of a graph means its 2-vertex connectivity, and 
we consider a loopless graph on two vertices with at least one edge as 2-connected. 
Loops and isthmuses have the ususal meaning.
A graph can be called \emph{loop}, or \emph{isthmus}, if it has a unique edge and this unique edge is a loop, or not a loop, respectively.
\emevder{a verifier/confirmer, sinon "signle loop", "single isthmus" comme dans filtrations}%
A \emph{digraph} is a directed graph, and an \emph{ordered graph} is a graph $G=(V,E)$ on a linearly ordered set of edges $E$. 
Edges of a directed graph are supposed to be \emph{directed} or equally \emph{oriented}.
A directed graph will be denoted with an arrow, $\G$, and the underlying undirected graph without arrow, $G$. Reversing the directions of a subset of edges $A$ in a directed graph $\G$ is called %\emph{redirecting} or equally 
\emph{reorienting}, 
and the resulting directed graph is denoted $-_A\G$.
The digraph obtained by reorienting all edges is called \emph{the opposite} digraph.
The cycles,  cocycles, and spanning trees of a graph $G=(V,E)$ are considered as subsets of $E$,
hence their edges can be called their \emph{elements}. 
The cycles and cocycles of $G$ are always understood as being minimal for inclusion.
Given $F\subseteq E$, we denote $G(F)$ the graph obtained by restricting the edge set of $G$ to $F$, that is the minor $G\s (E\s F)$ of $G$ (observe that $G(F)$ is not necessarily connected, isolated vertices are pointless and can be ignored).
For $e\in E$, a minor $G/\{e\}$, resp. a minor $G\bk\{e\}$, resp. a subset $A\bk \{e\}$ for $A\subseteq E$, can be denoted for short $G/e$, resp. $G\bk e$, resp. $A\bk e$.
If $\mathcal F$ is a set of subsets of $E$, then $\cup \mathcal F$ denotes the subset of $E$ obtained by taking the union of all elements of $\mathcal F$.
In the paper, $\subset$ denotes the strict inclusion, and $\uplus$ (or $+$) denotes the disjoint union.
%
%Throughout the paper, 
We call \emph{correspondence} when several objects (e.g. some orientations) are associated with the same object  (e.g. a spanning tree) by a surjection (hence a bijection can be seen as a one-to-one~correspondence).%
%\red{*** AUSSI $T\s e$ *********}
%If  $\G=(V,E)$ is a directed graph whose underlying undirected graph is $G$, we call $\G$ an \emph{orientation} of $G$.
%An orientation $\G$ of $G$ is called acyclic

%\bigskip

\subsection{Spanning tree activities}
\label{subsec:prelim-sp-trees}

Let $G$ be an ordered (connected) graph and let $T$ be  a spanning tree of $G$. 
The next definitions are almost not practically used in the rest of the paper, but we use them here to define the Tutte polynomial and to settle the general setting of the paper.
For $b\in T$, the \emph{fundamental cocycle} of $b$ with respect to $T$, denoted $C_G^*(T;b)$, or $C^*(T;b)$ for short,  is the cocycle joining the two connected components of $T\setminus \{b\}$. Equivalently, it is the unique cocycle contained in $(E\s T)\cup\{b\}$.
For $e\not\in T$, the \emph{fundamental cycle} of $e$ with respect to $T$, denoted $C_G(T;e)$, or $C(T;e)$ for short,
 is the unique cycle contained in $T\cup\{e\}$.
Let $$\Int(T)=\Bigl\{\ b\in T \mid b=\ \min \ \bigl(\ C^*(T;b)\ \bigr)\ \ \Bigr\},$$
$$\Ext(T)=\Bigl\{\ e\in E\setminus T \mid e=\ \min \ \bigl(\ C(T;e)\ \bigr)\ \ \Bigr\}.$$
The elements of $\Int(T)$, resp. $\Ext(T)$, are called \emph{internally active}, resp. \emph{externally active}, with respect to $T$. The cardinality of $\Int(T)$, resp. $\Ext(T)$ is called \emph{internal activity}, resp. \emph{external activity}, of $T$. 
Observe that $\Int(T)\cap \Ext(T)=\emptyset$ and that, for $p=\min (E)$,  we have $p\in \Int(T)\cup \Ext(T)$. If $\Int(T)=\emptyset$, resp.
$\Ext(T)=\emptyset$, then $T$ is called \emph{external}, resp. \emph {internal}.
If 
%$\mid \Int(T)\cup \Ext(T)\mid=1$ 
$\Int(T)\cup \Ext(T)=\{p\}$
then $T$ is called \emph{uniactive}.
Hence, a spanning tree with internal activity $1$ and external activity $0$ can be called uniactive internal, and a spanning tree with internal activity $0$ and external activity $1$ can be called uniactive external.
Let us mention that exchanging the two smallest elements of $E$ yields a canonical bijection between uniactive internal and uniactive external spanning trees, see \cite[Section 4]{GiLV05}.
\eme{remplacer partout ?}%
Also, we recall that if $T_{\min}$ is the smallest (lexicographic) spanning tree of $G$, then $\Int(T_{\min})=T_{\min}$, $\Ext(T_{\min})=\emptyset$ and $\Int(T)\subseteq T_{\min}$ for every spanning tree $T$.
In the paper, we can also denote $\Int_G$ for $\Int$, resp. $\Ext_G$ for $\Ext$, to highlight the graph $G$.
%If $b=min\ C^*(T;b)$ then $b$ is \emph{internally active} in $T$.
%The set of internally active elements of $T$ is denoted $\Int(T)$,
%the cardinality of this set is the \emph{internal activity} of $T$.
%If $e=\min\ C^*(T;e)$ then $e$ is \emph{externally active} in $T$.
%The set of externally active elements of $T$ is denoted $\Ext(T)$,
%the cardinality of this set is the \emph{external activity} of $T$.
\emevder{serait-il utile de redonner  caracterisation de internal uniactive de GiVL05 ?}

By \cite{Tu54}, the Tutte polynomial of $G$ is
$$t(G;x,y)=\sum_{\io,\ep}\ t_{\io,\ep}\ x^\io\ y^\ep$$ 
where $t_{\io,\ep}$ is the number of spanning trees of $G$ 
with internal activity $\io$ and external activity $\ep$.

%\bigskip

\subsection{Orientation activities}
\label{subsec:orientation-activity}

If  $\G=(V,E)$ is a directed graph whose underlying undirected graph is $G$, we call $\G$ an \emph{orientation} of $G$.
A \emph{directed cycle} of $\G$ is a cycle of $G$ such as all orientations of edges are consistent with a running direction of the cycle.
A \emph{directed cocycle} of $\G$ is a cocycle of $G$ such as all orientations of edges go from one of the two parts of the vertex set of $G$ induced by the cocycle to the other.
The directed graph $\G$ is \emph{acyclic} if it has no directed cycle, or, equivalently, if every edge belongs to a directed cocycle.
The directed graph $\G$ is \emph{strongly connected} (or \emph{totally cyclic}), if every edge belongs to a directed cycle, or, equivalently, if it has no directed cocycle.

Let $\G$ be an orientation of an ordered connected graph $G$.
Let $$O^*(\G)=\Bigl\{\ a\in E \mid a=\ \min \ \bigl(\ D\ \bigr)\ \hbox{for a directed cocycle } D\ \Bigr\},$$
%Let $$\mathrm{O}^*(\G)=\Bigl\{\ a\in E \mid a=\ \min \ \bigl(\ D\ \bigr)\ \hbox{for a directed cocycle } D\ \Bigr\},$$
$$O(\G)=\Bigl\{\ a\in E \mid a=\ \min \ \bigl(\ C\ \bigr)\ \hbox{for a directed cycle } C\ \Bigr\}.$$
The elements of $O^*(\G)$, resp. $O(\G)$, are called \emph{dual-active}, resp. \emph{active}, with respect to $\G$. The cardinality of $O^*(\G)$, resp. $O(\G)$, is called \emph{dual-activity}, resp. \emph{activity}, of $\G$.
Observe that $O^*(\G)\cap O(\G)=\emptyset$ and that, for $p=\min (E)$,  we have $p\in O^*(\G)\cup O(\G)$. Observe also that we have $O^*(\G)=\emptyset$, resp.
$O(\G)=\emptyset$, if and only if  $\G$ is strongly connected, resp. acyclic.

By  \cite{LV84a}, we have the following theorem enumerating of orientation activities: % (EOA)
\begin{equation*}
%\tag{EOA}
\label{eq:orientation-activities}
%t(G;x,y)=\sum_{\io,\ep}o_{\io,\ep}2^{-\io-\ep}x^\io y^\epsilon
t(G;x,y)=\sum_{\io,\ep}\ o_{\io,\ep}\ {\Bigl({x\over 2}\Bigr)}^\io\  {\Bigl({y\over 2}\Bigr)}^\ep
\end{equation*}
%$$t(G;x,y)=\sum_{\io,\ep}o_{\io,\ep}2^{-\io-\ep}x^\io y^\ep$$
where $o_{\io,\ep}$ is the number of orientations of $G$ with  dual-activity $\io$ and activity $\ep$.
%\red{EOA a virer si pas utilise, moche}
%\red{numeroter equation pour s'en resservir}
%
%\begin{equation}
%\tag{EOA}
%\label{eq:orientation-activities}
%t(G;x,y)=\sum_{\io,\ep}o_{\io,\ep}2^{-\io-\ep}x^\io y^\ep
%\end{equation}

This last formula %(\ref{eq:orientation-activities}) 
generalizes various results from the literature,
for instance: counting acyclic orientations \cite{St73},
which is a special case of counting the number of regions of a (real central) hyperplane arrangement  \cite{Wi66, Za75, GrZa83}, 
%\red{reecrire mieux idem AB2b tel qu'annonce dans AB1 ---- XXXXXXXXXXXXX + GrZa???}
counting bounded regions in hyperplane arrangements or bipolar orientations in graphs \cite{Za75} (see below),
generalizations in (oriented) matroids \cite{LV77}, etc., see \cite{AB1}\emevder{[ChapterOM]} for further references. 
%DESSOUS REPORTE EN INTRO
%Let us mention that the idea of considering smallest elements of directed cycles/cocycles had been addressed in \cite{Be77}, but this paper was based on counting cycles/cocycles rather than those smallest elements, yielding wrong formulas.
%\red{citer Berman ?}
%.. \red{autres refs, zaslav etc?}
%generalizations in (oriented) matroids\cite{BrLu76, LV75, LV77, LV80, GrZa83}... 
%*pas-sur-pour-GrZa83?* *generalisation-matroides-pas-necessaires-ici?*
%\medskip

Comparing the above two expressions for $t(G;x,y)$ we get, for all $\io,\ep$:
$$o_{\io,\ep}\ =\ 2^{\io+\ep}\ t_{\io,\ep}.$$
%\smallskip

\subsection{Bipolar orientations and $\beta$ invariant}
\label{subsec:prelim-beta}

%\red{uniactive spanning trees ? bien dans titre, equilibre sp trees, mais peu utile ici, renvoyer a ABG2LP}

%\red{***attention convention pour graphe reduit a une seule ar?te}
We say that a directed graph $\G$  on the edge set $E$ is {\it bipolar with respect to $p\in E$} if  $\G$ is acyclic and has a unique source and a unique sink which are the extremities of $p$.
In particular, if $\G$ consists in a single edge $p$ which is an isthmus, then $\G$ is bipolar with respect to $p$.
Equivalently, $\G$ is bipolar with respect to $p$ if and only if every edge of $\G$ is contained in a directed cocycle and every directed cocycle contains $p$, see  \cite{GiLV05}.
We say that $\G$ is {\it cyclic-bipolar with respect to $p\in E$} if either $\G$ consists in a single edge $p$ which is a loop, or $\G$ has more than two edges and the digraph $-_p\G$ obtained from reorienting $p$ in $\G$ is bipolar with respect to $p$. Equivalently, $\G$ is cyclic-bipolar if and only if every edge of $\G$ is contained in a directed cycle and every directed cycle contains $p$, see  \cite[Proposition 5]{GiLV05}. 
%
%The above equivalences are proved in \cite{GiLV05}.
%It is shown in \cite{GiLV05} that the above equivalences are correct.
Therefore, for graphs with at least two edges, reorienting $p$ provides a canonical bijection between bipolar orientations with respect to $p$ and cyclic-bipolar orientations with respect to $p$ \cite[Section 4]{GiLV05}.
Another characterization is the following: $\G$ is bipolar w.r.t. $p$ (or equally $-_p\G$ is cyclic-bipolar w.r.t. $p$) if and only if $\G$ is acyclic and $-_p\G$ is strongly connected.
Let us mention that if $G$ is planar then bipolar orientations of $\G$ with respect to $p$ correspond to cyclic-bipolar orientations of $G^*$ with respect to~$p$.
%It is also shown in \cite{GiLV05} that, assuming $G$ is ordered, $\G$ is bipolar or cyclic-bipolar with respect to $p=\min(E)$ if and only if 
%%$\mid O^*(\G)\cup O(\G)\mid=1$.
%$ O^*(\G)\cup O(\G)=\{p\}$.

Assuming $G$ is ordered, we get by definitions that:  
$\G$ is bipolar with respect to $p=\min(E)$ if and only if 
$O(\G)=\emptyset$ (i.e. $\G$ is acyclic, i.e. $\G$ has an activity equal to zero) and $O^*(\G)=\{p\}$ (i.e. it has exactly one dual-active edge, i.e. $\G$ has a dual-activity equal to one).
Similarly, $\G$ is cyclic-bipolar if and only if 
$O^*(\G)=\emptyset$ (i.e. $\G$ is totally cyclic, i.e. $\G$ has a dual-activity equal to zero) and $O(\G)=\{p\}$  (i.e. it has exactly one active edge, i.e. $\G$ has an activity equal to one).
For an ordered digraph, being (cylic-)bipolar is always meant w.r.t. its smallest edge (for short, we might omit this precision).
%\bigskip
%\red{bijection bipolar et cyclic-bipolar}
%
%\red{**OU cyclic-bipolar ???}

%Comparing the two above expressions for $t(G;x,y)$ we get, for all $\io,\ep$:
%$$o_{\io,\ep}=2^{\io+\ep}t_{\io,\ep}.$$ 

In particular $$\beta(G)=t_{1,0}={o_{1,0}\over 2},$$ counts the number of uniactive internal spanning trees, as well as  the number of bipolar orientations of $G$ with respect to a given edge with fixed orientation. This number does not depend on the linear ordering of the edge set $E$. This value is known as the \emph{beta invariant} of $G$ \cite{Cr67} and denoted $\beta(G)=t_{1,0}$. Assuming $\mid E\mid >1$, it is known $\beta(G)=t_{1,0}=t_{0,1}$, and that $\beta(G)\not=0$ if and only if the graphic matroid of $G$ is connected, that is if and only if $G$ is loopless and 2-connected. 
Note that, if $\mid E\mid=1$, then we have $\beta(G)=1$ if the single edge is an isthmus of $G$,
and $\beta(G)=0$ if the single edge is a loop of $G$.
%\bigskip

%\red{introduire $\beta^*$ ??? OUI !!! ou est-il introduit ?}

Finally, for our constructions, we need to introduce the following dual slight variation of $\beta$:
$$
 \beta^*(G)=t_{0,1}={o_{0,1}\over 2}= \
\Biggr\{
\begin{array}{ll}
       \beta(G) &\text{ if }|E|>1 \\
       0 &\text{ if $G$ is an isthmus} \\
       1 &\text{ if $G$ is a loop.} 
\end{array}
$$
%It is also knwon that $\beta(G)=t_{0,1}={o_{1,0}\over 2}$.
%In this paper, we will adopt as a convention that $\beta(G)=1$ if $\mid E\mid =1$ even if $G$ is a loop.\com{*ici?*}

\subsection{Subset activities refining spanning tree activities}
\label{prelim:subset-activities}

This section can be skipped in a first reading. It is crucial only for the refined active bijection in Section \ref{subsec:refined}, which relates it to its counterpart for orientations developed in Section \ref{subsec:act-classes}. This section can also be seen as completing Section \ref{sec:spanning-trees} which addresses the spanning tree viewpoint. % on the active bijection.

Let $G$ be a graph on a linearly ordered set of edges $E$.
Let $T$ be a spanning tree of $G$. The set of subsets of $E$ containing $T\setminus \Int(T)$ and contained in $T\cup \Ext(T)$ will be called the \emph{interval of} $T$,\emevder{"interval of T" utilis\'e?}
denoted $[T\setminus \Int(T), T\cup \Ext(T)]$.
%
%Observe that the interval of $T$ has a boolean lattice structure and can be also denoted:
%%$$\bigl\{T'\subseteq E\mid T'=T\triangle\bigl(\cup_{i\in I}\{i\}\bigr) \text{ for some }I\subseteq \Ext_G(T)\cup \Int_G(T)\bigr\}.$$
%$$\bigl\{T'\subseteq E\mid T'=T\triangle\bigl(\cup_{a\in P\cup Q}\{a\}\bigr) \text{ for  }P\subseteq \Int_G(T),\ Q\subseteq \Ext_G(T)\bigr\}.$$
%
It is a classical result from \cite{Cr69} (see also \cite{Da81, GoMM97, LV13} for generalizations) 
%\red{verifier references} 
that these sets considered for all spanning trees form a partition of $2^E$:
$$2^E=\biguplus_{T\hbox{ spanning tree}}[T\setminus \Int(T), T\cup \Ext(T)].$$

%Theorem \ref{EG:th:expansion-orientations} above provides a counterpart to 
%Theorem \ref{EG:th:Tutte-4-variables} recalled below 
%%the following one 
%in terms of four \emph{subset activities} defined below and obtained by refining spanning tree activities. 

The four refined activities defined below, which we call \emph{subset activities}, can be seen as situating a subset in the interval $[T\setminus \Int(T), T\cup \Ext(T)]$ to which it belongs for some spanning tree $T$.
They are obviously consistent with the definition of activities for a spanning tree (Section \ref{subsec:prelim-sp-trees}).
%\break
%Therefore, the refined active bijection provides a bijective interpretation of the equality of these two expressions, see Section \ref{subsec:refined}.
%%

%of these notions.% (at the level of matroid perspectives).
%
%Let us follow the notations used in \cite{LV13}. %, rephrased in terms of graphs.%

%\red{see also \cite{BaHoTr18}}

%\red{A summary about these subset activities can be found in \cite{GiChapterPerspectives}.}

%\new{partition en intervalles !! attention a redits avec section suivante sur refines... reordonner ?}

\begin{definition} %[\cite{LV13}]
\label{def:gene-act-base}
Let $G$ be a graph on a linearly ordered set $E$. 
Let $T$ be a spanning tree of $G$. Let $A$ be in the boolean interval $[T\s \Int_{G}(T), T\cup \Ext_G(T)]$. We denote:
\begin{eqnarray*}
\Ext_G(A) & =& \Ext_G(T)\s A; \\
Q_G(A) & =& \Ext_G(T)\cap A; \\
\Int_{G}(A)& =& \Int_{G}(T)\cap A; \\
P_{G}(A) & =& \Int_{G}(T)\s A. 
\end{eqnarray*}
\end{definition}

%\red{remplacer B par T}

Let us mention that
these four parameters can be defined directly from $A$ without using $T$.
In particular, $Q_G(A)$, resp. $P_G(A)$, counts smallest edges of cycles, resp. cocycles, contained in $A$, resp. $E\s A$.
This yields
$|P_G(A)|=r(G)-r_G(A)$ %(corank of $A$)
and
$|Q_G(A)|=|A|-r_G(A)$, where $r$ is the usual rank function.
%$\mid P_G(A)\mid$ and  $\mid Q_G(A)\mid$ 
These two values do not depend on the associated spanning tree.
See \cite{LV13} for~details.

Finally, Theorem \ref{EG:th:Tutte-4-variables} below provides an expansion formula for the Tutte polynomial in terms of these activities.
It is a specialization of a similar %but more general 
theorem 
%yielding various expressions of the Tutte polynomial
 in terms of generalized activities 
% \cite{GoTr90}
 \cite[Theorem~3]{GoTr90}. %(stated in matroids). 
 The formulations used in this section and paper follow \cite{LV13} (which generalized these notions from matroids to matroid perspectives).
 %\cite[Theorem 3.5]{LV13}. % (generalized to matroid perspectives).
Let us mention that numerous Tutte polynomial formulas are directly derived from this general four parameter formula, see \cite{GoTr90, LV13}.
Notably, setting $(x,u,y,v)$ to $(x,0,y,0)$ yields the Tutte polynomial expession in terms of spanning tree activities (Section \ref{subsec:prelim-sp-trees}), and setting $(x,u,y,v)$ to $(1,x-1,1,y-1)$ yields the classical Tutte polynomial expression in terms of rank function \cite{Tu54}.
See also \cite{GiChapterPerspectives, Gi18} for overviews on the notions of this~section.%
\emevder{and for alternative proofs of the Theorem below ?}

%\red{extait de LV13 :
%PM(A) := {e ? E \ A | e is the smallest element of some cocircuit of M contained in E \ A}
%QM(A) := {e ? A | e is the smallest element of some circuit of M contained in A}.
%}
%
%\red{attention bien verifier sur exemple que les defs correspodnent avec celles des orientations}

%\begin{thm}[{\cite[Theorem 3]{GoTr90}, see also \cite[Theorem 3.5]{LV13}}]
\begin{thm}[{\cite{GoTr90, LV13}}]
\label{EG:th:Tutte-4-variables}
Let $G$ be a graph on a linearly ordered set of edges $E$. We have
%\index{Tutte polynomial!4-variable expansion}
\begin{large}
$$T(G;x+u,y+v)=\sum_{A\subseteq E}\ x^{\mid \Int_G(A)\mid}\ u^{\mid P_G(A)\mid}\ y^{\mid \Ext_G(A)\mid}\ v^{\mid Q_G(A)\mid}$$
\end{large}
\end{thm}

\eme{dessous en commenatire formualtion avec cardianux}
%%\begin{thm}[{\cite[Theorem 3]{GoTr90}, see also \cite[Theorem 3.5]{LV13}}]
%\begin{thm}[{\cite{GoTr90, LV13}}]
%\label{EG:th:Tutte-4-variables}
%Let $G$ be a graph on a linearly ordered set of edges $E$. We have
%%\index{Tutte polynomial!4-variable expansion}
%$$T(G;x+u,y+v)=\sum_{A\subseteq E}x^{\io_G(A)}u^{cr_G(A)}y^{\ep_G(A)}v^{nl_G(A)}$$
%where $\io_G(A)=\mid \Int_G(A)\mid$, $cr_G(A)= r(G)-r_G(A)=\mid P_G(A)\mid$, $\ep_G(A)=\mid \Ext_G(A)\mid$ and $nl_G(A)=|A|-r_G(A)=\mid Q_G(A)\mid$.
%\end{thm}

%\red{ajotuer cardinaux dans formule, et preciser cardfianux de P et Q plus haut hors theoreme}

%\red{peut-etre pas enoncer comme th, juste dans texte}

%\red{AJOUTER ACT DE SUBSETS/BASES}

\subsection{Some tools and terminology from (oriented) matroid theory}
\label{subsec:om}

%\new{mettre a jour vu que LP est enlevee..}

The technique used in the paper is close from (oriented) matroid technique, which notably means that it focuses on edges, whereas vertices are usually not used.
%
%\red{***enlever cette prhase ???***}Finally, at a few places in the paper, we will have to deal with combinations of cycles or of cocycles.
%To achieve this, we will use some practical notations and classical properties from oriented matroid theory.
%
Given an orientation $\G$ of a graph $G$, we will have sometimes to deal with directions of edges in cycles and cocycles of the underlying graph $G$, and, at a few places,  to deal with combinations of cycles or cocycles. To achieve this, we will use some practical notations and classical properties from oriented matroid theory \cite{OM99}.
%Useful results under this formalism, typical of oriented matroid technique, will be recalled when needed in proofs.\red{XXX quand ?}
%% (notably at the beginning of the proof of Theorem~\ref{th:fob}).
%%
%%\red{voir si on utilise signes, sinon mettre cete def dans preuve ou c'est tuilise...}

%The useful formalism for graphs inherited from oriented matroid theory is the following.
A \emph{signed edge subset} is a subset $C\subseteq E$ provided with a partition  into a positive part $C^+$ and a negative part $C^-$.
A cycle, resp. cocycle, of $G$ provides two opposite signed edge subsets called \emph{signed cycles}, resp. \emph{signed cocycles}, of $\G$ by giving a sign in $\{+,-\}$ to each of its elements accordingly with the orientation $\G$ of $G$ the natural way.  
Precisely: two edges having the same direction with respect to a running direction of a cycle will have the same sign in the associated signed cycles, and  two edges having the same direction with respect to the partition of the vertex set induced by a cocycle will have the same sign in the associated signed cocycles.
%\red{a dire mieux ??? introduire direction ?}
%
In particular, a directed cycle, resp. a directed cocycle, of $\G$ corresponds to a signed cycle, resp. a signed cocycle, all the elements of which are positive (and to its opposite all the elements of which are negative).
%have the same sign. 
%The set of positive, resp. negative, elements of a signed edge subset $C$ is denoted $C^+$, resp. $C^-$.
We will often use the same notation $C$ either for a signed edge subset (formally a couple $(C^+,C^-)$, e.g. signed cycle) or for the underlying subset ($C^+\uplus C^-$, e.g. graph cycle).
%Unless it is ambiguous, we will make the abuse to consider cycles and cocycles of $\G$ as signed subsets, using the same notation for unsigned edge-subset.
%
Given a spanning tree $T$ of $G$ and an edge $b\in T$, resp. an edge $e\not\in T$, the fundamental cocycle $C^*(T;b)$, resp. the  fundamental cycle $C(T;e)$, induces two opposite signed cocycles, resp. signed cycles, of $\G$; then, by convention, 
the (signed) fundamental cocycle $C^*(T;b)$, resp. the (signed) fundamental cycle $C(T;e)$, is considered to be the one in which $b$ is positive, resp. $e$ is positive.\emevder{utilis\'e ou pas ?}
%
%
%At these places, proofs could be seen as a presentation of oriented matroid technique.
%
%
%\bigskip

We will also use some terminology inherited from classical matroid theory. 
Let $G$ be a graph with edge set $E$.
A \emph{flat} $F$ of $G$ is a subset of $E$ such that $E\setminus F$ is a union of cocycles, equivalently: if $C\setminus \{e\}\subseteq F$ for some cycle $C$ and edge $e$, then $e\in F$; equivalently: $G/F$ has no loop.
A \emph{dual-flat} $F$ of $G$ is a subset of E which is a union of cycles (in fact its complement is a flat of the dual matroid), equivalently: if $D\setminus \{e\}\subseteq E\s F$ for some cocycle $D$ and edge $e$, then $e\in E\s F$; equivalently: $G(F)$ has no isthmus.
A \emph{cyclic flat} $F$ of $G$ is both a flat and a dual-flat of $G$;
 equivalently: $G/F$ has no loop and $G(F)$ has no isthmus.
 
%Technically, we will use some classical properties of cycles and cocycles in minors, coming from classical (oriented) matroid theory.
Lastly, in Section \ref{sec:tutte},
we will extensively use properties of cycles and cocycles in minors. 
%Let us finally recall some combinatorial technique, coming from classical (oriented) matroid theory.
So, let us recall some combinatorial technique, coming from classical (oriented) matroid theory. For $F\subseteq E$, it is known that:
cycles of $G(F)$ are cycles of $G$ contained in $F$; cocycles of $G(F)$ are non-empty inclusion-minimal intersections of $F$ and cocycles of $G$; cycles of $G/F$ are non-empty inclusion-minimal intersections of $E\s F$ and cycles of $G$ (that is inclusion-minimal subsets obtained by removing $F$ from cycles of $G$); cocycles of $G/F$ are cocycles of $G$ contained in $E\s F$.

\section{The active partition/filtration of an ordered digraph}
\label{sec:tutte}

We investigate into the details the notion of active partition (and active filtration) of the edge set of an ordered digraph (introduced in previous works, e.g. \cite{Gi02,GiLV05}).
This notion turns out to be fundamental for various results: 
a canonical decomposition of an ordered digraph into bipolar and cyclic-bipolar minors (Section \ref{subsec:act-part-decomp});
a decomposition of the set of all orientations, yielding a Tutte polynomial formula in terms of filtrations and  beta invariants of minors (Section \ref{subsec:act-part-tutte}); 
a notion of activity classes of orientations 
%obtained by reversing parts of the active partition/filtration 
that partition the set of orientations into boolean lattices, yielding a Tutte polynomial formula in terms of 4 variable orientation-activities  (Section \ref{subsec:act-classes});
 and the extension of the canonical active bijection from the uniactive case to the general case (Section \ref{subsec:alpha-def-decomp}).
The reader can see Section \ref{sec:intro} for a global and more detailed  introduction to the constructions of this section and their role in the whole construction.
\emevder{A METTRE OU PAS ? Let us recall that these constructions extend to oriented matroids \cite{AB2-b}, and that proofs 
% in the oriented matroid setting are rather similar to the following proofs in the graph setting, 
  in this setting are rather similar to the following ones, 
 and even shorter as they can use duality (they often contain both a primal and a dual viewpoint, both of which we address here since a graph has no graph dual).
Also, we use orientability of graphs to derive the Tutte polynomial formula of Section \ref{subsec:act-part-tutte}, which extends to matroids  \cite{AB2-a} whereas such a proof is not available in (non-orientable) matroids.}%
%\teal{We mention that proofs in the oriented matroid setting are rather similar to the following proofs in the graph setting, and even shorter as they can use duality (they often contain both a primal and a dual viewpoint, both of which we address here since a graph has no graph dual).
%Moreover, 
%let us point out that graphs are orientable, in contrast with matroids which are not always orientable. Hence, the decomposition for directed graphs allows us to get this expression of the Tutte polynomial for general graphs, without requiring the spanning tree decomposition, whereas the same expression generalized to matroids has to be proved in terms of matroid bases, not only in terms of oriented matroids.
%}

%\new{attention donner references a Viennot et cie}

%\subsection{Tutte polynomial decomposition}
%\label{sub:tutte}
%\bigskip

%\section{Tutte polynomial decomposition}
%\section{Filtrations and Tutte polynomial}
%\section{Tutte polynomial and decomposition into bipolar and cyclic-bipolar orientations}
%\section{Decomposition of an ordered directed graph into bipolar and cyclic-bipolar orientations, and Tutte polynomial in terms of beta invariants of minors}

%\subsection{Definition, and decomposition of an ordered digraph into bipolar and cyclic-bipolar minors}

\subsection{Definition and examples - Decomposition of an ordered digraph into bipolar and cyclic-bipolar minors}
\label{subsec:act-part-decomp}

%\red{ATTENTION NON-EMPTY MINIMAL AVAIT ETE OUBLIE !!!! verifier preuves}

%\new{order differnt than in matroid appers}

%The construction addressed in this section refines 
Let us refine the classical partition of the edge set of a directed graph $\G$ as $E=F_c\uplus (E\s F_c)$ where $F_c$ and $E\s F_c$ are  respectively the union of directed cycles and cocycles of $\G$, %and the induced
which yields a
decomposition of $\G$ into an acyclic minor $\G/F_c$ and a strongly connected minor $G(F_c)$. %\red{pas redit apres ?}

A simple example is provided in Figure \ref{EG:fig:K4-dec}.
A more involved example is provided in Figure \ref{fig:gros-ex-decomp}.

\begin{figure}[h]
\centering
\includegraphics[scale=1.85]{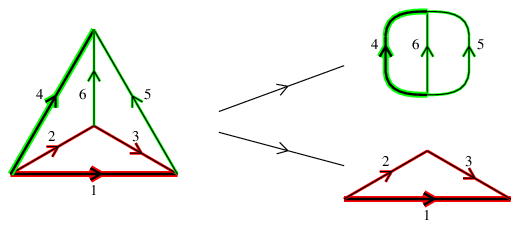}
\hfill
\includegraphics[scale=1.5]{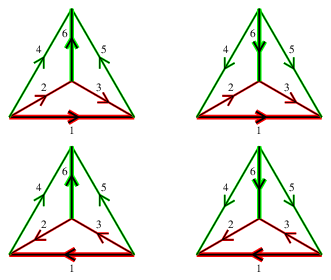}
\caption[]{Active decomposition of an ordered acyclic digraph $\G$ (on the left).
The active filtration is $\emptyset=F_c\subset 123\subset 123456$ (Definition \ref{def:act-seq-dec}). The active partition is $123+456$, with cyclic flat $F_c=\emptyset$  (Definition \ref{def:act-part}). The  active minors (in the middle) are $\G(123456)/123$, which is bipolar w.r.t. $4$, and $\G(123)$, which is bipolar w.r.t. $1$ (Definition \ref{def:active-minors}).
On the right part, we show the four digraphs in the same activity class (Definition \ref{def:act-class}). They share the same active partition/filtration, and the same active minors up to reorienting all their edges.}
\label{EG:fig:K4-dec}
\end{figure}

%An illustration of Proposition \ref{prop:unique-dec-seq} is given in Figure \ref{fig:gros-ex-decomp}.

\begin{figure}[H]
\centering
\includegraphics[scale=1.3]{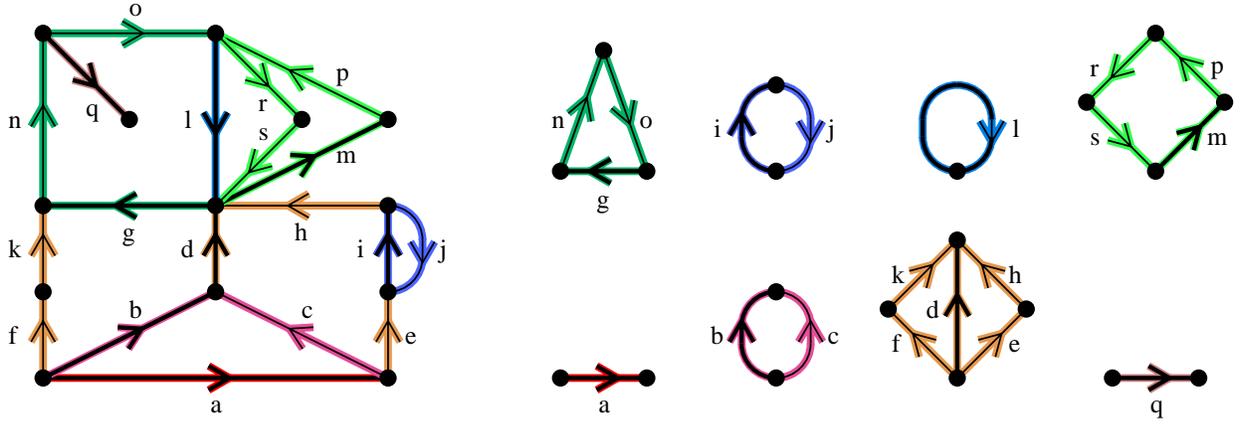}
\caption[]{Active decomposition of an ordered digraph $\G$ (shown on the left). 
%The ordering of $E$ is given by: $1<\dots<9<a<\dots<j$. The active partition is given by: $F_c=7ef+9a+c+dgij$ and $E\s F_c=1+23+4568b+h$ (Definition \ref{def:act-part}). In particular, the active edges are $O(\G)=\{7,9,c,d\}$ and the dual-active edges are $O^*(\G)=\{1,2,4,h\}$ (bold edges). 
The ordering of $E$ is given by: $a<b<c<\dots<q<r<s$. The active partition is given by: $F_c=gno+ij+l+mprs$ and $E\s F_c=a+bc+defhk+q$ (Definition \ref{def:act-part}). In particular, the active edges are $O(\G)=\{g,i,l,m\}$ and the dual-active edges are $O^*(\G)=\{a,b,d,q\}$ (bold edges).
The  bipolar, resp. cyclic-bipolar, active minors w.r.t. their smallest edges, whose edge sets are given by the active partition, are shown in the bottom right line, resp. the upper right line (Definition \ref{def:active-minors}).
%\new{ai mis seulement lettres a la place, OK ????}
}
\label{fig:gros-ex-decomp}
\end{figure}

%\purp{PEUT6ETRE DEFINIR ACTIVE PARTITIONAVANT ABSTRACT.... bof, plutot expliquer en intro}

Recall that $\subset$ denotes a strict inclusion.
%
% We will use properties of cycles and cocycles in minors recalled in Section \ref{subsec:om}.
See Section \ref{subsec:om} for properties of cycles and cocycles in minors, and for some terminology (e.g. cyclic flats), inherited from classical matroid theory. % that we will extensively use.
\emevder{a mettre ? ici ? dans preuve seulement ?}
%\red{METTRE EXEMPLES DES ICI ??? OUI PAS MAL}

\emevder{bien recomparer avec ab2-b et avec note expansions}

\begin{definition}
\label{def:act-seq-dec}
%\red{ATTENTION INDICES $a_0$ pas defini, faut $k>0$ +++ atention a vide et E, a definir}

%\red{definir cyclic part/ acyclic aprt ??? bien pour erpsecptvies}

%\red{attention notation Min}

Let $\G$ be an ordered directed graph, % whose underlying graph is $G$, 
with $\io$ dual-active edges $a_1<...<a_\io$ and $\ep$ active edges $a'_1<...<a'_\ep$.
The \emph{active filtration of $\G$} is the sequence of subsets of $E$:
$$\emptyset=F'_\ep\subset F'_{\ep-1}\subset\ldots\subset F'_0= F_c = F_0\subset \ldots\subset F_{\io-1}\subset F_\io=E,$$ %\red{*ecrire inclusions? oui!*}
that can be also denoted $(F'_\ep, \ldots, F'_0, F_c , F_0, \ldots, F_\io)$,
defined by the following.
%\red{ecrire suite emboitee / preciser avant qu'on peut ecrire les deux}
The subset $F_c$, called \emph{the cyclic flat} of the sequence, is 
%the union of all directed cycles of $\G$, or equivalently $E\setminus F_c$ is the union of all directed cocycles of $\G$.
$$F_c\ =\ \bigcup_{C \hbox{\small\ directed cycle}}C
\ =\ E\ \setminus\ \bigcup_{C \hbox{\small\ directed cocycle}}C.$$
We have $F_\io=E$, and for every $0\leq k\leq\io-1$, we have
$$F_k=E\ \setminus\ \bigcup_{\substack{ D \hbox{\small\ directed cocycle}\\ {{\small  \Min}}(D)\; \geq\; a_{k+1}}}D.$$
We have $F'_\ep=\emptyset$, and for every $0\leq k\leq\ep-1$, we have
$$F'_{k}=\bigcup_{\substack{ C \hbox{\small\ directed cycle}\\ {{\small  \Min}}(C)\; \geq\; a'_{k+1}}}C.$$
%
%VERSION EN 1 LIGNE
%This implies that, for $1\leq k\leq\io$, we have 
%$$F_k\setminus F_{k-1}=\bigcup \{D\mid D \hbox{\small\ directed cocycle}{\hbox{\small  \Min\ }}D=a_k\} \setminus
%\bigcup\{D\mid { D \hbox{\small\ directed cocycle}} {{\hbox{\small \Min\ }}D>a_k}\},$$
%and that, for $1\leq k\leq \ep$, we have $$F'_{k-1}\setminus F'_{k}=\bigcup_{\buildrel{ D \hbox{\small\ directed cycle}}\over{{\hbox{\small  \Min\ }}D=a'_k}}D \setminus
%\bigcup_{\buildrel{ D \hbox{\small\ directed cycle}}\over{{\hbox{\small \Min\ }}D>a'_k}}D
%.$$
%
%VERSION AVEC BIGCUP D
\end{definition}

One can note that, for $0\leq k\leq\io$, $F_k$ is a flat of $G$ and, for $0\leq k\leq\ep$, $F'_k$ is a dual flat of~$G$ .

\begin{definition}
\label{def:act-part}
The \emph{active partition of $\G$} is the partition of $E$ 
induced by successive differences of sets in the active filtration:
%Using notations of Definition \ref{def:act-seq-dec}, 
%\blue{The active filtration of $\G$ induces a partition of the edge-set $E$:}
$$E= (F'_{\ep-1}\s F'_\ep)\ \uplus\ \dots \ \uplus\ (F'_{0}\s F'_1)\ \uplus\ (F_1\s F_0)\ \uplus\ \dots \ \uplus\ (F_{\io}\s F_{\io-1}),$$
%Observe that we have: 
with:
\vspace{-3mm}
\begin{center}
$
   \begin{array}{ll}
\min(F'_{k-1}\s F'_k)=a'_k & \text{for } 1\leq k\leq \ep,\\[1mm]
\min(F_{k}\s F_{k-1})=a_k & \text{for } 1\leq k\leq \io.
   \end{array}
   $
\end{center}
%We call this partition \emph{the active partition of $\G$} and 
We assume that the active partition is always given with the cyclic flat $F_c$ (i.e. it can be thought of as a pair of partitions, one for $F_c$, the other for $E\s F_c$).
%\medskip
For convenience, we can refer to $F_c$, or to the parts forming $F_c$, as the \emph{cyclic part} of $\G$, and to $E\s F_c$, or to the parts forming $E\s F_c$, as the \emph{acyclic part} of $\G$.
\end{definition}

Observe that knowing the subsets forming the active partition of $\G$ %(together with the subset $F_c$), 
allows us to build the active filtration of $\G$. Indeed, the sequence $\min(F_k\setminus F_{k-1})$, $1\leq k\leq\io$,  is increasing with $k$, and the sequence  $\min(F'_{k-1}\setminus F'_k)$, $1\leq k\leq\ep$, is increasing with $k$, so the position of each part of the active partition with respect to the active filtration is identified.
%
%\red{A VERIFIER } 
Also, we have, for $1\leq k\leq\io$, 
$$F_k\setminus F_{k-1}=\bigcup_{\substack{ D \hbox{\small\ directed cocycle}\\ {{\small  \Min}}(D)\; =\; a_k}}D\ \ \setminus\
\bigcup_{\substack{ D \hbox{\small\ directed cocycle}\\ {{\small \Min}}(D)\; >\; a_k}}D,$$
and, for $1\leq k\leq \ep$,  $$F'_{k-1}\setminus F'_{k}=\bigcup_{\substack{ D \hbox{\small\ directed cycle}\\ {{\small  \Min}}(D)\; =\; a'_k}}D\ \ \setminus\
\bigcup_{\substack{ D \hbox{\small\ directed cycle}\\ {{\small \Min}}(D)\; >\; a'_k}}D
.$$

Let us point out that
%in the definitions that precede and the results that follow, 
the particular case of acyclic digraphs is addressed as the case where $F_c=\emptyset$, and the strongly connected case is addressed as the case where $F_c=E$.
Those cases can be thought of as being dual to each other (they are actually dual in an oriented matroid setting).
%Besides this, i
By the same token, in the planar case, 
%filtrations and related results can be effectively dualized. Indeed, 
 $(F'_\ep, \ldots, F'_0, F_c , F_0, \ldots, F_\io)$ is
%  a filtration of $G$, resp. 
  the active filtration of $\G$ if and only if
  $(E\s F_\io, \ldots, E\s F_0, E\s F_c, E\s F'_0, \ldots, E\s F'_\ep)$
%is a filtration of $G^*$, resp. 
is the active filtration of a dual $\G^*$ of $\G$
(which is the reason for the symmmetry in the two subscript orderings).
\emevder{attention ordre change ? !!! TOUT REVERIFIER}
Also, one can see  that if the active filtration of $\G$ is
 $(F'_\ep, \ldots, F'_0, F_c , F_0, \ldots, F_\io)$ then 
the active filtration of $\G(F_c)$ (strongly connected digraph) is  $\emptyset=F'_\ep\subset F'_{\ep-1}\subset\ldots\subset F'_0= F_c = F_c,$
and the active filtration of $\G/F_c$ (acyclic digraph) is
$\emptyset= F_c\s  F_c = F_0\s  F_c\subset \ldots\subset F_{\io-1}\s  F_c\subset F_\io\s  F_c=E\s  F_c$\emevder{inclusions ou sequence?}
(an extensive refinement of these properties is provided in Observation  \ref{obs:induced-dec-seq-ori} below).

\begin{remark}
\label{rk:active-partitions-vs-components}
\rm
%Let us mention that, a
As shown in \cite[Section 7]{GiLV05}, the notion of active partition for 
an ordered digraph
%a directed graph on a linearly ordered set of edges 
generalizes the notion of components of acyclic orientations with a unique sink.
This last notion, studied in \cite{La01} in relation with the chromatic polynomial, in \cite{Vi86} in terms of heaps of pieces, and in \cite{CaFo69} in terms of non-commutative monoids (see also \cite{Ge01}), relies on certain linear orderings of the vertex set. 
%It matches exactly active partitions for consistent orderings of the edge set that always exist.
For every such vertex ordering, there exists a consistent edge ordering such that active partitions exactly match acyclic orientation components. 
%Our generalization allows us to consider any orientation, any ordering of the edge set, see details in \cite[Section 7]{GiLV05}, 
%along with a generalization to oriented matroids \cite{AB2}.
Our generalization allows us to consider any orientation and any ordering of the edge set % see details in \cite[Section 7]{GiLV05}, 
(along with a generalization to oriented matroids).%  \cite{AB2}.
% and a geometrical interpretation \cite{AB2} (active partitions describe positions of regions w.r.t. the minimal flag).
%\red{OU RK??? ATTENTION ADAPTER DANQS INTRO paragraphe sur section 3.1}
\end{remark}

%\begin{definition}
%\label{def:active-minors}
%\red{*********** OUGJENSUIS *********** --- ceci a ete deplace d'apres}
%Using notations of Definition \ref{def:act-seq-dec}, and recursively using Lemma \ref{lem:induction-dec-seq}, we introduce
%%we define 
%%Let $\G$ be directed graph on a linearly ordered set of edges $E$,
%%with $\io$ dual-active elements $a_1<...<a_\io$ and $\ep$ active elements $a'_1<...<a'_\ep$, and with active filtration 
%%$\emptyset=F'_\ep\subset F'_{\ep-1}\subset\ldots\subset F'_0= F_c = F_0\subset \ldots\subset F_{\io-1}\subset F_\io=E$.
%%%$(F'_\ep, \ldots, F'_0, F_c , F_0, \ldots, F_\io)$.
%%%Using recurisvely using Lemma \ref{lem:induction-dec-seq}
%%The 
%the \emph{active minors} of $\G$ as the $\io$ minors $$\G_k=\G(F_k)/F_{k-1},\ \ 1\leq k\leq\io,$$ 
%which are %either an isthmus $a_k$ or 
%bipolar with respect to %$a_k$, with 
%$a_k=\min(F_k\setminus F_{k-1})$, and the $\ep$ minors $$\G'_k=\G(F'_{k-1})/F'_{k}, \ \ 1\leq k\leq \ep,$$ which are 
%%either a loop $a'_k$ or 
%cyclic-bipolar with respect to %$a'_k$,  with 
%$a'_k=\min(F'_{k-1}\setminus F'_{k})$.
%%
%%\red{***attention convetion pour graphe reuit a une seule ar?te}
%%The properties of these minors of being bipolar or cyclic-bipolar are given by recurisvely using Lemma \ref{lem:induction-dec-seq}.
%See examples in Figures \ref{EG:fig:K4-dec} and~\ref{fig:gros-ex-decomp}.
%\end{definition}

\begin{definition}
\label{def:active-minors}
The \emph{active minors} of $\G$ are the % $\io +\ep$ 
minors 
%$$\G_k=\G(F_k)/F_{k-1}, \text{ for } 1\leq k\leq\io,\text{ and }$$ 
% $$\G'_k=\G(F'_{k-1})/F'_{k},  \text{ for } 1\leq k\leq \ep.\hphantom{\text{ and }}$$
 $$\G(F_k)/F_{k-1}, \text{ for } 1\leq k\leq\io,\text{ and }$$ 
 $$\G(F'_{k-1})/F'_{k},  \text{ for } 1\leq k\leq \ep.\hphantom{\text{ and }}$$
%
%\red{***attention convetion pour graphe reuit a une seule ar?te}
%The properties of these minors of being bipolar or cyclic-bipolar are given by recurisvely using Lemma \ref{lem:induction-dec-seq}.
\emevder{Rk : on ne se servait pas de la notation $\G_k$, a part dans preuve alors je l'ai enlevee sauf dans preuve}%
%See examples in Figures \ref{EG:fig:K4-dec} and~\ref{fig:gros-ex-decomp}.
\end{definition}

\begin{prop}
\label{prop:pty-active-minors}
With the notations of Definitions \ref{def:act-seq-dec},
 and \ref{def:active-minors}, 
we have:
\begin{itemize}
\itemsep=0mm
\partopsep=0mm 
\topsep=0mm 
\parsep=0mm
\item the $\io$ active minors 
$\G(F_k)/F_{k-1}$, 
%$\G_k=\G(F_k)/F_{k-1}$, 
$1\leq k\leq\io,$ 
are %either an isthmus $a_k$ or 
bipolar 
w.r.t. %with respect to %$a_k$, with 
$a_k=\min(F_k\setminus F_{k-1})$, 
\item the $\ep$ active minors 
$\G(F'_{k-1})/F'_{k}$,
%$\G'_k=\G(F'_{k-1})/F'_{k}$,
 $1\leq k\leq \ep,$  are 
%either a loop $a'_k$ or 
cyclic-bipolar 
w.r.t. 
%with respect to %$a'_k$,  with 
$a'_k=\min(F'_{k-1}\setminus F'_{k})$.
\end{itemize}
\end{prop}

\begin{proof}
Direct by recursively using Lemma \ref{lem:induction-dec-seq} below.
\end{proof}

%\begin{definition}
%\label{def:active-minors}
%We call \emph{active minors} of $\G$ the minors addressed in proposition \ref{prop:pty-active-minors}.
%See examples in Figures \ref{EG:fig:K4-dec} and~\ref{fig:gros-ex-decomp}.
%\end{definition}
%

%\red{ajotuer figure plus compliquee dont loop et isthme}

%\red{***bipolar => 2-conneted ?***}

%\begin{remark}
%orietned matroid duality in the planar case, acyclic case, strongly connected case
%\end{remark}

%\red{ajouter def activity class et proposition, se montre avec composition ou par th suivant}

\begin{lemma}
\label{lem:induction-dec-seq}
%Let $\G$ be an ordered directed graph 
%%on a linearly ordered set of edges $E$ 
%with $\io\geq 0$ dual-active edges $a_1<...<a_\io$, with $\ep\geq 0$ active edges $a'_1<...<a'_\ep$, and with active decomposing  sequence 
% $\emptyset=F'_\ep\subset F'_{\ep-1}\subset\ldots\subset F'_0= F_c = F_0\subset \ldots\subset F_{\io-1}\subset F_\io=E$. 
 \emevder{repeter hypothese comme ci-dessus dans source, ou ecrire 'with the notations...' comme dans prop ci-dessus?}%
%Denote $F=F_{\io-1}$ and $F'=F'_{\ep-1}$.

We use the notations of Definitions \ref{def:act-seq-dec}.
If $\io>0$ then, denoting $F=F_{\io-1}$, we have:
\begin{itemize}
\itemsep=0mm
\partopsep=0mm 
\topsep=0mm 
\parsep=0mm
\item $\G /F$ is %either a single isthmus $a_\io$ or 
bipolar with respect to $a_\io$,
\item the active filtration of $\G(F)$ is 
 $(F'_\ep, \ldots, F'_0, F_c , F_0, \ldots, F_{\io-1})$.
\end{itemize}
 If $\ep>0$ then, denoting $F'=F'_{\ep-1}$, we have:
 \begin{itemize}
 \itemsep=0mm
\partopsep=0mm 
\topsep=0mm 
\parsep=0mm
\item $\G(F')$ is %either a single loop $a_\ep$ or 
cyclic-bipolar with respect to $a'_\ep$,
\item the active filtration of $\G / F'$ is 
 $(F'_{\ep-1}\setminus F', \ldots, F'_0\setminus F', F_c\setminus F' , F_0\setminus F', \ldots, F_{\io}\setminus F')$.
\end{itemize}
\end{lemma}

\eme{Enlever notation $F$ et $F'$ de enonce ? garder juste pour preuve ?}%

%\red{AJOUTER NOMS DE PARAGRAPHES DANS PREUVE}

%\red{attention bipolar autorise isthme et pas boucle, et cyclic bipolar autorise boucle - verifier preuve qui a ete ecrite aveant cette convention}

\begin{proof}
The proof  separately deals with the two parts of the statement. We begin with the second part, %which we call the \emph{Cycle part}, 
in which we assume $\ep>0$ and handle cycles.
%deal with $F'=F'_{\ep-1}$.
The other half
%next part 
of the proof, in which we assume $\io>0$ and handle cocycles,
%, which assumes $\io>0$ and deals with $F=F_{\io-1}$, which we call the \emph{Cocycle part}, 
is dual from the previous one.
%, as it deals with cocycles and the dual part of the filtration. 
In an oriented matroid setting, we would not have to prove the two halves, we would just have to apply one half to the dual (see \cite{AB2-b}). Here, in a graph setting, we have to adapt it. Essentially, the dual part consists in replacing terms and constructions with their dual corresponding ones, except that a supplementary technicality is used to handle cocycles.
Also, recall that cycles and cocycles of $G$ and of its minors are all considered as subsets of $E$.
%

%{\it (Cycle part.)}
\tiret {\it Cycles part.}
Assume $\ep>0$ and let $F'=F'_{\ep-1}$.
The cycles of $\G(F')$ are the cycles of $\G$ contained in $F'$, %with same signs, 
% where $E\s F= \cup_{\buildrel{ D \hbox{\small\ directed cocycle}}\over{{\hbox{\small  \Min\ }}D= a_{\io}}}\u D$.
 where $F'$ is the union of all %(supports $\u D$ of) 
 directed cycles $C$ of $\G$ with smallest edge $a'_\ep$.
 Hence every edge of $\G(F')$ belongs to a directed cycle, hence $\G(F')$ is totally cyclic.
% \red{def totally cyclic}
 And $a'_\ep$ belongs to a directed cycle of $\G(F')$, hence $a'_\ep$ is active in $\G(F')$.
If another element was active in $\G(F')$, then it would also be the smallest element of a directed cycle in $\G$ and active in $\G$, a contradiction with $a'_\ep$ being the greatest active element of $\G$.
So we have $O(\G(F'))=\{a'_\ep\}$ and $O^*(\G(F'))=\emptyset$, that is: $\G(F')$ is cyclic-bipolar with respect to $a'_\ep$. %  (a single loop).

As $F'$ is a union of directed cycles of $\G$, the directed cocycles of $\G/ F'$ are the directed cocycles of $\G$. Hence, $\G$ and $\G/ F'$ have the same dual-active edges and the same unions of directed cocycles with given smallest element. Hence, the ``dual part''  $(F_c, F_0, \ldots, F_{\io-1}, F_\io=E)$ of their active filtration is the same up to removing $F'$ from each subset. % for $\G$ and $\G(F)$.

%\red{XXXXX}
The cycles of $\G/ F'$ are exactly the non-empty inclusion-minimal
intersections of cycles of $\G$ with $E\s F'$.
%exactly the (non-empty inclusion-minimal) (signed) subsets 
%Hence, they are of the form $C\s F'$ where $C$ is a cycle of $\G$.
More precisely, the signed subsets of the form $C\s F'$, where $C$ is a cycle of $\G$, are unions of cycles of $\G/F'$.
Since every element of $F'$ is greatest than $a'_\ep$ by definition of $a'_\ep$, we have that $a'_k\in E\s F'$ for every $1\leq k<\ep$. A directed cycle $C$ of $\G$ with smallest element $a'_k$, for $1\leq k<\ep$, induces a directed cycle  of $G/F'$ contained in $C\s F'$ with smallest element $a'_k$, hence $a'_1,\dots, a'_{\ep-1}$ are active in $\G/F'$. 
%Moreover, 
Let $H'_k=\cup\{C\mid C\hbox{ directed cycle of }\G/F', \ \min(D)>a'_k\}$.
Independently, by definition of $F'_k$, we have $F'_k\s F'=\cup\{C\s F'\mid C\hbox{ directed cycle of }\G, \ \min(C)>a'_k\}$.
For every directed cycle $C$ of $\G$, $C\s F'$ is a union of directed  cycles of $G/F'$, 
so we  have $F'_k\s F'\subseteq H'_k$.
%\red{definir $\cup\{ \}$}

Now, conversely,  let $e$ be an element of $H'_k$, for some $1\leq k<\ep$. It belongs to be a directed cycle  $C$ of $\G/F'$ with smallest element $a>a'_k$.
%The cycle $C$ is contained in a cycle $C_G$ of $\G$ with $C_G\s F'=C$. %and the same signs for common elements.
As $F'$ is a union of directed cycles of $\G$, it is easy to see that 
there exists a directed cycle $C'$ of $\G$ containing $e$ and contained in $C\cup F'$.
Since every element of $F'$ is greater than $a'_\ep$ and $a'_\ep\geq a >a'_k$, the smallest element of $C'$ is greater than $a$, hence strictly greater than $a'_k$.
Since $e$ belongs to $C'\s F'$, we get that $e\in F'_k\s F'$.
We have proved  $H'_k\subseteq F'_k\s F'$, that is finally $F'_k\s F'\subseteq H'_k$, which provides the active filtration of $\G/F'$.
%

%{\it (Cocycle dual part.)}
\tiret {\it Cocycles dual part.}
Assume $\io>0$ and let $F=F_{\io-1}$.
%The next part of the proof, dealing with cocycles and the dual part of the filtration, is dual from the above one. In an oriented matroid setting, we would not have to translate it, just to apply it to the dual. But in a graph setting we have to adapt it. Essentially, what follows comes directly from replacing terms and constructions with their dual corresponding ones.
%%
%
The cocycles of $\G/F$ are the cocycles of $\G$ contained in $E\s F$, %with the same signs, 
% where $E\s F= \cup_{\buildrel{ D \hbox{\small\ directed cocycle}}\over{{\hbox{\small  \Min\ }}D= a_{\io}}}\u D$.
 where $E\s F$ is the union of all %(supports $\u D$ of) 
 directed cocycles $D$ of $\G$ with smallest edge $a_\io$.
 Hence every edge of $\G/F$ belongs to a directed cocycle, hence $\G/F$ is acyclic.
 And $a_\io$ belongs to a directed cocycle of $\G/F$, hence $a_\io$ is dual-active in $\G/F$.
If another element was dual-active in $\G/F$, then it would also be the smallest element of a directed cocycle in $\G$ and dual-active in $\G$, a contradiction with $a_\io$ being the greatest dual-active element of $\G$.
So we have $O^*(\G/F)=\{a_\io\}$ and $O(\G/F)=\emptyset$, that is $\G/F$ is bipolar with respect to $a_\io$. %  (or is a single isthmus).

As $E\s F$ is a union of directed cocycles of $\G$, the directed cycles of $\G(F)$ are the directed cycles of $\G$. Hence, $\G$ and $\G(F)$ have the same active edges, and the ``primal part'' %\red{***a definir***}  
$(F'_\ep, \ldots, F'_0, F_c)$ of their active filtration is the same. % for $\G$ and $\G(F)$.

The cocycles of $\G(F)$ 
are exactly the non-empty inclusion-minimal intersections of 
 intersections of $F$ and cocycles of $\G$.
 More precisely, the signed subsets of the form $C\cap F$, where $C$ is a cocycle of $\G$, are unions of cocycles of $\G(F)$.
Since every element of $E\s F$ is greatest than $a_\io$ by definition of $a_\io$, we have that $a_k\in F$ for every $1\leq k<\io$. A directed cocycle $D$ of $\G$ with smallest element $a_k$, for $1\leq k<\io$, induces a directed cocycle contained in $D\cap F$ of $G(F)$ with smallest element $a_k$, hence $a_1,\dots, a_{\io-1}$ are dual-active in $\G(F)$. 
%Moreover, 
Let $H_k=F\s \cup\{D\mid D\hbox{ directed cocycle of }\G(F), \ \min(D)>a_k\}$.
Independently, by definition of $F_k$, we have $F_k=F\cap F_k=F\s \cup\{F\cap D\mid D\hbox{ directed cocycle of }\G, \ \min(D)>a_k\}$.
For every directed cocycle $D$ of $\G$, $D\cap F$ is a union of directed cocycles of $G(F)$, 
so we  have  $F\s F_k\subseteq F\s H_k$, that is $H_k\subseteq F_k$.
%\red{definir $\cup\{ \}$}

Now, conversely,  let $e$ be an element of $F\s H_k$, for some $1\leq k<\io$. It belongs to be a directed cocycle  $D$ of $\G(F)$ with smallest element $a>a_k$.
We want to prove that $e$ belongs to $F\cap D'$ for some directed cocycle $D'$ of $\G$ contained in $D\cup (E\s F)$.
This is less easy to see than for cycles as in the above part of the proof.
%A reasoning similar to the above one given for cycles is less direct in terms of cocycles. We left it to the reader and rather 
We give a proof using usual oriented matroid technique.
Let us recall that the \emph{composition} $A\circ B$ between two signed edge subsets as the edge subset $A\cup B$ with signs inherited from $A$ for the element of $A$ and inherited from $B$ for the elements of $B\setminus A$.
\eme{Suite utilise composition. peut etre donner la composition conforme des covecteurs en, preliminaires ?}%
The cocycle $D$ is contained in a cocycle $D_G$ of $\G$ with $D_G\cap F=D$. % and the same signs for common elements.
Let $D'_G$ be the composition of all directed cocycles of $\G$ with smallest element $a_\io$, whose support is $E\s F$ and whose signs are all positive (since given by directed cocycles).
Then $D'_G\circ D_G$ is positive, since it is positive on $E\s F$ as $D'_G$, and positive on $D_G\cap F=D$ as $D$.
And $D'_G\circ D_G$ has smallest element $a$, since $a<a_\io$. 
By the conformal composition property of covectors in oriented matroid theory, there exists a directed cocycle $D'$ of $\G$ containing $e$ and contained in $D_G\cup (E\s F)$.
Since every element of $E\s F$ is greater than $a_\io$ and $a_\io\geq a >a_k$, the smallest element of $D'$ is greater than $a$, hence strictly greater than $a_k$.
Since $e$ belongs to $F\cap D'$, we get that $e\in F\s F_k$.
We have proved  $F\s H_k\subseteq F\s F_k$, that is finally $F_k=H_k$, which provides the active filtration of $\G(F)$.
\end{proof}

%\red{attention $\underline D$ pour support de D dans les defs de $F_k$...! prevenir qu'on fait cet abus... cocycles contained in F= its spport is cntained in F etc !}
%
%\red{attention dns ection rpecedente on utilise composition, donc subsets signed, mais pas dit dans preliminaires...}

%\subsection{Decomposition for all orientations of an ordered graph, and Tutte polynomial in terms of beta invariants of minors}
\subsection{Decomposition of the set of all orientations of an ordered graph -  Tutte polynomial in terms of filtrations and beta invariants of minors}
%\subsection{Decomposition of the set of orientations of an ordered graph, and Tutte polynomial in terms of beta invariants of minors}
\label{subsec:act-part-tutte}

%Let us now identify which sequences of subsets can be obtained as active filtrations, in order to derive general results involving all orientations of the underlying graph, not only a given directed graph.
%%
%Let us now characterize and build on the set of  all possible sequences of subsets that can be active filtrations of an orientation of a given graph.
%%
Let us now characterize and build on the set of  all possible sequences of subsets that can be active filtrations of an orientation of a given graph, and let us obtain general results involving all orientations of the underlying graph, not only a given directed graph.
After giving definitions %that will characterize these sequences,
for
%%that will characterize 
these sequences, 
we first complete Proposition \ref{prop:pty-active-minors} with a  uniqueness property in Proposition \ref{prop:unique-dec-seq}, then we extend this result to a bijective result taking into account all possible sequences in Theorem \ref{th:dec-ori}, whose enumerative counterpart is the  Tutte polynomial formula of Theorem \ref{th:tutte}.

\begin{definition}
\label{def:abst-dec-seq}
Let $E$ be a linearly ordered set. Let $G=(V,E)$ be a graph with set of edges $E$.
%with a linearly ordered set of edges $E$. 
We call \emph{filtration of $G$ (or of $E$)} a 
sequence $(F'_\ep, \ldots, F'_0, F_c , F_0, \ldots, F_\io)$ of subsets of $E$ such that:
\vspace{-1mm}
\begin{itemize}
\itemsep=0mm
\item $\emptyset= F'_\ep\subset...\subset F'_0=F_c=F_0\subset...\subset F_\io= E$;
\item the sequence $\min(F_k\setminus F_{k-1})$, $1\leq k\leq\io$  is increasing with $k$;
\item the sequence  $\min(F'_{k-1}\setminus F'_k)$, $1\leq k\leq\ep$, is increasing with $k$.
%\item the sequence $\min(F'_k)$, $1\leq k\leq\ep$  is increasing ;
%\item the sequence  $\min(F_k)$, $1\leq k\leq\io$, is decreasing.\red{increasing, non?***}
\end{itemize}
\end{definition}

%\red{MODIFIER DEF AVEC CONNECTED ET AJOUTER LEMME --- VOIR NOTES MANUSCRITES}

%\red{VERIFIER connected et beta et matroide, et attention a aretes multiples sur 2 sommets seulement}
For convenience, in the rest of the paper, we can equally use the notations $(F'_\ep, \ldots, F'_0, F_c , F_0, \ldots, F_\io)$
or $\emptyset= F'_\ep\subset...\subset F'_0=F_c=F_0\subset...\subset F_\io= E$ to denote a filtration of $G$.

\begin{definition}
\label{def:graph-dec-seq}
%We call \emph{connected filtration of $G$} %\red{***ou valid***}
%a filtration $(F'_\ep, \ldots, F'_0,$ $F_c , F_0, \ldots, F_\io)$ of $E$ such that: \red{XXX CHAPTER!!!}
A filtration $(F'_\ep, \ldots, F'_0,$ $F_c , F_0, \ldots, F_\io)$ of $G$ is called \emph{connected} if, in addition:
\vspace{-1mm}
%\begin{itemize}
%\itemsep=0mm
%%\item for every $0\leq k\leq\io$, the subset $F_k$ is a flat of $G$;
%\item for every $1\leq k\leq\io$, the minor $G(F_{k})/F_{k-1}$ is either a single isthmus (if it has one edge), %\red{necessaire?}, 
%or is loopless and 2-connected with at least two edges (otherwise); %\red{****************?}
%\eme{en fait isthme est loopless 2-connected...}%
%%\item for every $0\leq k\leq\ep$, the subset $F'_k$ is a dual-flat of $G$;
%\item for every $1\leq k\leq\ep$, the minor $G(F'_{k-1})/F'_k$ is either a single loop (if it has one edge), or is loopless and 2-connected  with at least two edges 
%%(otherwise, being understood that it is not a single isthmus). %
%(otherwise). %
%%\item the subset $F_c$ is a cyclic flat of $G$ (this condition is implied by the previous ones).
%\end{itemize}
\begin{itemize}
\itemsep=0mm
%\item for every $0\leq k\leq\io$, the subset $F_k$ is a flat of $G$;
\item for every $1\leq k\leq\io$, the minor $G(F_{k})/F_{k-1}$ is either loopless and 2-connected with at least two edges, or a single isthmus; %\red{****************?}
%\item for every $0\leq k\leq\ep$, the subset $F'_k$ is a dual-flat of $G$;
\item for every $1\leq k\leq\ep$, the minor $G(F'_{k-1})/F'_k$ is either  loopless and 2-connected  with at least two edges, or a single loop.
%\item the subset $F_c$ is a cyclic flat of $G$ (this condition is implied by the previous ones).
\end{itemize}
\end{definition}

%\red{(recall that we consider a loopless graph on two vertices with at least one edge as 2-connected) ok deja en dessous}

\eme{
To motivate this definition, let us recall that, for a graph with at least two edges, $\beta(G)\not= 0$ if and only if the matroid of $G$ is connected, that is if and only if $G$ is loopless and 2-connected (see also forthcoming Lemma \ref{lem:dec-seq-equiv}). 
}%

%To motivate this definition, 
%Recall that 
The minors involved in Definition \ref{def:graph-dec-seq} are said to be \emph{associated with} or \emph{induced by} the filtration.
Let us recall that the 2-connectivity of a graph means its 2-vertex connectivity, and that
we consider a loopless graph on two vertices with at least one edge as 2-connected (Section \ref{subsec:prelim-gene}). 
%\purp{mis aussi en prelilminaires, bien ou pas ? peut-etre separer le cas de 2 sommets ?}
%
Let us recall that, for a graph $G$ with at least two edges, 
$G$ is 
%either an isthmus or is 
loopless 2-connected
if and only if $\beta(G)\not=0$ 
if and only if $\beta^*(G)\not=0$ 
 (if and only if there exists a bipolar orientation of $G$ 
 if and only if there exists a cyclic-bipolar orientation of $G$ 
 if and only if the cycle matroid of $G$ is connected,  see Section \ref{subsec:prelim-beta}).
\eme{phrase suivante utile ou ok apr definition de 2-connected?}%
\eme{appramment : 2-connected = enlever 1 sommet ne deconnecte pas. isthme est 2-connexe}%
Let us lastly recall that, for a graph $G$ with one edge, we have $\beta(G)=1$ and $\beta^*(G)=0$ if it is an isthmus, and $\beta(G)=0$ and $\beta^*(G)=1$ if it is a loop.
From these results, we %directly 
derive the next lemma.
%
%Let us also recall that a digraph with only one edge $p$ is bipolar w.r.t. $p$ if $p$ is an isthmus, and cyclic-bipolar w.r.t. $p$ if $p$ is a loop.

%\medskip
%\red{AJOUT SUR  sur beta --- comparer avec AB2-b --- referer a ce lemme}

\begin{lemma}
\label{lem:connected-filtration-beta}
A filtration $(F'_\ep, \ldots, F'_0,$ $F_c , F_0, \ldots, F_\io)$ of $G$ is connected  if and only~if
%$$\Bigl(\prod_{1\leq k\leq \io}
%\beta \bigl( M(F_k)/F_{k-1}\bigr)\Bigr)
% \ \Bigl(\prod_{1\leq k\leq \ep}\beta^* \bigl( M(F'_{k-1})/F'_{k}\bigr)\Bigr)\ \not=\ 0.$$
\ss

\hfill $\displaystyle \Bigl(\prod_{1\leq k\leq \io}
\beta \bigl( G(F_k)/F_{k-1}\bigr)\Bigr)
 \ \Bigl(\prod_{1\leq k\leq \ep}\beta^* \bigl( G(F'_{k-1})/F'_{k}\bigr)\Bigr)\ \not=\ 0.$\hfill
 \qed
 \end{lemma}

\emevder{ai enleve lemme disant que connected filtratin => flat s and dual flats, inutile}%

\emevder{attention de ne pas employer "cyclic flat" of the conected fitlration du coup--- normalement j'ai verifier, c'est OK}%

\begin{prop}
\label{prop:unique-dec-seq}
Let $\G$ be an ordered directed graph.
% on a linearly ordered set of edges $E$.
The active filtration of $\G$ is the unique (connected) filtration $(F'_\ep, \ldots, F'_0, F_c , F_0, \ldots, F_\io)$ of $G$ %(or $E$) 
such that
the $\io$ minors 
%$$\G_k=\G(F_k)/F_{k-1},\ \ 1\leq k\leq\io,$$ 
$$\G(F_k)/F_{k-1},\ \ 1\leq k\leq\io,$$ 
are %either an isthmus $a_k$ or 
bipolar with respect to %$a_k$, with 
$a_k=\min(F_k\setminus F_{k-1})$, and the $\ep$ minors 
%$$\G'_k=\G(F'_{k-1})/F'_{k}, \ \ 1\leq k\leq \ep,$$ are 
$$\G(F'_{k-1})/F'_{k}, \ \ 1\leq k\leq \ep,$$ are 
%either a loop $a'_k$ or 
cyclic-bipolar with respect to %$a'_k$,  with 
$a'_k=\min(F'_{k-1}\setminus F'_{k})$.
%\red{***attention convetion pour graphe reuit a une seule ar?te}
\end{prop}

\eme{$a_k=\min(F_k)$ ? ou seulement $a_k=\min(F_k\setminus F_{k-1}$ ?}

%\red{attention bipolar autorise isthme et pas boucle, et cyclic bipolar autorise boucle - verifier preuve qui a ete ecrite aveant cette convention}

\begin{proof}
First, we check that the active filtration $\emptyset= F'_\ep\subset...\subset F'_0=F_c=F_0\subset...\subset F_\io= E$ of $\G$ is a filtration of $G$. 
%By construction, we have 
%$\emptyset= F'_\ep\subset...\subset F'_0=F_c=F_0\subset...\subset F_\io= E$.
%Also, for $1\leq k\leq\io$, we have  $F_k\setminus F_{k-1}=\bigcup_{\buildrel{ D \hbox{\small\ directed cocycle}}\over{{\hbox{\small  \Min\ }}D=a_k}}D \setminus \bigcup_{\buildrel{ D \hbox{\small\ directed cocycle}}\over{{\hbox{\small \Min\ }}D>a_k}}D $. \red{***a  remonter sous def***}
Assume $\G$ has  $\io$ dual-active edges $a_1<...<a_\io$, and $\ep$ active edges $a'_1<...<a'_\ep$.
By definition of $a_k$, for $1\leq k\leq \io$, there exists a directed cocycle of $\G$ whose smallest element is $a_k$, hence $a_k\in F_k\setminus F_{k-1}$ according to the definition of $F_k\setminus F_{k-1}$ given above. So we have $a_k=\min(F_k\setminus F_{k-1})$,  $1\leq k\leq\io$, which is increasing with $k$ by definition of $a_k$.
Similarly, for $1\leq k\leq \ep$, there exists a directed cycle of $\G$ whose smallest element is $a'_k$,
%we have $F'_{k-1}\setminus F'_{k}=\bigcup_{\buildrel{ D \hbox{\small\ directed cycle}}\over{{\hbox{\small  \Min\ }}D=a'_k}}D \setminus \bigcup_{\buildrel{ D \hbox{\small\ directed cycle}}\over{{\hbox{\small \Min\ }}D>a'_k}}D$, 
so we get $a'_k=\min(F'_{k-1}\setminus F'_k)$, %$1\leq k\leq\ep$,
which is increasing with $k$.
Hence the result.
%So the active filtration of $\G$ is an abstract filtration.

%Second, by definition of $F_k$, we directly have that, for every $0\leq k\leq\io$, the subset $F_k$ is a flat of $G$, as the complement of a union of cocycles. Similarly, for every $0\leq k\leq\io$, the subset $F_k$ is a dual-flat of $G$ as a union of cycles.

%The fact that the minors have the require connectivity properties

Second, 
by Proposition \ref{prop:pty-active-minors}, the active minors exactly satisfy the property stated in the statement.
% 
%we apply recursively Lemma \ref{lem:induction-dec-seq}.
%We directly get that 
%the $\io$ minors $\G_k=\G(F_k)/F_{k-1}$, $1\leq k\leq\io$
%are %either an isthmus $a_k$ or 
%bipolar with respect to $a_k$;
%and that the $\ep$ minors $\G'_k=\G(F'_{k-1})/F'_{k}$, $1\leq k\leq \ep$, are %either a loop $a'_k$ or 
%cyclic-bipolar with respect to $a'_k$. This proves the property stated in the lemma. 
This also proves that those minors are loopless and 2-connected as soon as they have more than one edge, which 
%achieves the proof 
shows that the active filtration of $\G$ is a connected filtration of $G$.
%we prove that  the active filtration satisfies the given properties.
%Assume that, for some $1\leq k\leq\io$, $\G_k=\G(F_k)/F_{k-1}$, is not bipolar with respect to  $a_k=\min(F_k\setminus F_{k-1})$. Recall that bipolar orientations are exactly those with activity $0$ and dual-activity $1$.
%Then, $\G_k$ contains either a directed cycle $C$, or a directed cocycle $D$ with $\min(D)>a_k$.
%In the first case,
%
%...
%\blue{***}

%\red{prelim : rappel des circ et cocirc de $G/F$ et $G(F)$}

Now, it remains to prove the uniqueness property.
%Assume $(F'_\ep, \ldots, F'_0, F_c , F_0, \ldots, F_\io)$ is a connected filtration of $G$ satisfying the properties given in the lemma.
Assume $(F'_\ep, \ldots, F'_0, F_c , F_0, \ldots, F_\io)$ is a filtration of $E$ satisfying the properties given in the statement.
Then it is obviously  a connected filtration of $G$, by the definitions, since being bipolar, resp. cyclic-bipolar, implies being either connected  or reduced to an isthmus, resp. a loop.
First, we prove that $F_c$ is the union of all directed cycles of $\G$.
Assume $C$ is a directed cycle of $\G$, not contained in $F_c$.
Let $k$ be the smallest such that $C\subseteq F_k$, $1\leq k\leq \io$.
Then $C\s F_{k-1}\not=\emptyset$ (otherwise $k$ would not be minimal), so $C\s F_{k-1}$ contains a directed cycle of $\G/F_{k-1}$.
Moreover $C\s F_{k-1}\subseteq F_k\s F_{k-1}$ by definition of $k$, so  $C\s F_{k-1}$ contains a directed cycle of $\G_k=\G(F_k)/F_{k-1}$, a contradiction with $\G_k$ being acyclic.
Hence the union of directed cycles of $\G$ is contained in $F_c$.
With exactly the same reasoning from the dual viewpoint, we get that the union of directed cocycles of $\G$ is contained in $E\s F_c$.
Precisely: assume $D$ is a directed cocycle of $\G$, not contained in $E\s F_c$.
Let $k$ be the smallest such that $D\subseteq E\s F'_k$, $1\leq k\leq \ep$.
Then $D\cap F'_{k-1}\not=\emptyset$ (otherwise $k$ would not be minimal), so $D\cap F'_{k-1}$ contains a directed cocycle of $\G(F'_{k-1})$.
Moreover $D\cap F'_{k-1}\subseteq F'_{k-1}\s F'_{k}$ by definition of $k$, so  $D\cap F'_{k-1}$ contains a directed cocycle of $\G'_k=\G(F'_{k-1})/F'_{k}$, a contradiction with $\G'_k$ being strongly connected.
Finally, $F_c$ contains the union of directed cycles of $\G$ and has an empty intersection with the union of all directed cocycles of $\G$, so $F_c$ is exactly the union of all directed cycles of $\G$.

%\red{XXXX}
Second, we prove the following claim: for every directed cycle $C$ of $\G$, the smallest element of $C$ equals 
$a'_{k+1}$, where $k$ is the greatest possible such that $C\subseteq F'_k$, $0\leq k\leq \ep-1$.
%we prove that the edges $a_k$, $1\leq k\leq \io$, and $a'_k$, $1\leq k\leq \ep$, are the dual-active and active edges of $\G$, respectively.
%
%\red{non en fait on ne montre que une inclusion, l'autre reviexnrait a repeter le lemme precedent, penible....}
%
%Assuming $\ep>0$, let $C$ be a directed cycle of $\G$ with smallest element $a$.
%Let $k$ be the greatest possible such that $C\subseteq F'_k$, $0\leq k\leq \ep-1$.
Indeed, for such $C$ and $k$, we have
$C\s F'_{k+1}\not=\emptyset$ (otherwise $k$ would not be maximal), so $C\s F'_{k+1}$ is a union of directed cycles of $\G/F'_{k+1}$.
Moreover, $C\s F'_{k+1}\subseteq F'_k\s F'_{k+1}$ by definition of $k$, so $C\s F'_{k+1}$ is a union of directed cycles of $\G'_{k+1}=\G(F'_k)/F'_{k+1}$. 
By assumption that $\G'_{k+1}$ is cyclic-bipolar with respect to $a'_{k+1}$, we have that $a'_{k+1}$ belongs to every directed cycle of $\G'_{k+1}$,
so $a'_{k+1}$ %is an edge of $C\s F'_{k+1}$, and
is the smallest edge of $C\s F'_{k+1}$. 
%So $a'_{k+1}$ %is an edge of $C\s F'_{k+1}$, and
%is the smallest edge of $C\s F'_{k+1}$, by assumption that $\G'_{k+1}$ is cyclic-bipolar with respect to $a'_{k+1}$ (it belongs to every directed cycle of $\G'_{k+1}$).
By definition of a filtration, $a'_{k+1}$ is the smallest edge in $F'_k$ (it is the smallest in $F'_k\s F'_{k+1}$ and the sequence $\min(F'_i\s F'_{i+1})$ is increasing with $i$), hence we have $\min(C)=a'_{k+1}$. 
%This means that 
In particular, we have proved that the active edges of $\G$ are of type $a'_k$, $1\leq k\leq \ep$.

%ATTENTION LA RECIPROQUE NECESSITE REDITE DU LEMME PRECEDENT
%Conversely, it is obvious that every $a'_k$, $1\leq k\leq \ep$, is an active edge of $\G$. Indeed, $a'_k$ is the smallest edge of a directed cycle $C$ in  $\G'_{k}=\G(F'_{k-1})/F'_{k}$ by assumption; hence $C$ is a directed cycle of $\G/F'_{k}$; hence it is contained in a cycle $C'$

Dually, we prove - the same way - the following claim: for every directed cocycle $D$ of $\G$, the smallest element of $D$ equals 
$a_{k+1}$, where $k$ is the greatest possible such that $D\subseteq E\s F_k$, $0\leq k\leq \io-1$.
%We obtain the dual result the same way.
%Assuming $\io>0$, let $D$ be a directed cocycle of $\G$ with smallest element $a$.
%Let $k$ be the greatest possible such that $D\subseteq E\s F_k$, $0\leq k\leq \io-1$.
%Then 
%Second, we prove that, for every directed cycle $C$ of $\G$, the smallest element of $C$ equals $a'_{k+1}$, where $k$ is the greatest possible such that $C\subseteq F'_k$, $0\leq k\leq \ep-1$.
%we prove that the edges $a_k$, $1\leq k\leq \io$, and $a'_k$, $1\leq k\leq \ep$, are the dual-active and active edges of $\G$, respectively.
%
%\red{non en fait on ne montre que une inclusion, l'autre reviexnrait a repeter le lemme precedent, penible....}
%
%Assuming $\ep>0$, let $C$ be a directed cycle of $\G$ with smallest element $a$.
%Let $k$ be the greatest possible such that $C\subseteq F'_k$, $0\leq k\leq \ep-1$.
Indeed, for such $D$ and $k$, we have
$D\cap F_{k+1}\not=\emptyset$ (otherwise $k$ would not be maximal), so $D\cap F_{k+1}$ is a union of directed cocycles of $\G(F_{k+1})$.
Moreover, $D\cap F_{k+1}\subseteq F_{k+1}\s F_{k}$ by definition of $k$, so $D\cap F_{k+1}$ is a union of directed cocycles of $\G_{k+1}=\G(F_{k+1})/F_{k}$. 
By assumption that $\G_{k+1}$ is bipolar with respect to $a_{k+1}$, we have that $a_{k+1}$ belongs to every directed cocycle of  $\G_{k+1}$,
so $a_{k+1}$ %is an edge of $C\s F'_{k+1}$, and
is the smallest edge of $D\cap F_{k+1}$.
%So $a_{k+1}$ %is an edge of $C\s F'_{k+1}$, and
%is the smallest edge of $D\cap F_{k+1}$, by assumption that $\G_{k+1}$ is bipolar with respect to $a_{k+1}$.
By definition of a filtration, $a_{k+1}$ is the smallest edge in $E\s F_k$ (it is the smallest in $F_k\s F_{k-1}$ and the sequence $\min(F_i\s F_{i-1})$ is increasing with $i$), hence we have $\min(C)=a_{k+1}$. 
In particular, we have proved that the dual-active edges of $\G$ are of type $a_k$, $1\leq k\leq \io$.

Third, we prove that the parts of the considered filtration are indeed the parts of the active filtration.
Let us denote $F=F_{\ep-1}$ and so $a'_\ep=\min (F)$.
We want to prove that $F=\cup\{C\mid C\hbox{ directed cycle of }\G, \ \min(C)=a'_\ep\}$.
By assumption, $G'_\ep=\G(F)$ is cyclic-bipolar.
So, every edge of $\G(F)$ belongs to a directed cycle of $\G(F)$ with smallest element $a'_\ep$. 
The cycles of $\G(F)$ are the cycles of $\G$ contained in $F$. Hence, every edge of $\G$ belonging to $F$ belongs to a directed cycle of $\G$ with smallest element $a'_\ep$,
which proves that
$F\subseteq \cup\{C\mid C\hbox{ directed cycle of }\G, \ \min(C)=a'_\ep\}$.
Conversely, let $C$ be a directed cycle of $\G$ with smallest element $a'_\ep$.
By the above claim, we have that $\ep-1$ is the greatest possible such that $D\subseteq F'_{\ep-1}$, that is $D\subseteq F$, hence the result.

Dually, let us denote $F=F_{\io-1}$ and so $a_\io=\min (E\s F)$.
We want to prove that $F=E\s \cup\{D\mid D\hbox{ directed cocycle of }\G, \ \min(D)=a_\io\}$.
By assumption, $G_\io=\G/F$ is bipolar.
So, every edge of $\G/F$ belongs to a directed cocycle of $\G/F$ with smallest element $a_\io$. 
The cocycles of $\G/F$ are the cocycles of $\G$ contained in $E\s F$. Hence, every edge of $\G$ belonging to $E\s F$ belongs to a directed cocycle of $\G$ with smallest element $a_\io$,
which proves that
$E\s F\subseteq \cup\{D\mid D\hbox{ directed cocycle of }\G, \ \min(D)=a_\io\}$.
Conversely, let $D$ be a directed cocycle of $\G$ with smallest element $a_\io$.
By the above claim, we have that $\io-1$ is the greatest possible such that $D\subseteq E\s F_{\io-1}$, that is $D\subseteq E\s F$, hence the result.

%assume $e$ is an edge belonging both to $F$ and to a directed cocycle $D$ of $\G$ with smallest element $a_\io$.

Now, we can conclude by induction, assuming the proposition is true for minors of $\G$.
Assume $\io>0$ and denote again $F=F_{\io-1}$, we have proved above that $F$ is indeed the largest part different from $E$ in the active filtration of $\G$.
It is easy to check that the sequence of subsets
$(F'_\ep, \ldots, F'_0, F_c , F_0, \ldots, F_\io-1)$  is a filtration of $G(F)$.
Moreover this filtration obviously satisfies the properties of the proposition for the directed graph $\G(F)$, as the involved minors are unchanged.
Hence, this filtration is the active filtration of $\G(F)$, by induction assumption.
Hence, by Lemma \ref{lem:induction-dec-seq}, we have that the subsets $F'_\ep, \ldots, F'_0, F_c , F_0, \ldots, F_\io-1$ are indeed the same subsets as in the active filtration of $\G$.
Finally, assume that $\ep>0$ and denote again $F'=F_{\ep-1}$, 
we have proved above that $F'$ is indeed the largest part different from $\emptyset$ in the active filtration of $\G$.
It is easy to check that the sequence of subsets
$(F'_{\ep-1}\s F', \ldots, F'_0\s F', F_c\s F' , F_0\s F', \ldots, F_\io\s F')$  is a filtration of $G/F'$.
Moreover this filtration obviously satisfies the properties of the proposition for the directed graph $\G/F'$, as the involved minors are unchanged.
Hence, this filtration is the active filtration of $\G/F'$, by induction assumption.
Hence, by Lemma \ref{lem:induction-dec-seq}, we have that the subsets $F'_{\ep-1}, \ldots, F'_0, F_c , F_0, \ldots, F_\io$ are indeed the same subsets as in the active filtration of $\G$.
\end{proof}

\begin{observation}\emevder{corollaire?}
\label{obs:induced-dec-seq-ori}
\rm
%From Proposition \ref{prop:unique-dec-seq}, one gets the following remarkable property (that refines the observation made below Definition \ref{def:act-part} which corresponds  to the case $F=G=F_c$ below).
Let $\emptyset= F'_\ep\subset...\subset F'_0=F_c=F_0\subset...\subset F_\io= E$ be the active filtration of~$\G$.
Let $F'$ and $F$ be two subsets in this sequence such that $F'\subseteq F$ (with possibly $F\subseteq F_c$ or $F_c\subseteq F'$). Then, by Proposition \ref{prop:unique-dec-seq}, 
the active filtration  of $\G(F)/F'$ is obtained from the subsequence with extremities $F'$ and $F$ (i.e. $F'\subset \dots \subset F$) of the active filtration of $\G$ by   subtracting $F'$ from each subset of the subsequence (with $F_c\s F'$ as cyclic flat).
In particular, the subsequence ending with $F$ (i.e. $\emptyset\subset \dots \subset F$) yields the active filtration of $M(F)$, and the subsequence beginning with $F$  (i.e. $F\subset \dots \subset E$) yields the active filtration of  $M/F$ by subtracting  $F$ from each subset.
\emevder{faire une propositin? pa mal pour vision algebrique...}
\end{observation}

\begin{thm}
\label{th:dec-ori}
Let $G$ be an ordered graph. We have

%$$\{\ \hbox{orientations $\G$ of $G$}\ \}\ $$
%$$ = \ \biguplus \ \Biggl\{\ \ \G \ \ \mid\ \    
%\G(F_k)/F_{k-1},\ \ 1\leq k\leq\io, \hbox{ bipolar with respect to } \min (F_k\setminus F_{k-1}),$$
%$$
%\hbox{ and }\ \ 
%\G(F'_{k-1})/F'_{k}, \ \ 1\leq k\leq \ep, \hbox{ cyclic-bipolar with respect to } \min (F'_{k-1}\setminus F'_{k})\ \ \Biggr\}$$
\vspace{2.5mm}
\centerline{$\Bigl\{\ \hbox{orientations $\G$ of $G$}\ \Bigr\}\ $}
\vspace{-5.5mm}
\begin{align*}
%\omit\rlap{\bigl\{\ \hbox{orientations $\G$ of $G$}\ \bigr\}}\\
 = \ \biguplus \ &\Biggl\{\ \G \ \ \mid\ \    
G(F_k)/F_{k-1},\ \ 1\leq k\leq\io, \hbox{ bipolar with respect to } \min (F_k\setminus F_{k-1}), \\
&\hphantom{\Biggl\{\ \G \ \ \mid\ \    }\hbox{and}\ \ 
\G(F'_{k-1})/F'_{k}, \ \ 1\leq k\leq \ep, \hbox{ cyclic-bipolar with respect to } \min (F'_{k-1}\setminus F'_{k})\ \ \Biggr\}.
\end{align*}
%\red{attention arete isolee... convention bipolar reversers bipolar ok pour arete seule ?}
%\vspace{-1mm}
where the disjoint union is over all connected filtrations $(F'_\ep, \ldots, F'_0, F_c , F_0, \ldots, F_\io)$ of $G$. The connected filtration of $G$ associated to an orientation $\G$  in the right-hand side of the equality is the active filtration of $\G$.
\end{thm}

\begin{proof}
This result consists in a bijection between all orientations $\G$ of $G$ and sequences of orientations of the minors involved in decomposition sequences of $G$. It is directly given  by  Proposition \ref{prop:unique-dec-seq}.
From the first set to the second set, 
%the result is directly given by Proposition \ref{prop:unique-dec-seq}, since 
the active filtration of $\G$ provides the required decomposition.
Conversely, from the second set to the first set,
first choose a connected filtration of $G$.
Then, for each minor of $G$ defined by this sequence, choose a bipolar/cyclic-bipolar orientation for this minor as written in the second set statement.
This defines an orientation $\G$ of $G$ (since every edge of $G$ appears in one and only one of these minors).
Now, for this orientation $\G$, the chosen filtration satisifies the property of Proposition \ref{prop:unique-dec-seq},
hence this filtration is the active filtration of this orientation $\G$ of $G$. 
%
%Conversely, %the result is also given by  Proposition \ref{prop:unique-dec-seq}.
%consider an orientation $\G$ of $G$ formed from the second set by orientations for each minor involved in a filtration of $G$. Then this filtration has to be equal to the active filtration of $\G$. 
%%
Finally, the uniqueness in Proposition \ref{prop:unique-dec-seq} ensures that the union in the second set is disjoint.
%It remains to prove the converse.
%Given a sequence of orientations of minors in the second set for the filtration $(F'_\ep, \ldots, F'_0, F_c , F_0, \ldots, F_\io)$ of $G$, let $\G$ be the orientation of $G$ given by these orientations of minors. 
%We have to prove that the filtration is actually the active dcomposing sequence of $\G$.
%\blue{***}
%
\end{proof}

\begin{thm}
\label{th:tutte}
Let $G$ be a graph on a linearly ordered set of edges $E$. We have
$$t(G;x,y)= \ \ \sum \ \ \Bigl(\prod_{1\leq k\leq \io}
\beta \bigl( G(F_k)/F_{k-1}\bigr)\Bigr)
 \ \Bigl(\prod_{1\leq k\leq \ep}\beta^* \bigl( G(F'_{k-1})/F'_{k}\bigr)\Bigr)\ {x^\io\  y^\ep}$$
% $$t(G;x,y)= \ \ \sum \ \ \Bigl(\prod_{1\leq k\leq \io}
%\bar\beta \bigl( G(F_k)/F_{k-1}\bigr)\Bigr)
% \ \Bigl(\prod_{1\leq k\leq \ep}\bar\beta \bigl( G(F'_{k-1})/F'_{k}\bigr)\Bigr)\ {x^\io\  y^\ep}$$
%}
%%%%%%
%DESSOUS version v3
%\noindent where the sum is over all possible active filtrations in $G$, and
%where, by convention, $\beta(G)=1$ if~\hbox{$\mid E\mid=1$.}
\noindent 
%where, by convention, 
%\red{ref a lemme pour $\beta^*$}
where
%$\bar\beta=1$ for a graph with one edge, and $\bar\beta=\beta$ otherwise,
$\beta^*=\beta$ for a graph with at least two edges, 
% and  
%$\beta^*=1-\beta$ for a graph with one edge ($\beta^*$ of a graph with one edge equals $1$ if this edge is a loop and equals $0$ is this edge is an isthmus),
%$\beta^*=1-\beta$ for a graph with one edge,
%$\beta^*$ of a graph with one edge equals $1$ if this edge is a loop and equals $0$ is this edge is an isthmus,
$\beta^*$ of a loop equals $1$, 
$\beta^*$ of an isthmus equals $0$,
%single loop edge equals $1$, $\beta^*=1$ for a graph with a single isthmus edge (that is $\beta^*=1-beta$ for a graph with one edge), 
%and $\beta^*=\beta$ otherwise,
%where
%$\bar\beta(H)=\beta(H)$ if $\mid E(H)\mid>1$, and $\bar\beta(H)=1$ if~\hbox{$\mid E(H)\mid=1$,}
and where the sum can be equally:
\vspace{-1mm}
\begin{itemize}
\itemsep=0mm
\partopsep=0mm 
\topsep=0mm 
\parsep=0mm
\item either over all connected filtrations $(F'_\ep, \ldots, F'_0, F_c , F_0, \ldots, F_\io)$ of $G$;
\item or over all filtrations $(F'_\ep, \ldots, F'_0, F_c , F_0, \ldots, F_\io)$ of $E$.
\end{itemize}
%\begin{itemize}
%\item either over all sequences of sets type
%$\emptyset= F'_\ep\subset...\subset F'_0=F_c=F_0\subset...\subset F_\io= E$
%where $F_c$ is a union of cycles \com{*?*}, the sequence $\min(F'_k)$, $1\leq k\leq\ep$  is increasing, and the sequence  $\min(F_k)$, $1\leq k\leq\io$, is decreasing,
%\item or over all possible active filtrations in $G$, that is sequences of the above type where each minor in the formula is loopless and $2-connected$.
%\end{itemize}
\end{thm}
%
%\com{*enoncer-comme-theoreme?*}

\begin{proof}
By Lemma \ref{lem:connected-filtration-beta}, we directly have that the sum  over all filtrations and the sum over all connected filtrations yield the same result.
The result where the sum is over all connected filtrations  of $G$ is 
exactly the enumerative translation of Theorem \ref{th:dec-ori}.
%directly given by Theorem \ref{th:dec-ori}, which constitutes a bijective proof of the result.
More precisely, consider the set of orientations $\G$ with dual-activity $\io$ and activity $\ep$, whose cardinality is $o_{\io,\ep}$.
This set  bijectively corresponds to the set 
$\biguplus \bigl\{ \G  \mid  \G(F_k)/F_{k-1},\ \ 1\leq k\leq\io,$
$\hbox{bipolar with respect to } \min (F_k\setminus F_{k-1}),$
$\hbox{ and }\ 
\G(F'_{k-1})/F'_{k}, \ \ 1\leq k\leq \ep,$ $\hbox{ cyclic-bipolar with respect}$ 
$\hbox{to } \min (F'_{k-1}\setminus F'_{k})\  \bigr\}$
where the union is over all connected filtrations of $G$ with fixed $\io$ and $\ep$.
The cardinality of each part of this set is obviously 
$\Bigl(\prod_{1\leq k\leq \io}
2.\beta \bigl( G(F_k)/F_{k-1}\bigr)\Bigr)
 \ \Bigl(\prod_{1\leq k\leq \ep}2.\beta^* \bigl( G(F'_{k-1})/F'_{k}\bigr)\Bigr)$
 since $\beta$ counts half the number of bipolar or cyclic-bipolar orientations of a graph with more than two edges, $\beta=1$ for a graph with a single isthmus (which can happen for minors of type  $G(F_k)/F_{k-1}$), and $\beta^*=1$ for a graph with a single loop (which can happen for minors of type  $ G(F'_{k-1})/F'_{k}$).
To achieve the proof, we use that the coefficient $t_{\io,\ep}$ of the Tutte polynomial equals $o_{\io,\ep}/{2^{\io+\ep}}$, as shown in \cite{LV84a} (see Section~\ref{subsec:orientation-activity}).
%By \cite{LV84a} (see Section \ref{subsec:orientation-activity}), the coefficient $t_{\io,\ep}$ of the Tutte polynomial equals $o_{\io,\ep}/{2^{\io+\ep}}$, hence the result.
% 
%The second result, where the sum is over all abstract filtrations  of $E$, is equivalent to the first by Lemma \ref{lem:dec-seq-equiv}.
\end{proof}

\begin{remark}
\rm
%From Definitions \ref{def:act-seq-dec} and \ref{def:active-minors}, Lemma \ref{lem:induction-dec-seq} and Definition \ref{def:active-minors}, 
From Proposition \ref{prop:pty-active-minors}, we already have that any orientation $\G$ can be decomposed into  bipolar/cyclic-bipolar  minors induced by a (connected) filtration of $G$ (the active one of $\G$). Then we could directly deduce a weaker version of Theorem \ref{th:dec-ori} with a union instead of a disjoint union, and  a weaker version of Theorem \ref{th:tutte} with an inequality instead of an equality.
It is the uniqueness result of Proposition \ref{prop:unique-dec-seq} that allows us to state 
Theorems \ref{th:dec-ori} and \ref{th:tutte} as they~are.%
\end{remark}

%\red{commetnaire avant}

\begin{cor}[\cite{EtLV98, KoReSt99}]
\label{cor:convolution}
Let $G$ be a graph. We have
$$t(G;x,y)=\sum t(G/F_c;x,0)\ t(G(F_c);0,y)$$ where the sum can be either over all subsets $F_c$ of $E$, or over all cyclic flats  $F_c$ of $G$.
\end{cor}

\begin{proof}
By fixing $y=0$ in Theorem \ref{th:tutte}, we get

\centerline{$t(G;x,0)= \ \ \sum \ \ \Bigl(\prod_{1\leq k\leq \io}
\beta \bigl( G(F_k)/F_{k-1}\bigr)\Bigr)
 \  {x^\io}$}
 
\noindent where the sum is over all connected filtrations of type $\emptyset=F'_0= F_c = F_0\subset \ldots\subset F_\io=E$ of $G$.
By fixing $x=0$, we get

\centerline{$t(G;0,y)= \ \ \sum \ \ 
 \ \Bigl(\prod_{1\leq k\leq \ep}\beta^* \bigl( G(F'_{k-1})/F'_{k}\bigr)\Bigr)\ {y^\ep}$}
 
\noindent where the sum is over all connected filtrations of type $\emptyset=F'_\ep\subset \ldots\subset F'_0= F_c = F_0=E$ of $G$.
For a given cyclic flat $F_c$ of $G$, pairs of connected filtrations of $G/F_c$ and $G(F_c)$ of the  above type, respectively for $y=0$ and $x=0$, obviously correspond to the connected filtrations of $G$ of type $(F'_\ep, \ldots, F'_0,$ $F_c , F_0, \ldots, F_\io)$ involving this $F_c$ (note that the subset $F_c$ in a connected filtration has to be a cyclic flat of the graph, since it is the cyclic flat of some active partition, see also the similar observation below Definition \ref{def:act-part} in terms of the active filtration).
Then, by decomposing the sum in Theorem \ref{th:tutte} as $\sum_{F_c}\sum_{i,j}  \Pi_{1\leq k\leq \io}\dots\Pi_{1\leq k\leq \ep}\dots$, 
we get the formula
$t(G;x,y)=\sum t(G/F_c;x,0)\ t(G(F_c);0,y)$ where the sum is over all cyclic flats  $F_c$ of $G$.
\emevder{preuve ci-dessus pas super bien dit, vaut le coup d'ameliorer ???}%
If $F_c$ is not a cyclic flat, then either $G/F_c$ has a loop or $G(F_c)$ has an isthmus, implying that the corresponding term in the sum equals zero.
\end{proof}

%\begin{color}{red}
%For information, let us give an alternative direct proof of Corollary \ref{cor:convolution} using orientation activities and the fact that each edge belongs either to a directed cycle or a directed cocycle. It is available for graphs (and directly generalizable to oriented matroids, but not to non-orientable matroids).
%
%\begin{proof}[An alternative direct proof of Corollary \ref{cor:convolution}]
%
%\end{proof}
%
%\end{color}

The formula in Corollary \ref{cor:convolution} is called  {\it convolution formula for the Tutte polynomial} in \cite{KoReSt99}, and it is also the enumerative translation of the bijection given in \cite{EtLV98}.
For information, we mention that Corollary \ref{cor:convolution} can also be proved very shortly and directly for graphs, using the enumeration of orientation activities formula of the Tutte polynomial from \cite{LV84a} (see Section~\ref{subsec:orientation-activity}), along with the fact that a digraph $\G$ can be uniquely decomposed into an acyclic digraph $\G/F_c$ and a strongly connected digraph $\G(F_c)$ where $F_c$ is the union of directed cycles of $\G$ (however, such a proof does not generalize  to non-orientable matroids). 
The reader can complete details or find them in \cite{GiChapterOriented}. 
\emevder{verfiier si prevue donnee dans chpaitre comme prevu + faut il mettre cette remarque ? peut etre juste "see also chapter for a short direct proof using..."}%
%We leave details to the reader (or see \cite{GiChapterOriented})
%\new{ref a chapitre ?}
%
Let us also mention that an algebraic proof of the formula in Theorem \ref{th:tutte} could be obtained using matroid set functions, a technique introduced in \cite{La97}, 
according to its author \cite{Lass-perso}.
\subsection{Activity classes in the set of all orientations of an ordered graph - Tutte polynomial expansion in terms of four refined orientation activities}
%\section{Four-parameter expansion of the Tutte polyomial}
\label{subsec:act-classes}

%\red{The conent of this secion is generalized to oriented matroid perspectives in \cite{Gi18}.}

Let us continue to build on active partitions. We define the notion of activity classes of orientations of an ordered graph. They are a central concept in this paper, and they will be put in bijection with spanning trees by the canonical active bijection in Section \ref{subsec:alpha-def-decomp} (as in \cite{GiLV05}).
Next, we develop this notion to derive further structural and enumerative results, which are interesting on their own and will be used later for the refined active bijection in Section \ref{subsec:refined}.
Let us mention that the whole content of this section is generalized to oriented matroid perspectives in \cite{Gi18}.

\begin{prop}
\label{prop:act-classes} 
Let $\G$ be a directed graph on a linearly ordered set of edges $E$, 
with $\io$ dual-active edges and $\ep$ active edges.
%with active partition $E=A'_\ep\uplus \ldots \uplus A'_1\uplus A_1\uplus\ldots\uplus A_\io$.
The $2^{\io+\ep}$ orientations of $G$ obtained by reorienting any union of parts of the active partition of $\G$ have the same active filtration/partition as $\G$ (and hence also the same active and dual-active  edges, and the same active minors up to taking the opposite).
%\red{in terms of decomposign sequences ?}
\end{prop}

%\red{definir ici activity class, et mentionner partiiton de $2^E$ et referencer pour section 6}

%\red{can be proved directly also, detaller preuve aleternaive ou pas ?}

%\new{verifier proof --- citer extension to persepctives}

\begin{proof}
%We give two proofs, one using previous results, and the other not.
%
The result is not difficult to prove directly from Definition \ref{def:act-seq-dec}: consider $A$  the union of all directed cycles (or cocycles) of $\G$ whose smallest edge is greater than a given edge $e$, and prove that every union of all directed cycles (or cocycles) whose smallest edge is greater than any edge is the same in $\G$ and $-_A\G$. 
We leave this proof as an exercise (see \cite{AB2-b} for a similar short proof in oriented matroid terms, see \cite{Gi18} for a detailed more general proof in oriented matroid perspectives).
\eme{Let $C$ be a  directed cycle of $\G$, with $\min(C)=a$. If $a\geq e$ then $C\subseteq A$ then $C$ is a directed cycle of $-_A\G$.
If $a<e$ then consider $A$ as the composition of cycles...}%
Alternatively, the result can also be seen as a direct corollary of Theorem \ref{th:dec-ori} or Proposition \ref{prop:unique-dec-seq}. Indeed, reorienting a union of parts of the active partition of $\G$ implies reorienting completely some of the active minors 
%involved in the decomposition 
of $\G$. 
%provided by Theorem \ref{th:dec-ori}. 
Then, for the resulting orientation, the resulting minors still satisfie the property of Proposition \ref{prop:unique-dec-seq}, hence the active filtration is the same as that of $\G$.
\end{proof}

%\red{mettre ici def de act class + mentinner aprtition des orietnations}

\begin{definition}
\label{def:act-class}
We call \emph{(orientation) activity class of }$\G$ the set of all orientations of $G$ obtained by reorienting any union of parts of the active partition of $\G$. 
\end{definition}

An illustration %of activity class
 is given in Figure \ref{EG:fig:K4-dec}. %is given in the left part of forthcoming Figure \ref{EG:fig:K4-iso}.
From Proposition \ref{prop:act-classes}, we directly get the following result.

\begin{prop}
\label{prop:partition-into-act-classes}
Activity classes of orientations of $G$ 
%form a 
partition  
%of
the set of orientations of $G$:
\ms

\noindent $\displaystyle\bigl\{\ \text{orientations of }G\ \bigr\}=\biguplus_{\substack{\text{activity classes}\\ \text{of orientations of }G\\ \text{\tiny (one $-_A\G$ chosen in each class)}}}
\Bigl\{\  2^{|O(-_A\G)|+|O^*(-_A\G)|} \ \
\substack{\text{\small orientations obtained by}\\
\text{\small  active partition reorienting}}
\ \Bigr\}.$

\vspace{-4mm}
\qed
\end{prop}

\begin{definition}
\label{def:active-fixed}
Let $\G$ be a digraph on a linearly set of edges $E$ (thought of as a \emph{reference orientation} of the graph $G$). 
%Let $\G$ be an orientation of $G$ .
%Let $G$ be a graph on a linearly set of edges $E$. 
%Let $\G$ be an orientation of $G$ (thought of as a reference orientation).\emevder{un peu lourd... ptet mettre juste Let $\G$}
Let $A\subseteq E$. 
The digraph $-_A\G$ is said to be \emph{active-fixed}, resp. \emph{dual-active fixed}, (with respect to $\G$) if the directions of all active, resp. dual-active, edges of $-_A\G$ agree with their directions in $\G$, that is if $O(-_A\G)\cap A=\emptyset$, resp. $O^*(-_A\G)\cap A=\emptyset.$
\end{definition}

From Propositions  \ref {prop:act-classes} and \ref{prop:partition-into-act-classes}, the Tutte polynomial formula in terms of orientations activities (Section \ref{subsec:orientation-activity}), and the usual cyclic/acyclic decomposition, we derive the following counting results.
\emevder{reresentatives dans corollarie ?}

\begin{cor}
\label{cor:enum-classes}
Let $G$ be an ordered graph.
The number of activity classes of orientations of $G$ with activity $i$ and dual activity $j$ equals $t_{i,j}$.

Let $\G$ be a reference orientation of $G$.
Each activity class of $G$ contains exactly one orientation of $G$ which is
%and only one 
active-fixed and dual-active-fixed (w.r.t. $\G$).
The number of such orientations of $G$ with activity $i$ and dual activity $j$ thus equals $t_{i,j}$.

In this way, we also obtain the enumerations given by Table~\ref{table:enum-classes}.
\qed
\end{cor}

\begin{table}[H]
\def\interligne{&\\[-11pt]}%
\vspace{-5mm}
\begin{center}
\begin{tabular}{|l|c|}
\hline
\interligne
%\multicolumn{1}{c}
{\bf orientations of $G$} /
%\multicolumn{1}{c}
{\bf activity classes of $G$}   
& \multicolumn{1}{|c|}{\bf number}\\
\hline
\interligne
 active-fixed and   dual-active-fixed / all
&  $t(G;1,1)$\\
 \interligne
  acyclic and dual-active-fixed / acyclic
& $t(G;1,0)$\\
 \interligne
 active-fixed and strongly connected / strongly connected
& $t(G;0,1)$\\
\interligne
 active-fixed (/ non-applicable) & $t(G;2,1)$\\
\interligne
  dual-active-fixed (/ non-applicable) & $t(G;1,2)$\\
\hline
\end{tabular}  
\end{center}
\vspace{-5mm}
\caption{Enumeration of certain orientations based on  representatives of activity classes (Corollary  \ref{cor:enum-classes}).
%\purp{These classes are explicitly addressed in oriented matroids and related to certain subsets int erms of subset activities in \cite{ABG2, AB2-b}.}
 }
 \label{table:enum-classes}%
\end{table}

\vspace{-2mm}
Now, let us refine orientation activities into four parameters
%Let us now go a step further in construction. We refine orientation activities into four parameters
w.r.t. a reference orientation~$\G$. 
%, in order to catch the positions orientations in their activity classes. 
%These parameters have been introduced in \cite{LV12}.

%\red{deja dit dans intro de section}

\eme{dessous en commentaire DEF avec cardinaux en plus}
%\begin{definition} %[\cite{LV12}\red{ou pas ? ou footnote ? verifier def dans LV12}]
%\label{def:gene-act-ori}
%Let us define:
%% as follows:
%%
%%\vspace{-0.3cm}
%\begin{eqnarray*}
%   \Theta_\G(A)=O(-_A\G)\s A, &\hspace{2cm} &\theta_\G(A)=|\Theta_\G(A)|, \\
% \bar\Theta_\G(A)=\Theta_\G(E\s A)=O(-_A\G)\cap A, & & \bar\theta_\G(A)=|\bar\Theta_\G(A)|.  
%\end{eqnarray*}
%%
%%\vspace{-0.5cm}
%%
%Hence we have $O(-_A\G)=\Theta_\G(A)\uplus\bar\Theta_\G(A)$ and 
%$o_\G(A)=\theta_\G(A)+\bar\theta_\G(A)$.
%%
%Dually, we define:
%%
%%\vspace{-0.3cm}
%\begin{eqnarray*}
%   \Theta^*_\G(A)=\Theta_{\G^*}(A)=O^*(-_A\G)\s A, &\hspace{2cm} &\theta^*_\G(A)=|\Theta^*_\G(A)|,\\
%\bar\Theta^*_\G(A)=\bar\Theta_{\G^*}(A)=O^*(-_A\G)\cap A,& & \bar\theta^*_\G(A)=|\bar\Theta^*_\G(A)|. 
%\end{eqnarray*}
%%\index{activities!orientation-activities}
%%\index{activities!subset orientation-activities}
%%\vspace{-0.2cm}
%%$$\Theta^*_M(A)=\Theta_{M^*}(A)=O^*(-_AM)\s A,$$
%%$$\bar\Theta_M(A)=\bar\Theta_{M^*}(A)=O^*(-_AM)\cap A,$$
%%$$\theta^*_M(A)=|\Theta^*_M(A)|,$$
%%$$\bar\theta^*_M(A)=|\bar\Theta^*_M(A)|.$$
%%
%%\vspace{-0.5cm}
%%
%Hence we have $O^*(-_A\G)=\Theta^*_\G(A)\uplus\bar\Theta^*_\G(A)$ and 
%$o^*_\G(A)=\theta^*_\G(A)+\bar\theta^*_\G(A)$.
%%\end{notation}
%\end{definition}

\begin{definition} %[\cite{LV12}\red{ou pas ? ou footnote ? verifier def dans LV12}]
\label{def:gene-act-ori}
Let $\G$ be an ordered directed graph (reference orientation of $G$). We define: %\red{agree, etc.}
% as follows:
%
\vspace{-1mm}
\begin{eqnarray*}
   \Theta_\G(A)&=&O(-_A\G)\s A,  \\
 \bar\Theta_\G(A)&=&O(-_A\G)\cap A, \\ 
   \Theta^*_\G(A)&=&O^*(-_A\G)\s A, \\
\bar\Theta^*_\G(A)&=&O^*(-_A\G)\cap A. 
\end{eqnarray*}
\eme{dessous avec complementaires en plus}
%\begin{align*}
%   \Theta_\G(A)&&& &=& &&&O(-_A\G)\s A,  \\
% \bar\Theta_\G(A)&&=&&\Theta_\G(E\s A)&&=&&O(-_A\G)\cap A, \\ 
%%\end{eqnarray*}
%%%
%%%\vspace{-0.5cm}
%%%
%%%
%%%
%%Dually, we define:
%%%
%%%\vspace{-0.3cm}
%%\begin{eqnarray*}
%   \Theta^*_\G(A)&&=&&\Theta_{\G^*}(A)&&=&&O^*(-_A\G)\s A, \\
%\bar\Theta^*_\G(A)&&=&&\bar\Theta_{\G^*}(A)&&=&&O^*(-_A\G)\cap A. 
%\end{align*}
%
%\index{activities!orientation-activities}
%\index{activities!subset orientation-activities}
%\vspace{-0.2cm}
%$$\Theta^*_M(A)=\Theta_{M^*}(A)=O^*(-_AM)\s A,$$
%$$\bar\Theta_M(A)=\bar\Theta_{M^*}(A)=O^*(-_AM)\cap A,$$
%$$\theta^*_M(A)=|\Theta^*_M(A)|,$$
%$$\bar\theta^*_M(A)=|\bar\Theta^*_M(A)|.$$
%
%\vspace{-0.5cm}
%
Hence we have $O(-_A\G)=\Theta_\G(A)\uplus\bar\Theta_\G(A)$ 
%and $o_\G(A)=\theta_\G(A)+\bar\theta_\G(A)$.
%Hence we have
and $O^*(-_A\G)=\Theta^*_\G(A)\uplus\bar\Theta^*_\G(A)$.
%and  $o^*_\G(A)=\theta^*_\G(A)+\bar\theta^*_\G(A)$.
%\end{notation}
\end{definition}

%\red{met on notation petit theta etc. ?}
%\bs

These parameters can be seen as situating a reorientation of $\G$ in its activity class. 
Indeed, the representative $-_A\G$ of its activity class which is active-fixed and dual-active fixed (Corollary \ref{cor:enum-classes}) satisfies
%Notice that in each activity class, 
$\bar\Theta_\G(A)=O(-_A\G)\cap A=\emptyset$ and $\bar\Theta^*_\G(A)=O^*(-_A\G)\cap A=\emptyset,$
and the other orientations in the same activity class correspond to other possible values of $\bar\Theta_\G(A)\subseteq O(-_A\G)$ and $\bar\Theta^*_\G(A)\subseteq O^*(-_A\G)$.
A way of understanding the role of the reference orientation $\G$ is that it breaks the symmetry in each activity class, so that its boolean lattice structure can be expressed relatively to its aforemetioned representative.
This feature will be taken up in  Section \ref{subsec:refined} on the refined active bijection, in connection with the similar four parameters for spanning trees from Section \ref{prelim:subset-activities}.
Let us also mention that, for suitable orderings
(roughly when all branches of the smallest spanning tree are increasing),
 unique sink acyclic orientations are also representatives of their activity classes, see \cite[Section 6]{GiLV05}.

Finally, we derive the following Tutte polynomial expansion formula in terms of these four parameters.
A (technical) proof for Theorem \ref{EG:th:expansion-orientations} below is proposed in 
%\cite[Theorem 2.1]{LV12} 
the preprint \cite{LV12}
by deletion/contraction in the more general setting of oriented matroid perspectives.
%%
%%
%\new{+ direct proof in article soumis}
%\red{See also \cite{Gi18} for a short proof in terms of active partitions for oriented matroid perspectives.}
%\red{???OU???}
%\red{We mention that this short proof generalizes to oriented matroid perspectives \cite{Gi18}.}
%%
As announced in \cite{LV12},
this theorem can be proved by means of the above construction on activity classes (this theorem can also be seen as a direct corollary of the similar formula for subset activities from Theorem  \ref{EG:th:Tutte-4-variables} and the refined active bijection from Theorem \ref{EG:th:ext-act-bij}).
We give this short proof below for completeness of the paper, though it is a translation of the proof given in \cite{Gi18} for oriented matroid perspectives.%
%
%We give below a short and direct proof of this theorem using the partition of the set of orientations into activity classes (that is Proposition \ref{prop:act-classes}), choosing a representative in each class as done above,
%and the Tutte polynomial formula in terms of orientation activities.

%\red{ATTENTION P pour dual/int et Q pour primal/ext !!!}

% \red{***********OUGJENSUIS}
 
\begin{thm}[\cite{LV12, Gi18}] %[\cite{LV12}, \cite{Gi18}] %[{\cite[Theorem 2.1]{LV12}}\cite{ABG2}\cite{AB2}\eme{TO BE UPDATED}]
%\begin{thm}[{\cite[Theorem 2.1]{LV12}}\cite{ABG2}\cite{AB2}\eme{TO BE UPDATED}]
\label{EG:th:expansion-orientations}
%Let $M$ be an oriented matroid on a linearly ordered set $E$. We have
Let $G$ be a graph on a linearly ordered set of edges $E$, and $\G$ be an orientation of $G$. We have
%
%$$T(G;x+u,y+v)=\sum_{A\subseteq E} x^{\theta^*_\G(A)}u^{\bar\theta^*_\G(A)}y^{\theta_\G(A)}v^{\bar\theta_\G(A)}.$$
%
\eme{dessous en commanriter enonce vec cardianux}
%\begin{large}
%$$T(G;x+u,y+v)=\sum_{A\subseteq E} \ x^{\theta^*_\G(A)}\ u^{\bar\theta^*_\G(A)}\ y^{\theta_\G(A)}\ v^{\bar\theta_\G(A)}.$$
%\end{large}
\begin{large}
$$T(G;x+u,y+v)=\sum_{A\subseteq E} \ x^{|\Theta^*_\G(A)|}\ u^{|\bar\Theta^*_\G(A)|}\ y^{|\Theta_\G(A)|}\ v^{|\bar\Theta_\G(A)|}.$$
\end{large}
\end{thm}

%\red{ENLEVER INDICES $M$ POUR LA PREUVE}

\begin{proof}
The proof is obtained by a simple combinatorial transformation.
%, using the partition of the set of orientations into activity classes discussed above.
Let us start with  the right-hand side of the equality, where we denote $\theta^*_\G(A)$ instead of $|\Theta^*_\G(A)|$, etc., by setting:
\vspace{-1mm}
$$\displaystyle [Exp] = \sum_{A\subseteq E} x^{\theta^*_\G(A)}u^{\bar\theta^*_\G(A)}y^{\theta_\G(A)}v^{\bar\theta_\G(A)}.$$
%\noindent 

\vspace{-2mm}
Since $2^E$ is isomorphic to the set of orientations, which is partitioned into orientation activity classes of $\G$ (Proposition \ref{prop:partition-into-act-classes}), and by choosing a representative for each activity class which is active-fixed an dual-active-fixed (as discussed above), we get:
\vspace{-1mm}
$$[Exp] =  \sum_{\substack{\text{orientation activity classes of }G\\ \text{with one $-_A\G$ chosen in each class}\\ \text{such that $O(-_A\G)\cap A=\emptyset$ and $O^*(-_A\G)\cap A=\emptyset$}}}\ \ \sum_{\substack{-_{A'}\G\text{ in the }\\\text{class of }-_A\G}} x^{\theta^*_\G(A')}u^{\bar\theta^*_\G(A')}y^{\theta_\G(A')}v^{\bar\theta_\G(A')}$$
%\noindent The active partition of $-_A\G$ for $A\subseteq E$ can be denoted as
%$$\biguplus_{i\in O(-_A\G)}A_i\uplus \biguplus_{j\in O^*(-_A\G)}A^*_j.$$
%Let us recall \red{(ref Cor ?)} that every reorientation in the same activity class has the same (dual-)active elements, the same active partition.
%Then the associated reorientation activity class of $\G$ can be denoted $$cl(A)=\Biggl\{ A \triangle \Bigl[\biguplus_{i\in I}A_i\uplus \biguplus_{j\in I^*}A^*_j\Bigr]\ \ \mid \ \ I\subseteq O(A),\ I^*\subseteq O^*(A)\Biggr\}.$$

\vspace{-2mm}
As discussed above, when $-_{A'}\G$ ranges the activity class of $-_A\G$, 
$\bar\Theta_\G(A')$ and $\bar\Theta^*_\G(A')$  range subsets of 
$O(-_A\G)$ and  $O^*(-_A\G)$,  respectively. So, we get the following expression (where ``idem'' refers to the text below the first above sum), which we then transform using the binomial formula:%
%
%As specified above the theorem statement, let us choose $-_A\G$ to represent $cl(-_A\G)$ such that $O(-_A\G)\cap A=\emptyset$ and $O^*(-_A\G)\cap A=\emptyset$. For such $A$ we have also $O(-_A\G)\s A=O(-_A\G)$ and $O^*(-_A\G)\s A=O^*(-_A\G)$.
%Then, for $-_{A'}\G\in cl(-_A\G)$, we have by definition: 
%$\Theta_\G(A')=O(-_{A'}\G)\s A'$, and by the above choice of $A$ we get: $P=P\cap A'\subseteq O(-_{A}\G)=O(-_{A'}\G)$, hence we obtain: $$\Theta_\G(A')=O(-_{A}\G)\triangle P=O(-_{A}\G)\s P.$$
%Similarly, we have:
%$\bar\Theta_\G(A')=\bar O(-_{A}\G)\triangle P=P,$
%$\Theta^*_\G(A')= O^*(-_{A}\G)\triangle Q=O^*(-_{A}\G)\s Q,$
%and $\bar\Theta^*_\G(A')=\bar O^*(-_{A}\G)\triangle Q=Q.$
%\eme{ou bien detailler en remplacant $A'$ par $A\triangle...$ dans def de $\Theta$}%
%Hence we get:
%
\begin{eqnarray*}
[Exp] & = & \sum_{\text{idem}}\ \ \sum_{\substack{P\subseteq O^*(-_A\G)\\ Q\subseteq O(-_A\G)}} x^{\mid O^*(-_A\G)\s P\mid}u^{\mid P\mid}y^{\mid O(-_A\G)\s Q\mid}v^{\mid Q\mid}\\
&=&  \sum_{\text{idem}} \ \
\Biggl(\sum_{P\subseteq O^*(-_A\G)} 
x^{\mid O^*(-_A\G)\s P\mid}
u^{\mid P\mid}\Biggr)
\Biggl(\sum_{Q\subseteq O(-_A\G)} 
y^{\mid O(-_A\G)\s Q\mid}
v^{\mid Q\mid}\Biggr)
\\
&=&  \sum_{\text{idem}}\ \ (x+u)^{|O^*(-_A\G)|} (y+v)^{|O(-_A\G)|}
\end{eqnarray*}
%

%\new{binomial formula a simplifier comme dans AB2}

%Now we can simply rewrite this expression as follows (where we denote for short $o(A)=|O(-_A\G)|$ and $o^*(A)=|O^*(-_A\G)|$, and where the sum $\sum_{\text{idem}}$ is the the same as the first sum of the above equality):
%\begin{eqnarray*}
%%$$[Exp] =  \sum_{\text{idem}} \sum_{\substack{I\subseteq O(A)\\ I^*\subseteq O^*(A)}} x^{\mid O^*(A)\s I^*\mid}u^{\mid I^*\mid}y^{\mid O(A)\s I\mid}v^{\mid I\mid}$$
%%
%[Exp]&=&  \sum_{\text{idem}} \Biggl(\sum_{P\subseteq O(-_A\G)} y^{\mid O(-_A\G)\s P\mid}v^{\mid P\mid}\Biggr)\Biggl(\sum_{Q\subseteq O^*(-_A\G)} x^{\mid O^*(-_A\G)\s Q\mid}u^{\mid Q\mid}\Biggr)\\
%%
%&=& \sum_{\text{idem}} \Biggl(\sum_{i=0\ldots o(A)}\sum_{\substack{P\subseteq O(-_A\G)\\ \mid P\mid=i}} y^{o(A)-i}v^{i}\Biggr)\Biggl(\sum_{i=0\ldots o^*(A)}\sum_{\substack{Q\subseteq O^*(-_A\G)\\ \mid Q\mid=i}} x^{o^*_\G(A)-i}u^{i}\Biggr)\\
%%
%&=&  \sum_{\text{idem}} \Biggl(\sum_{i=0\ldots o(A)} y^{o(A)-i}v^{i}\sum_{\substack{P\subseteq O(-_A\G)\\ \mid P\mid=i}}1\Biggr)\Biggl(\sum_{i=0\ldots o^*(A)} x^{o^*_\G(A)-i}u^{i}\sum_{\substack{Q\subseteq O^*(-_A\G)\\ \mid Q\mid=i}} 1\Biggr)\\
%%
%&=&  \sum_{\text{idem}} \Biggl(\sum_{i=0\ldots o(A)} y^{o(A)-i}v^{i}
%\binom {i} {o(A)}\Biggr)\Biggl(\sum_{i=0\ldots o^*(A)} x^{o^*_\G(A)-i}u^{i}\binom {i} {o^*(A)}\Biggr)\\
%%
%&=&  \sum_{\text{idem}} (y+v)^{o(A)}(x+u)^{o^*(A)}\\
%%
%\end{eqnarray*}

%\red{Observe that this particular equality is interesting on its own, we also put it in corollary below???? bor evident depuis tutte d'en bas...} Let us write it completely:

%Finally, 
\vspace{-2mm}
Since the activity class of $-_A\G$ has ${2^{|O(-_A\G)|+|O^*(-_A\G)|}}$ elements with the same orientation activities,
we have (denoting for short $o(A)=|O(-_A\G)|$ and $o^*(A)=|O(-_A\G)|$):
% we simply transform this equality as follows:
%\vspace{-1cm}
\begin{eqnarray*}
%   x + 4 & =  & 0 \\
%   8 - y & =  & 0
%
[Exp] & = & 
%\sum_{\substack{cl(A)\\ \text{reorientation activity classes of }\G\\ \text{with one $A$ chosen in each class}\\\text{such that $O(A)\cap A=\emptyset$ and $O^*(A)\cap A=\emptyset$}}} (y+v)^{o(A)}(x+u)^{o^*(A)}\\
%
%
%&=& 
\sum_{\text{idem}}\ \ {1\over{2^{o(A)+o^*(A)}}}\ \sum_{\substack{-_{A'}\G\text{ in the  class of }-_A\G}}\ (x+u)^{o^*(A')}(y+v)^{o(A')}\\
&=& \sum_{\text{idem}}\ \ \sum_{\substack{-_{A'}\G\text{ in the  class of }-_A\G}}\ {\Bigl({x+u\over 2}\Bigr)}^{o^*(A')}{\Bigl({y+v\over 2}\Bigr)}^{o(A')}\\
&=&  \sum_{A\subseteq E}\ {\Bigl({x+u\over 2}\Bigr)}^{o^*(A)}{\Bigl({y+v\over 2}\Bigr)}^{o(A)}\\
&=&  t(G;x+u,y+v)
\end{eqnarray*}
\vspace{-2mm}
using at the end the orientation activity enumeration formula from \cite{LV84a} recalled %as (\ref{eq:orientation-activities}) 
in Section \ref{subsec:orientation-activity}.
\end{proof}

\begin{remark}
\rm
Numerous Tutte polynomial formulas can be obtained from Theorem \ref{EG:th:expansion-orientations}, for instance by replacing variables ($x$, $u$, $y$, $v$) with $(x/2,x/2,y/2,y/2)$, 
%with 
or $(x+1,-1,y+1,-1)$, 
%with 
or $(2,0,0,0)$. One can also obtain expressions for Tutte polynomial derivatives.
%These results, more derived results of this flavour, and a detailed example for this set of results are given in \cite{LV12} (see also \cite{AB2}).
Such formulas, and a detailed example for this set of formulas, are given in \cite{LV12} (see also \cite{Gi18, GiChapterOriented, AB2-b}).
\emevder{verifier refs du "see also", a mettre ou pa ?}%
\end{remark}

%Let us mention the following corollaries (proofs are direct from  Theorem \ref{EG:th:expansion-orientations}, use Taylor formula for the third). 
%and let us give also an expression for derivatives which is a counterpart of Theorem \ref{EG:th:derivatives-perspectives} (proof is direct from Theorem \ref{EG:th:expansion-orientations-perspectives} by Taylor formula).

\eme{dessous en commentaire intro de corollaires enleves remis a LV12}%
%Various formulas can be obtained from this result. 
%Let us mention the following corollaries (proofs are direct from  Theorem \ref{EG:th:expansion-orientations}, use Taylor formula for the third). 
%%and let us give also an expression for derivatives which is a counterpart of Theorem \ref{EG:th:derivatives-perspectives} (proof is direct from Theorem \ref{EG:th:expansion-orientations-perspectives} by Taylor formula).
%These results, more derived results of this flavour, and a detailed example for this set of results are given in \cite{LV12}.

%\red{corollarires dessus pas forc?ment utiles ? citer ? oui car LV12 unpublished ?}

\eme{dessous en commentaire corollaires enleves remis a LV12}%
\section{The three levels of the active bijection of an ordered graph}
%\section{Definitions and properties of the three levels of the active bijection of an ordered graph}

%\section{Construction of the active bijection 
%%between orientation activity classes and spanning trees, 
%and of a refined active bijection between orientations and edge subsets 
%%preserving the four-parameter expansion
%}
\label{sec:edges}

In this section, we give the definitions and main properties of the three levels of the active bijection, as depicted in the diagram of Figure \ref{fig:diagram} and as globally described in Section \ref{sec:intro}. 
%See Section \ref{sec:intro} for a global introduction.
We focus on the construction from digraphs/orientations to spanning trees/subsets. The inverse direction is summarized in Section \ref{sec:spanning-trees}.

\subsection{The uniactive bijection -  The fully optimal spanning tree of an ordered bipolar digraph}
%\subsection{The canonical uniactive bijection -  The fully optimal spanning tree of an ordered bipolar digraph  (reminders from \cite{GiLV05,AB1})}
\label{subsec:fob}

%\new{ --- az mettre partout !!!}

%\red{$1/0$ ??? canonical uniactive bijection ???}

%\red{active duality (assuming $|E|>1$)}

%\red{active duality diagram?}

Let us  recall or reformulate the precise definitions and main properties of the uniactive case of the active bijection, subject of \cite{GiLV05,AB1}. See Section \ref{sec:intro} for a global introduction with %references and 
related results.

%

%\red{IDEM DANS HANDOBBOK et AB2, mettre alternatives equivalentes}

\begin{definition}
\label{def:acyc-alpha}
%Let $\G=(V,E)$ be an ordered digraph, bipolar 
Let $\G=(V,E)$ be a directed graph, on a linearly ordered set of edges, 
which is bipolar 
with respect to the minimal element $p$ of $E$. The \emph{fully optimal spanning tree} $\alpha(\G)$ of $\G$ is the unique spanning tree $T$ of $G$ such~that:
\smallskip
%\red {itemize ???}

%\noindent 
\bul for all $b\in T\setminus \{p\}$, the directions (or the signs) of $b$ and $\min(C^*(T;b))$ are opposite in $C^*(T;b)$;
%\bul for all $b\in T\setminus p$, the signs of $b$ and $\min(C^*(T;b))$ are opposite in $C^*(T;b)$
%\smallskip

%\noindent 
\bul for all $e\in E\s T$, the directions (or the signs) of $e$ and $\min(C(T;e))$ are opposite in $C(T;e)$.
%\bul for all $e\not\in T$, the signs of $e$ and $\min(C(T;e))$ are opposite in $C(T;e)$.
%\medskip
%\red{verifier pour $p$ premeir cocirc est bien positif OK par orhtogonalite des que $p\in T$... verifier que $p\in T$ necessaire}
\end{definition}

Note that a directed graph and its opposite have the same fully optimal
%are mapped onto the same 
spanning tree.
Let us mention that the above definition has equivalent formulations involving unions of successive fundamental cycles/cocycles \cite{GiLV05, AB1} (recalled in \cite[Section \ref{ABG2LP-subsec:prelim-fob}]{ABG2LP}).\emevder{verifier ref}
A detailed illustration of the above definition on a bipolar orientation of $K_4$ is given in \cite{AB2-b} (see also \cite{GiLV05} for another example).
\ss

The existence and uniqueness of the fully optimal spanning tree  
is the main result of \cite{GiLV05, AB1}. This is a deep and tricky result with several intepretations, mainly geometrical  (see \cite[Section 5]{AB2-b} for a recap\emeref{verifier}). It is equivalent to the key theorem below, before which we give  dual definitions extending the above one, and after which we give further precisions from the constructive viewpoint.%

%Thus, the mapping $\alpha$ is well-defined and it is shown that it provides a bijection between all  bipolar reorientations of $G$, with respect to $p$ with fixed orientation, and all uniactive internal spanning trees of $G$.
%\red{ou Therefore, we get the following key theorem.}

%\medskip

\begin{definition}[Dual and very similar to Definition \ref{def:acyc-alpha}]
\label{def:cyc-alpha1}
Let $\G=(V,E)$ be a directed graph on a linearly ordered set of edges,
cyclic-bipolar with respect to the minimal element $p$ of $E$. 
%A first way to define the active mapping $\alpha(\G)$ can be obtained from the one for bipolar orientations by cycles/cocycles duality, as it can be expressed more generally in oriented matroids \cite{AB1}.
%A first way to define  $\alpha(\G)$ comes from cycles/cocycles duality.
%, as it can be expressed more generally in oriented matroids \cite{AB1}.
%
%In this way,  $\alpha(\G)$ is defined as the unique spanning tree $T$ of $G$ such~that:
We define  $\alpha(\G)$ as the unique spanning tree $T$ of $G$ such~that:
\smallskip

\bul for all $b\in T$,  the directions (or signs) of $b$ and $\min(C^*(T;b))$ are opposite in $C^*(T;b)$;
%\bul for all $b\in T\setminus p$, the signs of $b$ and $\min(C^*(T;b))$ are opposite in $C^*(T;b)$
\smallskip

%\bul for all $e\in E\s (T\cup\{p\})$, 
\bul for all $e\in (E\s T)\s \{p\}$, 
the directions (or signs) of $e$ and $\min(C(T;e))$ are opposite in $C(T;e)$.
%\bul for all $e\not\in T$, the signs of $e$ and $\min(C(T;e))$ are opposite in $C(T;e)$.
%
\end{definition}

\begin{definition}[equivalent to Definition \ref{def:cyc-alpha1}]
\label{def:cyc-alpha2}
Let $\G=(V,E)$ be a directed graph on a linearly ordered set of edges,
cyclic-bipolar with respect to the minimal element $p$ of $E$. We assume $|E|>1$.
Then $-_p\G$ has the property to be bipolar with respect to $p$, and we define $\alpha(\G)$ as:
\begin{equation}
\tag{Active Duality}
\label{eq:act_mapping_strong_duality}
\alpha(\G)=\alpha(-_p\G)\setminus \{p\}\cup \{p'\}
\end{equation}
% $$\alpha(\G)=\alpha(-_p\G)\setminus \{p\}\cup \{p'\}$$ 
where $p'$ is the smallest edge of $E$ distinct from $p$.
\end{definition}

%\red{def uniactive itnernal ?}
The equivalence of these two definitions is given by \cite[Theorem 5.3]{AB1}%
\footnote{Let us correct here an unfortunate typing error in \cite[Proposition 5.1 and Theorem 5.3]{AB1}. The statement has been given under the wrong hypothesis 
$T_{\min}=\{p<p'<\dots\}$ 
%$B_{\min}=\{p<p'<\dots\}$ 
instead of the correct one $E=\{p<p'<\dots\}$. Proofs are unchanged
(independent typo: in line 10 of the proof of Proposition 5.1, instead of $B'-f$, read $(E\setminus  B')\setminus\{f\}$).
In \cite[Section 4]{GiLV05}, the statement of the \ref{eq:act_mapping_strong_duality} property is correct.\emevder{*** ref a ERRATUM ***}
% An erratum has been published \com{***to be done***}.
}.
Only the second one was used in \cite{GiLV05}.
%With either of these two equivalent definitions, we get an extension of Theorem \ref{thm:bij-10} to cyclic-bipolar orientations: cyclic-bipolar orientations with respect to $p$ with fixed orientation are in bijection with uniactive external spanning trees %which extends Theorem \ref{thm:bij-10}, see 
%\cite{GiLV05, AB1}. 
%
%Then, if $\G$ is cyclic-bipolar with respect to $p$, we have the property that $$\alpha(\G)=\alpha(-_p\G)\setminus \{p\}\cup \{p'\}$$ where $p'$ is the smallest edge of $E$ distinct from $p$, and where $-_p\G$ has the property to be bipolar with respect to $p$.
%\medskip
%The equivalence of these two definitions
Let us observe that 
Definition \ref{def:cyc-alpha1} comes from Definition \ref{def:acyc-alpha} and cycle/cocycle duality.
%The cycles/cocycles duality property of the active mapping (different from the \ref{eq:act_mapping_strong_duality}) is available in oriented matroids as well.
Actually, for a planar ordered graph $\G$, assumed to be (cyclic-)bipolar w.r.t. the smallest edge, and a dual $\G^*$ of $\G$, we have: 
\vspace{-3mm}
\begin{equation}
\tag{Duality}
\label{eq:act_mapping_duality}
\alpha(\G)=E\s \alpha({\G}^*).
\end{equation}

\vspace{-2mm}
%$$\alpha(\G)=E\s \alpha({\G}^*).$$
%\bigskip
\noindent Then, the \ref{eq:act_mapping_strong_duality} property provided by Definition \ref{def:cyc-alpha2} 
means that Definitions \ref{def:acyc-alpha} and \ref{def:cyc-alpha1}
%, coming from Definition \ref{def:acyc-alpha} and cycle/cocycle duality, 
are compatible
with the canonical bijection between bipolar and cyclic-bipolar orientations (see Section \ref{subsec:orientation-activity}), and the canonical bijection between internal and external uniactive spanning trees (see Section \ref{subsec:prelim-sp-trees}), as %recalled in Section \ref{sec:prelim} and
 detailed in \cite[Section 4]{GiLV05}.
 Let us mention that the Active Duality property 
 %amounts to a strengthening of 
% strengthens
can also be seen as a strengthening of 
 linear programming duality, see \cite[Section 5]{AB1}.
%So, the property provided by Definition \ref{def:cyc-alpha2} 
%So, this property can be alternatively used to define $\alpha(\G)$ of a cyclic-bipolar orientation from the definition for bipolar orientations, as was done in \cite{GiLV05}.
%This property turns out to be a graph formulation of a more general duality property of the active bijection, which we call \emph{active duality property}, and which is notably a strengthening of linear programming duality \cite[Section 5]{AB1}.
We sum up these duality properties of $\alpha$ in the diagram of Figure \ref{fig:bounded-duality-diagram}.
%Notice that, by cycles/cocycles duality (not the same duality), we also have that if $\G$ is cyclic-bipolar with respect to $p$, then 
%\red{A VERIFIER ! a mettre ?}

\def\hadistance{3cm}
\def\vadistance{0cm}
\def\hbdistance{6cm}
\def\vbdistance{0cm}
\def\vdistance{2.1cm}

\begin{figure}[h]
\centering

\scalebox{0.75}{
\begin{tikzpicture}[->,>=triangle 90, thick,shorten >=1pt,auto, node distance=\hdistance,  thick,  
%main node/.style={rectangle,fill=blue!10,draw,font=\sffamily\large}
main node/.style={rectangle,font=\sffamily\large}
]

  \node[main node] (1) {\begin{tabular}{c}$\G$\\ \small (bipolar)\end{tabular}};
  \node[main node] (1a)[below left= \vadistance and \hadistance of 1] {\begin{tabular}{c}$\G^*$\\ \small (cyclic-bipolar)\end{tabular}};
  \node[main node] (1b)[below right= \vbdistance and \hbdistance of 1] {\begin{tabular}{c}$-_p\G$\\ \small (cyclic-bipolar)\end{tabular}};
   \node[main node] (1ab) [below left= \vadistance and \hadistance of 1b] {\begin{tabular}{c}$-_p\G^*$\\ \small (bipolar)\end{tabular}};

  \node[main node] (2) [below =\vdistance of 1] {\begin{tabular}{c}$T$\\ \small (uniactive internal)\end{tabular}};
  \node[main node] (2a) [below =\vdistance of 1a] {\begin{tabular}{c}$E\s T$\\  \small (uniactive external)\end{tabular}};
  \node[main node] (2b) [below =\vdistance of 1b] {\begin{tabular}{c} $T\s\{p\}\cup\{p'\}$\\ \small (uniactive external)\end{tabular}};
   \node[main node] (2ab)  [below =\vdistance of 1ab] {\begin{tabular}{c}  $(E\s T)\s\{p'\}\cup\{p\}$\\ \small (uniactive internal)\end{tabular}};

\path[every node/.style={font=\sffamily}]
	
	(1) edge [->, >=latex, line width=1mm] node [above left] 
	{\bf Duality}
	(1a)
	
	(1) edge [->, >=latex, line width=1mm] node [above right] 
	{\bf Active duality}
%	{\begin{tabular}{c}\bf Active \\ \bf Duality\end{tabular}}
	(1b)
		
	(1a) edge [->, >=latex]
	(1ab)

	(1b) edge [->, >=latex] 
	(1ab)
		
	(2) edge [->, >=latex] %node [above left] 
	%{\bf Duality}
	(2a)
	
	(2) edge [->, >=latex] %node [above right] 
	%{\bf Active duality}
	(2b)
		
	(2a) edge [->, >=latex]
	(2ab)

	(2b) edge [->, >=latex] 
	(2ab)
				
	(1) edge [->, >=latex, line width=1mm] node [left] 
	{$\alpha$}
	(2)
	
	(1a) edge [->, >=latex] node [left] 
	{$\alpha$}
	(2a)
	
	(1b) edge [->, >=latex] node [left] 
	{$\alpha$}
	(2b)
	
	(1ab) edge [->, >=latex] node [left] 
	{$\alpha$}
	(2ab)
;
\end{tikzpicture}
}
\caption{Commutative diagram of duality properties of the uniactive bijection in the planar case. 
It involves the usual duality ($\vec G^*$ is any planar dual of $\vec G$) and the active duality (Definition \ref{def:cyc-alpha2}).
%
%%%, which can be seen as a strengthening of linear programming duality, see \cite[Section 5]{AB1}). \red{bounded active biejction ? uniactive bijection?}
%
It generalizes to general graphs using the oriented matroid dual $M^*(\vec G)$ instead of $\vec G^*$ (and, beyond, to general oriented matroids \cite{AB1, AB2-b}).
}
\label{fig:bounded-duality-diagram}
\end{figure}
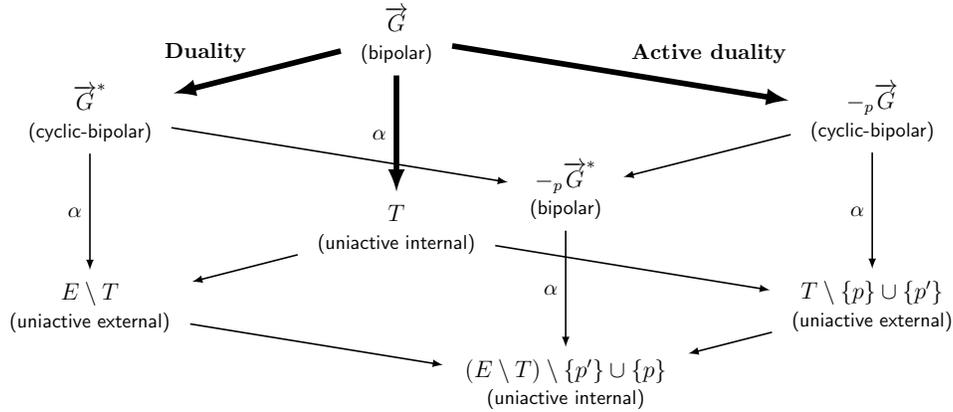

\begin{thm}[Key Theorem \cite{GiLV05, AB1}]
\label{thm:bij-10}
Let $G$ be a graph on a linearly ordered set of edges $E$ with 
smallest edge $p$. %$\min(E)=p$.

The mapping $\G\mapsto \alpha(\G)$ yields a bijection between all bipolar orientations of $G$ w.r.t. $p$ with fixed orientation for $p$ and all spanning trees of $G$ with internal activity $1$ and external activity $0$.

Also, it yields a bijection between all cyclic-bipolar orientations of $G$ w.r.t. $p$ with fixed orientation for $p$ and all spanning trees of $G$ with internal activity $0$ and external activity $1$.
\end{thm}

%\red{XXX Report 1: the next paagraph breaks up the flow}

%\blue{*bof*The case of bipolar digraphs is crucial and the above result is the combinatorial key of all forthcoming constructions.} 
The bijection provided by Theorem \ref{thm:bij-10} is called \emph{the uniactive bijection of the ordered graph $G$}.
This bijection was built in \cite{GiLV05,AB1} by its inverse, from uniactive internal spanning trees to bipolar orientations, provided by a single pass algorithm over the spanning tree, or equally (dually) over its complement.
Actually, it is easy to see that, given a uniactive spanning tree, one just has to choose orientations one by one in a single pass over $E$ (following the ordering) so as to build an orientation for which this spanning tree satisfies the criterion of Definition \ref{def:acyc-alpha} or \ref{def:cyc-alpha1}.
We recall this algorithm in Proposition \ref{prop:alpha-10-inverse} in Section \ref{subsec:basori}.
%It is obvious so that the criterion for edge directions from Definition \ref{def:acyc-alpha} is satisfied, and it works for uniactive internal or external spanning trees as well.
%
The problem of computing the direct image of a bipolar ordered digraph under $\alpha$ is not easy, and it is precisely addressed  in the companion paper \cite{ABG2LP}.
%A direct  computation of $\alpha$ for bipolar orientations is given  in the companion paper \cite{ABG2LP}, as an adaptation for graphs of a 
%more general construction by means of elaborations on linear programming \cite{AB3}.
An efficient but technical solution \cite[Section \ref{ABG2LP-sec:fob}]{ABG2LP} uses a linear number of minors, and consists in an adaptation for graphs of a 
more general construction by means of elaborations on linear programming \cite{GiLV09, AB3}.
%It is built directly more generally in \cite{AB3} by means of elaborations on linear programming. 
%
Alternatively, the uniactive bijection $\alpha$ can also be built by deletion/contraction, quite naturally but using an exponential number of minors, see Section \ref{subsec:alpha-10-ind} %Theorem \ref{thm:ind-10} 
and \cite[Section \ref{ABG2LP-sec:induction}]{ABG2LP} for details.
Let us emphasize that those two constructions of $\alpha$ do not give a proof of Theorem \ref{thm:bij-10}, or of the existence and uniqueness 
of $\alpha(\G)$
 in Definition \ref{def:acyc-alpha}: on the contrary, 
this fundamental result is used to prove their correctness.
\subsection{The canonical active bijection - The active spanning tree of an ordered digraph}
%\subsection{The canonical active bijection, and the active spanning tree of an ordered digraph}
\label{subsec:alpha-def-decomp}

%\red{pas active mapping}

%\red{$\alpha(G)/F$ by Observation \ref{obs:induced-dec-seq-ori}}

First, we give three equivalent definitions for the active spanning tree of an ordered digraph, consistently with the definition given in \cite{GiLV05}. Then, we give the main theorem stating the consistency and properties of the construction, yielding the canonical active bijection of an ordered graph. Then, we give its complete proof, that mainly makes the link between spanning trees and constructions of Section \ref{sec:tutte} for orientations. See Section \ref{sec:intro} for a global introduction.

An important feature of the canonical active bijection is that it preserves active partitions, meaning that the active partition of a digraph is the same as the active partition of its active spanning tree.
We will mention this second notion of active partition though it has not yet been defined in the paper. For convenience, we postpone this definition to Section \ref{subsec:dec-bases} %(actually 
(it can be defined by several ways, the proof of the main theorem of this section will prove at the same time this spanning tree decomposition, which is also shortly defined in \cite{GiLV05} and detailed in \cite{AB2-a}).%
\emevder{fin de derniere parenthese ?}
%Filtrations also allow us to define some  filtrations and active partitions associated with spanning trees, as described in \cite{GiLV05} and detailed in \cite{Gi02, AB2-a, AB2-b} (this result will not be proved in the present paper, focused on building $\alpha$ from orientations). 
%%
%These notions of active filtrations/partitions for  orientations and spanning trees are highly compatible in the sense that they relie upon the same increasing sequences of edges subsets (filtrations) and they are preserved by the active bijection (which is much stronger than preserving activities).
%

An interesting 
underlying feature, that we will not develop in this paper, is how the active spanning tree is characterized by a sign criterion directly on its fundamental cycles/cocycles (obtained by applying the criterion for the uniactive case used in the previous section to suitable subsets of these cycles/cocycles obtained by the decomposition used in the present section, which also yields the algorithm of Section \ref{subsec:basori}).
%This feature is addressed in \cite{AB2-b}, and 
We invite the reader to look at the several detailed examples in \cite{AB2-b} of active spanning trees of orientations of  the graph $K_4$
%(which we do not copy here for lack of space,  and 
(which are consistent with the $K_4$ example in Section~\ref{sec:example}). 
%Indeed, it is interesting to see how the sign criterion of the uniactive bijection applies to fundamental cycles and cocycles of 
\emevder{phrase lourde ?}

%\red{interessant de voir comment signe sur fund graph etc....}
%
%\new{see AB2-b for detailed examples for several spanning trees of $K_4$ --- ''here we do not illsutrate the concrete signing''}

\emevder{donner refs precises des exemples de $K_4$ de AB2-b}

\emevder{eventuellement mettre ici exemple du graphe donne en decomp ?, oui serait bien, avec explication sommaire... sauf que je ne l'ai pas mis dans examle de AB2-b ?! }

\begin{definition}
\label{def:alpha-ind-decomp}
Let $\G$ be a directed graph on a linearly ordered set of edges.
The \emph{active spanning tree} $\alpha(\G)$ is defined by extending the definition of $\alpha$ from (cyclic-)bipolar (Definitions \ref{def:acyc-alpha} to \ref{def:cyc-alpha2}) to general ordered digraphs by the two following characteristic properties:
\smallskip

\noindent {\rm (1)} \bul  $\alpha(\G)=\alpha(\G/ F)\ \uplus\ \alpha(\G(F))$
where $F$ is the union of all directed cycles of $\G$ whose smallest element is the greatest active element~of~$\G$.  %\hfill (1)
\smallskip

\noindent {\rm (2)} \bul  $\alpha(\G)=\alpha(\G/ F)\ \uplus\ \alpha(\G(F))$
where $F$ is the complementary of the union of all directed cocycles of $\G$ whose smallest element is the greatest dual-active element~of~$\G$. %\hfill (2)
%\noindent {\rm (1)} \bul  $\alpha(\G)=\alpha(\G/ A)\ \uplus\ \alpha(\G(A))$
%where $A$ is the union of all directed cycles of $\G$ whose smallest element is the greatest active element~of~$\G$.  %\hfill (1)
%\smallskip
%
%\noindent {\rm (2)} \bul  $\alpha(\G)=\alpha(\G\setminus A)\ \uplus\ \alpha(\G/(E\setminus A))$
%where $A$ is the union of all directed cocycles of $\G$ whose smallest element is the greatest dual-active element~of~$\G$. %\hfill (2)
%\medskip
%\red{$F$ la place de $A$  et $E\s A$ ?}
\end{definition}
 
 Let us briefly  justify why  $\alpha(\G)$ is well defined, by Lemma \ref{lem:induction-dec-seq}:
  in property~(1), if  $F\not= \emptyset$, then $\G(F)$ is cyclic-bipolar, and $\G/F$ has one active element less than $\G$; 
in property~(2), if  $E\s F\not= \emptyset$, then $\G/F$ is bipolar, and $\G(F)$ has one dual-active element less than $\G$.
Then, the two properties are consistent and can be used recursively in any way, finally yielding the next definition.

\begin{definition}[equivalent to Definition \ref{def:alpha-ind-decomp}]
\label{def:alpha-seq-decomp}
Let $\G$ be a directed graph on a linearly ordered set of edges, with active filtration
 $\emptyset= F'_\ep\subset...\subset F'_0=F_c=F_0\subset...\subset F_\io= E$.
%$(F'_\ep, \ldots, F'_0, F_c , F_0, \ldots, F_\io)$.
We define $$\alpha(\G)=\ \biguplus_{1\leq k\leq \io} \alpha\bigl( \G(F_k)/F_{k-1}\bigr)\ \uplus\ \biguplus_{1\leq k\leq \ep}
\alpha\bigl( \G(F'_{k-1})/F'_{k}\bigr).$$
\end{definition}

Observe that this definition is valid since each active minor is either bipolar or cyclic-bipolar, by Proposition \ref{prop:pty-active-minors}, and its image has been defined in Definitions \ref{def:acyc-alpha} to \ref{def:cyc-alpha2}.
Observe also that the $2^{\io+\ep}$ digraphs in the same activity class as $\G$ have the same image under $\alpha$ as $\G$ (as they have the same active minors up to taking the opposite, cf. Proposition \ref{prop:act-classes}).

\begin{observation}
\label{obs:induc-alpha}
\rm
Let us consider an ordered digraph $\G$ and continue Observation \ref{obs:induced-dec-seq-ori}.
Let $\emptyset= F'_\ep\subset...\subset F'_0=F_c=F_0\subset...\subset F_\io= E$ be the active filtration of~$\G$.
Let $F$ and $F'$ be two subsets in this sequence such that $F'\subseteq F$.
In particular, $F'$ can be a $F'_i$, $0\leq i\leq \ep$, that is any union of directed cycles whose smallest edge is greater than a given edge, and $F$ can be a $F_i$, $0\leq i\leq \io$, that is the complementary of any union of directed cocycles whose smallest edge is greater than a given edge.
 Then, by Definition \ref{def:alpha-seq-decomp}, we have
$$\alpha(\G)=\alpha\bigl(\G(F')\bigr)\uplus \alpha\bigl(\G(F)/F'\bigr)\uplus \alpha\bigl(\G/F\bigr).$$
\end{observation}

In this way, we can also derive the following  equivalent relaxed definition.

\begin{definition}[equivalent to Definitions \ref{def:alpha-ind-decomp} and \ref{def:alpha-seq-decomp}]
\label{def:alpha-ind-decomp-relaxed}
Let $\G$ be an ordered directed graph.
We define $\alpha(\G)$ by Definitions \ref{def:acyc-alpha} to \ref{def:cyc-alpha2} if $\G$ is (cyclic-)bipolar w.r.t. its smallest edge, and by the following characteristic property:
$$\alpha(\G)=\alpha(\G/ F)\ \uplus\ \alpha(\G(F))$$
where $F$ is either any union of all directed cycles of $\G$ whose smallest edge is greater than a fixed edge of $\G$, or  the complementary of any union of all directed cocycles of $\G$ whose smallest edge is greater than a fixed edge of $\G$.
\end{definition}

Let us finally observe that the definition of $\alpha$ is consistent with cycle/cocycle duality.
For a dual pair of planar graphs $\G$ and $\G^*$, this is expressed  the following way: 
\begin{equation}
\tag{Duality}
\label{eq:act_mapping_duality}
\alpha(\G)=E\s \alpha({\G}^*)
\end{equation}
(which is also valid for general graphs by replacing $\G^*$ with the dual oriented matroid $M^*(\G)$).

\begin{thm} %[\cite{Gi02, GiLV05, AB2}]
\label{EG:th:alpha}
Let $G$ be a graph on a linearly ordered set of edges $E$. 
\begin{enumerate}
\item For any orientation $\G$ of $G$, $\alpha(\G)$  is well defined, and Definitions \ref{def:alpha-ind-decomp}, \ref{def:alpha-seq-decomp} and \ref{def:alpha-ind-decomp-relaxed} are equivalent.
%\item The mapping $\alpha$ is well-defined. 
%Let us denote the active filtration of $\G$ (Definition \ref{EG:def:ori-dec-seq})
%by
%$(F'_\ep, \ldots, F'_0, F_c , F_0, \ldots, F_\io)$.
%We have $$\alpha(\G)=\ \biguplus_{1\leq i\leq \io} \alpha\bigl( \G(F_k)/F_{k-1}\bigr)\ \uplus\ \biguplus_{1\leq i\leq \ep}
%\alpha\bigl( \G(F'_{k-1})/F'_{k}\bigr)$$
%where the involved minors are all bounded or having a bounded dual with respect to their smallest element.
%\item The mapping $\G\mapsto \alpha(\G)$ maps an orientation $\G$ of $G$ onto 
\item For any orientation $\G$ of $G$, $\alpha(\G)$ is
a spanning tree of $G$ with the same active filtration as~$\G$
(Definitions \ref{def:act-seq-dec} in Section \ref{subsec:act-part-decomp}, and \ref{def:sp-tree-dec-seq} in Section \ref{subsec:dec-bases}),
%(Theorems \ref{th:dec-ori} and \ref{EG:th:dec_base}), 
which implies in particular 
\begin{eqnarray*}
\Int\bigl(\alpha(\G)\bigr)&=&O^*(\G),\\
\Ext\bigl(\alpha(\G)\bigr)&=&O(\G).
\end{eqnarray*}
%$$\Int(\alpha(\G))=O^*(\G),$$ $$\Ext(\alpha(\G))=O(\G).$$
\item The $2^{i+j}$ orientations of $G$ in a given activity class with activity $j$ and dual-activity $i$ are mapped onto the same  spanning tree by $\alpha$.
\item The mapping $\G\mapsto \alpha(\G)$ from orientations of $G$ to spanning trees of $G$ is surjective. It provides a bijection between all activity classes of orientations of $G$ and all  spanning trees of $G$ (see Table \ref{table:intro} in Section \ref{sec:intro} for noticeable restrictions).
%\hfill \qed
%\index{active bijection}
%\index{oriented matroid!active bijection}
\end{enumerate}
\end{thm}

The bijection provided by Theorem \ref{EG:th:alpha} is called the \emph{canonical active bijection} of the ordered graph $G$.
%.
% 
%{\it  Finally,  as already stated in \cite{GiLV05}, $\alpha$ provides an activity preserving - and active partition preserving - bijection between activity classes of orientations and spanning trees  of $G$}.
%
%\red{ 
Observe that it depends only on $G$ and its edge-set ordering, not on a particular orientation $\G$ of $G$.
Let us mention that the inverse mapping can be computed by a single pass over $E$, see Section \ref{subsec:basori}\emevder{a mettre ici ?}
(and let us mention a deletion/contraction construction, see Section \ref{subsec:alpha-ind}).
%}

%\begin{proof}
%\red{a faire bref}
%\end{proof}

%\begin{color}{brown}

%\begin{color}{darkbrown}

%Theorem \ref{EG:th:alpha} is proved in next Section \ref{subsec:dec-bases}.
%\bigskip

%\subsection{Proof of Theorem \ref{EG:th:alpha}, proof of Theorem \ref{EG:th:dec_base}, and equivalent definitions of active filtrations for spanning trees}
%\subsection{Proof of Theorem \ref{EG:th:alpha}}

%\subsection{Proof for the theorem of Section \ref{subsec:alpha-def-decomp} (and of Section \ref{subsec:dec-bases})}
%\label{subsec:main-proof}

\bigskip
The rest of the section is devoted to proving Theorem \ref{EG:th:alpha}.
Using the decomposition for orientations from Theorem \ref{th:dec-ori} and the bijection in the bipolar case from Theorem \ref{thm:bij-10}, we can prove % the properties and bijection of
 Theorem \ref{EG:th:alpha} by means of two lemmas establishing the link with spanning trees, and we can derive at the same time a decomposition for spanning trees. This decomposition will be stated later as Theorem \ref{EG:th:dec_base} and will yield equivalent definitions of the active filtration of a spanning tree in Definition \ref{def:sp-tree-dec-seq}, see Section \ref{sec:spanning-trees}. 
Let us observe that this last decomposition generalizes to matroid bases \cite{AB2-a}, but such proofs based on orientations are not possible in non-orientable matroids. 

\def\F{T}
\def\S{{\mathcal S}}

%\red{orientatbility of graphs}

%\begin{proof}[Proof of Theorem \ref{EG:th:alpha}, first part]
%
%\end{proof}

%\red{ATTENTIO UMEOTATION SPECIALE}

\begin{property}
\label{pty:fund}
Let $T$ be a spanning tree of a graph $G$ with set of edges $E$.
Let $B\subset A\subseteq E$.
Assume $T'=T\cap A\s B$ is a spanning tree of $G'=G(A)/B$.
Then, for all $b\in T'$, we have $$C^*_{G'}(T';b)=C^*_G(T;b)\cap A=C^*_G(T;b)\cap A\s B,$$
and, for all $e\in (A\s B)\s T'$, we have $$C_{G'}(T';e)=C_G(T;e)\s B=C_G(T;e)\cap A\s B.$$
%\red{A VERIFIER !!!!!}
%\hfill\square
\end{property}

\begin{proof}
Direct by the definitions and the properties of cycles and cocycles in minors.
\end{proof}

\begin{lemma}
\label{lem:th-bij-gene-part1}
%Let $\G$ be an orientation of a graph $G$ on a linearly ordered set of edges $E$.
Let $\G$ be an ordered digraph. % (with underlying graph $G$).
Then, $\alpha(\G)$ is well defined by Definitions \ref{def:alpha-ind-decomp}, \ref{def:alpha-seq-decomp} and \ref{def:alpha-ind-decomp-relaxed}, equivalently.
Moreover, denoting $T=\alpha(\G)$, we have that $T$ is a spanning tree of $G$, with $\Int(T)=O^*(\G)$ and $\Ext(T)=O(\G)$.
%\item Assume $\Int(T)\not=\emptyset$. Let $a=\max(\Int(T))$.
%Let $A\subseteq E$ be the smallest subset for inclusion such that:
%\red{auter def?}
%\begin{enumerate}[(i)]
%\partopsep=0mm \topsep=0mm \parsep=0mm \itemsep=0mm
%\item \label{it1} $a\in A$;
%\item \label{it2} if $b\in T\cap A$ then $C^*(T;b)\subseteq A$;
%\item \label{it3} if $b\in T\s \Int(T)$ and every $e\in C^*(T;b)$ with $e<b$ belongs to $A$ then $b\in A$.
%\end{enumerate}
%Then $A$ is the part of the active partition of $\G$ with smallest element $a$.
%\item \red{idem pour cycles}
%\end{itemize}
%
%\red{separer lemme en deux ?}
%
\end{lemma}

\begin{proof}
%\red{first part}
The fact that $\alpha(\G)$ is a well defined subset of $E$ and that Definitions \ref{def:alpha-ind-decomp}, \ref{def:alpha-seq-decomp} and \ref{def:alpha-ind-decomp-relaxed} of $\alpha(\G)$ are equivalent directly  comes from the above discussion and from Section \ref{subsec:act-part-decomp}.

Using notations of Definition \ref{def:alpha-seq-decomp}, Proposition \ref{prop:pty-active-minors} that ensures that the active minors are \hbox{(cyclic-)}bipolar, and Theorem \ref{thm:bij-10} stating the active bijection for (cyclic-)bipolar orientations, we have $T=\alpha(\G)=\ \biguplus_{1\leq k\leq \io} T_k\ \uplus\ \biguplus_{1\leq k\leq \ep}
 T'_k$, where, for $1\leq k\leq \io$, $T_i=\alpha\bigl( \G(F_k)/F_{k-1}\bigr)$ is a uniactive internal spanning tree of $G_k=G(F_k)/F_{k-1}$, and, for $1\leq k\leq \ep$, $T'_i=\alpha\bigl( \G(F'_{k-1})/F'_{k}\bigr)$ is a uniactive external spanning tree of $G'_k=G(F'_{k-1})/F'_{k}$. % (by Theorem \ref{thm:bij-10} for bipolar orientations and its dual extension to cyclic-bipolar orientations, as recalled in Section \ref{subsec:fob}).
% \red{donner ce thm pour bases 0,1 ?}
%
Recall that, for every subset $F\subseteq E$, the union of a spanning tree of $G/F$ and of a spanning tree of $G(F)$ is a spanning tree of $G$. Then a direct induction shows that $T$ is a spanning tree of $G$.
%
%\red{lemma}
%
%\red{preuve de \Int}
%
%
%
%\begin{color}{purple}
%***dessous : preuve de $\Int(\F)=\{a_1,\dots, a_\io\}$, issue de AB2***
%\end{color}

The directed graph $\G$ has $\io\geq 0$ dual-active elements, which we denote $a_1<...<a_\io$, and $\ep\geq 0$ active elements, which we denote $a'_1<...<a'_\ep$. Also, let us denote $\S=(F'_\ep, \ldots, F'_0, F_c , F_0, \ldots, F_\io)$ the active filtration of $\G$.
%\red{xxxxxxxxxxxxx OUGJENSUIS xxxxxxxxxxxx}
%\red{FAIRE LEMME pour ce "second..." ? car bien reutilise ensuite dans preuve}
\ss

\tiret {\it External activity part.}
Let us prove that $\Ext(\F)=\{a'_1,\dots, a'_\ep\}$. 

Let $e\in \Ext(\F)$. By definition, $e\not\in T$ and $e=\min(C(T;e))$.
Since $\S$ induces a partition of $E$, $e$ is an edge of a minor $H$ of $G$ induced by this sequence $\S$: either $H=G_k$ for some $1\leq k\leq \io$ or $H=G'_k$ for some $1\leq k\leq \ep$.
In any case, $e$ is an element of the spanning tree $T_H$ induced by $T$ in $H$: either $T_H=T_k$ if $H=G_k$, or $T_H=T'_k$ if $H=G'_k$.
%
%It is easy to see that 
By Property \ref{pty:fund},
$C_H(T_H;e)$ is obtained from $C_G(T;e)$ by removing elements not in the edge set of $H$. 
So $e=\min(C_H(T_H;e))$, so $e$ is externally active in the spanning tree $T_H$ of the graph $H$. 
By properties of the spanning trees induced by $T$ in the minors induced by the  sequence $\S$, this implies that $e=a'_k$ for some $1\leq k\leq \ep$.
Hence $\Ext(\F)\subseteq\{a'_1,\dots, a'_\ep\}$.

Conversely, let $1\leq k\leq \ep$.
By properties of the sequence $\S$, we have $a'_k=\min(C_{G'_k}(T_k;a'_k))$.
As above, it is easy to see that we have $C_{G'_k}(T_k;a'_k)=C_G(T;a'_k)\cap (F'_{k-1}\s F'_{k})$. 
Let $e=\min(C_G(T;a'_k))$ and assume that $e<a'_k$.
Since $\S$ is a filtration, by Definition \ref{def:abst-dec-seq}, the sequence $a'_j=\min(F'_{j-1}\s F'_{j})$ is increasing with $j$. Hence 
$a'_k=\min(F'_{k-1})$. Hence $e\in E\s F'_{k-1}$.
On the other hand, since $T\cap F'_{k-1}$ is a spanning tree of $G(F'_{k-1})$, then we have $C_{G}(T;e)\s F'_{k-1}=\emptyset$, which is a contradiction with $e\in E\s F'_{k-1}$.
So we have $e=a'_k$. So $a'_k\in \Ext(\F)$ and we have proved $\Ext(\F)\supseteq\{a'_1,\dots, a'_\ep\}$.
Finally, we have proved $\Ext(\F)=\{a'_1,\dots, a'_\ep\}$. 
\smallskip

\tiret {\it Internal activity dual part.}
Exactly the same reasoning can be directly adapted (using cycle/cocycle duality) in order to prove 
$\Int(\F)=\{a_1,\dots, a_\io\}$.
We could leave the details to the reader (fundamental cocycles are used instead of fundamental cycles, deletion is used instead of contraction, etc.). However, for the sake of completeness, we give the proof below.

%Let us prove that $\Int(\F)=\{a_1,\dots, a_\io\}$. 
%%where $a_k=\min(F_k\s F_{k-1})$, $1\leq k\leq \io$.
%%and $\Ext(\F)=\{a'_1,\dots,a'_\ep\}$, where $a'_k=\min(F'_{k-1}\s F'_{k})$, $1\leq k\leq \ep$.
%%\red{NB : serait utile aussi pour ABG2 !}

Let $b\in \Int(\F)$. By definition, $b\in T$ and $b=\min(C^*(T;b))$.
%By assumption on the sequence $\S$, %$\S$, 
Since $\S$ induces a partition of $E$, $b$ is an edge of a minor $H$ of $G$ induced by this sequence $\S$: either $H=G_k$ for some $1\leq k\leq \io$ or $H=G'_k$ for some $1\leq k\leq \ep$.
In any case, $b$ is an element of the spanning tree $T_H$ induced by $T$ in $H$: either $T_H=T_k$ if $H=G_k$, or $T_H=T'_k$ if $H=G'_k$.
%
%It is easy to see that 
By Property \ref{pty:fund},
%Moreover, since $H$ is of type $G(F')/F$ and its spanning tree $T_H$ of type $T\cap (F'\s F)$, we have by Property \ref{pty:fund_graph_flat} that
$C^*_H(T_H;b)$ is obtained from $C^*_G(T;b)$ by removing elements not in the edge set of $H$. %\red{vraiment easy ? detailler ?}
%\red{ne devrait on pas montrer cette pty pour $M/A\s B$ ??? ou detailler ceci ? BOF...}
So $b=\min(C^*_H(T_H;b))$, so $b$ is internally active in the spanning tree $T_H$ of the graph $H$. 
By properties of the spanning trees induced by $T$ in the minors induced by the  sequence $\S$, this implies that $b=a_k$ for some $1\leq k\leq \io$.
Hence $\Int(\F)\subseteq\{a_1,\dots, a_\io\}$.

Conversely, let $1\leq k\leq \io$.
By properties of the sequence $\S$, we have $a_k=\min(C^*_{G_k}(T_k;a_k))$.
As above, it is easy to see that we have $C^*_{G_k}(T_k;a_k)=C^*_G(T;a_k)\cap (F_k\s F_{k-1})$. 
Let $e=\min(C^*_G(T;a_k))$ and assume that $e<a_k$.
Since $\S$ is a filtration, by Definition \ref{def:abst-dec-seq}, the sequence $a_j=\min(F_j\s F_{j-1})$ is increasing with $j$. Hence 
$a_k=\min(E\s F_{k-1})$. Hence $e\in F_{k-1}$.
On the other hand, since $T\s F_{k-1}$ is a spanning tree of $G/F_{k-1}$, then we have $C^*_{G}(T;b)\cap F_{k-1}=\emptyset$, which is a contradiction with $e\in F_{k-1}$.
So we have $e=a_k$. So $a_k\in \Int(\F)$ and we have proved $\Int(\F)\supseteq\{a_1,\dots, a_\io\}$.
Finally, we have proved $\Int(\F)=\{a_1,\dots, a_\io\}$. 
%Notice that, by duality (i.e. applying this result to $M^*$), we have also $\Ext(\F)=\{a'_1,\dots,a'_\ep\}$.
%\red{attention notation $\Int(\F)$ et $\Int_M(T)$. pas genant comme c'est dit dessous}
%\smallskip
%
%
%{\it (External activity dual part.)}
%Exactly the same reasoning can be easily adapted to prove $\Ext(\F)=\{a'_1,\dots, a'_\ep\}$. We leave the details to the reader (fundamental cycles of edges in  complements of spanning trees are used instead of fundamental cocycles of edges in spanning trees).
%
%\red{donner plutot external part? ou les deux ? ou external en premier?}
%
%\begin{color}{purple}
%***FIN de preuve de $\Int(\F)=\{a_1,\dots, a_\io\}$, peut etre suffisant pour ABG2***
%\end{color}
\end{proof}

%\red{definir $\AA(\{a\})$ called \emph{active closure} in \cite{AB2-a} (where much more properties are studied).}

\begin{definition}
\label{def:active-closure}
Let $G$ be an ordered graph with set of edges $E$.
Let $T$ be a spanning tree of $G$.

%For $a\in Ext(T)$, the \emph{active closure} of $a$, denoted $\AA(\{a\})$,
For $X\subseteq Ext(T)$, the \emph{active closure} of $X$, denoted $\AA(X)$, is  the smallest (for inclusion) subset $A\subseteq E$  such that:
%Second, assume $\Ext(T)\not=\emptyset$. Let $a=\max(\Ext(T))$.
%Let $A\subseteq E$ be the smallest subset for inclusion such that:
%\red{auter def?}
\vspace{-1mm}
\begin{enumerate}[(i)]
\partopsep=0mm \topsep=0mm \parsep=0mm \itemsep=0mm
%\item \label{it1} $a\in A$;
\item \label{it1} $X\subseteq A$;
\item \label{it2} if $e\in (E\s T)\cap A$ then $C(T;e)\subseteq A$;
\item \label{it3} if $e\in (E\s T)\s \Ext(T)$ and every $b\in C(T;e)$ with $b<e$ belongs to $A$ then $e\in A$.
\end{enumerate}

%For $a\in Int(T)$, the \emph{active closure} of $a$, denoted $\AA(\{a\})$,
For $X\subseteq Int(T)$, the \emph{active closure} of $X$, denoted $\AA(X)$, 
 is  the smallest (for inclusion) subset $A\subseteq E$  such that:
\vspace{-1mm}
\begin{enumerate}[(i)]
\partopsep=0mm \topsep=0mm \parsep=0mm \itemsep=0mm
%\item \label{it1} $a\in A$;
\item \label{it1} $X\subseteq A$;
\item \label{it2} if $b\in T\cap A$ then $C^*(T;b)\subseteq A$;
\item \label{it3} if $b\in T\s \Int(T)$ and every $e\in C^*(T;b)$ with $e<b$ belongs to $A$ then $b\in A$.
\end{enumerate}
\end{definition}

\begin{lemma}
\label{lem:th-bij-gene-part2}
%%Let $\G$ be an ordered digraph with underlying graph $G$.
%Let $G$ be a graph on a linearly ordered set of edges $E$.
%Let $T$ be a spanning tree of $G$ and let $\G$ be an orientation of $G$.
%Assume $T=\alpha(\G)$.
%%\red{ou Let $\G$ be an ordered digraph and $T=\alpha(\G)$ its active spanning tree.}
Let $\G$ be an ordered digraph and $T=\alpha(\G)$ be the active spanning tree of $\G$.
Assume
either $a=\max(\Ext(T))=\max(O(\G))$, or $a=\max(\Int(T))=\max(O^*(\G))$.
Then, %in both cases, 
the part of the active partition of $\G$ containing $a$ is $\AA(\{a\})$.
\end{lemma}

Before proving this lemma, let us emphasize the following observation.

\begin{observation}
\rm
The part of the active partition constructed in Lemma \ref{lem:th-bij-gene-part2} is built from Definition \ref{def:active-closure}, hence only from $T$, and even more precisely only from the fundamental cycles and cocycles of $T$. % (or, in other words, from the fundamental graph of $T$). 
By this way, the active closure of Definition \ref{def:active-closure} allows us to define the part of \emph{the active partition of the spanning tree $T$} containing $a$, see Definition \ref{def:sp-tree-dec-seq} in Section \ref{subsec:dec-bases} (which is consistent with \cite{GiLV05, AB2-a}).
Actually, the notion of active closure is the central tool of \cite{AB2-a}, to which the reader is referred for several equivalent constructions (independent of orientations).
%for several equivalent constructions of the active closure of Definition \ref{def:active-closure}.
%\purp{See \cite{AB2-a} for several equivalent constructions of the active closure of Definition \ref{def:active-closure}.}
%\red{Here we use the active closure only on the maximal (dual) active edge, but we give the complete definition for consistency with \cite{GiLV05}.
%Definition \ref{def:active-closure} of the active closure is obviously consistent with  the definition given in \cite{GiLV05} for decomposing spanning trees, and it is studied into the details in \cite{AB2-a}.
%***  METTRE DANS DEC SP TREES ?}
\emevder{observation repetitive avec section sur spanning trees et avec commentaire avant preuve apres thm de cette section... a memttre ou pas ?}
\end{observation}

\begin{proof}[Proof of Lemma \ref{lem:th-bij-gene-part2}]
%\red{*********\max part********}
%
%\red{peut etre plus facile n termes de cycles ?}
%
%\red{remplacer b par t pour elements de spanning trees ?}
%\red{donner plutot external en premier?}
For the sake of completeness of the paper, we give below the two parts of the proof, for internal activities and for external activities. However, each part can be directly adapted from the other, following exactly the same reasoning, simply using dual objects with respect to cycle/cocycle duality, and one of the two parts could have been left as an exercise to the reader .
\ss

\tiret {\it External activity part.}
Assume that $\ep>0$ and denote $a=\max(\Ext(T))$. 
Let $A$ be the smallest subset of $E$
satisfying properties (\ref{it1})(\ref{it2})(\ref{it3}).
Let $F_{\ep-1}$ be the part of the active partition of $\G$ with smallest element $a'_\ep$.
Let us prove that $A=F'_{\ep-1}$.
%Exactly the same reasoning can be adapted  to prove that 
%prove that $A=F'_{\ep-1}$. It could be left as an exercise to the reader (simply use dual objects), but we detail this proof here for completeness of the paper.
%
%\red{justifier induction sur $/$}
%
%\red{detailler partie duale ?}

First, we show that $F'_{\ep-1}$ satisfies the same properties (\ref{it1})(\ref{it2})(\ref{it3}) as $A$. % except possibly that of being minimal for inclusion.
By the previous lemma, Lemma \ref{lem:th-bij-gene-part1}, we have $\Ext(T)=\{a'_1,\dots, a'_\ep\}$, hence we have $a=a_\ep$, hence $a\in F'_{\ep-1}$. So property (\ref{it1}) is satisfied.
Let $e\in F'_{\ep-1}\s T'_\ep$, with $T'_\ep=T\cap F'_{\ep-1}$.
Since $T'_\ep$ is a spanning tree of $G(F'_{\ep-1})$, % by Property \ref{pty:fund}, 
we have $C_G(T;e)\subseteq F'_{\ep-1}$.
So property (\ref{it2}) is satisfied.
Finally, for $e\in E\s T$, let us denote $C(T;e)^<=\{b<e\mid b\in C(T;e)\}$.
%We have $b\in \Int(T)$ if and only if $C^*(T;b)^<=\emptyset$.
Assume that there exists $e\in E\s T$ such that $\emptyset \subset C(B;e)^<\subseteq F'_{\ep-1}$ and $e\not\in F'_{\ep-1}$. 
Then $e\not\in \Ext(\F)$ as $C(B;e)^<\not=\emptyset$.
And $e=\min( C(T;e)\s F'_{\ep-1})$ as $C(T;e)^<\subseteq F'_{\ep-1}$.
By Property \ref{pty:fund}, $C(T;e)\s F'_{\ep-1}$
is the fundamental cycle of $e$ w.r.t. the spanning tree $T\s T'_\ep$ of $G/F'_{\ep-1}$
(it is a spanning tree by Lemma \ref{lem:th-bij-gene-part1} applied to $G/F'_{\ep-1}$). 
%\red{easy ?????}
Hence we have $e$ externally active in the spanning tree $T\s T'_\ep$ of $G/F'_{\ep-1}$, hence $e=a'_k$ for some $1\leq k\leq \ep-1$ by the above lemma, Lemma \ref{lem:th-bij-gene-part1}, applied to $\G/F'_{\ep-1}$ (precisely: $\Ext_{G/F'_{\ep-1}}(T\s T'_\ep)=\{a'_1,\dots,a'_{\ep-1}\}$). % \red{lemma?}.
This is a contradiction with $a'_\ep=\min( F'_{\ep-1})$.
So property (\ref{it3}) is satisfied.
Since $F'_{\ep-1}$ satisfies the three properties (\ref{it1})(\ref{it2})(\ref{it3}) and $A$ is the smallest set satisfying those three properties, we have shown $A\subseteq F'_{\ep-1}$.
%Hence $\AA(X)\subseteq A$ by Property \ref{prop:dec_operation} (i).\par

To conclude, let us assume that there exists $e\in F'_{\ep-1}\s A$. 
In a first case, let us assume that $e\not\in T$.
Then $C(T;e)\subseteq  F'_{\ep-1}$ since $T'_\ep$ is a spanning tree of $G(F'_{\ep-1})$. And moreover $C(T;e)$ is the fundamental cycle of $e$ w.r.t. $T'_\ep$ in $G(F'_{\ep-1})$.
If $e=\min (C(T;e))$ then we have $e$ externally active 
in the spanning tree $T'_\ep$ of $G(F'_{\ep-1})$. 
Then $e=a'_\ep=a$ by properties of $T'_\ep$, which is a contradiction with $e\not\in A$.
So there exists $f<e$ in $C(T;e)\s A$ (otherwise $\emptyset \subset C(T;e)^<\subseteq A$, which implies $e\in A$ by definition of $A$). So there exists $f<e$ with $f\in F'_{\ep-1}\s A$.
In a second case, let us assume that $e\in T$.
Then,  there exists $f\in E\s T$ with $f<e$ and $f\in C^*(T;e)\cap F'_{\ep-1}$ (otherwise, by Property \ref{pty:fund},
$e$ is externally active in $T'_\ep$, in contradiction with properties of $T'_\ep$). 
By assumption we have $e\in T\s A$. %\red{lemma?}
If $f\in A$ then, by definition of $A$, we have $C(T;f)\subseteq A$, hence $e\in A$ (since $f\in C^*(T;e)$ is equivalent to $e\in C(T;f)$), which is a contradiction with $e\not\in A$.
So we have $f\in F'_{\ep-1}$.
In any case, the existence of $e$ in $F'_{\ep-1}\s A$ implies the
existence of $f<e$ in $F'_{\ep-1}\s A$, which is impossible. 
So we have proved $F'_{\ep-1}=A$.
\smallskip

\tiret {\it Internal activity part.}
Assume that $\io>0$ and denote $a=\max(\Int(T))$. Let $A$ be the smallest subset of $E$
satisfying properties (\ref{it1})(\ref{it2})(\ref{it3}). 
Let $E\s F_{\io-1}$ be the part of the active partition of $\G$ with smallest element $a_\io$.
Let us prove that $A=E\s F_{\io-1}$ 
%By the above result, we have $a=a_\io$.
%\red{justifier induction sur $/$}
%\blue{dessous AB2}
\def\AA{acl}

First, we show that $E\s F_{\io-1}$ satisfies the same properties (\ref{it1})(\ref{it2})(\ref{it3}) as $A$. % except possibly that of being minimal for inclusion.
By the previous lemma, Lemma \ref{lem:th-bij-gene-part1}, we have $\Int(T)=\{a_1,\dots, a_\io\}$, hence we have $a=a_\io$, hence $a\in E\s F_{\io-1}$. So property (\ref{it1}) is satisfied.
Let $b\in T_\io$ with $T_\io=T\cap (E\s F_{\io-1})$. Since $T_\io$ is a spanning tree of $G/F_{\io-1}$, %by Property \ref{pty:fund}, 
we have $C^*_G(T;b)\subseteq (E\s F_{\io-1})$.
So property (\ref{it2}) is satisfied.
Finally, for $b\in T$, let us denote $C^*(T;b)^<=\{e<b\mid e\in C^*(T;b)\}$.
%We have $b\in \Int(T)$ if and only if $C^*(T;b)^<=\emptyset$.
Assume that there exists $b\in T$ such that $\emptyset \subset C^*(B;b)^<\subseteq E\s F_{\io-1}$ and $b\not\in E\s F_{\io-1}$. 
Then $b\not\in \Int(\F)$ as $C^*(B;b)^<\not=\emptyset$.
And $b=\min( C^*(T;b)\s (E\s F_{\io-1}))$ as $C^*(T;b)^<\subseteq (E\s F_{\io-1})$.
By Property \ref{pty:fund}, $C^*(T;b)\s (E\s F_{\io-1})$
is the fundamental cocycle of $b$ w.r.t. the spanning tree $T\s T_\io$ of $G(F_{\io-1})$
(it is a spanning tree by Lemma \ref{lem:th-bij-gene-part1} applied to $G(F_{\io-1})$). 
%\red{easy ?????}
Hence we have $b$ internally active in the spanning tree $T\s T_\io$ of $G(F_{\io-1})$, hence $b=a_k$ for some $1\leq k\leq \io-1$ by the above lemma, Lemma \ref{lem:th-bij-gene-part1}, applied to $\G(F_{\io-1})$ (precisely: $\Int_{G(F_{\io-1})}(T\s T_\io)=\{a_1,\dots,a_{\io-1}\}$). % \red{lemma?}.
This is a contradiction with $a_\io=\min( E\s F_{\io-1})$.
So property (\ref{it3}) is satisfied.
Since $E\s F_{\io-1}$ satisfies the three properties (\ref{it1})(\ref{it2})(\ref{it3}) and $A$ is the smallest set satisfying those three properties, we have shown $A\subseteq E\s F_{\io-1}$.
%Hence $\AA(X)\subseteq A$ by Property \ref{prop:dec_operation} (i).\par

To conclude, let us assume that there exists $e\in (E\s F_{\io-1})\s A$. 
In a first case, let us assume that $e\in T$.
Then $C^*(T;e)\subseteq E\s F_{\io-1}$ since $T_\io$ is a spanning tree of $G/F_{\io-1}$. And moreover $C^*(T;e)$ is the fundamental cocycle of $e$ w.r.t. $T_\io$ in $G/F_{\io-1}$.
If $e=\min (C^*(T;e))$ then we have $e$ internally active 
in the spanning tree $T_\io$ of $G/F_{\io-1}$. 
Then $e=a_\io=a$ by properties of $T_\io$, which is a contradiction with $e\not\in A$.
So there exists $f<e$ in $C^*(T;e)\s A$ (otherwise $\emptyset \subset C^*(T;e)^<\subseteq A$, which implies $e\in A$ by definition of $A$). So there exists $f<e$ with $f\in (E\s F_{\io-1})\s A$.
In a second case, let us assume that $e\not\in T$.
Then,  there exists $f\in T$ with $f<e$ and $f\in C(T;e)\cap (E\s F_{\io-1})$ (otherwise, by Property \ref{pty:fund},
$e$ is externally active in $T_\io$, in contradiction with properties of $T_\io$). 
By assumption we have $e\in E\s(A\cup T)$. %\red{lemma?}
If $f\in A$ then, by definition of $A$, we have $C^*(T;f)\subseteq A$, hence $e\in A$ (since $f\in C(T;e)$ is equivalent to $e\in C^*(T;f)$), which is a contradiction with $e\not\in A$.
So we have $f\in (E\s F_{\io-1})$.
In any case, the existence of $e$ in $(E\s F_{\io-1})\s A$ implies the
existence of $f<e$ in $(E\s F_{\io-1})\s A$, which is impossible. 
So we have proved $E\s F_{\io-1}=A$.
%\blue{fin AB2}
%\smallskip
\end{proof}

\begin{proof}[Proof of Theorem \ref{EG:th:alpha}]
Let $T=\alpha(\G)$, which, by Lemma \ref{lem:th-bij-gene-part1}, is well-defined and is a spanning tree of $G$ with $\Int(T)=O^*(\G)=\{a_1,\dots,a_\io\}$ and $\Ext(T)=O(\G)=\{a'_1,\dots,a'_\ep\}$.

By Definition \ref{def:acyc-alpha}, two opposite bipolar orientations are mapped onto the same spanning tree. Hence, by Definition \ref{def:act-class} and Definition \ref{def:alpha-seq-decomp}, the $2^{\io+\ep}$ orientations of $G$ in the activity class of $\G$ are mapped onto the same  spanning tree $T$.

It remains to prove that $\alpha$ yields a bijection between activity classes of orientations and spanning trees.
Assume that an orientation $\G'$ of $G$ is mapped onto the same spanning tree $T$ as $\G$.
Then, by Lemma \ref{lem:th-bij-gene-part1}, we have  $O^*(\G)=\Int(T)=O^*(\G')$ and
$O(\G)=\Ext(T)=O(\G')$.

Assume $\io>0$.
By Lemma \ref{lem:th-bij-gene-part2}, the part $E\s F_{\io-1}$ of the active partition of $\G$ containing $a_\io$ depends only on $T$, hence it is the same for $\G'$.
By Definition \ref{def:alpha-ind-decomp}, we have
$\alpha(\G)=\alpha(\G(F_{\io-1}))\uplus \alpha(\G/F_{\io-1})$
and $\alpha(\G')=\alpha(\G'(F_{\io-1}))\uplus \alpha(\G'/F_{\io-1})$.
By hypothesis, we have $\alpha(\G)=\alpha(\G')=T$ and $\alpha(\G/F_{\io-1})=\alpha(\G'/F_{\io-1})=T\s F_{\io-1}$.
So we have $\alpha(\G(F_{\io-1}))=\alpha(\G'(F_{\io-1}))$.

Similarly, if $\ep>0$, by Lemma \ref{lem:th-bij-gene-part2}, we obtain that 
the part $F_{\ep-1}$ of the active partition of $\G$ containing $a'_\ep$ is the same for $\G'$ and that $\alpha(\G/F_{\ep-1})=\alpha(\G'/F_{\ep-1})$.

Now we can conclude by induction, by Lemma \ref{lem:induction-dec-seq}.
We finally have that $\G$ and $\G'$ have exactly the same active partitions.
By hypothesis, the image by $\alpha$ of each induced bipolar or cyclic-bipolar minor is the same for $\G$ and $\G'$. 
Hence, by Theorem \ref{thm:bij-10}, those bipolar minors are either equal or opposite for $\G$ and $\G'$, that is: $\G$ and $\G'$ are in the same activity class
(Definition \ref{def:act-class}).

So we have proved that $\alpha$ yields an injection from activity classes of orientations with dual-activity $\io$ and activity $\ep$ to spanning trees
with internal activity $\io$ and external activity $\ep$. It is a bijection because the sets have  the same cardinality, by the equality $o_{\io,\ep}=2^{\io+\ep}t_{\io,\ep}$ from  \cite{LV84a}, as recalled in Section \ref{sec:prelim}.
\end{proof}

%\red{mettre $t_{i,j}$ au lieu de $b_{i,j}$ ???}

%\end{color}

%\subsection{Tutte polynomial in terms of four subset activities (reminder), and the refined active bijection w.r.t. a reference orientation}
%\subsection{The refined active bijection w.r.t. a reference orientation, and (classical) Tutte polynomial expansion in terms of four subset activities}
%\subsection{The refined active bijection w.r.t. a reference orientation - Tutte polynomial in terms of four subset activities (classical)}
\subsection{The refined active bijection (with respect to a reference orientation)}
\label{subsec:refined}

The present construction is a %straightforward 
natural development of the canonical active bijection.
Let us consider an ordered directed graph $\G$ and its active spanning tree $T=\alpha(\G)$.  On one hand, the activity class of $\G$ (Definition \ref{def:act-class}) obviously has a boolean lattice structure isomorphic to the power set of $O(\G)\cup O^*(\G)$.
On the other hand, the interval $[T\setminus \Int(T), T\cup \Ext(T)]$ of $T$ (Section \ref{prelim:subset-activities})
also has a boolean lattice structure isomorphic to the power set of $\Int(T)\cup \Ext(T)$.
Since we have $\Int(T)\cup \Ext(T)=O(\G)\cup O^*(\G)$ by properties of $\alpha$ (Theorem \ref{EG:th:alpha}), those two boolean lattices are isomorphic.
See Figure \ref{EG:fig:K4-iso} for an illustration.
\emevder{la figue illustre aussi le choix par rapport a orientation de reference, a preciser plus loin? dans la caption?}%
Furthermore, activity classes of orientations of $G$ form a partition of the set of orientations of $G$ (Proposition \ref{prop:partition-into-act-classes}), intervals of spanning trees form a partition of the power set of $E$ (Section \ref{prelim:subset-activities}), and activity classes of orientations are in bijection with spanning trees 
under $\alpha$
%by the canonical active bijection of $G$ 
(Theorem \ref{EG:th:alpha}).
Hence, selecting a boolean lattice isomorphism for each couple formed by an activity class and its active spanning tree directly yields a bijection between all orientations and all subsets of $E$, which refines the canonical active bijection of $G$, and transforms activity classes of orientations into intervals of spanning trees. 
The most natural way   to select such isomorphisms (see also Section \ref{subsec:act-map-class-decomp} for variants) is to use an orientation $\G$ as a reference orientation, whose role is to break the symmetry in activity classes, just as in Section \ref{subsec:act-classes}. By this way, we shall obtain below \emph{the refined active bijection $\alpha_\G$ of $G$ w.r.t. $\G$}, which relates\emevder{preserves ?} the refined activities 
%for orientations (Definition \ref
%{def:gene-act-ori}) and for subsets (Definition~\ref{def:gene-act-base}).
for orientations  and for subsets from Definitions \ref{def:gene-act-base} and \ref{def:gene-act-ori}
\emevder{renvoyer a (see Section \ref{subsec:act-map-class-decomp}) pour varaintes ?}%
(as announced in \cite{LV12}%
\footnote{
Beware that the definition for such a bijection proposed at the very end of \cite{LV12} in terms of the active bijection is not correct: it is not complete, and given with a wrong parameter correspondence. It is different from the present one, which is consistent with the definition 
%of the refined mapping 
given in \cite{Gi02,GiLV07,AB2-b}.}%
),
giving a bijective transformation between the formulas of Theorems \ref{EG:th:Tutte-4-variables} and \ref{EG:th:expansion-orientations}:
%(as announced in \cite{LV12}%
%%
%\footnote{
%Beware that the definition for such a bijection proposed at the very end of \cite{LV12} in terms of the active bijection is not correct: it is not complete, and given with a wrong parameter correspondence. It is different from the present one, which is consistent with \cite{Gi02,GiLV07,AB2-b}.}%
%)
%%Hence we get a bijective interpretation of 
%\begin{eqnarray*}
%T(G;x+u,y+v)&
%=&\sum_{A\subseteq E} x^{\io_M(A)}u^{cr_M(A)}y^{\ep_M(A)}v^{nl_M(A)}\\
%&=&\sum_{A\subseteq E} x^{\theta^*_M(A)}u^{\bar\theta^*_M(A)}y^{\theta_M(A)}v^{\bar\theta_M(A)}
%\end{eqnarray*}
%%
%$T(G;x+u,y+v)=\sum_{A\subseteq E}\ x^{\mid \Int_G(A)\mid}\ u^{\mid P_G(A)\mid}\ y^{\mid \Ext_G(A)\mid}\ v^{\mid Q_G(A)\mid}$
%$T(G;x+u,y+v)=\sum_{A\subseteq E} \ x^{|\Theta^*_\G(A)|}\ u^{|\bar\Theta^*_\G(A)|}\ y^{|\Theta_\G(A)|}\ v^{|\bar\Theta_\G(A)|}.$
%%
\begin{eqnarray*}
%\displaystyle
T(G;x+u,y+v)&
=&\sum_{A\subseteq E}\ x^{\mid \Int_G(A)\mid}\ u^{\mid P_G(A)\mid}\ y^{\mid \Ext_G(A)\mid}\ v^{\mid Q_G(A)\mid}\\
&=&\sum_{A\subseteq E} \ x^{|\Theta^*_\G(A)|}\ u^{|\bar\Theta^*_\G(A)|}\ y^{|\Theta_\G(A)|}\ v^{|\bar\Theta_\G(A)|}.
\end{eqnarray*}
\bs

\begin{figure}[h]
\centering
\includegraphics[scale=1.1]{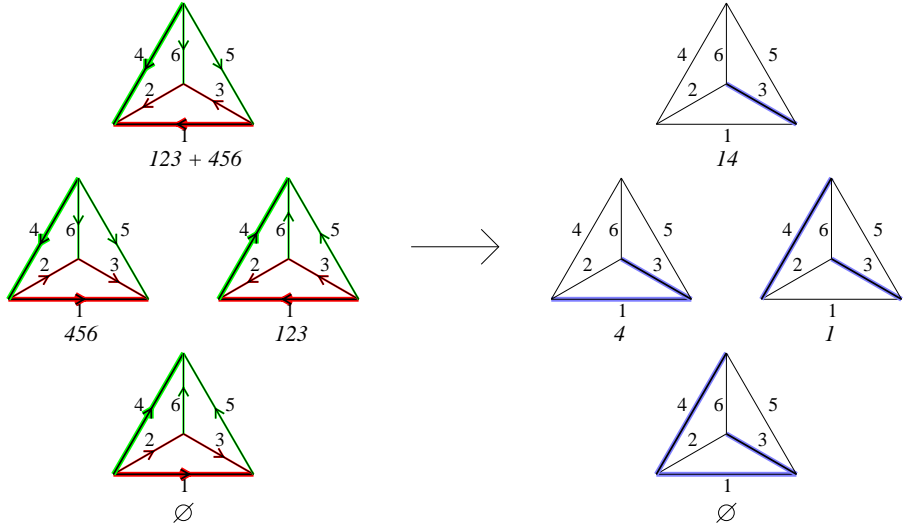}
\caption[]{Boolean lattice isomorphism between an activity class of (re)orientations and the interval of the corresponding spanning tree, figured for the class of the graph $\G$ from Figure \ref{EG:fig:K4-dec} with active partition $123+456$ and the spanning tree $T=\alpha(\G)$ with $[T\s \Int(T),T]=[3,134]$. Edges written below the graphs on the right are those removed from $T$, they correspond to reoriented parts in the digraphs on the left.}
%\eme{reduire figure et agrandir caracteres dans figure}
\label{EG:fig:K4-iso}
\end{figure}

Technically, let $G$ be a graph on a linearly ordered set $E$.
Let $\G$ be an orientation of $G$, thought of as  \emph{the reference orientation}.
Let $A\subseteq E$.
%Let us rephrase the definition of the activity class of $\G$ a follows (Definition \ref{def:act-class}).
%
The active partition of $-_A\G$ can be denoted as:
%$$E\ =\ \biguplus_{a\in O(-_A\G)}A_a\uplus \biguplus_{a\in O^*(-_A\G)}A^*_a$$
$$E\ =\ \biguplus_{a\ \in\ O(-_A\G)\ \cup\ O^*(-_A\G)}A_a$$
where the index of each part is the smallest element of the part.
%Let us recall \red{(ref Cor ?)} that every reorientation in the same activity class has the same (dual-)active elements, the same active partition.
%
Then, the activity class of $-_A\G$ can be denoted in the following way (where $\triangle$ denotes the symmetric difference): 
%$$cl(A)=\Biggl\{ \ A\ \triangle\ \Bigl(\ \biguplus_{a\in P}A_a\uplus \biguplus_{a\in Q}A^*_a\ \Bigr)\ \ \mid \ \ P\subseteq O(A),\ Q\subseteq O^*(A)\Biggr\}.$$
$$cl(-_A\G)=\Biggl\{\ -_{A'}\G \ \mid\ A'= A\ \triangle\ \Bigl(\ \biguplus_{a\in P\cup Q}A_a\Bigr)\ \text{ for } \ P\subseteq O^*(-_A\G),\ Q\subseteq O(-_A\G)\ \Biggr\}.$$
Let $T=\alpha(-_A\G)$ be the active spanning tree of $-_A\G$. 
%The set of subsets of $E$ containing $T\setminus \Int(T)$ and contained in $T\cup \Ext(T)$ will be called the \emph{interval of} $T$, denoted $[T\setminus \Int(T), T\cup \Ext(T)]$.
%
The interval %$[T\setminus \Int(T), T\cup \Ext(T)]$ 
of $T$ 
%has a boolean lattice structure and 
can be also denoted:
%$$\bigl\{T'\subseteq E\mid T'=T\triangle\bigl(\cup_{i\in I}\{i\}\bigr) \text{ for some }I\subseteq \Ext_G(T)\cup \Int_G(T)\bigr\}.$$
%$$[T\setminus \Int(T), T\cup \Ext(T)]\ =\ \bigl\{T'\subseteq E\mid T'=T\triangle\bigl(\cup_{a\in P\cup Q}\{a\}\bigr) \text{ for  }P\subseteq \Int_G(T),\ Q\subseteq \Ext_G(T)\bigr\}.$$
$$[T\setminus \Int(T), T\cup \Ext(T)]\ =\ \Biggl\{\ T'\subseteq E\ \mid\  T'=T\triangle\bigl(\biguplus_{a\in P\cup Q}\{a\}\bigr) \text{ for  }P\subseteq \Int(T),\ Q\subseteq \Ext(T)\ \Biggr\}.$$
%
%And it is a classical result \cite{Cr69} (see also \cite{Da81, GoMM97, LV13} for generalizations) 
%%\red{verifier references} 
%that these sets considered for all spanning trees form a partition of $2^E$:
%$$2^E=\biguplus_{T\hbox{ spanning tree}}[T\setminus \Int(T), T\cup \Ext(T)].$$
%\ss
The above notations emphasize the two boolean lattice structures. 
Then, we define an isomorphism between the two by choosing that the representative orientation of the activity class which is active-fixed and dual-active fixed w.r.t. $\G$ (Definition \ref{def:active-fixed} and Corollary \ref{cor:enum-classes})
%(Section \ref{subsec:act-classes})
is associated with the spanning tree~$T$.
Assume $-_{A}\G$ is the representative of its class with these properties, then we formally have:
$\bar\Theta_\G(A)=O(-_A\G)\cap A=\emptyset$ and $\bar\Theta^*_\G(A)=O^*(-_A\G)\cap A=\emptyset,$ which corresponds to $A'=A$, $P=\emptyset$ and $Q=\emptyset$ in the above setting, and which corresponds to the subset $T'=T$ in the interval of the spanning tree $T$, that is to $P_G(T')=\Int(T)\cap T'=\emptyset$ and $Q_G(T')=\Ext(T)\cap T'=\emptyset$ (Definitions \ref
{def:gene-act-ori} and \ref{def:gene-act-base}).
Finally, all orientations in the same activity class and all subsets in the same interval correspond to all possible values of $P$ and $Q$ in the above notations, so that:
%$P=\bar\Theta^*_\G(A')=P_G(T')\subseteq O^*(-_A\G)=\Int(T)$ and $Q=\bar\Theta_\G(A')=Q_G(T')\subseteq O(-_A\G)=\Ext(T)$.
\begin{center}
   $
   \begin{array}{lclclclcl}
    P&=&\bar\Theta^*_\G(A')&=&P_G(T')&\subseteq& O^*(-_A\G)&=&\Int(T), \\
    Q&=&\bar\Theta_\G(A')&=&Q_G(T')&\subseteq& O(-_A\G)&=&\Ext(T).
   \end{array}
   $
\end{center}
%
%This construction %can be extensively 
%is
%condensed in 
By this way, we naturally obtain the following definition and theorem.

\begin{definition}
\label{EG:def:act-bij-ext}
Let $\G$ be a directed graph on a linearly ordered set of edges $E$.
For  $A\subseteq E$, 
set
%and denoting $T=\alpha(-_A\G)$, 
%let there be mappings defined
%\index{active mapping}
%\index{oriented matroid!active mapping}
%\index{active bijection}
%\index{oriented matroid!active bijection}
% the following %non-equivalent 
%manners:
%\begin{enumerate}
%\item 
%,
%set
%
%\new{reviewer : def un peu longue, peut elle etre simplifiee ? NON...}
%
 $$\alpha_\G(A)=\alpha(-_A\G)\ \setminus\ \Bigl(A\cap O^*(-_A\G)\Bigr)\ \cup\ \Bigl(A\cap O(-_A\G)\Bigr).$$
% $$T=\alpha(-_A\G),$$
%$$P=A\cap \Int(T)=A\cap O^*(-_A\G),$$ 
%$$Q=A\cap \Ext(T)=A\cap O(-_A\G),$$
% $$\alpha_\G(A)=T\ \setminus\ P\ \cup\ Q.$$
In other words, we set $\alpha_\G(A)=T \setminus P \cup Q$ with $T=\alpha(-_A\G),$
$P=A\cap \Int(T)=A\cap O^*(-_A\G),$ and $Q=A\cap \Ext(T)=A\cap O(-_A\G).$
The mapping $\alpha_\G$ is called \emph{the refined active bijection of $G$ w.r.t.~$\G$.}\emevder{ce mapping est la refined bijection ? voir enonce thm, ou bien labijetion n'est pas un mapping mais un couplage ?}%
%
%set $$\alpha_\G(A)=T\ \setminus\ (A\cap \Int(T))\ \cup\ (A\cap \Ext(T))=T\ \setminus\ (A\cap O^*(-_A\G))\ \cup\ (A\cap O(-_A\G)).$$
%
%\item more generally, for $X\subset E$, set $$\alpha^X_\G(A)=B\ \setminus\ \bigl((X\triangle A)\cap \Int(B)\bigr)\ \cup\ \bigl((X\triangle A)\cap \Ext(B)\bigr);$$
%\item more generally, the subsets $X$ can change at each activity class, that
% is: consider $X$ as a function and for every spanning tree $B$ define $X(B)\subset E$, then set $$\alpha^{X}_\G(A)=B\ \setminus\ \bigl((X(B)\triangle A)\cap \Int(B)\bigr)\ \cup\  \bigl((X(B)\triangle A)\cap \Ext(B)\bigr).$$
%\end{enumerate}
\end{definition}

\begin{thm} %[see also \cite{Gi02, GiLV07, AB2}]
\label{EG:th:ext-act-bij}
%\red{verifier referecnes}
Let $G$ be a graph on a linearly ordered set of edges $E$, and 
$\G$ be an orientation of $G$ (thought of as a reference orientation).
%Let $\G$ be a directed graph on a linearly ordered set of edges $E$.
%The mapping $\alpha_\G$
%%of Definition \ref{EG:def:act-bij-ext} 
%yields a bijection of $2^E$ between reorientations of $\G$ and subsets of $E$, mapping %activity classes of reorientations of $\G$ onto intervals of spanning trees of $G$.
%
%In particular, for the mapping $\alpha_\G$, we have the following.
We have the following.
\begin{itemize}
\item The mapping $-_A\G\mapsto \alpha_\G(A)$ for $A\subseteq E$
%of Definition \ref{EG:def:act-bij-ext} 
effectively yields a bijection % of $2^E$ 
between all reorientations of $\G$ and all subsets of $E$. It maps activity classes of orientations of $G$ onto intervals of spanning trees of~$G$ (and these restrictions are boolean lattice isomophisms).\emevder{boolean lattice isomorphism?}
%
%\item \red{OPTIONAL ITEM this item explicits the construction but is not necessary.
%Let $T$ be a spanning tree of $\G$.
%Let us denote the active partition of $T$  by $E=\uplus_{i\in \Ext(T)\cup \Int(T)} A_i$, with $i=\min(A_i)$.
%%
% Amongst the reorientations $A$ of $\G$ such that
% $\alpha(-_A\G)=T$, let $A_T$ be the unique one such that $A_T\cap \Int(T)=A_T\cap \Ext(T)=\emptyset$. Then we have  $\alpha_\G(A_T)=T$. 
% %
%Moreover, for $P\subseteq \Int(T)$ and $Q\subseteq \Ext(T)$, let $A$ be 
% obtained from $A_T$ by reorienting the parts corresponding to $P\cup Q$, that is 
%$A=A_T\triangle\bigl(\cup_{i\in P\cup Q}A_i\bigr)$.
%Then we have $$\alpha_\G(A)=T\s P\cup Q.$$
%}
%\item With notations of Sections \ref{EG:sec:basis-expansion} and \ref{EG:sec:orientation-activities-expansion}, we have for all $A\subseteq E$:
%\begin{align*}
%\Int_G(\alpha_\G(A))&=&\Int_G(T)\s P&=&\Theta^*_\G(A),\\
%P_G(\alpha_\G(A))&=&P&=&\bar\Theta^*_\G(A),\\
%\Ext_G(\alpha_\G(A))&=&\Ext_G(T)\s Q&=&\Theta_\G(A),\\
%Q_G(\alpha_\G(A))&=&Q&=&\bar\Theta_\G(A).\\
%\end{align*}
%\vspace{-1cm}
\item  For all $A\subseteq E$, 
$T=\alpha(-_A\G)$ and $\alpha_\G(A)=T\setminus P\cup Q$,
%(and with $T=\alpha(-_A\G)$ and $\alpha_\G(A)=T\setminus P\cup Q$),
%s \ref{EG:sec:basis-expansion} and \ref{EG:sec:orientation-activities-expansion}, 
we have  (with notations of Definitions \ref
{def:gene-act-ori} and \ref{def:gene-act-base}):
%\red{verifier coherence avec def pour bases, bizarre le Int + ajouter en fonctionde I et J des booelan lattices}
\begin{align*}
\Int_G(\alpha_\G(A))&&=&&\scriptstyle \Int_G(T)\s P&&\scriptstyle =&&\scriptstyle O^*(-_A\G)\s P&&=&&\Theta^*_\G(A),\\
P_G(\alpha_\G(A))&&=&&&&\scriptstyle P&&&&=&&\bar\Theta^*_\G(A),\\
\Ext_G(\alpha_\G(A))&&=&&\scriptstyle \Ext_G(T)\s Q&&\scriptstyle =&&\scriptstyle O(-_A\G)\s Q&&=&&\Theta_\G(A),\\
Q_G(\alpha_\G(A))&&=&&&&\scriptstyle Q&&&&=&&\bar\Theta_\G(A).
\end{align*}

\item  In particular, $\alpha_\G(A)$ equals the active spanning tree $\alpha(-_A\G)$ if and only if $-_A\G$ is active fixed and dual-active fixed w.r.t. $\G$. Similarly, restrictions of the mapping $\alpha_\G$ yield the  bijections listed in Table \ref{table:thm-refined}.

%
%In particular, $\alpha_\G$ yields the following bijections (we indicate also the Tutte polynomial evaluation giving the number of involved subsets). 
%\new{Table \ref{table:thm-refined}.}
%%\red{def a reporter}
%%We call \emph{internal}, resp. \emph{external} a spanning tree with external activity $0$, resp. internal activity $0$.
%%\red{def a laisser?}
%We call 
%%\emph{$\G$-fixed-active orientation}, resp. \emph{$\G$-fixed-dual-active orientation},
%\emph{active-fixed orientation}, resp. \emph{dual-active-fixed orientation},
% an orientation in which all active, resp. dual-active, elements have the same direction as in the reference orientation $\G$ (corresponding to $Q=\emptyset$, resp. $P=\emptyset$).
% 
% \new{changer aka, coherence avec table intro, enlever abbreviations}
%\end{cor}
%
% \hfill\qed
% the orientation associated with the spanning tree $T$ is the unique in its activity class such that all active and dual-active elements have the same orientation as in the reference orientation $\G$.
\end{itemize}
\end{thm}

\begin{table}[h]
\centering
{
\ptirm
\def\interligne{&&\\[-11pt]}
\begin{tabular}{|l|l|c|}
\hline
\interligne
orientations & subsets & $t(G;2,2)$\\
\interligne
acyclic orientations & subsets of internal spanning trees & $t(G;2,0)$\\
&  (or no-broken-circuit subsets)& \\
\interligne
 strongly connected orientations & supersets of external spanning trees & $t(G;0,2)$\\
 \interligne
dual-active-fixed acyclic orientations & internal spanning trees & $t(G;1,0)$\\
% {\scriptsize (or unique sink acyclic orientations for suitable orderings, see \cite{GiLV05})}& & \\
 \interligne
dual-active-fixed acyclic orientations (w.r.t. $\scriptsize -_E\G$)& min. subsets of internal sp. tree intervals & $t(G;1,0)$\\
%{\scriptsize (up to reorienting $E$)}& & \\
\interligne
active-fixed strongly connected orientations & external spanning trees & $t(G;0,1)$\\
\interligne
active-fixed strongly connected orientations  (w.r.t. $\scriptsize -_{\tiny E}\G$) & max. subsets of external sp. tree intervals & $t(G;0,1)$\\
%{\scriptsize (up to reorienting $E$)}& & \\
\interligne
active-fixed orientations & subsets of spanning trees & $t(G;2,1)$\\
& (or forests)&\\
\interligne
dual-active-fixed orientations & supersets of spanning trees  & $t(G;1,2)$\\
& (or connected spanning subgraphs) \eme{??????????}&\\
\interligne
active-fixed and dual-active-fixed orientations & spanning trees & $t(G;1,1)$\\
\hline
%
% \hline
% \multicolumn{3}{|c|}{{\it \small \new{DESSOUS A REPORTER DESSUS + tableau a resserrer}}}\\
% \hline
%  				 orientations	& subsets of the edge set							& $t(G;2,2)$ \\ 
%%
%
%orientations with fixed orientation for active edges  & 	forests	&  $t(G;2,1)$ \\ %\bysame  \\ 
%orientations with fixed orientation for dual-active edges	& connected spanning subgraphs			&  $t(G;1,2)$ \\ 
%
%%  & & \\
% 				 acyclic orientations	& no-broken-circuit subsets			& $t(G;2,0)$ \\ 
%% & & \\
% 				 strongly connected orientations & {\scriptsize supersets of external spanning trees} & $t(G;0,2)$\\
%
%  \begin{tabular}{@{}l@{}}
%orientations with fixed orientation\\
%\hskip 1cm  for active edges and dual-active edges 
%\end{tabular} 	
%		&  spanning trees 			&   $t(G;1,1)$\\ 
% \begin{tabular}{@{}l@{}}
%acyclic orientations  with\\
%\hskip 1cm  fixed orientation for dual-active edges 
%\end{tabular} 	
%& internal spanning trees			&  $t(G;1,0)$ \\ 
% \begin{tabular}{@{}l@{}}
%  strongly connected orientations   with \\
%\hskip 1cm  fixed orientation for active edges
%\end{tabular}  
%& external  spanning trees			& $t(G;0,1)$ \\ 
%%
\end{tabular}
}
\caption{Remarkable restrictions of the refined active bijection of $G$  w.r.t. $\vec G$, between particular types of orientations (first column, in terms of Definition \ref{def:active-fixed}) and  particular types of edge subsets (second column) enumerated by Tutte polynomial evaluations (third column). See Theorem \ref{EG:th:ext-act-bij}.}
\label{table:thm-refined}
\end{table}

\eme{POUR PLUS TARD : voir si \min subsetes of itnernal sp. tree intervals peuvent etre decrits combinatoriement autrement ! probable vu que P ne depend pas de arbre...}%

As written above, the \emph{refined active bijection of the ordered graph $G$ w.r.t. (the reference reorientation) $\G$} is the bijection %$\alpha_M$ 
provided by Theorem \ref{EG:th:ext-act-bij}.
\emevder{repetititf deja ecrit deux fois en emphasize !}
It is important to insist that, in contrast with the canonical one, this bijection is induced by the choice of a reference orientation.
\emevder{mieux : is induced by que depends on pour refined bijection ?}%
%
%It gives a bijective transformation between the formulas of Theorems \ref{EG:th:Tutte-4-variables} and \ref{EG:th:expansion-orientations}:
%%(as announced in \cite{LV12}%
%%%
%%\footnote{
%%Beware that the definition for such a bijection proposed at the very end of \cite{LV12} in terms of the active bijection is not correct: it is not complete, and given with a wrong parameter correspondence. It is different from the present one, which is consistent with \cite{Gi02,GiLV07,AB2-b}.}%
%%)
%%%Hence we get a bijective interpretation of 
%%\begin{eqnarray*}
%%T(G;x+u,y+v)&
%%=&\sum_{A\subseteq E} x^{\io_M(A)}u^{cr_M(A)}y^{\ep_M(A)}v^{nl_M(A)}\\
%%&=&\sum_{A\subseteq E} x^{\theta^*_M(A)}u^{\bar\theta^*_M(A)}y^{\theta_M(A)}v^{\bar\theta_M(A)}
%%\end{eqnarray*}
%%%
%%$T(G;x+u,y+v)=\sum_{A\subseteq E}\ x^{\mid \Int_G(A)\mid}\ u^{\mid P_G(A)\mid}\ y^{\mid \Ext_G(A)\mid}\ v^{\mid Q_G(A)\mid}$
%%$T(G;x+u,y+v)=\sum_{A\subseteq E} \ x^{|\Theta^*_\G(A)|}\ u^{|\bar\Theta^*_\G(A)|}\ y^{|\Theta_\G(A)|}\ v^{|\bar\Theta_\G(A)|}.$
%%%
%\begin{eqnarray*}
%%\displaystyle
%T(G;x+u,y+v)&
%=&\sum_{A\subseteq E}\ x^{\mid \Int_G(A)\mid}\ u^{\mid P_G(A)\mid}\ y^{\mid \Ext_G(A)\mid}\ v^{\mid Q_G(A)\mid}\\
%&=&\sum_{A\subseteq E} \ x^{|\Theta^*_\G(A)|}\ u^{|\bar\Theta^*_\G(A)|}\ y^{|\Theta_\G(A)|}\ v^{|\bar\Theta_\G(A)|}.
%\end{eqnarray*}
%
Let us mention that the inverse mapping can be computed by a single pass over $E$, see Section \ref{subsec:basori}.
%\emevder{a mettre ici ?}
Let us mention a deletion/contraction construction, see Section \ref{subsec:alpha-refined-ind}.
Lastly, let us mention that variants can be defined, for instance by exchanging the correspondences between the four parameter activities for spanning trees and orientations, see Section \ref{subsec:act-map-class-decomp} below.

%\red{
%From the constructive viewpoint, let us mention
%that other characterizations and constructions exist for the canonical/refined active bijections, 
%notably: from spanning trees/subsets to orientations, by a single pass algorithm from the smallest to the greatest edge using only fundamental cycles/cocycles, see Section \ref{subsec:basori} and \cite{AB2-b}; or in terms of deletion/contraction of the greatest element, by an algorithm involving an exponential number of minors, see Section \ref{sec:induction} and \cite{AB4}.
%}

\begin{proof}[Proof of Theorem \ref{EG:th:ext-act-bij}]
%\red{a faire bref}
The first point directly  comes from the discussion above the theorem.
The second point also easily comes from this discussion.
Let us precisely check the equalities of parameters in the second point anyway.
In order to simplify notations, we omit subscripts ($G$ or $\G$) of activity parameters.
Let $A_T$ be the reorientation of $\G$ whose image under $\alpha_\G$ is the spanning tree $T$. % (obtained for $P=Q=\emptyset$).
Let $E=\uplus_{a\in O(-_{A_T}\G)\cup O^*(-_{A_T}\G)} A_a$, with $a=\min(A_a)$, be the active partition associated with $T$ or $-_{A_T}\G$.
Let $A$ be a subset in the associated activity class, we have
$A=A_T\triangle\bigl(\cup_{a\in P\cup Q}A_a\bigr)$ for some $P\subseteq \Int(T)=O^*(-_{A_T}\G)=O^*(-_A\G)$ and $Q\subseteq \Ext(T)=O(-_{A_T}\G)=O(-_A\G)$ with $P\cap A_T=\emptyset$ and $Q\cap A_T= \emptyset$.
By Definition \ref{EG:def:act-bij-ext}, we have $\alpha_\G(A)=T\setminus P\cup Q$ .
%
%Set $P\subset \Int_G(B)$ and $Q\subseteq \Ext_G(T)$ according to the hypothesis.
%Then, as in the above discussion, the activity class associated with $T$ can be denoted
%$cl(A_T)=\bigl\{A\subseteq E\mid A=A_T\triangle\bigl(\cup_{a\in P\cup Q}A_a\bigr) \text{ for  }P\subseteq O^*(-_{A_T}\G), \ Q\subseteq O(-_{A_T}\G)\bigr\}.$
%\red{nON mettre we have $A= A_t\triangle...$}
%Let $A\in cl(A_T)$. By Theorem \ref{EG:th:alpha}, we have $\Int(T)=O^*(-_{A_T}\G)=O^*(-_A\G)$ and $\Ext(T)=O(-_{A_T}\G)=O(-_A\G)$.

By Definition \ref{def:gene-act-base}, we have
$\Int(\alpha_\G(A))=\Int(T)\cap \alpha_\G(A)$.
We have $\Int(T)\cap \alpha_\G(A)=\Int(T)\cap(T\s P\cup Q)=\Int(T)\s P$.
By Theorem \ref{EG:th:alpha}, we have $\Int(T)\s P=O^*(-_A\G)\s P$.
By properties of $P$, we have
$O^*(-_A\G)\s P=O^*(-_A\G)\s \bigl(A_T\triangle (\cup_{a\in P\cup Q}A_a)\bigr)=O^*(-_A\G)\s A$.
By Definition \ref{def:gene-act-ori}, we have
$O^*(-_A\G)\s A=\Theta^*(A)$.
So finally $\Int_G(\alpha_\G(A))=\Theta^*(A)$.

On one hand, by Definition \ref{def:gene-act-base}, we have $\Int(\alpha_\G(A))\cup P(\alpha_\G(A))=\Int(T)$.
On the other hand, by Definition \ref{def:gene-act-ori}, we have $\Theta^*(A)\cup \bar\Theta^*(A)=O^*(-_A\G)$. By Theorem \ref{EG:th:alpha}, we have $\Int(T)=O^*(-_A\G)$, so, by the above result, we get $P(\alpha_\G(A))=\bar\Theta^*(A)$.

Similarly, by Definition \ref{def:gene-act-base}, we have
$\Ext(\alpha_\G(A))=\Ext(T)\s \alpha_\G(A)$.
We have $\Ext(T)\s \alpha_\G(A)=\Ext(T)\s(T\s P\cup Q)=\Ext(T)\s Q$.
By Theorem \ref{EG:th:alpha}, we have $\Ext(T)\s Q=O(-_A\G)\s Q$.
As above, %since $Q\cap A_T=\emptyset$, 
by properties of $Q$, we have 
$O(-_A\G)\s Q=O(-_A\G)\s \bigl(A_T\triangle (\cup_{a\in P\cup Q}A_a)\bigr)=O(-_A\G)\s A$.
As above, by Definition \ref{def:gene-act-ori}, we have
$O(-_A\G)\s A=\Theta(A)$.
So finally $\Ext(\alpha_\G(A))=\Theta(A)$.
And, as above, we deduce that $Q(\alpha_\G(A))=\bar\Theta(A)$.

%On one hand, by Definition \ref{def:gene-act-base}, we have $\Ext(\alpha_\G(A))\cup Q(\alpha_\G(A))=\Ext(T)$.
%On the other hand, by Definition \ref{def:gene-act-ori}, we have $\Theta^*(A)\cup \bar\\bar\Theta^*(A)=O^*(-_A\G}$.By Theorem \ref{EG:th:alpha}, we have $\Int(T)=O^*(-_A\G}$, so, by the above result, we get $P(\alpha_\G(A))=\bar\Theta^*(A)$.

Now, let us consider the list of bijections of the third point.
They are all obtained as restrictions of $\alpha_\G$.
Observe that an orientation is active-fixed, resp. dual-active-fixed, if it is obtained by $Q=\emptyset$, resp. $P=\emptyset$.
Therefore, all these bijections are obvious by the definitions, except the two ones involving $t(G;1,2)$ and $t(G;2,1)$. 
For the first one, resp. second one, of these two, we can use that subsets, resp. supersets, of spanning trees are exactly the subsets of type $T\s P$, resp. $T\cup Q$, for some spanning tree $T$ and $P\subseteq \Int(T)$, resp. $Q\subseteq \Ext(T)$.
This result is stated separately in Lemma \ref{lem:decomp-intervals} below.
%
%The list of bijections in the third point is direct.
%\red{lemme peut-etre ?}
%
%\red{detailler ?}
%\red{pourqoi pour 1,2, et 2,1 ? pas evident !}
%
%
%\red{justifier que tout subset de sp.tr. est compte}
\end{proof}

%\new{reviewer : veut allonger article avec sp. trees;;... voir rk 55}

\emevder{met on le resultat ci-dessous ou juste dans AB2b?}

\begin{lemma}
\label{lem:decomp-intervals}
Let $G$ be an ordered graph. 
%Subsets of spanning trees are exactly the subsets of type $T\s P$ for some spanning tree $T$ of $G$ and $P\subseteq \Int_G(T)$.
The set of subsets of spanning trees of $G$ is the union of intervals $[T\s \Int_G(T),T]$ over all spanning trees $T$ of $G$.
The set of supersets of spanning trees of $G$ is the union of intervals $[T, T\cup \Ext_G(T)]$ over all spanning trees $T$ of $G$.
\end{lemma}

\begin{proof}
%Among different possible proofs, ae give a proof the interest of whoch is to relate several structural results together.
By the main result of \cite{EtLV98} (implying Corollary \ref{cor:convolution}, and extended in Theorem \ref{EG:th:dec_base} below), we know that
spanning trees $T$ of $G$ are exactly subsets of the form $T_\io\uplus T_\ep$
where $T_\io$ is an internal spanning tree of $G/F$, 
$T_\ep$ is an external spanning tree of $G(F)$, 
and $F$ is a cyclic flat of $G$. Moreover $\Int(T)=\Int_{G/F}(T_\io)$ and 
$\Ext(T)=\Ext_{G(F)}(T_\ep)$ (for short, we omit these subscripts below).

We have $[T\s \Int(T),T\cup \Ext(T)]=[(T_\io\uplus T_\ep)\s \Int(T_\io),(T_\io\uplus T_\ep)\cup \Ext(T_\ep)]$.
Using the classical partition of $2^E$ into spanning tree intervals \cite{Cr69} recalled 
in the previous discussion,
%(see Section \ref{subsec:act-map-class-decomp}), 
we have:
\centerline{$\displaystyle 2^E=\biguplus_{T\text{ spanning tree}}[T\s \Int(T),T\cup \Ext(T)]
=\biguplus_{F,\ T_\io,\ T_\ep\text{ as above}}
[T_\io\s \Int(T_\io),T_\io]\times [T_\ep, T_\ep\cup \Ext(T_\ep)]$}
%$$\displaystyle 2^E=\biguplus_{T\text{ spanning tree}}[T\s \Int(T),T\cup \Ext(T)]
%=\biguplus_{F,\ T_\io,\ T_\ep\text{ as above}}
%[T_\io\s \Int(T_\io),T_\io]\times [T_\ep, T_\ep\cup \Ext(T_\ep)]$$
(where $\times$ yields all unions of a subset of the first set and a subset
 of the second set).
So we have 
\centerline{$\displaystyle \biguplus_{T\text{ spanning tree}}[T\s \Int(T),T]
=\biguplus_{F,\ T_\io,\ T_\ep\text{ as above}}
[T_\io\s \Int(T_\io),T_\io]\times [T_\ep]$}
%$$\biguplus_{T\text{ spanning tree}}[T\s \Int(T),T]
%=\biguplus_{F,\ T_\io,\ T_\ep\text{ as above}}
%[T_\io\s \Int(T_\io),T_\io]\times [T_\ep]$$
%
The size of the second set of the equality equals 
$\sum_{F}t(G/F;2,0)t(G(F);0,1)$ by classical evaluations of the Tutte polynomial.
And this number equals $t(G;2,1)$ by the convolution formula (Corollary \ref{cor:convolution}), which equals the number of subsets of spanning trees (as well known).
The first set of the equality is included in the set of subsets of spanning trees, and it has the same size,
hence it equals the set of subsets of spanning trees.
Similarly (dually in fact), we get the result involving supersets of spanning trees, whose number equals $t(G;2,1)$.
\eme{a verifier ! pas si evident...}%
%have $$\biguplus_{T\text{ spanning tree}}[T,T\cup \Ext(B)]
%=\biguplus_{F,\ T_\io,\ T_\ep\text{ as above}}
%[T_\io\s \Int(T_\io),T_\io]\times [T_\ep]$$
%
%
%
%Let $A$ be a subset of a spanning tree. 
%Let $T$ be the spanning tree such that $A\in [T\s \Int(T),T\cup \Ext(T)]$.
%Let $F,\ T_\io,\ T_\ep$ be the three parameters associated with $A$ in the above decomposition of $2^E$.
%We have $A=A_\io\uplus A_\ep$ with $A_\io=A\cap [T_\io\s \Int(T_\io),T_\io]$ and
%$A_\ep=A\cap [T_\ep, T_\ep\cup \Ext(T_\ep)]$.
%If there exists $e\in A_\ep\s T_\ep$, then $e\in \Ext(T_\ep)$, then $e\in \Ext(T)$, 
%
%
%
%So, we indeed have that the set of subsets of spanning trees is:
%$\uplus_{T\text{ spanning tree}}[T\s \Int(T),T]$.
%
%
\end{proof}
%
%\red{tableau a veifier avec exemple complet !!! $K_4$, et completer exemple}
%
%\red{peut etre mettre tableau en coroalliare et avec preuve, pas si evident...}
%
%\red{preuve par convolution}
%
%\red{voir note manuscrite pourpreuve des acts}

%
%Note that the parameters $P$ and $Q$ used in the statement of Theorem \ref{EG:th:ext-act-bij} are the same as the ones used above in the notations for intervals and activity classes.
%Note that more specific bijections can be obtained.

\eme{rk suivante inutile ? (verifiable dans preuve)
The reader might be surprised by the symmetry of the correspondence between parameters in Theorem \ref{EG:th:ext-act-bij}, compared with the non-symmetry of the definitions:
$\Int_{G}(A) = \Int_{G}(B)\cap A$ and $\Ext_G(A) = \Ext_G(B)\s A$ on one hand, 
and $\Theta^*_\G(A)=O^*(-_A\G)\s A$ and $\Theta_\G(A)=O(-_A\G)\s A$ on the other hand.
The reason is that we have chosen to associate a spanning tree with the active-fixed and dual-active-fixed orientation in its associated activity class, so in this case: an internally active element belongs to the spanning tree and is not reoriented on one hand, and an externally active element does not belong to the spanning tree and is also not reoriented on the other hand.}%
%\ss

%
%\begin{cor}[\cite{Gi02, GiLV07, ABG2, AB2}]
%\label{EG:cor:bij-regions-nbc}
%Let $\G$ be an oriented matroid on a linearly ordered set $E$.
%Let $-_A\G$ be an acyclic reorientation of $\G$ and let $T=\alpha(\G)$ with $\Ext(T)=0$ and $\Int(T)=O^*(-_A\G)$.
%Mapping $A$ onto $T\s (\Int(T)\cap A)$ yields a bijection between acyclic orientations and subsets of bases with external activity $0$ (also known as no-broken-circuit subsets of $\G$).
%
%Dually and similarly, we get a bijection between strongly connectedorientations and supersets of bases with internal activity $0$.
%\red{ou a mettre en item dans th precedent}
%\end{cor}

%For more information on no-broken-circuit subsets, see \cite[Section 7.4]{Bj92}.

\eme{th dessus provue sur feuilles volantes}%

\eme{peut etr alleger notations dans theoreme en definissant $T'$ et $A'$ comme dans intro de section}%

\eme{denote $\alpha_\G$ or denote $\bar\alpha_\G$}%

\eme{ATTENTION bien verifier tout ca, surtout derniere phrase, a ete vite fait !}%
%As noticed in Remark \ref{EG:rk:alpha-expansion}...

%\red{The refined active bijection preserves four-parameter activities defined in the next section, as it takes into account the positions in the boolean lattices.}

\eme{DESSOUS EN COMMENTAIRE premiere versio plusn litteraire de refined bijection}%
\subsection{A general decomposition framework  for classes of activity preserving bijections}
\label{subsec:act-map-class-decomp}

Let us briefly observe how the three level construction described in Sections \ref{subsec:fob}, \ref{subsec:alpha-def-decomp} and \ref{subsec:refined} can be relaxed so as to derive a whole 
%class of activity preserving, and active partition preserving,  bijections, satisfying some similar decomposition property based on active filtrations/partitions.
\emph{class of active partition preserving mappings} (hence also activity preserving), satisfying similar bijective and decomposition properties.
% \emph{active partition preserving mapping class.}
Among the bijections of this class, the active bijection is uniquely determined by its canonical construction at the first level, and its natural specification at the third level.
What we call \emph{preserving} is again the transformation of active elements, etc., into their counterpart for orientations/subsets.\emevder{utile ? a dire avant ?}

%\begin{itemize}
%\item {\it First level.} 
\paragraph{First level} 
Assume that, for any ordered graph $G$,  a mapping $\psi_G$ provides a bijection between orientations $\G$ of $G$ which are bipolar, resp. cyclic-bipolar,  w.r.t. their smallest edge with fixed orientation, and the spanning trees $\psi_G(\G)$ of $G$ which are internal, resp. external, uniactive. Assume also that two opposite orientations have the same image (so that the mapping $\psi_G$ has the same properties as the uniactive bijection $\G\mapsto\alpha(\G)$ stated in Theorem~\ref{thm:bij-10}).

%\item {\it Second level.} 
\paragraph{Second level} 
From the mappings $\psi_G$ available at the first level, one can extend their domains to all orientations of $\G$, using the same properties as for the canonical active bijection in Definitions \ref
{def:alpha-ind-decomp}, \ref{def:alpha-seq-decomp} and \ref{def:alpha-ind-decomp-relaxed}. Indeed, as discussed there, the  validity and equivalence 
of these definitions only relies upon properties of the active filtration/partition/minors addressed in Section \ref{sec:tutte}.
Precisely, for an ordered digraph $\G$ 
with active minors $\G_k$, ${1\leq k\leq \io}$, in the acyclic part, and 
$\G'_k$, ${1\leq k\leq \ep}$, in the cyclic part,
we define $$\psi_G(\G)=\ \biguplus_{1\leq k\leq \io} \psi_{G_k}\bigl( \G_k\bigr)\ \uplus\ \biguplus_{1\leq k\leq \ep}
\psi_{G'_k}\bigl( \G'_k\bigr).$$
\emevder{QUESTION RECHERCHE : n'est ce pas caracteristique de active partition preerving ???}%
%
%with active filtration
% $\emptyset= F'_\ep\subset...\subset F'_0=F_c=F_0\subset...\subset F_\io= E$,
%we define $$\psi_G(\G)=\ \biguplus_{1\leq k\leq \io} \psi_G\bigl( \G(F_k)/F_{k-1}\bigr)\ \uplus\ \biguplus_{1\leq k\leq \ep}
%\psi_G\bigl( \G(F'_{k-1})/F'_{k}\bigr).$$
%
At this step, similarly as for Theorem \ref{EG:th:alpha}, using the bijections at the first level and the decompositions of orientations and spanning trees provided by Theorems \ref{th:dec-ori} and \ref{EG:th:dec_base}, one can easily check that: \emph{$\psi_G$ yields an activity preserving, and active partition preserving, bijection between activity classes of orientations and spanning trees of $G$.}

%\end{itemize}

\paragraph{Third level}
As discussed in Section \ref{subsec:refined}, from any bijection $\psi_G$ between activity classes of orientations and spanning trees that preserves active elements, one can build a whole class of bijections %$\tilde\psi_G$
 between  orientations and subsets,  such that it maps each activity class of orientations onto a spanning tree interval. One can naturally demand that these restrictions are boolean lattice isomorphisms, which can be settled independently of each other.
For example, in each restriction, one can demand that the four activity parameters for orientations are transformed into the four activity parameter for subsets, but with possible exchanges in comparison with the refined active bijection
(i.e. make $\Int$ correspond to $\bar\Theta^*$ instead of $\Theta^*$,
and/or make $\Ext$ correspond to $\bar\Theta$ instead of $\Theta$).
Similarly, %For another example, %By this way, for instance, 
one can define active-fixed and dual-active-fixed orientations with respect to two different references orientations respectively, or with respect to variable reference orientations. 
\emevder{bien dit ???}
\ms

Lastly, let us roughly mention that one can also add a deletion/contraction property to the class of mappings considered in this section,
yielding the class mentioned in Section \ref{subsec:alpha-refined-ind}, option \ref{item:pres-act-parts} (which is thus at the intersection of the classes considered in this section and that one).

\section{Counterparts from the spanning tree viewpoint}
\label{sec:spanning-trees}
%
%\section{Brief counterparts from the spanning tree viewpoint}
%
%\section{Counterparts for spanning trees}
%
%\section{Counterparts of the constructions for/from (starting from???) %spanning trees}

%\red{advantag of graphs}
This section has a special status in the paper. While the above is essentially written from orientations to spanning trees, here we take the inverse viewpoint. We gather results intrinsically involving spanning trees  and constructions starting from spanning trees
(except subset activities refining spanning tree activities, that have been addressed in Section \ref{prelim:subset-activities}).

\subsection{The active partition/filtration of a spanning tree -  Decomposition of the set of all spanning trees of an ordered graph}
%\subsection{Active partition/filtration of a spanning tree, and decomposition of the set of spanning trees of an ordered graph}
%\subsection{Equivalent definitions of active filtrations for spanning trees}
\label{subsec:dec-bases}

First, we give a general decomposition theorem for spanning trees in terms of filtrations of an ordered graph. This theorem refines, at the uniactive level, the decomposition into internal/external spanning trees from \cite{EtLV98}, where only the cyclic flat $F_c$ was involved.
It is the counterpart for spanning trees of Theorem \ref{th:dec-ori}, and it can be derived from this latter theorem and the canonical active bijection (it is generalized to matroid bases in \cite{AB2-a}, in an intrinsic way, since a proof using orientations is not possible in non-orientable matroids: here we take benefit of graph orientability). 

Second, we define the active partition of a spanning tree. This fundamental notion can be defined in multiple ways 
%(it was briefly introduced in  \cite{GiLV05}, and, again, more details and constructions can be found in \cite{AB2-a}, as well as detailed examples for spanning trees of $K_4$).
(it was briefly introduced in  \cite{GiLV05}).
A noticeable feature of the active partition of a spanning tree is that it depends only on the fundamental cycles/cocyles of the spanning tree, but not on the whole graph (in fact, it can be generally seen %more generally 
as a decomposition of a bipartite graph on a linearly ordered set of vertices: edges of the spanning tree are considered as one part of a new set of vertices, the complementary set of edges form the other part, and two vertices 
%of the resulting bipartite graph 
are adjacent if they belong to the same fundamental cycle/cocycle).
%where vertices of one part are edges of a spanning tree, and vertices of the other part are edges of the complemntary set, 
%see \cite{AB2-a}).
\emevder{*** ce qui precede est-il a mettre ?***}%
Again, more details and constructions can be found in \cite{AB2-a}, as well as %a set of 
detailed examples on spanning trees of $K_4$ (consistently with Section~ \ref{sec:example}).

\begin{thm} %[\cite{Gi02, AB2-a}]  %[Main Theorem: Active Decomposition of Bases] %{50}
\label{EG:th:dec_base}
Let $G$ be a graph on a linearly ordered set of edges $E$.
%
%nd ${\cal B}$ the set of bases of $M$.\par
$$
%{\cal B}
\bigl\{\ \text{spanning trees of }G\ \bigr\}\ \ 
=\biguplus_
{\substack{
\emptyset=F'_\ep\subset...\subset F'_0=F_c\\[1mm]
F_c=F_0\subset...\subset F_\iota=E\\[1mm]
\text{\small connected filtration of $\scriptsize G$}
}}
\Biggl\{\ \ T'_1\plus...\plus T'_\ep\plus T_1\plus...\plus T_\iota
\ \ \mid$$
%$$\forall 1\leq k\leq \epsilon, \
%B'_k \hbox{ base with activities }(0,1)\hbox{ of } M(F'_{k-1})/F'_{k},$$
$$\text{for all }1\leq k\leq \ep, \
T'_k \hbox{ spanning tree of }G(F'_{k-1})/F'_{k}\text{ with $|\Int(T'_k)|=0$ and $|\Ext(T'_k)|=1$,}\hphantom{\ \ \Biggr\}}$$
%$$\forall 1\leq k\leq \iota, \
%B_k \hbox{ base with activities }(1,0)\hbox{ of } M(F_k)/F_{k-1} \}
%$$
$$\text{for all }1\leq k\leq \io, \
T_k \hbox{ spanning tree of }G(F_k)/F_{k-1}\text{ with $|\Int(T_k)|=1$ and $|\Ext(T_k)|=0$}\ \ \Biggr\}\hphantom{, }$$
%\bigskip
With $T=T'_1\plus...\plus T'_\ep\plus T_1\plus...\plus T_\iota$
we then have:
$$\Int(T)\ =\ \uplus_{1\leq k\leq \iota} \min(F_k\s F_{k-1})\ =\ \uplus_{1\leq k\leq \iota} \Int(T_k),$$
$$\Ext(T)\ =\ \uplus_{1\leq k\leq \ep} \min(F'_{k-1}\s F'_{k})\ =\ \uplus_{1\leq k\leq \ep} \Ext(T'_k).$$
\end{thm}

\begin{proof}[Proof of Theorem \ref{EG:th:dec_base}]
This is direct from  Theorem \ref{EG:th:alpha} and Theorem \ref{th:dec-ori}.
More precisely, by Theorem \ref{EG:th:alpha}, a spanning tree $T$ is the image of an orientation $\G$ by $\alpha$, hence it is a union of uniactive internal/external spanning trees in minors of $G$ induced by the active filtration of $\G$.
Conversely, for any connected filtration of $G$, the uniactive internal/external spanning trees of the minors induced by the sequence are images of some bipolar/cyclic-bipolar minors of an orientation $\G$, by Theorem \ref{th:dec-ori}.
\end{proof}

%\red{consitent with ABgraphs + algo thm + donner def dec seq bases avec equivalnces}

%\red{repeter def inductive du lemme ?}

From the above result and the constructions of 
Section \ref{subsec:alpha-def-decomp}, we can derive the next definition, followed by  its multiple equivalent constructions  (completed with statements from \cite{GiLV05,AB2-a}).

%Naturally, as for orientations in Section \ref{sec:tutte}, the active filtration of  spanning tree is a special connected filtration of the graph, and the active partition is the partition of the edge-set induced by successive differences of subsets in the filtration (so that the smallest elements of its parts are the internally/externally active elements of the spanning tree).

%\red{peut eter mieux une def  (via thm) + via alpha, suivi d'un thm de defs equivalentes constructives}

\begin{definition}
\label{def:sp-tree-dec-seq}
Let $G$ be a %(connected) 
graph on a linearly ordered set of edges $E$.
Let $T$ be a spanning tree of $G$.
The 
%\emph{active partition/filtration of $T$ in $G$}
\emph{active filtration of $T$ in $G$}
is the unique connected filtration of $G$ associated to $T$ in the decomposition given by Theorem \ref{EG:th:dec_base}. 
It is thus the unique filtration  $\emptyset= F'_\ep\subset...\subset F'_0=F_c=F_0\subset...\subset F_\io= E$ of $G$ such that:
\vspace{-1mm}
\begin{itemize}
\partopsep=0mm \topsep=0mm \parsep=0mm \itemsep=0mm
\item for  $1\leq k\leq\io$, $T\cap (F_{k}\s F_{k-1})$ is an internal  uniactive spanning tree of $G(F_k)/F_{k-1}$,
\item for  $1\leq k\leq\ep$, $T\cap (F'_{k-1}\s F_{k})$ is an external  uniactive spanning tree of $G(F'_{k-1})/F'_{k}$
\end{itemize}
\noindent (such a filtration is necessarily connected, otherwise one of the induced minors has no spanning tree with the required property, see Lemma \ref{lem:connected-filtration-beta}).
The \emph{active partition of $T$ in $G$} is the partition of $E$ formed by successive differences of subsets in the active filtration 
%(the smallest elements of its parts are the internally/externally active elements of $T$). 
(yielding parts whose smallest elements are the internally/externally active elements of $T$). 
%
%Equivalently, the active filtration/partition of $T$ in $G$ can be defined as the active filtration/partition  of any orientation $\G$ of $G$ such that $\alpha(\G)=T$.
%
%Equivalently, the active filtration/partition of $T$ in $G$ can be
%defined by the constructions of Proposition \ref{prop:act-part-sp-tree-construction}, which use only fundamental cycles/cocycles of $T$ in $G$.
%%directly  built from the fundamental cycles/cocycles of $T$ in $G$ 
\end{definition}

\begin{observation}
The active filtration/partition of $T$ in $G$ can also be defined as the active filtration/partition  of any orientation $\G$ of $G$ such that $\alpha(\G)=T$ (by Theorem \ref{EG:th:alpha}).
\end{observation}

\begin{prop}
\label{prop:act-part-sp-tree-construction}
Let $G$ be a %(connected) 
graph on a linearly ordered set of edges $E$.
Let $T$ be a spanning tree of $G$.
The 
active filtration/partition of $T$ in $G$
 can %also 
 be   directly  built using only the fundamental cycles/cocycles of $T$ in $G$ and the active closure operator by the following equivalent manners.
\begin{itemize}
\partopsep=0mm \topsep=0mm \parsep=0mm \itemsep=1mm
\item Using the active closure in an inductive way (see Definition \ref{def:active-closure} in Section \ref{subsec:alpha-def-decomp}).

Assume
either $a=\max(\Ext(T))$, or $a=\max(\Int(T))$.
Then,
the part of the active partition of $T$ containing $a$ is $\AA(\{a\})$
(by Lemma \ref{lem:th-bij-gene-part2}).
Then,  %we have that 
removing the part $\AA(\{a\})$ from the active partition of $T$ yields the active partition of $T\setminus \AA(\{a\})$ in $G/\AA(\{a\})$ if  $a=\max(\Ext(T))$, or in $G\setminus \AA(\{a\})$ if $a=\max(\Int(T))$
(this is obvious by Observations \ref{obs:induced-dec-seq-ori} and \ref{obs:induc-alpha} applied to an orientation $\G$ such that $\alpha(\G)=T$).

\item Using the active closure in a direct global way.

%Assume $T$ has $\io$ internally active elements $a_1<...<a_\io$ and $\ep$ externally active elements $a'_1<...<a'_\ep$.
Assume $\Int(T)=\{a_1,...,a_\io\}_<$ and $\Ext(T)=\{a'_1,...,a'_\ep\}_<$.
It turns out that the active filtration 
 $\emptyset= F'_\ep\subset...\subset F'_0=F_c=F_0\subset...\subset F_\io= E$ of $T$ satisfies:
 \begin{itemize}[$\circ$]
 \item $F_c=\AA(\Ext(\F))=E\s \AA(\Int(\F));$
 \item $F_k=E\setminus \AA(\{a_{k+1},\dots, a_\io\})$, for every $0\leq k\leq\io-1$;

\item $F'_k=\AA(\{a'_{k+1},\dots, a'_\ep\})$, for every $0\leq k\leq\ep-1$.

 \end{itemize}
%$$F_c=\AA(\Ext(\F))=E\s \AA(\Int(\F));$$
%for every $0\leq k\leq\io-1$, we have
%$$F_k=E\setminus \AA(\{a_{k+1},\dots, a_\io\});$$
%and for every $0\leq k\leq\ep-1$, we have
%$$F'_k=\AA(\{a'_{k+1},\dots, a'_\ep\}).$$
%
%

This is the definition that was given in \cite[Section 5]{GiLV05}. The equivalence with the above one is proved in \cite{AB2-a} (among various properties and alternative constructions %available for 
of the active closure).
% (they are consistent by setting $a=a_\io$ or $a=a'_\ep$).

\item Using a linear single pass algorithm over $E$.\emevder{OU using the active closure by a signle pass algorithm}

This construction is contained in Theorem \ref{th:basori} below, and it is  proved in \cite{AB2-a} too (by means of a more general single pass construction of the active closure).
%\item \red{use in $G$ first} recursively using Lemma \ref{lem:th-bij-gene-part2} in successive minors of the graph $G$ of type $G(F_{\io-1})$ of $G/F_{\ep-1}$; %, which is consistent with the definition given in \cite{GiLV05};
%\red{reecrire proprietes du lemme ! ou les citer explicitemnt }
%\item or by extending Lemma \ref{lem:th-bij-gene-part2} to compute all subsets of the sequence from the graph $G$, this is the definition that was given in \cite[Section 5]{GiLV05};  (see \cite{AB2-a} for more details).
%\item or by a linear single pass algorithm over $E$, which is recalled in Theorem \ref{th:basori}.
\end{itemize}
\end{prop}

\subsection{The three levels of the active bijection starting from spanning trees - An all-in-one single-pass construction from spanning trees}
%\subsection{An all-in-one single-pass construction from spanning trees/subsets to orientations (from \cite{AB2-a,AB2-b})}
\label{subsec:basori}

Concerning the uniactive bijection addressed in Section \ref{subsec:fob},
starting from a uniactive spanning tree, it is obvious to direct the edges one by one so that the criterion %for edge directions 
of Definitions \ref{def:acyc-alpha} or \ref{def:cyc-alpha1} is satisfied. 
We obtain the next algorithm which  works for uniactive internal or external spanning trees as well. 
%Note that \cite[Proposition 3]{GiLV05}
See \cite[Proposition 3]{GiLV05} for details and for 
two alternative dual formulations, in terms of cycles only or cocycles only.
Note that this algorithm consists in a single pass over the edge-set, which is extended to all spanning trees thereafter, whereas the direct computation of $\alpha$ is not easy.
%As mentioned in Section \ref{sec:intro} and \cite{ABG2LP}, 
This ``one way function'' feature  of the active bijection is noteworthy (see also Section \ref{sec:intro} and \cite{ABG2LP}).
%is a  noteworthy aspect of the active bijection.

\begin{prop}[uniactive bijection from spanning trees, see also {\cite[Proposition 3]{GiLV05}}]
\label{prop:alpha-10-inverse}
Let $G$ be a graph on  a linearly ordered set of edges $E=\{e_1,\dots,e_n\}_<$. 
For a spanning tree $T$ with internal activity $1$ and external activity $0$, or internal activity $0$ and external activity $1$,
the two opposite orientations of $G$ whose image under $\alpha$ is $T$ are computed by the following algorithm.

\begin{algorithme}
%\underbar{Direct definition of} $Basori^{(1,0)}_M(B)$\par
Orient $e_1$ arbitrarily.\par
For $k$ from $2$ to $n$ do\par
\hskip 10 mm if $e_k\in T$ then \par
\hskip 20 mm let $a=\min (C^*(T;e_k))$\par
\hskip 20 mm orient $e_k$ in order to have $a$ and $e_k$ with opposite directions in $C^*(T;e_k)$\par
\hskip 10 mm if $e_k\not\in T$ then \par
\hskip 20 mm let $a=\min (C(T;e_k))$\par
\hskip 20 mm orient $e_k$ in order to have $a$ and $e_k$ with opposite directions in $C(T;e_k)$\par
\end{algorithme}
%\qed
\end{prop}

The next definition is a direct rephrasing of Definition \ref{def:alpha-seq-decomp}, once Theorem \ref{EG:th:alpha} is acknowledged.

\begin{prop}[canonical active bijection from spanning trees]
\label{prop:preimage-sp-tree}
Let $G$ be an ordered graph. Let $T$ be a spanning tree of $G$, with active filtration  $\emptyset= F'_\ep\subset...\subset F'_0=F_c=F_0\subset...\subset F_\io= E$.
Let us denote $\alpha_G^{-1}(T)$ the set of orientations of $G$ whose image under $\alpha$ is $T$. Then we have:
% and, for two orientations $\G_1$ and $\G_2$ of minors of $G$ with disjoint edge-set, let us denote $\G_1\times\G_2$ the or
$$\alpha_G^{-1}(T)\ =\ 
\bigtimes_{1\leq k\leq\io}
\alpha_{G(F_k)/F_{k-1}}^{-1}(T\cap (F_{k}\s F_{k-1}))\
\times \ 
\bigtimes_{1\leq k\leq\ep}
\alpha_{G(F'_{k-1})/F'_{k}}^{-1}(T\cap (F'_{k-1}\s F_{k}))\
   $$
where $\times$ means that the $2^{\io+\ep}$ resulting orientations of $G$ are inherited from the orientations of the involved minors the natural way (and where each induced spanning tree of a minor is uniactive).%
\qed
\end{prop}

The inverse image under the refined active bijection can be directly defined by specifying the orientation within its activity class, %defined above, 
as discussed in Section \ref{subsec:refined},
and as reformulated below.
%. An equivalent formulation is the following.

\begin{prop}[refined active bijection from subsets]
\label{prop:preimage-subset}
Let $G$ be an ordered graph with reference orientation $\G$. 
Let 
%$A=T\setminus P\cup Q$ be a subset of $E$ 
$A$ be a subset 
in the interval of a spanning tree  of $G$ with active filtration  $\emptyset= F'_\ep\subset...\subset F'_0=F_c=F_0\subset...\subset F_\io= E$.
%Let us denote $\alpha_G^{-1}(T)$ the set of orientations of $G$ whose image under $\alpha$ is $T$. 
Then we have:
% and, for two orientations $\G_1$ and $\G_2$ of minors of $G$ with disjoint edge-set, let us denote $\G_1\times\G_2$ the or
%
%\ms
$$
\displaystyle\alpha_\G^{-1}(A)\ =\ 
\biguplus_{1\leq k\leq\io}
\alpha_{\G(F_k)/F_{k-1}}^{-1}(A\cap (F_{k}\s F_{k-1}))\
\uplus \ 
\biguplus_{1\leq k\leq\ep}
\alpha_{\G(F'_{k-1})/F'_{k}}^{-1}(A\cap (F'_{k-1}\s F_{k})).
   $$
%\hfill
%$\displaystyle\alpha_\G^{-1}(A)\ =\ 
%\biguplus_{1\leq k\leq\io}
%\alpha_{\G(F_k)/F_{k-1}}^{-1}(A\cap (F_{k}\s F_{k-1}))\
%\uplus \ 
%\biguplus_{1\leq k\leq\ep}
%\alpha_{\G(F'_{k-1})/F'_{k}}^{-1}(A\cap (F'_{k-1}\s F_{k})).
%   $
%\hfill
%\qed
\end{prop}

\emevder{A BIEN REVERIFIER !!!! (je l'ai ecrit tres vite fait, ainsi que la preuve)}

\emevder{on pourrait aussi donner une formula directe pour refined de ce type du coup je pense !!! en terms d'orietnations et d'ative filtation d'orietnation...}

\begin{proof}
This  is a straightforward reformulation of the construction of the refined active bijection discussed in Section \ref{subsec:refined}. Let us give details anyway.
Consider any of the  involved minors $H$, and the uniactive spanning tree $T_H$ induced in the minor $H$ by the involved spanning tree $T$. %associated with $A$. 
The inverse image of $T_H$ under $\alpha$ in $H$  consists of two opposite orientations of $H$.
Now consider the refined active bijection of $H$ w.r.t. the orientation of $H$ induced by $\G$, and denote $a$ the smallest edge of $H$.
One of the two above orientations is associated to $T_H$ (the one for which the orientation of $a$ agrees with $\G$),  and the other to $T_H\triangle \{a\}$.
Applying this to each minor $H$ and to any subset $A$ in the same interval, we always obtain a reorientation of $\G$ whose image under  $\alpha_\G$ is $A$.\emevder{reorientation mal dit, c'est un subset}
\end{proof}

%\begin{proof}
%In each involved minor $H$, the inverse image under $\alpha$ of the uniactive spanning tree $T_H$ induced by $T$ in $H$ consists of two opposite orientations.
%By properties of spanning tree intervals (Section \ref{
%\end{proof}

%\subsection{An all-in-one single-pass construction from spanning trees/subsets to orientations (\cite{AB2-a,AB2-b})}
%\subsection{An all-in-one single-pass construction from spanning trees or subsets to orientations}
%%\subsection{An all-in-one single-pass construction from spanning trees/subsets to orientations (from \cite{AB2-a,AB2-b})}
%\label{subsec:basori}

For completeness of the overview given in this paper, we give below a direct construction from spanning trees/subsets to orientations, using graph terminology. It is stated in \cite{AB2-b}  for general oriented matroids and the proof essentially relies upon \cite{AB2-a} (or also  \cite{Gi02}).
It combines the inverse computation of fully optimal spanning trees in bipolar minors (Proposition \ref{prop:alpha-10-inverse})\emevder{attention si enlevee de ABG2} with the computation of the active partition from \cite{AB2-a}. Noticeably, it uses only the fundamental cycles and cocycles of the spanning tree, not the whole graph structure. 
%is given in \cite{AB2-a,AB2-b, Gi02} for general oriented matroids. 
The single pass linear algorithm below builds at the same time: the active partition of a spanning tree (Theorem \ref{EG:th:dec_base}, refining the partition into internal/external edges from \cite{EtLV98}),
the premiage of a spanning tree under the canonical active bijection (Theorem \ref{EG:th:alpha} and Proposition \ref{prop:preimage-sp-tree}), and the preimage of a subset under the refined active bijection (Theorem~\ref{EG:th:ext-act-bij}  and Proposition \ref{prop:preimage-subset}). \emevder{ajouter ma these en ref?}
%It combines the computation of the fully optimal spanning tree from and the computation of the active paertition.
%\red{ajouter refs aux defs/ths + a dec Etienne LV + combinaison avec 1,0}

\begin{thm}[all-in-one single-pass algorithm from spanning trees \cite{AB2-a,AB2-b}] %[\cite{AB2-b}]
\label{th:basori}
Let $G$ be a graph on a linearly ordered set of edges $E=e_1<\ldots<e_n$.
Let $T$ be a spanning tree of $G$.

In the algorithm below, the active partition of $T$ is computed as a mapping, denoted $\ass$, from $E$ to $\Int(T)\cup \Ext(T)$, that maps an edge onto the smallest element of its part in the active partition of $T$. An edge is called internal, resp. external, if its image is in $\Int(T)$, resp. $\Ext(T)$.

The set of $2^{|\Int(T)|+|\Ext(T)|}$ orientations formed by the preimages of $T$ under $\alpha$ in the graph $G$, denoted here $\alpha_G^{-1}(T)$,  is computed by doing all possible arbitrary choices to orient $e_k$
during the algorithm.
%$e_k\in A$ or $e_k\not\in A$  left open during the algorithm.
%
%The unique orientation $-_A\G$ preimage of $X$ under $\alpha_\G$ is computed by doing the choice \red{LEQUEL ?????} during the algorithm.
%
Equivalently, those preimages under $\alpha$ can also be retrieved from one another since we have
%$$\alpha^{-1}(T)=\{\ A\triangle \t A\mid \t A=\ass^{-1}(P\cup Q),\ P\subseteq \Int(B),\ Q\subseteq \Ext(B),\  A\in \alpha^{-1}(T) \ \}.$$
$$\alpha_G^{-1}(T)=\{\ A\ \triangle \ \ass^{-1}(P\cup Q)\ \mid\ 
P\subseteq \Int(T),\ Q\subseteq \Ext(T),\  A\in \alpha_G^{-1}(T) \ \}.$$
%
%Let $B$ be a basis of $M$, ordered oriented matroid on $E=e_1<...<e_n$.
%The image of $B$ by $Basori^{(1,0)}_M$ is calculated by the following algorithm which builds at the same time the active partition  $\ass$ of $B$ and an element $A$ of $Basori^{(1,0)}_M$.\par
%The $2^{\io(B)+\ep(B)}$ elements $B$ are obtained either by doing all possible choices $e\in A$ or $e\not\in A$ 
%during the algorithm,
%either beginning with an element $A\in Basori_M(B)$ with $$Basori_M(B)=\{\ A\triangle \t A\mid \t A=\ass^{-1}(X),\ 
%X\subseteq \Int(B)\cup \Ext(B),\  A\in Basori_M(B) \ \}$$

Let $\G$ be a reference orientation of $G$, 
and let $X$ be a subset of $E$. We assume that $X=T\setminus P\cup Q$ with $P\subseteq \Int(T)$ and $Q\subseteq \Ext(T)$, or equivalently that $T=X\setminus Q\cup P$ with $Q=Q_G(X)$ and $P=P_G(X)$ 
%(cf. interval partition \red{xxxxx}).
(see Definition \ref{def:gene-act-base}).
We also derive the preimage of $X$ under $\alpha_\G$.

%\red{ORIENTER EDGES COMME DANS PROP 4.1, avec variantes selon canonical ou refined}

\begin{algorithme} \par
%\underbar{Direct computation of $Basori_M(B)$}
\underbar{Input}: \vtop{\emph{either} a spanning tree $T$ of $G$ alone,\par
 \emph{or} a spanning tree $T$ of $G$ and a subset $X=T\setminus P\cup Q$  in the interval of $T$.}
 \par
\underbar{Output}: \vtop{\emph{either} all orientations of $G$ in $\alpha_G^{-1}(T)$, \par
\emph{or} the %orientation of $\G$ obtained by 
reorientation of $\G$ w.r.t. $\alpha_\G^{-1}(X)$.}
\par

For $k$ from $1$ to $n$ do\par

\hskip 5mmif $e_k\not\in T$ then\par
\hskip 5mm\hskip 5mm if $e_k$ is externally active w.r.t. $T$ then \par
\hskip 5mm\hskip 10mm $e_k$ is external, $\ass(e_k):=e_k$,
%and arbitrary choice $e_k\in A$ or $e_k\not\in A$
%\par
%\hskip 5mm\hskip 10mm 
orient $e_k$ \emph{either} arbitrarily \emph{(to compute $\alpha^{-1}(T)$)}\par
\hskip 5mm\hskip 10mm \emph{or} with the same direction as in $\G$ if and only if $e_k\not\in Q$ \emph{(to compute $\alpha_\G^{-1}(X)$)}\par
\hskip 5mm\hskip 5mm otherwise\par

\hskip 5mm\hskip 10mm if there exists $c<e_k$ internal in $C(T;e_k)$ then\par

\hskip 5mm\hskip 15mm $e_k$ is internal\par
\hskip 5mm\hskip 15mm let $c$ $\in$ $C(T;e_k)$ with $c<e_k$, $c$ internal and $\ass(c)$ the greatest possible\par
\hskip 5mm\hskip 15mm let $\ass(e_k):=\ass(c)$\par
\hskip 5mm\hskip 10mm otherwise\par

\hskip 5mm\hskip 15mm $e_k$ is external\par
\hskip 5mm\hskip 15mm let $c$ $\in$ $C(T;e_k)$ with $c<e_k$ and $\ass(c)$ the smallest possible\par
\hskip 5mm\hskip 15mm let $\ass(e_k):=\ass(c)$ \finsi\par

\hskip 5mm\hskip 10mm let $a$ be the smallest possible in $C(T;e_k)$ with $\ass(a)=\ass(e_k)$\par

%\hskip 5mm\hskip 10mm if $e_k$ and $a$ have opposite directions (or signs) in $C(T;e_k)$ then\par
%%\hskip 5mm\hskip 10mm if $\sigma_{C(T;e_k)}(e_k)\not=\sigma_{C(T;e_k)}(a)$ then\par
%\hskip 5mm\hskip 30mm $e_k\not\in A$ if and only if $a\not\in A$ \finsi\par
%\hskip 5mm\hskip 10mm if $e_k$ and $a$ have the same directions (or signs) in $C(T;e_k)$ then\par
%%\hskip 5mm\hskip 10mm if $\sigma_{C(T;e_k)}(e_k)=\sigma_{C(T;e_k)}(a)$ then\par
%\hskip 5mm\hskip 30mm $e_k\not\in A$ if and only if $a\in A$ \finsi\finsi\par
\hskip 5mm\hskip 10mm orient $e_k$ so that $e_k$ and $a$ have opposite directions in 
$C(T;e_k)$\par

\hskip 5mmif $e_k\in T$ then \emph{(note: the below rules are dual to the above ones)}\par
\hskip 5mm\hskip 5mm if $e_k$ is internally active  w.r.t. $T$ then\par
\hskip 5mm\hskip 10mm $e_k$ is internal, $\ass(e_k):=e_k$, 
%and arbitrary choice $e_k\in A$ or $e_k\not\in A$\par
orient $e_k$ \emph{either} arbitrarily \emph{(to compute $\alpha^{-1}(T)$)}\par
\hskip 5mm\hskip 10mm \emph{or} with the same direction as in $\G$ if and only if $e_k\not\in P$ \emph{(to compute $\alpha_\G^{-1}(X)$)}\par
\hskip 5mm\hskip 5mm otherwise\par

\hskip 5mm\hskip 10mm it there exists $c<e_k$ external in $C^*(T;e_k)$ then\par

\hskip 5mm\hskip 15mm $e_k$ is external\par
\hskip 5mm\hskip 15mm let $c$ $\in$ $C^*(T;e_k)$ with $c<e_k$, $c$ external and $\ass(c)$ the greatest possible \par
\hskip 5mm\hskip 15mm let $\ass(e_k):=\ass(c)$ \par

\hskip 5mm\hskip 10mm otherwise\par
\hskip 5mm\hskip 15mm $e_k$ is internal\par
\hskip 5mm\hskip 15mm let $c$ $\in$ $C^*(T;e_k)$ with $c<e_k$ and $\ass(c)$ the smallest possible\par
\hskip 5mm\hskip 15mm let $\ass(e_k):=\ass(c)$ \finsi\par

\hskip 5mm\hskip 10mm let $a$ be the smallest possible in $C^*(T;e_k)$ with $\ass(a)=\ass(e_k)$\par
%\hskip 5mm\hskip 10mm if $e_k$ and $a$ have opposite directions (or signs) in $C^*(T;e_k)$ then\par
%%\hskip 5mm\hskip 10mm if $\sigma_{C^*(T;e_k)}(e_k)\not=\sigma_{C^*(T;e_k)}(a)$ then\par
%\hskip 5mm\hskip 30mm $e_k\not\in A$ if and only if $a\not\in A$ \finsi\par
%\hskip 5mm\hskip 10mm if $e_k$ and $a$ have the same directions (or signs) in $C^*(T;e_k)$ then\par
%%\hskip 5mm\hskip 10mm if $\sigma_{C^*(T;e_k)}(e_k)=\sigma_{C^*(T;e_k)}(a)$ then\par
%\hskip 5mm\hskip 30mm $e_k\not\in A$ if and only if $a\in A$ \finsi\finsi\finpour \par
\hskip 5mm\hskip 10mm orient $e_k$ so that $e_k$ and $a$ have opposite directions in 
$C^*(T;e_k)$
\end{algorithme}%
\end{thm}
\eme{QUESTION A VOIR AVEC AB2:
dans algo basori, ne peut-on prendre le c (qui sert a definir partie) directement egal au e (ou a) qui sert a comparer 
orietnations dans cycle/cocyle fondamental ?}
%\red{a mentionner dans intro}
%

%%%%%%%%%%%%%%%%%%%%%%%%%%%%%%%%%%%%%%%%%%%%%%%%%%%%%%%%%

\vspace{-2mm}
\section{Constructions by deletion/contraction}
\label{sec:induction}

%\addtocontents{toc}{\protect\setcounter{tocdepth}{1}}

%\vspace{-2mm}
%\red{notation $o$ definie ? utilisee ?}

We address deletion/contraction constructions for the three levels of the active bijection. 
As we will show, in contrast with the previous constructions, these constructions can be thought of as building the whole bijections at once, roughly said as $1-1$ correspondences between orientations and spanning trees/subsets rather than as pairs of inverse mappings from one side to the other.
We state these inductive constructions 
%of the active bijection 
the simplest way, so that they are directly related to what precedes in this paper. 
This whole section is  generalized and developed further in \cite{AB4}, notably with more practical conditions equivalent to the ones used in the following algorithms. 
\emevder{pas pour uniactive, qui est tel que comme ca}%
At the end, we also present 
%in Section \ref{subsec:ind-framework} 
how these constructions fit in a general deletion/contraction framework
for building correspondences/bijections 
involving graduated activity preservation constraints, amongst which the active bijection is uniquely determined by its canonical or natural properties.

\subsection{The uniactive bijection}
\label{subsec:alpha-10-ind}

This section repeats  results and condenses remarks from the companion paper \cite[Section \ref{ABG2LP-sec:induction}]{ABG2LP}%
\emevder{verifier section ---- faire numerotation latex avec ref croisee comme pour ab2-a ab2-b ?}%
\footnote{Note: Theorem \ref{thm:ind-10} is also stated in the companion paper \cite{ABG2LP}, which is also submitted. At the moment, we give its proof in both papers, including Lemma \ref{lem:induc-fob-basis},  but we should eventually remove this repetition and give the proof in only one of the two papers.}%
.

%\red{ou the canonical uniactive bijection...}

%\new{the mapping $\alpha$ ??? the mapping $\G\mapsto \alpha(\G)$ is a bijection --- IDEM AB2....!!!!}

%\new{FIN : juste lister rk de ABG2LP et renvoyer a ABG2LP}

%\new{a mettre dans article sur comput fullyoptial sp tree ??? avec ou sans preuve, mettre note pour reviewer}

%\new{a mettre avant direct comput}

%\new{clarifier difference avec 4.3 !!!!! conceptuel et compelxite}
%
%\red{XXX induction dans ABG2LP} 
%We give an alternative construction to that of Theorem \ref{th:fob}. Its formulation is shorter and simpler but it involves potentially an exponential number of minors. %obtained by successive deletion/contraction of the greatest edge.

\emevder{voir si ce lemme utilise, sinon renvoyer a ABG2LP --- OUI utilise dans preuve de Prop pour canonical act bij, mais juste pour propriete qu'un des deux est bipolaire, quise montre facilemetn sans ca... QUE FAIRE ??? remontrer ca dans la prevue, citer le lemme de auter ppeir, recopier lemme faire autre lemme ??? j'en ai marre je laisse comme ca pour le mment !}%

\begin{lemma}
\label{lem:induc-fob-basis}
Let $\G$ be a digraph,  on a linearly ordered set of edges $E$, which is bipolar w.r.t. $p=\min(E)$. Let $\n$ be the greatest element of $E$. Let $T=\alpha(\G)$. 
If $\n\in T$ then $\G/\n$ is bipolar w.r.t. $p$ and $T\setminus\{\n\}=\alpha(\G/\n)$.
If $\n\not\in T$ then $\G\backslash\n$ is bipolar w.r.t. $p$ and $T=\alpha(\G\backslash\n)$.
In particular, we get that $\G/\n$ is bipolar w.r.t. $p$ or $\G\backslash\n$ is bipolar w.r.t. $p$.
\end{lemma}

\begin{proof}
First, let us recall that if a spanning tree of a directed graph satisfies the criterion of Definition \ref{def:acyc-alpha}, then this directed graph  is necessarily bipolar w.r.t. its smallest edge. 
This is implied by \cite[Propositions 2 and 3]{GiLV05}, or also stated explicitly in \cite[Proposition 3.2]{AB1},
and this is easy to see: if the criterion is satisfied, then 
the spanning tree is internal uniactive (by definitions of internal/external activities)
and the digraph is determined up to reversing all edges (see Proposition \ref{prop:alpha-10-inverse}\emevder{ATTENTION ptet pas dans ABG2 ???}), 
which implies that the digraph is in the inverse image  of $T$ by the uniactive bijection of Theorem \ref{thm:bij-10} and that it is bipolar w.r.t. its smallest edge.

Assume that $\n\in T$.
Obviously, the fundamental cocycle of $b\in T\s\{\n\}$ w.r.t. $T\s\{\n\}$ in $G/\n$ is the same as the fundamental cocycle of $b$ w.r.t. $T$ in $G$.
%$$C^*_{\G/\n}(T\s\{\n\};b)=C^*_{\G}(T;b).$$
And the fundamental cycle of $e\not\in T$ w.r.t. $T\s\{\n\}$ in $G/\n$ is obtained by removing $\n$ from the fundamental cycle of $e$ w.r.t. $T$ in $G$.
Hence, those fundamental cycles and cocycles in $G/\n$ satisfy the criterion of Definition \ref{def:acyc-alpha}, hence $\G/\n$ is bipolar w.r.t. $p$ and $T\setminus\{\n\}=\alpha(\G/\n)$.

Similarly (dually in fact), assume that $\n\not\in T$.
%
%For $b\in T\s\{\n\}$ and $e\not\in T\s\{\n\}$, we have: $$C^*_{\G\bk\n}(T\s\{\n\};b)=C^*_{\G}(T;b)\s\{\n\}\ \ \text{ and }\ \ C_{\G\bk\n}(T\s\{\n\};e)=C_{\G}(T;e).$$
%
The fundamental cocycle of $b\in T$ w.r.t. $T\s\{\n\}$ in $G\bk\n$ is obtained by removing $\n$ from the fundamental cocycle of $b$ w.r.t. $T$ in $G$.
%$$C^*_{\G/\n}(T\s\{\n\};b)=C^*_{\G}(T;b).$$
And the fundamental cycle of $e\not\in T\s\{\n\}$ w.r.t. $T\s\{\n\}$ in $G\bk\n$ is the same as the fundamental cycle of $e$ w.r.t. $T$ in $G$.
Hence, those fundamental cycles and cocycles in $G\bk\n$ satisfy the criterion of Definition \ref{def:acyc-alpha}, hence $\G\bk\n$ is bipolar w.r.t. $p$ and $T\setminus\{\n\}=\alpha(\G\bk\n)$.

Note that the fact that either $\G/\n$ is bipolar w.r.t. $p$, or $\G\backslash\n$ is bipolar w.r.t. $p$ could also easily be directly proved in terms of digraph properties.
\end{proof}

\begin{thm}
\label{thm:ind-10}
The fully optimal (or active) 
%active (or fully optimal) 
spanning trees %$\alpha(\G)$ 
of ordered bipolar digraphs %$\G$  
satisfy the
%The active mapping $\alpha$, restricted to ordered bipolar digraphs, satisfies the 
%%following.%
following inductive definition.
\emevder{OU the uniactive bijection satisfies...}
\emevder{NB debut d'enonce different de ABG2-LP, voir si on met cii comme la bas}

\begin{algorithme}

For any  ordered digraph $\G$ on $E$, bipolar w.r.t. $p=\min(E)$, and  with $\max(E)=\n$.%
%\smallskip

If $|E|=1$ then $\alpha(\G)=\n$.\par
If $|E|>1$  then:\par

%\blue{
%\hskip 10mm \emph{(See 
%%Lemma \ref{lem:cond-ind-10} NON 
%Proposition \ref{prop:cond-ind-gene} 
%for equivalent conditions to the three conditions below)}
%}

\hskip 10mm If $\G/\n$ is bipolar w.r.t. $p$ but not $\G\backslash \n$  then
$\alpha(\G)=\alpha(\G/\n)\cup\{\n\}$.
%\red{or equivalently if $-_\n\G$ is... se referer au lemme de subsection suivante}

\hskip 10mm If $\G\backslash\n$ is bipolar w.r.t. $p$ but not $\G/ \n$ then
$\alpha(\G)=\alpha(\G\backslash\n)$.

\hskip 10mm If both $\G\backslash\n$ and $\G/\n$ are bipolar w.r.t. $p$  then:

%\hskip 10mm 
%{\bf choice}\par 

\hskip 20mm
let $T'=\alpha(\G\backslash\n)$, $C=C_\G(T';\n)$ and $e=\min(C)$\par

\hskip 20mm
if $e$ and $\n$ have opposite directions in $C$ then $\alpha(\G)=\alpha(\G\backslash\n)$;\par
%if $\sigma_C(e)\not=\sigma_C(\n)$ then $\f(M):=0$;\par
\hskip 20mm
%if $\sigma_C(e)=\sigma_C(\n)$ then $\f(M):=1$;\par
if $e$ and $\n$ have the same directions in $C$ then $\alpha(\G)=\alpha(\G/\n)\cup\{\n\}$.\par
\smallskip

{\sl or equivalently:}\par
\hskip 20mm
let $T''=\alpha(\G/\n)$, $D=C^*_\G(T''\cup\n;\n)$ and $e=\min(D)$\par

\hskip 20mm
if $e$ and $\n$ have opposite directions in $D$ then $\alpha(\G)=\alpha(\G/\n)\cup\{\n\}$;\par
%if $\sigma_D(e)\not=\sigma_D(\n)$ then $\f(M):=1$;\par
\hskip 20mm
if $e$ and $\n$ have the same directions in $D$ then $\alpha(\G)=\alpha(\G\backslash\n)$.
%if $\sigma_D(e)=\sigma_D(\n)$ then $\f(M):=0$.\par
\smallskip

%{\bf end}\par
%\noindent If $\f(M)=0$ then
%$$Oribas^{(1,0)}(M):=Oribas^{(1,0)}(M\s\n)$$
%if $\f(M)=1$ then $$Oribas^{(1,0)}(M):=Oribas^{(1,0)}(M/\n)\cup \n$$
\end{algorithme}
\end{thm}

%\blue{
%\hskip 10mm \emph{(See 
%%Lemma \ref{lem:cond-ind-10} NON 
%Proposition \ref{prop:cond-ind-gene} 
%for equivalent conditions to the three conditions below)}
%}

%\red{a mettre en Remark ?}
%Observe that $-_\n\G$ is bipolar w.r.t. $p$ if and only if $\G\backslash\n$ and $\G/\n$ are bipolar w.r.t. $p$, and then the above algorithm builds at the same time $\alpha(\G)$ and $\alpha(-_\n\G)$, 
%
%\vspace{-6mm}

\begin{proof}
By Lemma \ref{lem:induc-fob-basis}, at least one minor among $\{\G/\n, \G\bk\n\}$ is bipolar w.r.t. $p$.
If exactly one minor among $\{\G/\n, \G\bk\n\}$ is bipolar w.r.t. $p$, then
by Lemma \ref{lem:induc-fob-basis} again, 
the above definition  is implied.  Assume now that both minors are bipolar w.r.t. $p$.

Consider $T'=\alpha(\G\bk\n)$. 
Fundamental cocycles of elements in $T'$ w.r.t. $T'$ in $\G$ are obtained by removing $\n$ from those in $\G\bk\n$. Hence they satisfy the criterion of Definition \ref{def:acyc-alpha}.
Fundamental cycles of elements in $E\s (T'\cup\{\n\})$ w.r.t. $T'$ in $\G$ are the same as in $\G\bk\n$.  Hence they satisfy the criterion of Definition \ref{def:acyc-alpha}.
Let $C$ be the fundamental cycle of $\n$ w.r.t. $T'$. 
If $e$ and $\n$ have opposite directions in $C$,
then $C$ satisfies the criterion of Definition \ref{def:acyc-alpha}, and
 $\alpha(\G)=T'$.
Otherwise, we have $\alpha(\G)\not=T'$, and, by Lemma \ref{lem:induc-fob-basis}, we must have $\alpha(\G)=\alpha(\G/\n)\cup\{\n\}$.

The second condition involving $T''=\alpha(\G/\n)$ is proved in the same manner. Since it yields the same mapping $\alpha$, then this second condition is actually equivalent to the first one, and so it can be used as an alternative. Note that the fact that these two conditions are equivalent is difficult and proved here in an indirect way (actually this fact is equivalent to the key result that $\alpha$ yields a bijection), see \cite[Remark \ref{ABG2LP-rk:ind-10-equivalence}]{ABG2LP}. %\new{raccourcir ???}
\end{proof}

\begin{cor}
\label{cor:ind-10}
%Observe that if 
We use notations of Theorem \ref{thm:ind-10}.
If $-_\n\G$ is bipolar w.r.t. $p$ then the above algorithm of Theorem \ref{thm:ind-10} builds at the same time $\alpha(\G)$ and $\alpha(-_\n\G)$, 
we have:
$$\Bigl\{\ \alpha(\G),\ \alpha(-_\n\G)\ \Bigr\}\ =\ \Bigl\{\ \alpha(\G\backslash\n),\ \alpha(\G/\n)\cup\{\n\}\ \Bigr\}.$$
Also, we have that $-_\n\G$ is bipolar w.r.t. $p$ if and only if $\G\backslash\n$ and $\G/\n$ are bipolar w.r.t. $p$.
%\hfill\qed
\end{cor}

\begin{proof}
Direct by Theorem \ref{thm:ind-10} and Theorem \ref{thm:bij-10} (bijection property).
\end{proof}

Important remarks on the above results are given in \cite[Section \ref{ABG2LP-sec:induction}]{ABG2LP}. Let us resituate them here.

\begin{remark}[equivalence in Theorem \ref{thm:ind-10}]
%[equivalence in Theorem \ref{thm:ind-10}%
%\footnote{\red{Note to the reviewer: as for today, this Remark \ref{rk:ind-10-equivalence} is also repeated in the companion paper \cite{ABG2LP}.}}]
\label{rk:ind-10-equivalence}
 \rm
% \red{A RELIRE, LOURD}
%Let us finally point out that, a priori, the equivalence in the above algorithm is a deep property. 
The equivalence of the two formulations in the algorithm of Theorem \ref{thm:ind-10} is a deep and difficult result, 
which we directly derive from the bijection provided by the Key Theorem \ref{thm:bij-10}.
% difficult to prove if no property of the computed mapping is known.
%%a deep property.
%%It is not obvious at all if no property of the computed mapping is known.
%Here, to prove it, we  implicitly use that $\alpha$ is already well-defined by Definition \ref{def:acyc-alpha}, 
%and bijective for bipolar orientations (Key Theorem \ref{thm:bij-10}).
%%with a bijectivity property when restricted to bipolar orientations with fixed orientation for $p$ (Theorem \ref{thm:bij-10}). % \cite[Theorem 4]{GiLV05}.
%But 
Actually, if one defines a mapping $\alpha$ from scratch as in the algorithm (with either one of the two formulations) and then investigates its properties, then it turns out that
the above equivalence result is %somehow %turns out to be 
equivalent to the  existence and uniqueness of the fully optimal spanning tree (Definition \ref{def:acyc-alpha}) and hence to this key theorem. 
See \cite{AB4} for precisions.
%
%It can be equivalently seen also as a duality result.

%, see also \cite[Section 5]{AB1} or \cite[Section \red{XXX}]{AB2-b}).
%
Furthermore, this equivalence result is also related to the active duality property addressed in Section~\ref{subsec:fob}.
First, recall that cyclic-bipolar orientations of $G$ w.r.t. $p$ with fixed orientation for $p$ are also in bijection with external uniactive spanning trees of $G$.
Thanks to the equivalence of these two dual formulations, % in the above algorithm, 
one can directly adapt the above algorithm for this second bijection, with no risk of inconsistency.
\emevder{ATTENTION: TOUTE CETTE RK A BIEN VERIFIER, ET A VOIR AVEC - ET REPRENDRE DANS - AB4 !!!!!!}%
\emevder{cette remark a mettre dans ABG2 ou dans ABG2LP ? je l'ai mise dans le s 2 en reformulant unpeu, c'est bien car important.}%
Second,
let us mention that, %in terms of linear programming,
as well as the active duality property is a strengthening of linear programming duality,  
the deletion/contraction algorithm of Theorem \ref{thm:ind-10} corresponds to a refinement of the classical linear programming solving by constraint/variable  deletion, see \cite{GiLV04, AB3}. In these terms, and in oriented matroid terms as well,  the above equivalence result means that 
 the same algorithm can be equally used in the dual, with no risk of inconsistency.
 \emevder{au debut j'avais mis beaucoupplus brievement :More precisely, in the context of a deletion/contraction construction, the above equivalence result turns out to be equivalent to this bijectivity result, and it can be equivalently seen also as a duality result, see \cite{AB4}.}%
\end{remark}

\begin{remark}[computational complexity]
\label{rk:difficult}
\rm
%Concerning the computational complexity, %it is easy to see that 
Using the construction of Theorem \ref{thm:ind-10} to build one single image under $\alpha$ involves an exponential number of images of minors,  see details in \cite[Remark \ref{ABG2LP-rk:difficult}]{ABG2LP}.
However, this algorithm is efficient 
%in terms of computational complexity 
for building the images of all bipolar orientations of $G$ at once,  in the sense that, with $|E|=n$, the number of calls to the algorithm to build  these $O(2^n)$ images %of all bipolar orientations 
is  $O(n.2^n)$.  See details in \cite[Remark~\ref{ABG2LP-rk:ind-10}]{ABG2LP}, and see Remark \ref{rk:ind-10} below. % on this global~approach.
An efficent algorithm for building one single image, involving just one minor for each edge of the resulting spanning tree, is the main result of \cite{ABG2LP}.
\end{remark}

\begin{remark}[building the whole bijection at once, and the \choice notion]
\label{rk:ind-10}
\rm
By Corollary \ref{cor:ind-10}, %Observe that 
%we have that
 the construction of Theorem \ref{thm:ind-10} can be used to build the whole active bijection for $G$ (i.e. the $1-1$ correspondence, or the matching, between all bipolar orientations of $G$ w.r.t. $p$ with fixed orientation, and all internal uniactive spanning trees of $G$), from the whole active bijections for $G/\n$ and $G\bk\n$.
%%In that sense it builds the 1-1 correspondence for $G$ from the 1-1 correspondences for $G/\n$ and $G\bk\n$.
%Therefore, it could be formulated equally from orientations to bases, or from bases to orientations, or both simultaneously 
%(we can build the 1-1 correspondence for $G$ from the 1-1 correspondences for $G/\n$ and $G\bk\n$, 
%%(we can build the matching for $G$ from the matchings for $G/\n$ and $G\bk\n$, 
%using a local test and criterion, as highlighted in Corollary \ref{cor:ind-10},  that does not depend on the way of the construction, this remark is extended in Remark \ref{rk:ind-gene} and Section \ref{subsec:ind-framework}).
%\ss
For each pair of bipolar orientations $\{\G, -_\n\G\}$, the algorithm  provides which ``local choice'' is right to associate one orientation with the orientation induced in $G/\n$ and the other  with the orientation induced in $G\s\n$.
This \choice notion is extended to the general active bijection in Remark \ref{rk:ind-gene} and it
is formally developed in Section \ref{sec:induction} (and in \cite{AB4}) as the basic component for a deletion/contraction framework.
\end{remark}

\emevder{peut-eter detailler plutot preuves et remarques dans cet article car ABG2LP est plutot sur LP}%

%---------------------------------------------------------------------------------------
\subsection{The canonical active bijection}
\label{subsec:alpha-ind}

\eme{dessous a reprendre apres redaction de AB4}%
In this overview paper, we choose to give an inductive definition of the active bijection as concise as possible (based on definitions and proofs from the previous sections), 
but we point out that the algorithm below can be detailed further as a more practical algorithm in several ways, see Remark \ref{rk:ind-gene-alt}. 
%Notably: one could use a direct comparison  of the active partitions of $\G$, $-_\n\G$, $\G/\n$ and $\G\bk\n$ using only directed cycles/cocycles containing $\n$, more complete than the one given in Proposition \ref{prop:ind-gene} below; one could use a direct characterization of the sign involved in the (cyclic-)bipolar involved minor $\G_\n$, using fundamental cycles/cocycles in the original digraph $\G$, without having to compute this minor;
%%, extending that used in Theorem \ref{thm:ind-10}; 
%and one could use a more explicit case by case formulation of the underlying duality. 
%Those results are detailed in \cite{AB4}. Also, a short alternative formulation in the acyclic case is given in \cite{GiLV06}. 
%%Results below are proved here for completeness of the rpesent paper, from the results previously proved in the paper and from Theorems  \ref{???} and \ref{??} **** bij 10 et bij gene, a enoncer dans section defe de AB)}
%
In the next proposition, we use the properties of $\alpha$ to derive the minimum inductive properties of active partitions required to derive an inductive definition of $\alpha$. More involved and intrinsic inductive properties of active partitions are given and used in \cite{AB4}.
Also, a short alternative formulation in the acyclic case is given in \cite{GiLV06}. 

We call \emph{removing the greatest element of $E$ from an active partition of $E$}  the natural operation that consists in removing this element from its part in the active partition (and from the associated cyclic flat if it contains it), yielding another partition of $E$.
%\red{DEF DE ACTIVE PARTITION : LUI ASSOCIER D'OFFICE LE CYCLIC FLAT !!!! PLUS SIMPLE}

\begin{prop}
\label{prop:ind-gene}
Let $\G$ be a digraph,  on a linearly ordered set of edges $E$. Let $\n$ be the greatest element of $E$. Assume $\n$ is not an isthmus nor a loop of $G$. 
%Then, removing $\n$ from the active partition of $\G$ yields either the active partition of  $\G/\n$ or that of $\G\bk\n$. 
%Moreover, we have 
\begin{enumerate}[(i)]
\itemsep=0mm
\parsep=0mm
\item \label{it1-prop:ind-gene}
We have
$$\Bigl\{\ \alpha(\G),\ \alpha(-_\n\G)\ \Bigr\}\ =\ \Bigl\{\ \alpha(\G\backslash\n),\ \alpha(\G/\n)\cup\{\n\}\ \Bigr\}.$$

\item \label{it2-prop:ind-gene}
Moreover, if $\alpha(\G)=\alpha(\G\backslash\n)$, resp.
$\alpha(\G)=\alpha(\G/\n)\cup\{\n\}$, then
removing $\n$ from the active partition of $\G$ yields the active partition of $\G\bk\n$, resp. $\G/\n$. 

\item \label{it3-prop:ind-gene} In particular, 
 removing $\n$ from the active partition of $\G$ yields either the active partition of  $\G/\n$ or that of $\G\bk\n$. 
 
%Moreover, if $\G\bk\n$ and $\G/\n$ have the same active partition, then $\G$ and $-_\n\G$ also have the same active partition.
\item \label{it4-prop:ind-gene} Moreover, $\G\bk\n$ and $\G/\n$ have the same active partition if and only if $\G$ and $-_\n\G$  have the same active partition.
\end{enumerate}
\end{prop}

%\red{dans section act part, definir ``minor associated with...''}

%\red{ici definir removing an elemnt from a partition}

\begin{proof}
In what follows, bipolar and cyclic-bipolar are always meant w.r.t. the smallest edge.
Moreover, we will consider the %sequence of 
bipolar or cyclic-bipolar active minors induced by the active partition of $\G$ (Proposition \ref{prop:unique-dec-seq}), and denote $\G_\n$ the minor containing $\n$ among them, 
%whose edge set will be denoted $E_\n$. 
with edge set $E_\n$. 

%\noindent\emph{Claim 1 - Removing $\n$ form the active partition of $\G$ yields either the active partition of  $\G/\n$ or that of $\G\bk\n$.}
First, we prove (\ref{it3-prop:ind-gene}). Let us prove that removing $\n$ from the active partition of $\G$ yields either the active partition of  $\G/\n$ or that of $\G\bk\n$.
%Let us consider the sequence of bipolar or cyclic-bipolar minors induced by the active partition of $\G$ (Proposition \ref{prop:unique-dec-seq}), and $\G_\n$ the minor containing $\n$ among them, whose edge set is denoted $E_\n$. 
By Lemma \ref{lem:induc-fob-basis} (or by a direct easy proof)\emevder{ici Lemme de ABG2LP utilise, mais juste pte de bipolar}, we have that $\G_\n/\n$ or $\G_\n\bk\n$ is bipolar or cyclic-bipolar (if $\G_\n$ is cyclic-bipolar, then apply the lemma to $-_\n\G_\n$ which is bipolar, as recalled in Section \ref{sec:prelim}).
Replacing $\G_\n$ by this minor in the sequence of minors associated with $\G$ yields the sequence of minors induced by the partition of $E$ obtained by removing $\n$ from $E_\n$. This partition obviously corresponds  to a filtration of $G\bk\n$ or $G/\n$, and its induced minors are either bipolar or cyclic-bipolar (with no change of nature w.r.t. the minors given by the active partition of $\G$). Hence, by Proposition \ref{prop:unique-dec-seq}, it is necessarily the active partition of $G\bk\n$ or $G/\n$.

Now,  we prove (\ref{it1-prop:ind-gene}) and (\ref{it2-prop:ind-gene}). Let us prove that  $\alpha(\G)\in  \bigl\{ \alpha(\G\backslash\n), \alpha(\G/\n)\cup\{\n\} \bigr\}.$
%
%Denote $\G_\n$ the minor of $\G$ containing $\n$ associated with the active partition of $\G$, on set of edges $E_\n$.
%with smallest edge $a_\n$.
The minor $\G_\n$ is either bipolar or cyclic-bipolar.
In what follows, we can assume that it is bipolar. If it is cyclic-bipolar, then the same reasoning holds, up to applying it to $-_\n\G_\n$ which is bipolar (see Section \ref{sec:prelim}), and using Definition \ref{def:cyc-alpha2} which ensures the compatibility of $\alpha$ with this canonical bijection between bipolar and cyclic-bipolar orientations. We leave the details.
By Definition \ref{def:alpha-seq-decomp}, we have $\alpha(\G)=A\uplus \alpha(\G_\n)$ for some $A\subseteq E$.

Assume that removing $\n$ from the active partition of $\G$ yields the active partition of  $\G\bk\n$. 
By assumption, we have that $\G_\n\bk\n$ is bipolar.
By Definition \ref{def:alpha-seq-decomp}, 
%$\alpha(\G)=A\uplus \alpha(\G_\n)$ for some $A\subseteq E$.
%On the other hand, 
we have $\alpha(\G\bk\n)=A\uplus \alpha(\G_\n\bk\n)$
since the other minors induced by the active partition of $\G$ are unchanged by assumption, which implies that $A$ is also unchanged.
If $\alpha(\G_\n)=\alpha(\G_\n\bk\n)$ then $\alpha(\G)=\alpha(\G\bk\n)$.

Assume now that $\alpha(\G_\n)\not=\alpha(\G_\n\bk\n)$.
Then, by Theorem \ref{thm:ind-10}, we have 
that $\G_\n/\n$ is bipolar
and $\alpha(\G_\n)=\alpha(\G_\n/\n)\cup\{\n\}$.
% (in the bipolar case, or by Theorem \ref{thm:ind-10} and Definition \ref{def:cyc-alpha2} in the cyclic-bipolar case).
Since the other minors induced by the active partition of $\G$ are unchanged by removing $\n$, we get (as in the first paragraph) that  
 removing $\n$ from the active partition of $\G$ also yields the active partition of  $\G/\n$. 
 Hence $\alpha(\G/\n)=A\uplus \alpha(\G_\n/\n)$.
Hence $\alpha(\G)=\alpha(\G/\n)\cup\{\n\}$.

We have proved $\alpha(\G)\in  \bigl\{ \alpha(\G\backslash\n), \alpha(\G/\n)\cup\{\n\} \bigr\}$.
Notice that we have proved in the meantime:
if $\alpha(\G)=\alpha(\G/\n))\cup\{\n\}$
and removing $\n$ from the active partition of $\G$ yields the active partition of $\G\bk\n$, then removing $\n$ from the active partition of $\G$ also yields the active partition of $\G/\n$. 
So, in every case, we have:
if $\alpha(\G)=\alpha(\G/\n)\cup\{\n\}$ then
removing $\n$ from the active partition of $\G$ yields the active partition of $\G/\n$. 
A similar reasoning holds if $\alpha(\G)=\alpha(\G\bk\n)$. 

%We have proved $\alpha(\G)\in  \bigl\{ \alpha(\G\backslash\n), \alpha(\G/\n)\cup\{\n\} \bigr\}$, 
Now, by symmetry, we have proved also $\alpha(-_\n\G)\in  \bigl\{ \alpha(\G\backslash\n), \alpha(\G/\n)\cup\{\n\} \bigr\}$.
%It remains to 
We prove that $\alpha(\G)\not=\alpha(-_\n\G)$.
%Assume for a contradiction that  $\alpha(\G)=\alpha(-_\n\G)$. 
Otherwise,  $\G$ and $-_\n\G$ belong to the same activity class (Definition \ref{def:act-class}), implying that $\n$ is the unique element of its part in the active partition of $\G$, implying that $\n$ is an isthmus or a loop, which is forbidden by hypothesis.
So we have $\{\ \alpha(\G),\ \alpha(-_\n\G)\ \}\ =\ \{\ \alpha(\G\backslash\n),\ \alpha(\G/\n)\cup\{\n\}\ \}.$

Now, we prove (\ref{it4-prop:ind-gene}).
Assume that $\G$ and $-_\n\G$ have the same active partition. Then, obviously, by (\ref{it1-prop:ind-gene}) and (\ref{it2-prop:ind-gene}), removing $\n$ from this active partition yields the active partition of $\G\bk\n$ and also the active partition of $\G/\n$, and so these two active partitions are also equal.

Finally, assume that $\G\bk\n$ and $\G/\n$ have the same active partition.
Assume that removing $\n$ from the active partition of $\G$ yields the common active partition of  $\G\bk\n$ and $\G/\n$.
%As above, $\G_\n$ is defined as the minor of $\G$ containing $\n$ associated with the active partition of $\G$, defined on the edge set $E_\n$.
Assume that $\G_\n$ is bipolar.
Then $\G_\n\bk\n$ is defined on the edge set $E_\n\setminus\{\n\}$, and, by Proposition \ref{prop:unique-dec-seq},
it is a bipolar minor associated with the active partition of $\G\bk\n$
(the other minors are unchanged, as in the first paragraph).
In the same manner, $\G_\n/\n$ is a bipolar minor as it is associated with the active partition of $\G/\n$.
By Corollary \ref{cor:ind-10}, we then have that $-_\n\G_\n$ is also bipolar.
Hence the active partition of $\G$ satisfies the characterization given by Proposition \ref{prop:unique-dec-seq} for the active partition of $-_\n\G$, and so
$\G$ and $-_\n\G$ have the same active partition (that is: $-_\n\G_\n$ is the minor containing $\n$ associated with the active partition of $-_\n\G$).

If $\G_\n$ is cyclic-bipolar, then the same reasoning as above holds. 
We recall that the active partition is given with the information on the associated cyclic flat, that is on the bipolar or cyclic-bipolar nature of each minor, hence in this case we have that both  $\G_\n\bk\n$ and $\G_\n/\n$ are cyclic-bipolar. We end the same way up to reversing $\n$ and using the canonical bijection between bipolar and cyclic-bipolar orientations. 
So, finally, $\G$ and $-_\n\G$ have the same active partition.
\end{proof}

%\red{ATTENTION : indispensable d'associer cyclic flat a active partition !}

% ancienne formulation
%\begin{itemize}
%\item $\G\backslash\n$ is bipolar w.r.t. $p$ but not $\G/\n$ if and only if $-_\n\G$ is not acyclic, i.e. $o(-_n\G)>0$.
%\item  $\G/\n$ is bipolar w.r.t. $p$ but not $\G\backslash\n$ if and only if  $-_\n\G$ is acyclic but not bipolar, i.e. $o(-_\n\G)=0$ and $o^*(-_\n\G)>1$.
%\item Both $\G/\n$ and $\G\backslash\n$ are bipolar w.r.t. $p$ if and only if $-_\n\G$ is bipolar w.r.t. $p$, i.e. $o(-_\n\G)=0$ and $o^*(-_\n\G)=1$.
%\end{itemize}
%Observe that, in any case, either $\G/\n$ is bipolar w.r.t. $p$, or $\G\backslash\n$ is bipolar w.r.t. $p$.

%\red{ajouter ou pas la comparaison des \max...?}

%\red{separer ou pas cas acyclique, cyuclique ? cf . these p76}

%\red{RK: $G_\n$ bipolar if and only if $\n$ belongs to a directed cocycle.}

\begin{thm}
\label{thm:ind-gene}
The active spanning trees %$\alpha$ for
of ordered digraphs satisfy the following inductive definition.
\emevder{OU the canonical active bijection}

\begin{algorithme}

For any ordered digraph $\G$ on edge-set $E$, and  with $\max(E)=\n$.
\smallskip

If $E=\emptyset$ then $\alpha(\G)=\emptyset$.

\textbf{\textit{\small (isthmus/loop case)}}

If $\n$ is an isthmus of $G$ then $\alpha(\G)=\alpha(-_\n\G)=\alpha(\G/\n)\cup\{\n\}$.

If $\n$ is a loop of $G$ then $\alpha(\G)=\alpha(-_\n\G)=\alpha(\G\bk\n)$.

If  $\n$ is not an isthmus nor a loop of $G$ then:\par

%\textbf{\textit{\small (comparison step)}}
\textbf{\textit{\small (choice by activity comparison)}}

%\hskip 10mm \red{\emph{(See Proposition \ref{prop:cond-ind-gene} and \cite{AB4} for equivalent conditions to the conditions below)}}

\hskip 10mm If $\G/\n$ and $\G\bk\n$ do not have the same active partition, then
let $\G_{0}$ be the unique minor within $\{\G/\n, \G\bk\n \}$ such that the active partition of $\G_{0}$ is obtained by removing $\n$ from the active partition of $\G$ {(well-defined by Proposition \ref{prop:ind-gene} (\ref{it2-prop:ind-gene}))}.

\hskip 10mm If $\G/\n$ and $\G\bk\n$ have the same active partition, then let 
$\G_\n$ be the minor containing $\n$ associated with the active partition of $\G$ on set of edges $E_\n$, %. % with smallest edge $a_\n$. 
%Then:
and then:%
\ss

%\hskip 10mm 
%\textbf{\textit{\small (specific choice step)}}
\textbf{\textit{\small (choice by full optimality)}}

%\red{ATTENTION ICI CAS ACYCLIQUE, A CHANGER SI CYCLIC-BIPOLAR !!!???!!!}

%\red{eventuellement utiliser cir fond dans G au lieu du mineur, comme dans supersolv, voir RAB ci-dessous}

%\red{\emph{(See Proposition \ref{prop:ind-direct-full-opt} for equivalent conditions to those below)}}

\hskip 20mm
let $T'=\alpha(\G_\n\backslash\n)=\alpha(\G\backslash\n)\cap E_\n$, $C=C_{\G_\n}(T';\n)$ and $e=\min(C)$\par

\hskip 20mm
if $e$ and $\n$ have opposite directions in $C$ then let
$\G_{0}=\G\backslash\n$
% $\alpha(\G)=\alpha(\G\backslash\n)$;\par
%if $\sigma_C(e)\not=\sigma_C(\n)$ then $\f(M):=0$;\par

\hskip 20mm
%if $\sigma_C(e)=\sigma_C(\n)$ then $\f(M):=1$;\par
if $e$ and $\n$ have the same directions in $C$ then let
$\G_{0}=\G/\n$
%$\alpha(\G)=\alpha(\G/\n)\cup\{\n\}$.\par
\smallskip

{\sl or equivalently:}\par
\hskip 20mm
let $T''=\alpha(\G_\n/\n)=\alpha(\G/\n)\cap E_\n$, $D=C^*_{\G_\n}(T''\cup\n;\n)$ and $e=\min(D)$\par

\hskip 20mm
if $e$ and $\n$ have opposite directions in $D$ then let
$\G_{0}=\G/\n$
%$\alpha(\G)=\alpha(\G/\n)\cup\{\n\}$;\par
%if $\sigma_D(e)\not=\sigma_D(\n)$ then $\f(M):=1$;\par

\hskip 20mm
if $e$ and $\n$ have the same directions in $D$ then let
$\G_{0}=\G\backslash\n$.
%$\alpha(\G)=\alpha(\G\backslash\n)$.
%if $\sigma_D(e)=\sigma_D(\n)$ then $\f(M):=0$.\par
\smallskip

\textbf{\textit{\small (assignment step)}}

\begin{tabular}{llllll}
\hspace{5mm} &
If $\G_{0}=\G\bk\n$ &then &
$\alpha(\G)=\alpha(\G\bk\n)$ 
&and &$\alpha(-_\n\G)=\alpha(\G/\n)\cup\{\n\}$. \\

 &
If $\G_{0}=\G/\n$ &then 
& $\alpha(\G)=\alpha(\G/\n)\cup\{\n\}$ 
& and& $\alpha(-_\n\G)=\alpha(\G\bk\n)$. \\

\end{tabular}
%\hskip 10mm 
%If $\G_{0}=\G\bk\n$ then $\alpha(\G)=\alpha(\G\bk\n)$ and $\alpha(-_\n\G)=\alpha(\G/\n)\cup\{\n\}$
%
%\hskip 10mm 
%If $\G_{0}=\G/\n$ then $\alpha(\G)=\alpha(\G/\n)\cup\{\n\}\alpha(\G\bk\n)$ and $\alpha(-_\n\G)=\alpha(\G\bk\n)$

%{\bf end}\par
%\noindent If $\f(M)=0$ then
%$$Oribas^{(1,0)}(M):=Oribas^{(1,0)}(M\s\n)$$
%if $\f(M)=1$ then $$Oribas^{(1,0)}(M):=Oribas^{(1,0)}(M/\n)\cup \n$$
\end{algorithme}
\end{thm}

Let us recall that, with notations used in Theorem \ref{thm:ind-gene},
$\G_\n$ bipolar, resp. cyclic-bipolar, if and only if $\n$ belongs to a directed cocycle, resp. directed cycle, of $\G$.

\begin{proof}
We follow the cases addressed during the algorithm and we prove that, in every case, the resulting definition of $\alpha$ is correct.
If $\n$ is an isthmus or a loop, then $\n$ is the only element of its part in the active partition of $\G$, and then the definition is obviously correct as it coincides with Definition \ref{def:alpha-seq-decomp}.
Assume now that $\n$ is not an isthmus nor a loop.

\eme{Cas suivant vite fait, a verifier uletrieurement.}%
Assume $\G/\n$ and $\G\bk\n$ do not have the same active partition.
By Proposition \ref{prop:ind-gene}, the active partition of at least one
of the two minors in  $\{\G/\n, \G\bk\n \}$ is obtained by removing $\n$ from the active partition of $\n$. Hence $\G_0$ is well defined.
Assume $\G_0=\G\bk\n$.
By Proposition \ref{prop:ind-gene}, we have $\alpha(\G)\in  \bigl\{ \alpha(\G\backslash\n), \alpha(\G/\n)\cup\{\n\} \bigr\}$. Moreover, if $\alpha(\G)=\alpha(\G/\n)\cup\{\n\}$ 
then removing $\n$ from the active partition of $\G$ yields the active partition of $\G/\n$, which contradicts the definition of $\G_0$.
Hence $\alpha(\G)=\alpha(\G\bk\n)$.
And hence, by Proposition \ref{prop:ind-gene}, we also have $\alpha(-_\n\G)=\alpha(\G/\n)\cup\{\n\}$.
So the definition given in the theorem is correct (the same reasoning holds for $\G_0=\G/\n$).

Assume now that $\G/\n$ and $\G\bk\n$ have the same active partition.
By Proposition \ref{prop:ind-gene}, we have at the same time that $\G_\n$ is 
the minor containing $\n$ associated with the active partition of $\G$, and that  $-_\n\G_\n$ is 
the minor containing $\n$ associated with the active partition of $-_\n\G$.
And since we have $\alpha(\G\bk\n)\in\{\alpha(\G),\alpha(-_\n\G)\}$, then we have (by Definition \ref{def:alpha-seq-decomp})
$\alpha(\G_\n\backslash\n)=\alpha(\G\backslash\n)\cap E_\n$. Similarly, we have $\alpha(\G_\n/\n)=\alpha(\G/\n)\cap E_\n$.

If $\G_\n\backslash\n$ (or equivalently $\G_\n/\n$)
is bipolar w.r.t. its smallest edge, then
the two equivalent conditions are the same as in Theorem \ref{thm:ind-10} and their validity is proved the same way.\emevder{ici on utilise preuve du thm de ABG2LP --- pourrait on appliquer thm a $\G_\n$ tut simplement ?}
If $\G_\n\backslash\n$ (or equivalently $\G_\n/\n$)
is cyclic-bipolar w.r.t. its smallest edge, then the conditions do not have to be changed, and the proof is the same except that one uses Definition \ref{def:cyc-alpha1} instead of Definition \ref{def:acyc-alpha}, which give exactly the same criterion for directions of edges distinct from the smallest edge.
\end{proof}

\begin{remark}[equivalence in Theorem \ref{thm:ind-gene}]
\label{rk:ind-gene-equivalence}
 \rm
The equivalence of the two formulations in the algorithm of Theorem \ref{thm:ind-gene} directly comes from 
the same equivalence %of the two formulations 
%in the algorithm of 
in Theorem \ref{thm:ind-10}, see %addressed in 
Remark~\ref{rk:ind-10-equivalence}.
\end{remark}

\begin{remark}[building the whole bijection at once, and the \choice notion]
\label{rk:ind-gene}
\rm
%\red{dessous deja dit}
Continuing Remark \ref{rk:ind-10},
observe that the above algorithm builds at the same time $\alpha(\G)$ and $\alpha(-_\n\G)$.
Again, this is due to the CHOICE notion, as highlighted in Proposition \ref{prop:ind-gene} (\ref{it1-prop:ind-gene}).
%since:
%$$\Bigl\{\ \alpha(\G),\ \alpha(-_\n\G)\ \Bigr\}\ =\ \Bigl\{\ \alpha(\G\backslash\n),\ \alpha(\G/\n)\cup\{\n\}\ \Bigr\}.$$
%
By this way, the above algorithm can be considered as building the whole active bijection from those for $G/\n$ and $G\bk\n$ (as a $1-1$ correspondence rather than as a pair of inverse mappings). 
\emevder{AU DEBUT j'avais dit qu'on pouvait aller des sapnning trees aux orietnations, maispas si evident, a reflechir, ai enelve, pas necessaire de dire ca}%
%It could be formulated equally from orientations to bases, or from bases to orientations, or both at the same time. 
%For instance, from spanning trees to orientations, if $\n$ belongs to the spanning tree then the minor to be considered is $G/\n$, and if $\n$ does not belong to the spanning tree then it is $G\bk\n$.
Observe that the local choice between $G/\n$ and $G\bk\n$ is made here in two steps, first by comparing active partitions, next by applying the same full optimality criterion as in Theorem \ref{thm:ind-10}.
This \choice notion is developed in Section\ref{subsec:ind-framework}.
See \cite{AB4} for more details. 
%A direct construction from spanning trees to orientations is also given in \cite{AB2-b}.
%\new{mouais, from bases to ori pas si evident, non ? A VOIR !!! a ete enelve de ABG2LP}
%\new{ICI aussi elle est donnee}
%
%Also, as already noted in Remark \ref{rk:difficult} for bipolar orientations, the above algorithm involves an exponential number of minors, but is efficient for building the whole canonical active bijection of a given graph in the sense that, with $|E|=n$, the number of calls to the algorithm to build  these $2^n$ images of all orientations 
%is exactly  $n.2^n$ (when no account is taken of the cost of handling active partitions).
%%\new{ceci ajoute}
%%\eme{reformuler la construction directe par lgo lineaire de facon similaire  a deletion/contraction, puis comparer les deux !}
%%
%%Finally it can be easily refined to get the refined active bijections considered in Section \ref{subsec:refined}, as explained in the next subsection \red{(consider option ...)}.
%%\bs
%\emevder{modif de l'algo (ou  pour AB4 ?) : comparer active partiions of $\alpha(\G/\n)$ and $\alpha(\G\s\n)$ rather than active partitions of $\G/\n$ and $\G\s\n$}
\end{remark}

\begin{remark}[computational complexity]
\label{rk:ind-gene-difficult}
\rm
Continuing Remark \ref{rk:difficult} and Remark \ref{rk:ind-gene}, % made for bipolar orientations, 
the above algorithm involves an exponential number of minors for building one image under $\alpha$, but it is efficient for building the whole canonical active bijection of a given graph in the sense that, with $|E|=n$, the number of calls to the algorithm to build  the $2^n$ images of all orientations 
is exactly  $n.2^n$ (when no account is taken of the cost of handling active partitions).
\emevder{a decreire ent ermes de "enumeration compelxity", voir courriels Daniel}%
%\new{ceci ajoute}
%\eme{reformuler la construction directe par lgo lineaire de facon similaire  a deletion/contraction, puis comparer les deux !}
%
%Finally it can be easily refined to get the refined active bijections considered in Section \ref{subsec:refined}, as explained in the next subsection \red{(consider option ...)}.
%\bs
\emevder{modif de l'algo (ou  pour AB4 ?) : comparer active partiions of $\alpha(\G/\n)$ and $\alpha(\G\s\n)$ rather than active partitions of $\G/\n$ and $\G\s\n$}
\end{remark}

\begin{remark}[practical improvments]
\label{rk:ind-gene-alt}
\rm
The above algorithm can be detailed further as a more practical algorithm in several ways. These refinements are detailed in \cite{AB3}. Let us mention them roughly. First, one could use a direct comparison  of the active partitions of $\G$, $-_\n\G$, $\G/\n$ and $\G\bk\n$ using only directed cycles/cocycles containing $\n$, more complete than the one given in Proposition \ref{prop:ind-gene} below. Second, one could use a direct characterization of the sign involved in the (cyclic-)bipolar involved minor $\G_\n$, using fundamental cycles/cocycles in the original digraph $\G$, without having to compute this minor. Third, one could use a more explicit case by case formulation of the underlying duality. 
\end{remark}

\emevder{ATTENTION COHERENCE AVEC AB3 QUAND AB3 SOUMIS}

\subsection{The refined active bijection}
\label{subsec:alpha-refined-ind}

The refined active bijection can be built by a simple refinement of the deletion/contraction construction of the canonical active bijection.

\begin{thm}
\label{thm:ind-gene-refined}
Let $\G$ be an ordered digraph.
An algorithm building the image $\alpha_\G(A)$ for $A\subseteq E$ is obtained by adding the following \emph{(refined isthmus/loop case)} and \emph{(refined assignment step)} to Theorem \ref{thm:ind-gene}, in parallel to the corresponding steps in this theorem, while using this theorem to compute $\alpha(-_A\G)$.
%with the following ones, and applying the \emph{(choice by activity comparison)} and the \emph{(choice by full optimlaity)} to the digraph $-_A\G$.

\eme{ATTENTION verifier coherence acec Def \ref{EG:def:act-bij-ext}, deja fait mais a reverifier}%

\begin{boxedalgorithme}

\textbf{\textit{\small (refined isthmus/loop case)}}

%\small
%\smallalgofont

\begin{tabular}{llll}
\multicolumn{4}{l}{If $\n$ is an isthmus of $G$ then }\\

\hspace{1.5cm} &$\alpha_\G(A\cup \{\n\})=\alpha_{\G\bk\n}(A\bk\{\n\})$ 
&and &$\alpha_\G(A\setminus\{\n\})=\alpha_{\G/\n}(A\bk\{\n\})\cup\{\n\}$.  \\

\multicolumn{4}{l}{If $\n$ is a loop of $G$ then }\\
\hspace{1.5cm} &$\alpha_\G(A\cup \{\n\})=\alpha_{\G\bk\n}(A\bk\{\n\})\cup\{\n\}$ 
&and &$\alpha_\G(A\setminus\{\n\})=\alpha_{\G/\n}(A\bk\{\n\})$.  \\

\end{tabular}

Otherwise then proceed with Theorem \ref{thm:ind-gene}.
%do \emph{(comparison step)}.
%If $\n$ is an isthmus of $G$ then $\alpha(\G)=\alpha(-_\n\G)=\alpha(\G/\n)\cup\{\n\}$.
%
%If $\n$ is a loop of $G$ then $\alpha(\G)=\alpha(-_\n\G)=\alpha(\G\bk\n)$.
%
%If  $\n$ is not an isthmus nor a loop of $G$ then:\par

\end{boxedalgorithme}
\ss

\begin{boxedalgorithme}
\textbf{\textit{\small (refined assignment step)}}

\small
\smallalgofont

\begin{tabular}{llll}
If $\G_{0}=-_A\G\bk\n$ then 
$\alpha_\G(A)=\alpha_{\G\bk\n}(A\bk\{\n\})$ 
&and &$\alpha_\G(A\triangle\{\n\})=\alpha_{\G/\n}(A\bk\{\n\})\cup\{\n\}$  \\

If $\G_{0}=-_A\G/\n$ then 
$\alpha_\G(A)=\alpha_{\G/\n}(A\bk\{\n\})\cup\{\n\}$ 
&and &$\alpha_\G(A\triangle\{\n\})=\alpha_{\G\bk\n}(A\bk\{\n\})$  \\

\end{tabular}

\end{boxedalgorithme}

\end{thm}

\eme{preuve desous vite faite, faut-il mieux justifier la fin ?}%

\begin{proof}
By Definition \ref{EG:def:act-bij-ext}, with  $T=\alpha(-_A\G)$, we must have $\alpha_\G(A)=T \setminus (A\cap \Int(T)) \cup (A\cap \Ext(T))$.
If $\n$ is an isthmus and $\n\in A$, resp. $\n\not\in A$, then $\n$ is dual-active in $\G$, $\n\in \Int(T)$ and so $\n\not\in \alpha_\G(A)$, resp. $\n\in \alpha_\G(A)$.
If $\n$ is a loop and $\n\in A$, resp. $\n\not\in A$, then $\n$ is active in $\G$, $\n\in \Ext(T)$ and so $\n\in \alpha_\G(A)$, resp. $\n\not\in \alpha_\G(A)$.
Hence the definition given in the isthmus/loop case is correct.

In parallel, the computation of $\alpha(-_A\G)$ is performed using Theorem \ref{thm:ind-gene}.
The two parts \emph{(choice by activity comparison)} and the \emph{(choice by full optimality)} of Theorem \ref{thm:ind-gene} are applied to the digraph $-_A\G$. They consist in choosing which minor $\G_0\in\{-_A\G\bk\n,-_A\G/\n\}$ allows us to compute $\alpha(-_A\G)$.

So, lastly, the final assignment step for computing $\alpha_\G$ has to be exactly a reformulation of the same step in Theorem  \ref{thm:ind-gene}, with $-_A\G$ instead of $\G$ and $-_{A\triangle\{\n\}}\G$ instead of $-_\n\G$.
Observe that handling the isthmus/loop case suffices to have that $\alpha_\G$ satisfies the above relation with $\alpha$ for all $A$.
Indeed, $A\cap \Int(\alpha(\G))$ and $A\cap \Ext(\alpha(\G))$ are the same as
$A\cap \Int(\alpha(\G_0))$ and $A\cap \Ext(\alpha(\G_0))$, as long as $\n$ is not an isthmus nor a loop.
%Indeed, we have $\alpha(-_A\G)=\alpha_\G(A)\cup (A\cap \Int(T))\setminus \cup (A\cap \Ext(T))$.
\end{proof}

\begin{remark}[variants]
\label{rk:ind-refined-variants}
\rm
Let us mention that changing the assignment performed in the \emph{(refined isthmus/loop case)}, and making it possibly depend on $\alpha(-_A\G)$,
yields variants of the refined active bijections as mentioned in Section  \ref{subsec:act-map-class-decomp}.
\end{remark}

%\red{dans deletion cotnraxction, choix trivial peut varier avec T}

%\red{Let us mention that changing the \emph{(isthmus/loop step)} yields other refined active bijections corresponding to other parameters than $\emptyset$ in $\alpha_\G$, as mentioned in Section \ref{subsec:refined} and detailed in \cite{AB2}.} \red{and general framework}

\eme{dessous en commentaire section supprimee, reportee a AB4, donnant enonces seuls de s variantes pratiques pour les tests dans algo}%
\subsection{A general deletion/contraction framework for classes of activity preserving bijections}
\label{subsec:ind-framework}

%\blue{COPIE DE SECTION FRMAEWORK DECOMP FIN :
%Lastly, let us roughly mention that one can also add a deletion/contraction property to the class of mappings considered in this section,
%yielding the class mentioned in Section \ref{subsec:alpha-refined-ind}, option \ref{item:pres-act-parts} (which is thus at the intersection of the classes considered in this section and that one).
%\emevder{bien dit ???}
%}
%
%\red{definir CHOICE PROPERTY ??? cf Backman et intro... }
%
%\red{mttre ici rk sur backman?}
%
%\new{preciser que choice vient de \cite{LV84a} et \cite{Gi02} --- reciter backman?}
%\red{dans AB4 definir en parallele depuis orientations ou bases}
%
%In this section, we briefly present  %\red{roughly} 
%how one can obtain general classes of bijections between spanning trees (or subsets) and orientations, satisfying  properties with respect to activities yielding graduated constraints, and satisfying a common deletion/contraction framework.
In this section, we present  %\red{roughly} 
how one can obtain general classes of bijections between spanning trees (or subsets) and orientations satisfying  properties with respect to activities and satisfying a common deletion/contraction framework.
See \cite{AB3} for more details.
%
%The idea is to begin with a completely arbitrary bijection between the set $2^E$ of orientations and the set $2^E$ of subsets, and to introduce graduated constraints so that this bijection becomes less and less arbitrary as it satisfies more and more involved properties with respect to activities.
%At each level, the arbitrariness in the definition defines a class of bijections.
%Amongst these bijections, the active bijection satisfies the most involved properties, and hence is uniquely determined.
%
The idea is the following. 

We begin with a bijection, between the set $2^E$ of orientations and the set $2^E$ of subsets, which is completely arbitrary except that it
satisfies some minimal consistency in terms of deletion/contraction. This arbitrariness is formally given by a property which we call {\tt CHOICE}. %\textbf{\textit{\choice}}.
Then, we introduce  constraints, which determine the \choice in some cases, so that this bijection becomes less  arbitrary as it satisfies more  involved properties with respect to activities.
%At each level, the arbitrariness in the definition can be seen as defining a class of authorized bijections.
At each level, the arbitrariness in the definition defines a class of bijections, with noticeable properties (amongst which one can always fix a bijection by some trivial artificial criterion using the reference~orientation).
% depending on a reference orientation.

By this manner, we define various 
\emph{activity preserving mapping classes by deletion/contraction}. 
Amongst all these bijections, the canonical active bijection satisfies the most involved properties, and is uniquely determined, without having to use an artifical criterion, while the refined active bijection uses a trivial choice depending on a reference orientation at the very last step, in order to break symmetries as explained in Section \ref{subsec:refined}. % or a reference orientation. 
%\new{pas more or less}
%
Note that the mapping classes considered in this section
are distinct from the 
active partition preserving mapping classes
%activity preserving mapping classes by activity decomposition
 considered in Section \ref{subsec:act-map-class-decomp}.
The canonical active bijection belongs to both classes, and satisfies further properties (duality and full optimality for bipolar orientations) that determine it within these classes.

%\red{XXXXX OUGJENSUIS XXXX}
%
%In the framework below, we consider different possible options, which can be combined further. 
%\new{What follows can be thought of as bricks that one can use to build an algorithm. XXX BOF}
We do not give proofs here, 
proofs can be either easily deduced from \cite[Chapter 1]{Gi02} or
adapted from similar proofs of previous results in this section, 
and all proofs are detailed in \cite{AB4}.
% (some also in \cite{Gi02}).
%
%Each time a \emph{choice} is available, we obviously get a class of correspondences/bijections by considering all possible arbitrary choices at this step.

\ms

Let us now be technical.
Here, we want to build either a mapping $\psi$ that associates an ordered directed graph to one of its spanning trees, or a mapping $\psi_{\G}$ that associates a subset of its edges (meant as a reorientation of $\G$) with another subset of its edges (meant as a subset/superset of a spanning tree of $G$).
The important difference between those two viewpoints is that the mapping $\psi$ applies directly to any directed graph with no common reference orientation, whereas the mapping $\psi_{\G}$ applies substantially to the graph $-_A\G$ but may use the graph $\G$ as a reference orientation.
For the sake of simplicity, in what follows we consider only a mapping $\psi_\G: 2^E\rightarrow 2^E$ that a priori depends on a reference digraph $\G$ whose edge-set is $E$.
If it satisfies the following property:
$$\psi_{-_A\G}(A')=\psi_{\G}(A\triangle A')$$
for all $A,A'\subseteq E$, then it induces a well-defined mapping $\psi$ applied to any ordered digraph by $$\psi(-_A\G)=\psi_\G(A)=\psi_{-_A\G}(\emptyset).$$
Equivalently, in this case, $\psi_\G(A)$ depends only on $-_A\G$, and we say that $\psi_\G$ \emph{does not depend on a reference orientation} $\G$.
In particular, by constuction, $\alpha_\G$ effectively satisfies the above property, yielding $\alpha$ as addressed before.
\emevder{paragraphe dessus a bien verifier !!! est-ce bien ca ???}%

In the framework below, 
we start with an ordered directed graph $\G$, denoting $\n$ the greatest element of its edge set $E$, and we build a mapping $\psi_\G$, by means of several successive constructions or options (that one can also combine).
The common feature is to build at the same time $\psi_\G(A)$ and $\psi_\G(A\triangle\{\n\})$, for any $A\subseteq E$, %by considering the two minors $G/\n$ and $G\bk \n$, and
always preserving a fundamental inductive property. 

Note that we choose to present the construction from orientations to spanning trees. However, the way it is presented relies on building a bijection/correspondence from the two bijections/correspondences built in the minors $G\bk\n$ and $G/\n$ (just as in  Remarks \ref{rk:ind-10} and \ref{rk:ind-gene}). Therefore, it can be understood as doing both ways at the same time (for instance, from spanning trees to orientations, if $\n$ belongs to the spanning tree then the minor to be considered is $G/\n$, and if $\n$ does not belong to the spanning tree then it is $G\bk\n$).\emevder{ parenthese a verifeir, j'ai enelve ca des remarques, ca me paraissait louche, A VERFIERI REFLECHIR !!!}

Recall that we call \emph{correspondence} when several objects (e.g. some orientations) are associated with the same object (e.g. a spanning tree), hence a bijection is a one-to-one correspondence.
Lastly, in order to shorten notations, for $A\subseteq E$, we will denote:
\ss

\noindent\begin{tabular}{cccc}
$T_\n=\psi_\G(A\cup\{\n\})$,\hphantom{xx}&
$T_{-\n}=\psi_\G(A\setminus\{\n\})$,\hphantom{xx}&
$T_/=\psi_{\G/\n}(A\bk \{\n\})$,\hphantom{xx}&
$T_\bk=\psi_{\G\bk\n}(A\bk \{\n\})$.\\
\end{tabular}

\emevder{pas eu le temps de relire ce qui suit avant resoumission, a faire !}
%\red{defs notations G, omega, a mettre dans cadre ci-dessous ?}

%\red{notation $\G/\n$ au lieu de $\G/\{\n\}$ dans prelim}

%\red{***NB*** en fait le refined ne se fait que avec les isthmes/boucles donc le separer !... A VERIFIER}
%\bs

%\begin{color}{blue}
%DESSOUS AE
%
%\noindent {\bf Main common property.}
%
%\begin{boxedalgorithme}
%If $\n$ is not an isthmus nor a loop, then:
%$$\Bigl\{\ \psi(\G),\ \psi(-_\n\G)\ \Bigr\}\ =\ \Bigl\{\ \psi(\G\backslash\n),\ \psi(\G/\n)\cup\{\n\}\ \Bigr\}.$$
%And similarly, for $A\subseteq E$:
%$$\Bigl\{\ \psi_\G(A),\ \psi_\G(A\triangle \{\n\})\ \Bigr\}\ =\ \Bigl\{\ \psi_{\G\bk\n}(A\bk\{\n\}),\ \psi_{\G/\n}(A\bk \{\n\})\cup\{\n\}\ \Bigr\}.$$
%\end{boxedalgorithme}
%\end{color}
%
%\red{A VERIFIER}

%\bs
\begin{enumerate}

\item \noindent{\bf Minimalist framework.}

\begin{enumerate}

\item \emph{Initialization.} Just set $\psi_\emptyset(\emptyset)=\emptyset$.
\emevder{bizarre decommncer avec vide alors que plus haut on dit qu'on commence avec $\G$.... pas trrible, a reprednre... apres AB4 peut-etre}

\item \emph{Orientations - subsets bijection.}
\label{item:sub-or-istloop-free}
For all $A\subseteq E$, make $\psi_\G$ satisfy the following property.

\begin{boxedalgorithme}
%\textbf{\textit{(fundamental choice)}}

\noindent\begin{tabular}{llrcl}

%For all $A\subseteq E$:
%\textbf{\textit{\choice}}\hphantom{xxxx}&
\textbf{\textit{\choice}}&set&
$\Bigl\{\ \psi_\G(A),\ \psi_\G(A\triangle \{\n\})\ \Bigr\}$&
$=$&
$\Bigl\{\ \psi_{\G\bk\n}(A\bk\{\n\}),\ \psi_{\G/\n}(A\bk \{\n\})\cup\{\n\}\ \Bigr\}.$\\

That is:&set&
$\{\ \ T_\n,\ \ T_{-\n}\ \ \}$&
$=$&
$\{\ \ T_\bk,\ \ T_/\cup\{\n\}\ \ \}.$\\
\end{tabular}
%
%\noindent\begin{tabular}{lrrcl}
%\multicolumn{2}{l}{Which means: set arbitrarily  }&&\\
%\hfill either&
%$T_\n=T_/\cup\{\n\}$ &and& $T_{-\n}=T_\bk$,\\
%
%\hfill or & $T_\n=T_\bk$ &and& $T_{-\n}=T_/\cup\{\n\}$.\\
%\end{tabular}

\noindent\begin{tabular}{lrlcl}
%That is: &set (arbitrarily) either&
That is: &set either\hphantom{xxx}&
$T_\n=T_\bk$ &and& $T_{-\n}=T_/\cup\{\n\}$,\\

& or\hphantom{xxx}&$T_\n=T_/\cup\{\n\}$ &and& $T_{-\n}=T_\bk$.\\
\end{tabular}

%That is:
%$$\{\ \ T_\n,\ \ T_{-\n}\ \ \}\ \ =\ \ \{\ \ T_\bk,\ \ T_/\cup\{\n\}\ \ \}.$$
%%For all $A\subseteq E$:
%$$\Bigl\{\ \psi_\G(A),\ \psi_\G(A\triangle \{\n\})\ \Bigr\}\ =\ \Bigl\{\ \psi_{\G\bk\n}(A\bk\{\n\}),\ \psi_{\G/\n}(A\bk \{\n\})\cup\{\n\}\ \Bigr\}.$$
%That is:
%$$\{\ \ T_\n,\ \ T_{-\n}\ \ \}\ \ =\ \ \{\ \ T_\bk,\ \ T_/\cup\{\n\}\ \ \}.$$
\end{boxedalgorithme}

Arbitrary choices satisfying this property yield orientations - subsets bijections.
\eme{??? Also, the $2^{i+j}$ subsets in the interval associated with a spanning tree, (where $i$, resp. $j$, is the internal, resp. external, activity of the spanning tree) will correspond to $2^{i+j}$ orientations. Associating those orientations with this spanning tree is equivalent to using the definition of the previous option. *****oulala....!!!!???}%
\ss

%\item \noindent{\bf Isthmus/loop step.}
%We consider the case when $\n$ is a loop or an isthmus of $\G$, that is $\G/\n=\G\bk\n$.
%We consider different options for the construction.
%We denote either $T=\psi(\G/\n)$, or  $A\subseteq E$ and $T=\psi_{\G/\n}(A\setminus\{\n\})$, depending obviously on the option.

\item \emph{Orientations - spanning trees correspondence.}

%\begin{boxedalgorithme}
%If $\n$ is an isthmus or a loop of $\G$ then (free choice):
%$$\Bigl\{\ \psi_\G(A),\ \psi_\G(A\triangle \{\n\})\ \Bigr\}\ =\ \Bigl\{\ T,\ T\cup\{\n\}\ \Bigr\}.$$
%\end{boxedalgorithme}

\begin{boxedalgorithme}
If $\n$ is an isthmus of $G$ then $T_\n=T_{-\n}=T_\bk\cup\{\n\}=T_/\cup\{\n\}$.

If $\n$ is a loop of $G$ then $T_\n=T_{-\n}=T_\bk=T_/$.

Otherwise then \emph{\choice}.
%:  $\{T_\n, T_{-\n}\}=\{T_\bk, T_/\cup\{\n\}\}$
\end{boxedalgorithme}

%\begin{boxedalgorithme}
%If $\n$ is an isthmus of $\G$ then $\psi(\G)=\psi(-_\n\G)=T\cup\{\n\}$.
%
%If $\n$ is a loop of $\G$ then $\psi(\G)=\psi(-_\n\G)=T$.
%\end{boxedalgorithme}

If $\n$ is a loop or an isthmus of $G$, then $\G/\n=\G\bk\n$, and $T_/=T_\bk$.
One can see, using classical properties in the Tutte polynomial area, that this construction yields  $2^{i+j} - 1$  orientations - spanning trees correspondences, where $i$, resp. $j$, is the internal, resp. external, activity of the spanning tree. 
%\red{???? a completer, attetnion auc as non-preserving!!! a l'air d'aller}.
\ss

\item \emph{Examples of trivial fixations of the \choice}
\label{item:trivial-specification}

\begin{boxedalgorithme}
\emph{Example 1.} Set $T_\n=T_/\cup\{\n\}$ and $T_{-\n}=T_\bk$.

\emph{Example 2.} Set $T_\n=T_\bk$ and $T_{-\n}=T_/\cup\{\n\}$.
\end{boxedalgorithme}

Such trivial fixations can be used in any of the present constructions, as soon as a choice is left arbitrary, in order to get a completely defined mapping within the considered class of mappings.
%
%The choice specifies the free choice in step (\ref{item:sub-or-istloop-free}).
 %
Such a mapping will obviously depend on $\G$.
%, and can have specific properties.
Notice that a fixation of this type is used in Theorem \ref{thm:ind-gene-refined} when $\n$ is an isthmus or a loop, with a different treatment of these two cases, yielding the required properties of $\alpha_\G$.
Variants can be used, as noted in Remark \ref{rk:ind-refined-variants}. %(and Section \ref{subsec:act-map-class-decomp}).
%, to identify orientations associated with spanning trees in a bijection between orientations and subsets, accordingly with Section \ref{subsec:refined}.
%
%\red{verifier que c'est bien ce qui est fait dans ce th}
% (e.g. Theorem \ref{thm:ind-gene-refined}).
 %
\eme{Obviously, various similar choices with specific properties are available and can be combined with other specifications below.}%
%This option will eventually yield a $1 - 1$ subsets - orientations correspondence.
%\red{with the property that..... \min cycle/cocycle ???}
%
%\red{lequel est utilise pour $\alpha$ ???}

\end{enumerate}

\item \noindent{\bf Fixing the \choice by activity comparison.} 

Each of the following options can be applied assuming that $\n$ is not an isthmus nor a loop. 
We give constructions in order of increasing fixation constraint.

%
%We consider here the case where $\n$ is not an isthmus nor a loop. 
%In every case we want to preverve the main common property stated previously. Stating this property in general with a compltely fre choice yields a spannign tree orientation correspondence but which has no interest besides that of satisfying this inductive property, hence we do not state it as a possibility. Each possibility tields a free choice when the choice is not determinde by the required property.
%
%We give constructions in order of increasing specification.
%
%\begin{boxedalgorithme}
%If $\n$ is not an isthmus nor a loop of $\G$ then:
%\end{boxedalgorithme}
%
%\red{OU BIEN la mettre en premiere option ici}
%
%\red{definir $T_\bk$, $T_/$, $ T_\n, T_{-\n} $}
%
%\begin{enumerate}
%
%\item \red{\emph{Main common property with most free choice} ????????}
%
%\begin{boxedalgorithme}
%Free choice: $\{T_\n, T_{-\n}\}=\{T_\bk, T_/\}$
%\end{boxedalgorithme}
% 
% This construction will ensure to have a $1-2^{i+j}$ correspondence or a bijective construction between all orientations and all spanning trees.
 \ss
 
 \begin{enumerate}
 
\item \emph{Separating acyclic/cyclic parts and internal/external parts}

\begin{boxedalgorithme}
If $\n$ belongs to a directed cycle of $-_{A\setminus \{\n\}}\G$ and a directed cocycle of $-_{A\cup\{\n\}}\G$ 

\hskip 1.5cm 
then $T_\n=T_\bk$ and $T_{-\n}=T_/\cup\{\n\}$.

If $\n$ belongs to a directed cocycle of $-_{A\setminus \{\n\}}\G$   and a directed cycle of $-_{A\cup\{\n\}}\G$ 

\hskip 1.5cm 
then $T_\n=T_/\cup\{\n\}$ and $T_{-\n}=T_\bk$.

Otherwise then \emph{\choice}.
%: $\{T_\n, T_{-\n}\}=\{T_\bk, T_/\}$.
\end{boxedalgorithme}
Note that this fixation does not depend on  a reference orientation $\G$, only on $-_A\G$ and $-_{A\triangle\n}\G$.
Applied to an orientations - spanning trees correspondence, 
 the construction will associate acyclic orientations with internal spanning trees, and strongly connected orientations with external spanning trees. 
\eme{QUESTION A VERIFIER !!! est-ce automatique que ca se reporte sur activite des arbres???? Si oui on devrait pouvoir le definir dans l'autre sens: si $\n$ est dans la partie externe alors... etc.}%
Applied to an orientations - subsets bijection, we get in particular a bijection between acyclic orientations and no broken circuit subsets and, and a bijection between strongly connected orientations and supersets of external spanning trees. See \cite{AB4}.
\eme{In this case, fixing the \choice further in the case where $\n$ is an isthmus or a loop with a trivial \choice of type (\ref{item:trivial-specification}), one can get a bijection between acyclic orientations with given orientation for active and dual-active elements and internal spanning trees (and a dual bijection for strongly connected orientations and external spanning trees).
See details in \cite{AB4}.
\red{vrai ou bien attendre preservation des activites ?}
\red{voir item ensuite sur trivial specification}
}%
\ss

\item \emph{Preserving active elements}
\label{item:pres-act-elts}

\begin{boxedalgorithme}
%
%\hskip 10mm If $A{\cal O}(M)\subset A{\cal O}(-_{\n}M)$ or $A{\cal O}^*(-_{\n}M)\subset A{\cal O}^*(M)$\par
%\hskip 20mm then $\ff(M)=0$
%\smallskip
%\hskip 10mm if $A{\cal O}(-_{\n}M)\subset A{\cal O}(M)$ or $A{\cal O}^*(M)\subset A{\cal O}^*(-_{\n}M)$\par
%\hskip 20mm then $\ff(M)=1$
%\smallskip
%\hskip 10mm if $A{\cal O}(M)=A{\cal O}(-_{\n}M)$ and $A{\cal O}^*(-_{\n}M)=A{\cal O}^*(M)$\par
%\hskip 20mm then {\bf (arbitrary choice)} $\{\ff(M),\ff(-_\n M)\}=\{0,1\}$.

\eme{DESSOUS EN COMMENTAIRE FORMULATION sans reorienter A}
%If $O(\G)\subset O(-_{\n}\G)$ or $O^*(-_{\n}\G)\subset O^*(\G)$
% then $T_\n=T_\bk$ and $T_{-\n}=T_/$.
%
%If $O(-_{\n}\G)\subset O(\G)$ or $O^*(\G)\subset O^*(-_{\n}\G)$ then $T_\n=T_/$ and $T_{-\n}=T_\bk$.
%
%If $O(\G)=O(-_{\n}\G)$ and $O^*(-_{\n}\G)=O^*(\G)$
%then \emph{free choice}: $\{T_\n, T_{-\n}\}=\{T_\bk, T_/\}$.

\small
\smallalgofont

If $O(-_{A\cup\{\n\}}\G)\subset O(-_{A\setminus\{\n\}}\G)$ or $O^*(-_{A\setminus\{\n\}}\G)\subset O^*(-_{A\cup\{\n\}}\G)$
 then $T_\n=T_\bk$ and $T_{-\n}=T_/\cup\{\n\}$.

If $O(-_{A\setminus\{\n\}}\G)\subset O(-_{A\cup\{\n\}}\G)$ or $O^*(-_{A\cup\{\n\}}\G)\subset O^*(-_{A\setminus\{\n\}}\G)$ then $T_\n=T_/\cup\{\n\}$ and $T_{-\n}=T_\bk$.

If $O(-_{A\cup\{\n\}}\G)=O(-_{A\setminus\{\n\}}\G)$ and $O^*(-_{A\setminus\{\n\}}\G)=O^*(-_{A\cup\{\n\}}\G)$
then \emph{\choice}.
%: $\{T_\n, T_{-\n}\}=\{T_\bk, T_/\}$.

\end{boxedalgorithme}

Note that this fixation does not depend on a reference orientation $\G$, only on $-_A\G$ and $-_{A\triangle\n}\G$.
Applied to an orientations - spanning trees correspondence, 
this fixation is necessary and sufficient to have that the construction will preserve active elements: active, resp. dual-active, elements of the orientation are equal to externally active, resp. internally active elements, of the associated spanning tree.
The proof that it is well-defined and yields this result is given in \cite{AB4} and \cite[Chapter 1]{Gi02}.
Also, this construction can be used as a set theoretical proof of the expression of the Tutte polynomial in terms of orientation activities from \cite{LV84a} recalled in Section \ref{sec:prelim} (see also Remark \ref{rk:fourientations}). 
%\red{en fait ensembliste version de \cite{LV84a}... REECRIRE}
\ss

\item \emph{Preserving active partitions}
\label{item:pres-act-parts}

\begin{boxedalgorithme}
\small
\smallalgofont

Assume %$\G/\n$ and $\G\bk\n$ 
$-_{A\cup\{\n\}}\G$ and $-_{A\setminus\{\n\}}\G$
 do not have the same active partition.
 
Let $T_\n \in\{T_\bk, T_/\cup\{\n\}\}$ corresponding (respectively) to the unique minor 
in $\{ -_A\G\bk\n , -_A\G/\n\}$ 

whose active partition is obtained by removing $\n$ from 
the active partition of $-_{A\cup\{\n\}}\G$. 

And let $T_{-\n}$ be the other element of $\{T_\bk, T_/\cup\{\n\}\}$

Otherwise then \emph{\choice}.
%: $\{T_\n, T_{-\n}\}=\{T_\bk, T_/\}$.
\end{boxedalgorithme}

Note that this fixation does not depend on a reference orientation $\G$, only on $-_A\G$ and $-_{A\triangle\n}\G$.
It reformulates the fixation used in Theorem \ref{thm:ind-gene}.
It allows the construction to preserve active partitions: the active partition of the orientation and of its associated spanning tree are equal.
Various practical conditions to compare active partitions of the two involved minors are provided 
%in Proposition \ref{prop:cond-ind-gene}. Proofs are given 
in \cite{AB4}, along with proofs for this result.
This yields a whole class of active partition preserving bijections/correspondences, refining the class mentioned in Section \ref{subsec:act-map-class-decomp} with a deletion/contraction property.
%Let us mention that a direct characterization equivalent to the above one involving only directed cycles/cocycles containing $\n$ is given in \cite{Gi02, AB4}.
%\red{OU ICI ???? A VOIR....}

%\red{QUESTION : est-ce que preserver active partitions <=> classe par decomposition des activites ???? ca semblerait naturel !!!}

%\red{mentionner remarque sur calsse de bijections \ref{rk:refined-bij}}
\end{enumerate}

\newpage

\item \noindent{\bf Further \choice fixation.}
%Each time a free choice can be made, one gets a class of correspondences. Various possibilities to specifiy the free choice left are available, providing that the correspondence will satisfy further properties.

\begin{enumerate}

%\item \emph{Example of a simple choice using a reference orientation}
%
%\begin{boxedalgorithme}
%For $A\subseteq E$,
%
%if $\n\in A$  then $T_\n=T_\bk$ and $T_{-\n}=T_/$,
%
%if $\n\not\in A$ then $T_\n=T_/$ and $T_{-\n}=T_\bk$.
%
%
%\end{boxedalgorithme}
%
%\red{dire que ces $T_...$ sont formels}
%
%This specification is allows us to have a completely defined bijection $\psi_\G$, depending on a refence orientation $\G$.
%\ss

\item \emph{The canonical active bijection}

It is defined  using option (\ref{item:pres-act-parts}),
specified as in Theorem \ref{thm:ind-gene} in order to have duality and full optimality properties for bipolar and cyclic-bipolar minors. 
We do not repeat this specification here.
Finally the mapping is uniquely determined, 
and does not depend on a reference orientation $\G$.
More details and geometrical interpretations can be found in \cite{AB1, AB2-b, AB3, AB4}.

%Let us mention that this specification turns out to be simply an adjacency property of a geometrical nature in an hyperplane arrangement reformulation between the flag of faces formed by a spanning tree and its associated orientation.
%See details and geometrical interpretations in \cite{AB1, AB2, AB3, AB4}.
%\red{a mettre ? A DETAILLER ? a mettre plutot dans intro ?}
%\red{pourrait etre un lemme : fob iff adjacency !!!}
%\red{en fait dans these adjacency = tous le flag, pas seulement plus gd element, il ...noraml... il faudra faire remarquer que parfois seul le plus gd element suffit, voir ca dans AB4}
%\red{a mettre en truc ulterieur}
\ss

\item \emph{The weak active bijection}

This variant consists in using, first, option (\ref{item:pres-act-elts}) above, and, next, a further fixation similar to that of the canonical active bijection.
It is defined and studied in \cite{GiLV06} in the case of triangulated (or chordal) graphs and supersolvable arrangements. 
%Since it does not involve active partitions,
It is more simple and direct to define than the canonical active bijection in this case.
It preserves active elements, it coincides with the canonical active bijection for (cyclic-)bipolar orientations, but it does not preserve active partitions, and the set of orientations associated with a spanning tree is not structured as an activity class. Anecdotally, it is proved in \cite{GiLV06} that the weak active bijection and the canonical active bijection coincide for acyclic orientations of the complete graph (equivalent to permutations, or to regions of the braid arrangement), yielding a classical bijection between permutations and increasing trees.%
%Coxeter arrangements $A_n$ (braid arrangement, equiovalent to acyclic orientations of the colmplete graph) and $B_n$ (hyperocathedral arrangement \red{a verifier}).
\ss

\item \emph{Adding trivial fixation}

Using a fixation of type (\ref{item:trivial-specification}) in the case where $\n$ is an isthmus or a  loop, in the case where a choice is left open by the previous ones, allows us for instance to  get bijections between orientations with given orientation for active and dual-active elements and spanning trees (see also below).
% with external activity zero (and a similar dual bijection for strongly connected orientations). 
%\red{called minimal-orientation-active before ???}
%\red{ref a resultat dans refined active mapping section}
%\red{A REFORMULER : a l'air de dire que refined active bij repose sur ce choix trivial...}
%\red{citer theoreme de refiend act bij}
\ss

\item \emph{The refined active bijection}

As stated in Theorem \ref{thm:ind-gene-refined}, it is obtained by applying the same fixations as for the canonical active bijection, but for an orientation - subset bijection and with a trivial fixation of type (\ref{item:trivial-specification}) in the case where $\n$ is an isthmus or a  loop.
As seen in Theorem \ref{EG:th:ext-act-bij}, it coincides with the canonical active bijection for active-fixed and dual-active-fixed orientations.

\end{enumerate}

%
%\item \noindent{\bf Final assignment step.}
%
%
%\begin{boxedalgorithme}
%If $T_\n=T_\bk$ and $T_{-\n}=T_/$ then $\psi(\G)=\psi(\G\bk\n)$ and $\psi(-_\n\G)=\psi(\G/\n)\cup\{\n\}$ 
%
%If $T_\n=T_/$ and $T_{-\n}=T_\bk$ then $\psi(\G)=\psi(\G/\n)\cup\{\n\}$ and $\psi(-_\n\G)=\psi(\G/\n)\cup\{\n\}\psi(\G\bk\n)$ 
%
%\end{boxedalgorithme}
%
%\red{a repeter pour $\psi_\G$, ou bien dire avant cet assigment en une fois...}

\end{enumerate}

\begin{remark}
\label{rk:fourientations}
\rm
In the seminal paper \cite{LV84a}, the proof of the expression of the Tutte polynomial in terms of orientation activities (Section \ref{subsec:orientation-activity}) is based on some sort of numerical \choice fixation at the level of  set cardinalities (orientation activities), see \cite[Lemma 3.2]{LV84a}.
This approach is generalized in \cite[Th\'eor\`eme 1.6]{Gi02} with a set theoretic approach, recalled here as option (\ref{item:pres-act-elts}), yiedling a proper correspondence and a preservation of active elements.
It is generalized further in the above framework, detailed in \cite{AB4}.
The fundamental property that enables these deletion/contraction constructions based on \choice fixation is roughly that
$\G$ and $-_\n\G$ on one side, and $\G/\n$ and $\G\bk\n$ on the other side, match to have the same properties (activities, active elements, or active partitions).
Let us mention the recent work \cite{BaHoTr18} which gathers both subset activity parameters (see Section \ref{prelim:subset-activities}) and orientation activity parameters (see Section \ref{subsec:act-classes}) in a large Tutte polynomial expansion formula in the context of graph fourientations. 
%Furthermore, this work 
%Let us point out that 
This work 
extends in some sense to graph fourientations the aforementioned fundamental property, at the level of active elements.
% that allows for deletion/contraction constructions of activity preserving bijections by means of local `choices' addressed in Section \ref{subsec:ind-framework} (see also \cite{AB4}). 
%\red{dessous pas dans intro}
Precisely, \cite[Lemma 3.3]{BaHoTr18} is a 
remarkable 
non-trivial 
extension to graph fourientations of (the restriction to graphs of) %the result 
\cite[Lemma 3.2]{LV84a} or \cite[Th\'eor\`eme 1.6]{Gi02}.
%given in numerical terms for oriented matroid perspectives  and more deeply addressed in set-theoretic terms  as  option (\ref{item:pres-act-elts}) in the construction of Section \ref{subsec:ind-framework} 
%following \cite[Th\'eor\`eme 1.6]{Gi02}, see also \cite{AB4}).
In this context, this fundamental property allows for defining activity preserving mappings in fourientations by deletion/contraction and  \choice fixations.
The fixation used to build %in the construction of 
the main mapping of \cite{BaHoTr18} is a ``tiebreaker'' induced by a reference orientation,
%(as written in \cite{BaHoTr18}), 
similar to the use of a ``trivial choice'' (\ref{item:trivial-specification}) in the above framework.
\end{remark}

%%%%%%%%%%%%%%%%%%%%%%%%%%%%%%%%%%%%%%%%%%%%%%%%%%%%%%%%%%
%\section{Complete detailed example of $K_4$}
\section{Detailed examples of $K_3$ and $K_4$}
\label{sec:example}

First, the constructions of the  active bijection are illustrated on the example%
\footnote{Anecdotally, the example of $K_3$ was highlighted by Tutte himself in conferences \cite{Tu01}. Comparing the symmetry of the graph with the non-symmetry of the polynomial he had defined, while the sum of its coefficients was equal to the number of spanning trees, made him think of introducing a linear ordering on the edges in order to break the symmetry. Thus began the long story of Tutte polynomial activities.}
 of $K_3$. The canonical and refined bijections are shown on Figure~\ref{fig:K3} and in the table of Figure \ref{fig:tabK3ori}.
The Tutte polynomial of $K_3$ is $$t(K_3;x,y)=x^2+x+y.$$

\begin{figure}[h]
\centering
{
\parindent=-1mm
\scalebox{1}[0.85]
{\includegraphics[width=9cm]{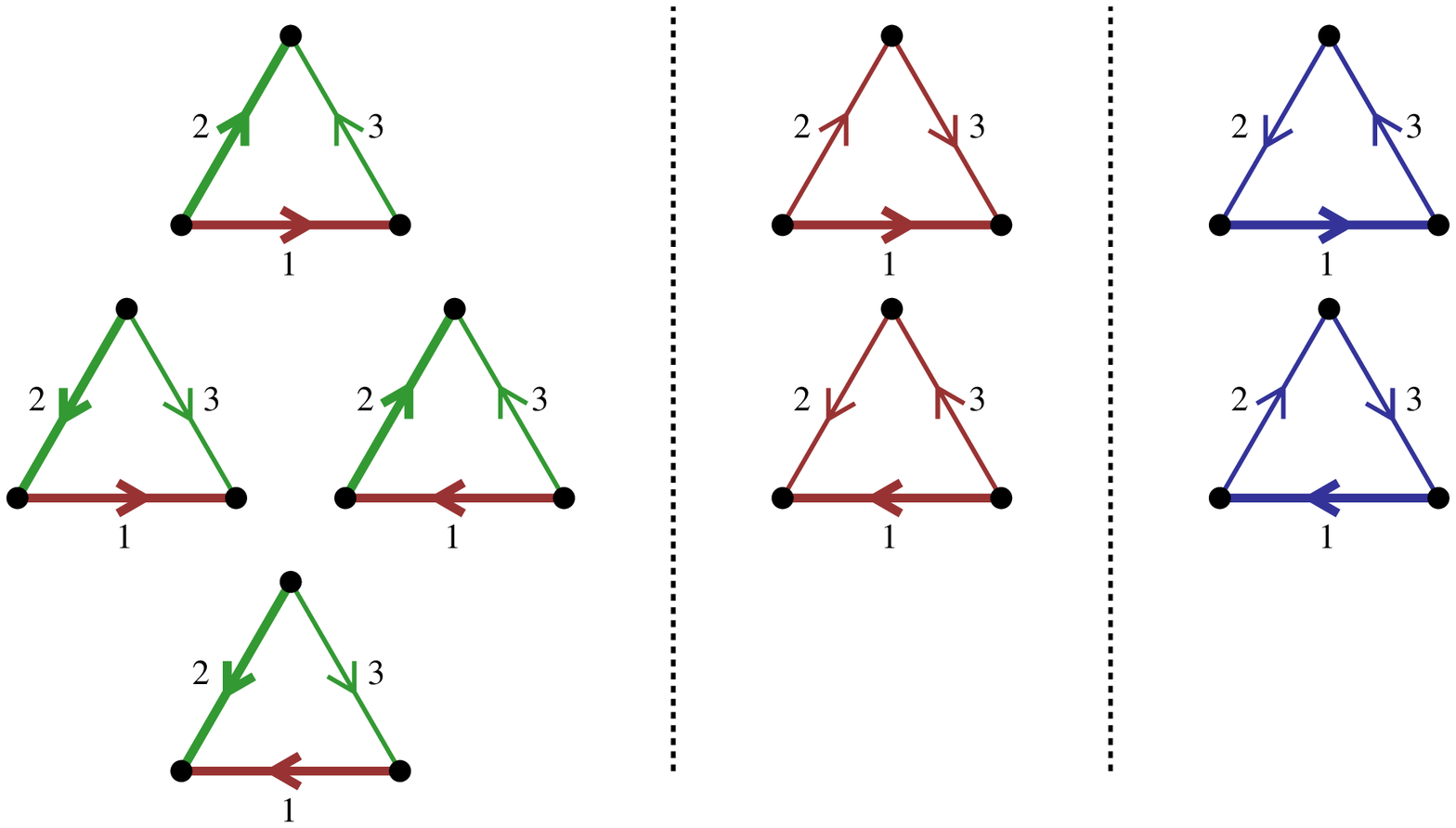}}
}
\hrule
\vspace{-1mm}
\flushleft
%\hspace{1cm}
%\begin{tabular}{l@{\hspace{1cm}}ccc@{\hspace{3mm}};{1pt/1pt}c;{1pt/1pt}c}
\begin{tabular}{l@{\hspace{15mm}}ccc;{1pt/1pt}c;{1pt/1pt}c}
 &  & $+ x^2$ & & $+x$ & $+y$\\
 $T(K_3;x+u,y+v)=$& $+xu$ & & $+ux$ & $ +u$ &$ +v$\\
  & & $+u^2$ & &   & \\
\end{tabular}
\vspace{1mm}
\hrule
\vspace{1mm}
%\centerline{$\mathlarger{\mathlarger{\mathlarger{\updownarrow}}}$ }
%\vspace{-2mm}
%
\centering
{
\scalebox{1}[0.85]
{\includegraphics[width=9cm]{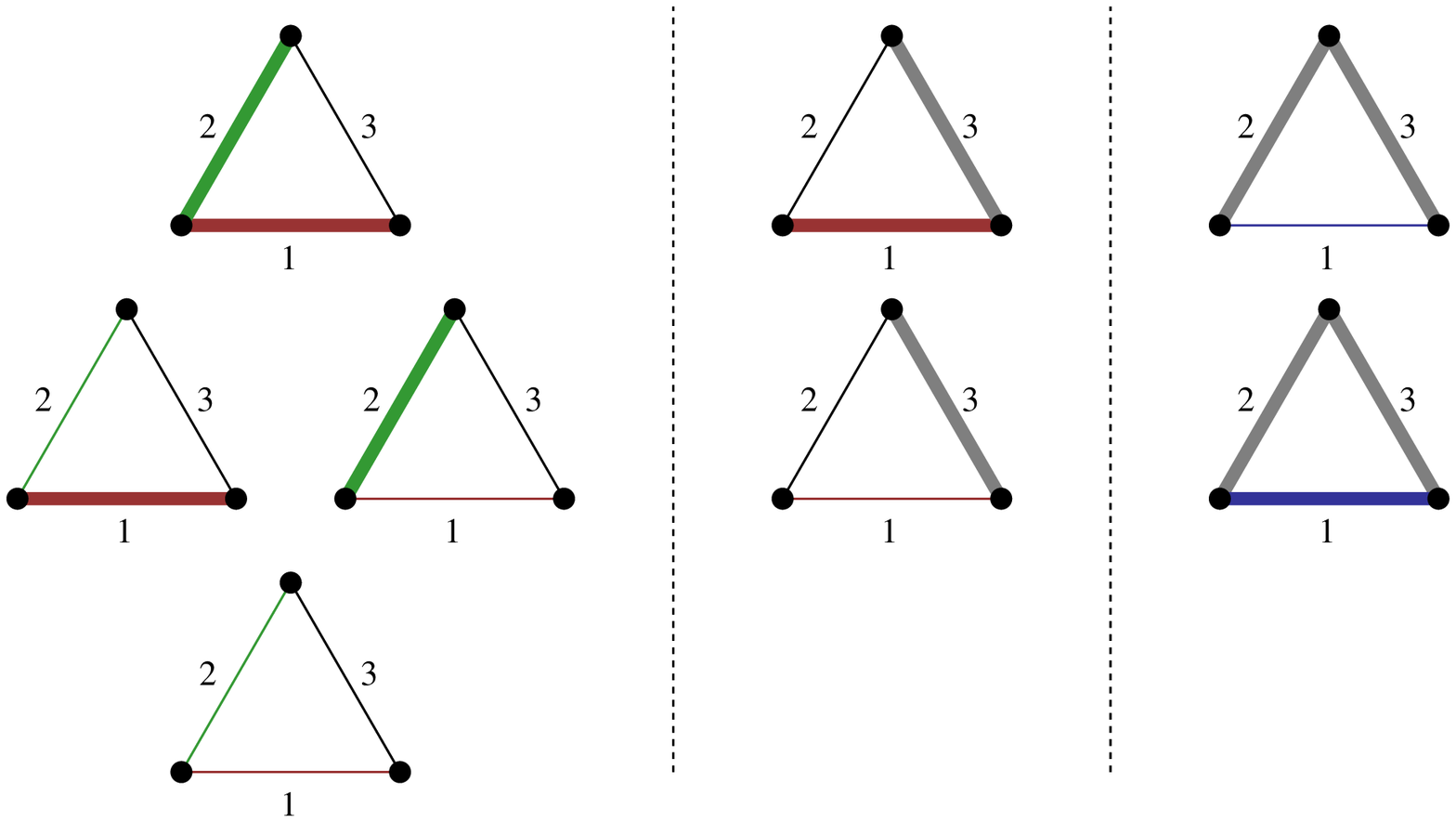}}
}
\caption{The active bijection illustrated on the graph $K_3$. We have $T(K_3;x,y)=x^2+x+y$.
%Each monomial corresponds a boolean lattice in the first line correspond to a column in the figures.
The layout reflects the bijections.
%Each monomial corresponds to the boolean lattice of an activity class  in the top line and to the boolean lattice of a basis interval in the bottom line.
%Those blocks are associated by the canonical active bijection.
%In each block, orientations in the top line and subsets (in bold) in the bottom line are associated by the refined active bijection, consistently with the four variable formula.
Each monomial corresponds to an activity class of orientations in the top part and to a spanning tree in the bottom part,
associated by the canonical active bijection.
Each spanning tree yields a boolean lattice of subsets (shown by bold edges). Orientations in the top part and subsets in the bottom part are associated by the refined active bijection (with respect to the 
%upper left 
orientation 
displayed
%which is %on the left
first
%orientation 
in the top row), consistently with the four variable~formula, in the way shown by the layout.
%consistently with their positions on the figure.
%as shown by their layout.
%
%At the first level, the opposite bipolar orientations are associated with the spanning trees $13$. At the second level, each block in the upper line is asociated with a spanning tree in the bottom line.
}
\label{fig:K3}
\end{figure}

\def\interligne{&&&\\[-11pt]}
\def\fcyc #1{\fbox{\hbox{#1}}}

\begin{figure}[H]
\begin{center}
\begin{tabular}{|c|c|c|c|}
\hline
Active filtrations & Active partitions & Orientation activity classes & {Spanning trees}   \\
\hline
\interligne
$\fcyc{\O}\subset 1\subset E$ & $1+23$ & $123$, $1\ovl{23}$,  $\ovl{1}23$, $\ovl{123}$ & 12 \\
$\fcyc{\O}\subset E$ & $123$ & $12\ovl{3}$, $\ovl{12}3$ & 13 \\
$\emptyset\subset \fcyc{E}$ & $123$ & $1\ovl{2}3$, $\ovl{1}2\ovl{3}$ & 23 \\
\hline
\end{tabular}
\caption{Table of the active bijection of $K_3$, where orientations are written with a bar over reoriented edges w.r.t. the reference orientation given in the upper left of Figure \ref{fig:K3}. The cyclic flat of each active filtration  is boxed in the first column.}
\label{fig:tabK3ori}
\end{center}
\vspace{-0.5cm}
\end{figure}

%\red{dessin de decomp d'une orietnation de $K_4$, cf AB2 vieux ??? ou bien expose cirm sur AB2}

%\red{table des filtrations de ma these ?}

Second, the constructions of the canonical active bijection (and the refined active bijection) described in the previous sections are completely illustrated on the example of $K_4$, with ordering (and reference orientation) given by Figure \ref{fig:K4}. 
The Tutte polynomial of $K_4$ is $$t(K_4;x,y)=x^3+3x^2+2x+4xy+2y+3y^2+y^3.$$

%\red{See AB2-b for details on examples of $K_4$}

\begin{figure}[h]
	\centering	%\includegraphics[scale=0.5]{figures/big_active_bijection_K4_ordre_actif_fonce.eps}
\includegraphics[scale=0.4]{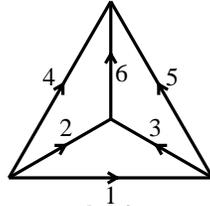}
\vspace{-0.5cm}
		\caption[]{Ordering and reference orientation of $K_4$}
		\label{fig:K4}
\end{figure}

\def\interligne{&&&\\[-11pt]}

\begin{figure}[h]
\begin{center}
\begin{tabular}{|c|c|c|c|}
\hline
Active filtrations & Active partitions & Orientation activity classes & {Spanning trees}   \\
\hline
\interligne
$\fcyc{\O}\subset 1\subset 123\subset E$ & $1+23+456$ & $123456$, $1\ovl{23}456$, $123\ovl{456}$, $1\ovl{23456}$, ... & 124 \\
$\fcyc{\O}\subset 1\subset E$ & $1+23456$ & $12345\ovl{6}$, $1\ovl{2345}6$, ... & 126 \\
$\fcyc{\O}\subset 145\subset E$ & $145+236$ & $1234\ovl{56}$,$1\ovl{23}4\ovl 56$, ... & 125 \\
$\fcyc{\O}\subset 123\subset E$ & $123+456$ & $12\ovl{3456}$, $12\ovl 3456$, ... & 134 \\
$\fcyc{\O}\subset E$ & $123456$ & $12\ovl 34\ovl{56}$, ... & 135 \\
$\fcyc{\O}\subset E$ & $123456$ & $12\ovl 34\ovl{5}6$,...& 136 \\
$\emptyset\subset \fcyc{123}\subset E$ & $123+456$ & $1\ovl 23456$, $1\ovl 23\ovl{456}$, ...& 234 \\
$\emptyset\subset \fcyc{145}\subset E$ & $145+236$ & $1\ovl{234}56$, $123\ovl 45\ovl 6$, ...& 245 \\
$\emptyset\subset \fcyc{246}\subset E$ & $246+135$ & $12\ovl{34}56$, $1\ovl{23}45\ovl 6$, ...& 146 \\
$\emptyset\subset \fcyc{356}\subset E$ & $356+124$ & $1234\ovl 56$, $12\ovl 345\ovl 6$, ... & 156 \\
$\emptyset\subset \fcyc{E}$ & $123456$ & $1\ovl 23\ovl 456$, ... & 235 \\
$\emptyset\subset \fcyc{E}$ & $123456$ & $1\ovl 23\ovl 45\ovl 6$, ... & 236 \\
$\emptyset\subset 246 \subset \fcyc{E}$ & $135+246$ & $1\ovl 2345\ovl 6$, $123\ovl 456$, ... & 346 \\
$\emptyset\subset 356 \subset \fcyc{E}$ & $124+356$ & $1\ovl{234}5\ovl 6$, $1\ovl 23\ovl{45}6$, ... & 256 \\
$\emptyset\subset 23456\subset \fcyc{E}$ & $1+23456$ & $1\ovl234\ovl{56}$, $12\ovl{34}56$, ... & 345 \\
$\emptyset\subset 356\subset 23456\subset \fcyc{E}$ & $1+24+356$ & $12\ovl{34}5\ovl 6$, $1\ovl{23}45\ovl 6$,
$1\ovl 234\ovl 56$, $123\ovl{45}6$, ...& 456 \\
\hline
\end{tabular}
\caption{Table of the active bijection of $K_4$, where orientations are written with a bar over reoriented edges w.r.t. the reference orientation given in Figure \ref{fig:K4}, and where ``...'' means ``and opposites''. The cyclic flat of each active filtration  is boxed in the first column.}
\label{fig:tabK4ori}
\end{center}
\vspace{-0.5cm}
\end{figure}

\begin{figure}[]
	\centering	%\includegraphics[scale=0.5]{figures/big_active_bijection_K4_ordre_actif_fonce.eps}
\includegraphics[scale=1.1]{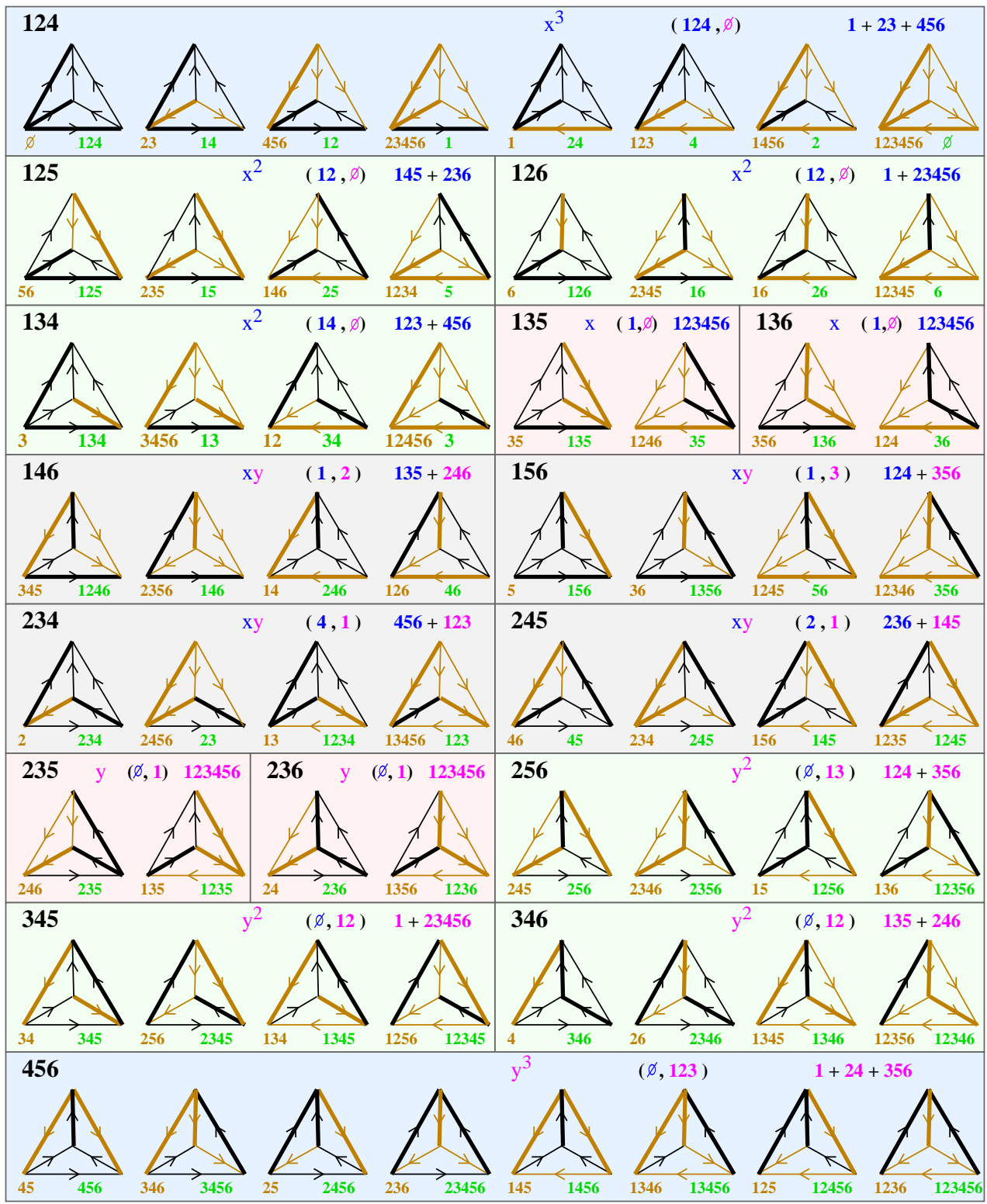}
%		\caption[]{The full construction for $K_4$ with linear ordering 		and with reference orientation for the big active bijection shown in Figure \ref{fig:K4}}
\vspace{-0.3cm}
\caption[]{The canonical and refined active bijections of $K_4$ w.r.t. ordering and reference orientation from Figure \ref{fig:K4}}
		\label{fig:big-active-bijection}
\end{figure}

\newpage
Figure \ref{fig:tabK4ori} sums up in a table the canonical active bijection of the underlying ordered graph (Theorem \ref{EG:th:alpha}). 
%Orientations are written as reorientations of the reference orientation of Figure \ref{fig:K4}, with a bar written above reoriented edges.
%
%A block has been detailed in figures of previous sections.
%
Notice that figures in previous sections are based on the same ordered graph.
Figure \ref{EG:fig:K4-dec} show the minors involved in the decomposition of 
the orientations associated with spanning tree $134$, and Figure \ref{EG:fig:K4-iso} show the bijection between this orientation activity class and the interval of this spanning tree.

The reader is advised to  see also \cite{AB2-a} and \cite{AB2-b}, were 
more illustrations are given for constructions on the same example, such as
several detailed examples of decompositions of orientations and of their active spanning trees by means of active partitions and suitably signed fundamental cycles/cocycles (represented in tableaux and in bipartite graphs), as well as a complete geometrical representation using two dual pseudoline arrangements.
%\red{table a mettre dans derniere section example ?}

\eme{ajouter subsets dans tableau ??? (comme dans figure) eventuellement comme elemtns soulignes}

%\new{geometri dans AB2b}

%\red{TTENTION NON CONFORME A DEF, notament pour act externe}

%\red{enlever variante ?}
\eme{desous en commaentaire, figure de variante enlevee}
%\begin{figure}[]
%	\centering	%\includegraphics[scale=0.5]{figures/big_active_bijection_K4_ordre_actif_fonce.eps}
%%\includegraphics[scale=0.65]{figures/big_active_bijection_K4_ordre_actif.eps}
%\includegraphics[scale=0.8]{figures/variante_K4_base124.eps}
%%		\caption[]{Example of another big active bijection for basis $124$ obtaied by symetric difference with $\{1,2\}\subseteq \Int(124)$ with respect to the big active bijectin of Figure \ref{fig:big-active-bijection}.}
%\vspace{-0.3cm}
%				\caption[]{Example of variant for the refined active bijection of Figure \ref{fig:big-active-bijection}.}				
%		\label{fig:variant}
%\end{figure}

Figure \ref{fig:big-active-bijection} illustrates completely the canonical and refined active bijection between orientations and spanning trees (Theorem \ref{EG:th:ext-act-bij}). 
Each block corresponds to an activity class of orientations,
and to its associated spanning tree by the canonical active bijection.
The following information is given:%

\begin{itemize}
\itemsep=0mm
\item In the upper left: the spanning tree $T$
\item As drawn digraphs: 
	\vspace{-2mm}
	\begin{itemize}
	\itemsep=0mm
	\item the $2^{\io+\ep}$ orientations $\G$ such that $T=\alpha(\G)$
	\item the first orientation of the block is the active-fixed and dual-active-fixed one
	\end{itemize}
\item In the upper right: 
	\vspace{-2mm}
	\begin{itemize}
	\itemsep=0mm
	\item the associated coefficient $x^\io y^\ep$ in the Tutte polynomial
	\item the pair of subsets $(\Int(T),\Ext(T))=(O^*(\G),O(\G))$ whose cardinalities are $(\io,\ep)$
	\item the active partition of both $T$ and $\G$ (reorienting the parts of the active partition provides the other graphs of a block from any of them, and, conversely, active partitions can be deduced from the set of reorientations associated with the same spanning tree by considering symmetric differences of these reorientations)
	\end{itemize}
\item Under each graph is illustrated the refined active bijection w.r.t. the digraph of Figure \ref{fig:K4}:
	\vspace{-2mm}
	\begin{itemize}
	\itemsep=0mm
	\item on the left, the corresponding reorientation with respect to the digraph of Figure \ref{fig:K4}.
	\item on the right, the corresponding edge-subset obtained by adding/removing active elements with respect to the digraph of Figure \ref{fig:K4} %(arbitrary choice)
	\end{itemize}
\item On each graph:
	\vspace{-2mm}
	\begin{itemize}
	\itemsep=0mm
	\item the bold edges form the spanning tree
	\item the grey edges are those that are reoriented
	\end{itemize}
\end{itemize} 

In particular, noticeable restrictions of the refined active bijection listed in Tables
%Theorem \ref{EG:th:ext-act-bij} 
\ref{table:intro} or \ref{table:thm-refined}
can be read on the figure:
the bijection between acyclic orientations and NBC subsets is given by the three first lines;
the bijection between strongly connected orientations and supersets of external trees is given by the three last lines;
the bijection between (dual-)active-fixed orientations and spanning trees is given by considering the first digraph of each block (in particular for acyclic orientations it is the only one of the block with the source of edge $1$ as unique source, yielding a bijection between such orientations and internal spanning trees); et caetera.
\eme{Anecdotically, let us mention that, since $K_4$ is a complete graph, the active bijection also consists in a bijection between permutations and increasing trees \cite{GiLV06}.}
\eme{Finally, Figure \ref{fig:variant} shows a variant (Remark \ref{rk:refined-bij}) of the refined active bijection for spanning tree $124$, obtained by symmetric difference with $\{1,2\}\subseteq \Int(124)$ of all edge-sets in this interval with respect to 
%the refined active bijection of
Figure \ref{fig:big-active-bijection}.
}

\eme{VARIANTE}
\eme{Concerning the big active bijection,
here the same reference orientation has been chosen for every interval.
The bijection between acyclic orientations and NBC subsets is given by the three first lines. The three last lines correspond dually to strongly connected orientations.
Figure \ref{fig:variant} shows an
example of another big active bijection for spanning tree $124$, obtained by symmetric difference with $\{1,2\}\subseteq \Int(124)$ of all edge 
-sets in this interval with respect to the big active bijection of Figure \ref{fig:big-active-bijection}.}
%\bigskip

%Also the reader may check the result of \cite{GiLV05} that each block of acyclic orientation contains exactly one orientations with unique source the source of the smallest edge, proving a biejction between internal spanning trees and such orientations.

\bs
%\newpage

\noindent{\bf Acknowledgments.}

\noindent Emeric Gioan 
wishes to thank a lot the following people:
two anonymous reviewers for numerous useful comments and presentation advices,
Joanna Ellis-Monaghan and Iain Moffatt for their 
%motivating 
unifying
%cohesive
development work on the Tutte polynomial that encouraged the completion of this 
%writing pending for a long time,
writing which had been pending for a long time,
%long-pending writing, %\break
 Bodo Lass for  stimulating discussions,
Cl\'ement Dupont for suggesting to use the %charming suitable 
nice term ``filtration'',
 Ilda da Silva for useful 
%comments and 
presentation advices,
%Lorenzo Traldi for 
%%stimulating discussions,
%%various 
%useful comments and english corrections, 
%%St?phane Bessy for his help on Remark xxx,
%and Spencer Backman for %\red{various ?non}
% useful comments and stimulating discussions.%
%% and to his wife for getting through a complicated life together!
and Lorenzo Traldi and Spencer Backman for 
some english corrections and numerous useful comments.%
%
%\bs

%\noindent{\bf Acknowledgments.}
%
%\noindent Emeric Gioan 
%wishes to thank a lot the following people:
%two anonymous reviewers for numerous useful comments and presentation advices,
%Joanna Ellis-Monaghan and Iain Moffatt for their 
%%motivating 
%unifying
%%cohesive
%development work on the Tutte polynomial that encouraged the completion of our long time pending writings, %\break
% Bodo Lass for  stimulating discussions,
%%
%Cl\'ement Dupont for suggesting to use the %charming suitable 
%nice term ``filtration'',
%%
% Ilda da Silva for useful 
%%comments and 
%presentation advices,
%%Lorenzo Traldi for 
%%%stimulating discussions,
%%%various 
%%useful comments and english corrections, 
%%%St?phane Bessy for his help on Remark xxx,
%%and Spencer Backman for %\red{various ?non}
%% useful comments and stimulating discussions.%
%%% and to his wife for getting through a complicated life together!
%Lorenzo Traldi  for 
%numerous english corrections and useful comments,
%and Spencer Backman for numerous useful comments.%
%%\bs

%\red{various useufl comments ? non !}

%\new{referee}

%\red{separer pti/gros ? + and all of them for their encouragements to achieve our slowly maturing publication process of this whole work.}
%%%%%%%%%%%%%%%%%%%%%%%%%%%%%%%%%%%%%%%%%%%%%%%%%

%\red{verfiier si refs utilisees}

%%%%%%%%%%%%%%%%%%%%%%%%%%%%%%%%%%%%%%%%%%%%%%%%%%%%%%%%%%%%
\addcontentsline{toc}{section}{References}

\addtocontents{toc}{\small Computing the fully optimal spanning tree of an ordered bipolar directed~graph \hfill Companion paper \cite{ABG2LP}}

\emevder{attention coherence des refs, notamment titres de AB3 et AB4}

%\bigskip
%\newpage
%%%%%%%%%%%%%%%%%%%%%%%%%%%%%%%%%%%%%%%%%%%%%%%%%%%%%%%%%%%%%

%\begin{bibliography}
%\noindent{\bf References.}

\vspace{-0.2cm}

\bibliographystyle{amsplain}

%\nocite{*}   %%% pour faire apparaitre toutes refs meme non citees

%\bibliography{AB-biblio}    %%% pour consutrire biblio a partir de base de donnees (attention : ordre non modifiable... donc j'ai modifie ensuite le .bbl)

\emevder{METTRE A JOUR BIBLIO !!!!!! faire fichier .bib}

\providecommand{\bysame}{\leavevmode\hbox to3em{\hrulefill}\thinspace}
\providecommand{\MR}{\relax\ifhmode\unskip\space\fi MR }
% \MRhref is called by the amsart/book/proc definition of \MR.
%\providecommand{\MRhref}[2]{%
%  \href{http://www.ams.org/mathscinet-getitem?mr=#1}{#2}
%}
\providecommand{\href}[2]{#2}

%\end{bibliography}
%%%%%%%%%%%%%%%%%%%%%%%%%%%%%%שש

\end{document}